%% file: FicheraLayer.tex
\documentclass[a4paper,reqno,oneside,12pt]{amsart}

\usepackage{amsmath}
\usepackage{amsfonts}
\usepackage{amsbsy}
\usepackage{amssymb}
\usepackage{mathrsfs}
\usepackage{exscale}
\usepackage{enumerate}

\usepackage{color}
\definecolor{orange}{rgb}   {0.75,   0.,   0.}
\definecolor{vert}{rgb}   {0.,  0.5,    0.25}

\usepackage{times}
\usepackage{hyperref}
\hypersetup{pdfborder={0000}, colorlinks=true, linkcolor=orange,citecolor=vert}

\usepackage{graphicx}
\input fig4tex.tex
\def\GrisClair{0.8}
\def\RadBullet{0.05}

\numberwithin{equation}{section}

\newtheorem{theorem}{Theorem}[section]
\newtheorem{lemma}[theorem]{Lemma}
\newtheorem{proposition}[theorem]{Proposition}
\newtheorem{corollary}[theorem]{Corollary}

\theoremstyle{definition}

\newtheorem{notation}[theorem]{Notation}

\theoremstyle{remark}
\newtheorem{remark}[theorem]{Remark}

\newcommand{\bx}{\boldsymbol{x}}

\newcommand{\cC}{\mathcal{C}}
\newcommand{\cD}{\mathcal{D}}
\newcommand{\cH}{\mathcal{H}}
\newcommand{\cI}{\mathcal{I}}
\newcommand{\cJ}{\mathcal{J}}
\newcommand{\cK}{\mathcal{K}}
\newcommand{\cL}{\mathcal{L}}
\newcommand{\cM}{\mathcal{M}}
\newcommand{\cN}{\mathcal{N}}
\newcommand{\cO}{\mathcal{O}}
\newcommand{\cP}{\mathcal{P}}
\newcommand{\cQ}{\mathcal{Q}}
\newcommand{\cS}{\mathcal{S}}
\newcommand{\cT}{\mathcal{T}}
\newcommand{\cV}{\mathcal{V}}
\newcommand{\cW}{\mathcal{W}}

\newcommand{\gC}{\mathfrak{C}}
\newcommand{\gF}{\mathfrak{F}}
\newcommand{\gG}{\mathfrak{G}}
\newcommand{\gL}{\mathfrak{L}}

\newcommand{\rd}{\mathrm{d}}

\newcommand{\sD}{\mathsf{D}}
\newcommand{\sN}{\mathsf{N}}

\newcommand{\sC}{\mathscr{C}}
\newcommand{\sX}{\mathscr{X}}
\newcommand{\sY}{\mathscr{Y}}
\newcommand{\sP}{\mathscr{P}}

\newcommand{\sG}{\mathscr{G}}
\newcommand{\sL}{\mathscr{L}}
\newcommand{\sS}{\mathscr{S}}

\newcommand{\sGf}{\Gamma^\flat}
\newcommand{\sLf}{\Lambda^\flat}
\newcommand{\sSf}{\mathscr{S}^\flat}

\newcommand{\sGs}{\Gamma^\sharp}
\newcommand{\sLs}{\Lambda^\sharp}
\newcommand{\sSs}{\mathscr{S}^\sharp}

\newcommand{\uQ}{\,\underline{\!\cQ\!}\,}

\newcommand{\Dir}{\mathsf{Dir}}
\newcommand{\Mix}{\mathsf{Mix}}
\newcommand{\Bro}{\mathsf{Bro}}
\newcommand{\tens}{\mathsf{tens}}
\newcommand{\red}{\mathsf{1D}}
\newcommand{\DN}{\mathsf{DN}}

\newcommand{\dis}{\mathsf{dis}}
\newcommand{\ess}{\mathsf{ess}}
\newcommand{\Dom}{\mathsf{Dom}}

\newcommand{\R}{\mathbb{R}}

\renewcommand{\proofname}{{\bf Proof}}

\DeclareMathOperator{\dist}{dist}
\DeclareMathOperator{\supp}{supp}
\DeclareMathOperator{\dom}{dom}

\title{Dirichlet spectrum of the Fichera layer}

\author{Monique Dauge} 

\address{Univ. Rennes, CNRS, IRMAR - UMR 6625, F-35000 Rennes, France}

\email{monique.dauge@univ-rennes1.fr}
\urladdr{https://perso.univ-rennes1.fr/monique.dauge/}

\author{Yvon Lafranche} 

\address{Univ. Rennes, CNRS, IRMAR - UMR 6625, F-35000 Rennes, France}

\email{yvon.lafranche@univ-rennes1.fr
}
\urladdr{https://perso.univ-rennes1.fr/yvon.lafranche/}

\author{Thomas Ourmi\`eres-Bonafos}
\address{\rm Laboratoire de Math\'ematiques d'Orsay, Univ.~Paris-Sud, CNRS, Universit\'e Paris-Saclay, 91405 Orsay, France}

\email{thomas.ourmieres-bonafos@math.u-psud.fr}
\urladdr{http://www.math.u-psud.fr/~ourmieres-bonafos/}

\begin{document}



\begin{abstract}
We investigate the spectrum of the three-dimensional Dirichlet Laplacian in a prototypal infinite polyhedral layer, that is formed by three perpendicular quarter-plane walls of constant width joining each other.  Alternatively,  this domain can  be viewed as an octant from which another ``parallel'' octant is removed. It contains six edges (three convex and three non-convex) and two corners (one convex and one non-convex).
It is a canonical example of non-smooth conical layer. We name it after Fichera because near its non-convex corner, it coincides with the famous Fichera cube that illustrates the interaction between edge and corner singularities. This domain could also be called an octant layer.

We show that the essential spectrum of the Laplacian on such a domain is a half-line and we characterize its minimum as the first eigenvalue of the two-dimensional Laplacian on a broken guide. By a Born-Oppenheimer type strategy, we also prove that its discrete spectrum is finite and that a lower bound is given by the ground state of a special Sturm-Liouville operator. By finite element computations taking singularities into account, we exhibit exactly one eigenvalue under the essential spectrum threshold leaving a relative gap of 3\%. We extend these results to a variant of the Fichera layer with rounded edges (for which we find a very small relative gap of 0.5\%), and to a three-dimensional cross where the three walls are full thickened planes.
\end{abstract}                                                            

\maketitle

\section{Introduction and main results}
\subsection{Motivations} 
In the usual three-dimensional Euclidean space, the term {\it layer} commonly denotes a tubular neighborhood of a reference surface. Such geometries are physically relevant because for example, in mesoscopic physics, properties of thin-films (also known as quantum layers) can be deduced by a spectral study of the Laplacian in hard-wall layer domains.

From a mathematical point of view, such operators have been considered in \cite{DEK01,CEK04} and the main results can be roughly summed up as follows: under adequate geometric conditions on the reference surface the essential spectrum is a half-line of the form $[a,+\infty)$ ($a\in\mathbb{R}_+$) whereas the existence of bound states depends on curvature properties (such a behavior is reminiscent of the study of quantum waveguides, see for instance the pioneering works \cite{GJ92,DE95}).

In fact, as mentioned in \cite{DEK01,CEK04} and specifically studied in \cite{ET10}, reference surfaces being circular conical layers exhibit an interesting behavior: they have infinitely many bound states and thus, they accumulate to the threshold of the essential spectrum. Moreover, we know that the accumulation rate is logarithmic as shown in \cite{DOR15}.

Recently, in \cite{OBP16}, it has been proved that the infiniteness of bound states and the logarithmic accumulation to the threshold of the essential spectrum still hold for conical layers constructed around any smooth reference conical surface (by smooth reference conical surface, we mean that the surface is smooth except in its vertex).

Then, a natural question is to know whether this structure of the spectrum is preserved if this smoothness hypothesis is violated. This is the very question we tackle in this paper for a reference surface that is a specific polyhedral cone and thus has edges. We show that the structure of the spectrum drastically differs from the smooth case: the threshold of the essential spectrum is lower than for smooth conical surfaces and there is only a finite number of bound states.

Before going any further, let us mention similar works about various realizations of the Laplacian interplaying with conical geometries. Schr\"odinger operators with singular potential, modeled by $\delta$-interactions supported on smooth cones, are investigated in \cite{BEL14,LOB16} whereas the case of the Laplacian with Robin boundary conditions in smooth conical domains is dealt with in \cite{BPP17,P16}. Finally, for problems related with polyhedral geometries, let us mention \cite{BDP14} where the magnetic Laplacian in three-dimensional corner domains is studied as well as \cite{BP16} where the bottom of the essential spectrum of the Robin Laplacian is characterized for polyhedral cones.

\subsection{Main results}\label{sec:mainres}
As a prototype for infinite polyhedral layers including edges, 
we investigate the layer $\Lambda$ obtained by removing the first octant $(\R_+)^3$ from the translated octant $(\R_+)^3 - (1,1,1)$ (see Figure \ref{fig:fig0}, right), namely:
\begin{equation}
\label{eq:Fichera}
   \Lambda = \big\{ (\R_+)^3 - (1,1,1) \big\} \setminus (\overline\R_+)^3.
\end{equation}
(Here $\R_+$ denotes $(0,\infty)$). This domain can be called an {\em octant layer}, but we choose to name it after Fichera since its non-convex polyhedral corner sitting at the origin $(0,0,0)$ is a celebrated example of interaction between edge and corner singularities: this interaction is described in \cite[\S\,17]{DaBook1988} and related numerical issues are addressed in \cite{ApelSW1996,ApelMehrmannWatkins02,CoDa2003} for instance.
We are interested in the positive Dirichlet Laplacian $\cL_{\Lambda}$ posed on $\Lambda$. We will show that its spectral properties heavily depend on its two-dimensional analogue posed on the broken guide $\Gamma$ of width $1$ and angle $\frac\pi2$ (see Figure \ref{fig:fig0}, left).

\begin{figure}[h]
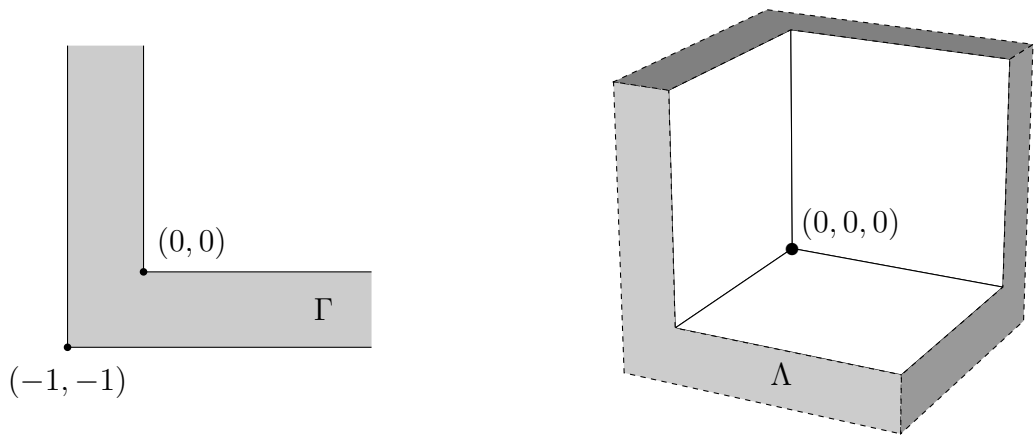

\begin{minipage}[l]{.49\linewidth}
\input fig1-2d.tex
\end{minipage}
\begin{minipage}[l]{.49\linewidth}
\input fig1-3d.tex
\end{minipage}
\caption{On the left: the two-dimensional {\em broken guide} $\Gamma$ with angle $\frac\pi2$. On the right: a representation of the three-dimensional {\em Fichera layer} $\Lambda$.}
\label{fig:fig0}
\end{figure}

\begin{theorem}[\cite{DaLaRa11,NazSha14}]
\label{th:GR}
With the broken guide $\Gamma=\big\{ (\R_+)^2 - (1,1) \big\} \setminus (\overline\R_+)^2$ and $\cL_\Gamma$ the positive Dirichlet Laplacian on $\Gamma$, there holds:
	\begin{enumerate}[i)]
		\item The essential spectrum of $\cL_\Gamma$ coincides with $[\pi^2,+\infty)$;
		\item The operator $\cL_\Gamma$ has exactly one eigenvalue below its essential spectrum.
	\end{enumerate}
\end{theorem}
It is proved in \cite{DaLaRa11} that $\cL_\Gamma$ has a finite number of eigenvalues below its essential spectrum, and in \cite{NazSha14} that this number is $1$, as expected after semi-analytical calculations \cite{ESS89}, or finite element approximation \cite{DaLaRa11}. We are going to revisit this result later on, see Remark \ref{rem:one}.
Let $\lambda_1(\Gamma)$ be the first eigenvalue of $\cL_\Gamma$. An approximate numerical value given in \cite{DaLaRa11} is $\lambda_1(\Gamma)\simeq0.929\,\pi^2$. Our main result concerning the Fichera layer $\Lambda$ is
\begin{theorem} 
With the Fichera layer $\Lambda = \big\{ (\R_+)^3 - (1,1,1) \big\} \setminus (\overline\R_+)^3$ and $\cL_\Lambda$ the positive Dirichlet Laplacian on $\Lambda$, there holds:
\begin{enumerate}[i)]
\item\label{thm:main:i} The essential spectrum of $\cL_\Lambda$ coincides with $[\lambda_1(\Gamma),+\infty)$;
\item\label{thm:main:ii} $\cL_\Lambda$ has at most a finite number of eigenvalues below its essential spectrum.
\end{enumerate}
\label{thm:main}
\end{theorem}
Theorem \ref{thm:main} exhibits a significant difference between non-smooth and smooth conical layers: First, the bottom of the essential spectrum in the non-smooth case is determined by the edge profile (the broken guide) and is lower than in the smooth case where it is given by the first eigenvalue $\pi^2$ of the one-dimensional Laplacian on the interval $\cI=(-1,0)$ across the layer. Second, in contrast with the present statement, for smooth conical layers the discrete spectrum is infinite as it was first observed for circular cones \cite[Thm.~3.1]{ET10}, \cite[Thm.~1.4]{DOR15}, and next generalized to smooth conical layers \cite[Thm.~2]{OBP16}.

Before going further, let us comment on relations between smooth layers and our Fichera layer. A smooth layer is defined like a shell in elasticity: from a reference unbounded smooth surface $\sSf$ without boundary (said \emph{midsurface}) and a positive thickness parameter $\varepsilon$, we may define the layer $\sLf[\varepsilon]$ as the set of points at distance strictly smaller than $\varepsilon/2$ to $\sSf$. This same definition is also adopted in \cite{OBP16} when $\sSf$ is a smooth conical surface. Trying this for our Fichera layer, we choose $\sSf$ as the union of three quarter planes
\[
   \sSf = \{\bx\in\mathbb{R}^3 :\ \  \min\{x_1,x_2,x_3\} = -\tfrac12\}
\]
and take $\varepsilon=1$. However the layer $\sLf:=\sLf[1]$ is distinct from $\Lambda$ outside any compact set: We can see that the section of $\Lambda$ by any plane $x_1=R$, $R>0$, is isometric to the broken guide $\Gamma$, whereas a similar section of $\sLf$ is isometric to the broken guide $\sGf$ with rounded exterior corner drawn in Figure \ref{fig:round}, left.

\begin{figure}[h]
\begin{minipage}[l]{.49\linewidth}
\input fig2-round.tex
\end{minipage}
\begin{minipage}[l]{.49\linewidth}
\input fig2-roundsharp.tex
\end{minipage}
\caption{Plane sections $\sGf$ and $\sGs$ of the variants $\sLf$ and $\sLs$ of the Fichera layer.}
\label{fig:round}
\end{figure}

An alternative in the same spirit would be to set $\sS^0=\{\bx\in\mathbb{R}^3 :\ \  \min\{x_1,x_2,x_3\} = 0\}$ and define the layer $\sLs$ as the set of points at signed distance less than $1$ from $\sS^0$ where the sign is given by the outward normal to the octant $(\R_+)^3$. We observe that the section of $\sLs$ by any plane $x_1=R$, $R>0$, is isometric to the broken guide $\sGs$ with rounded exterior corner drawn in Figure \ref{fig:round}, right. Note that $\sLs$ can also be viewed as a layer of thickness $\varepsilon=1$ associated with the midsurface $\sSs$ drawn in the same figure. 

In fact, the Dirichlet Laplacian on $\sGf$ or $\sGs$ has exactly one eigenvalue under the threshold of the essential spectrum \cite{GJ92,Pa17} and Theorem \ref{thm:main} generalizes to the layers $\sLf$ and $\sLs$ if we replace the guide $\Gamma$ by $\sGf$ and $\sGs$, respectively, see Section \ref{sec:ext} and Appendix \ref{sec:app2}.
 
Another interesting related geometry is given by the cross-like domains. Let $\sX$ and $\sY$ be the two- and three-dimensional ``crosses'', {\em cf.}\ Figure \ref{fig:X} (here $\cI$ is the bounded interval $(-1,0)$)
\begin{equation}\label{eqn:defXY}
   \sX = (\R\times\cI) \cup (\cI\times\R) \quad\mbox{and}\quad
   \sY = (\R^2\times\cI) \cup (\R\times\cI\times\R) \cup (\cI\times\R^2).
\end{equation}
\begin{figure}[h]
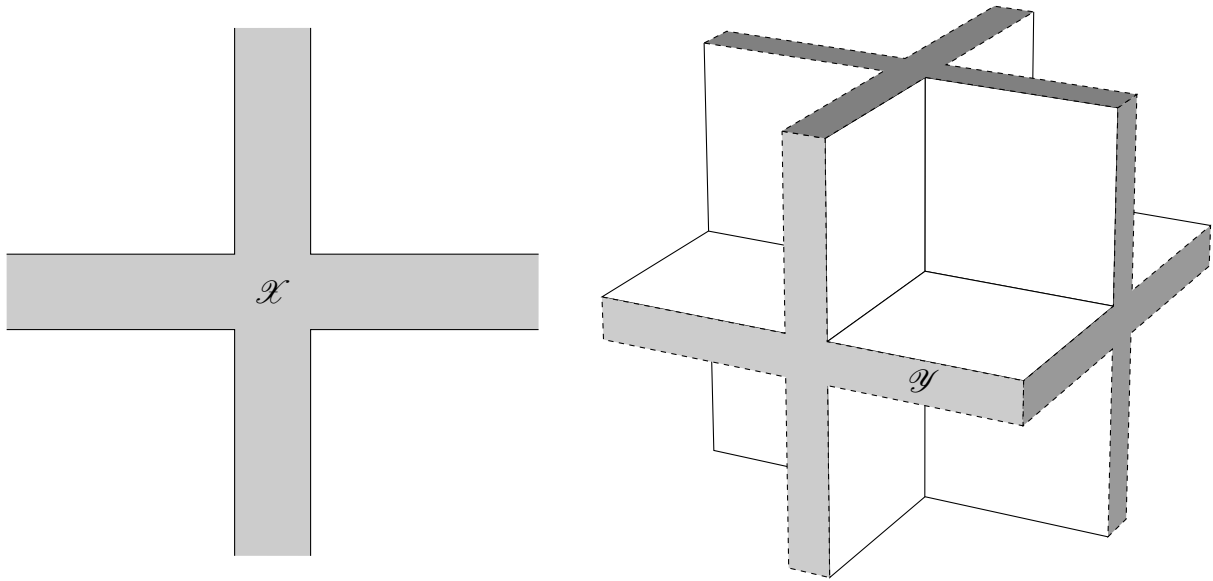

\begin{minipage}[l]{.49\linewidth}
\input fig3-cross.tex
\end{minipage}
\begin{minipage}[l]{.49\linewidth}
\input fig3-murs.tex
\end{minipage}
\caption{The 2- and 3-dimensional crosses.}
\label{fig:X}
\end{figure}%
Here, once more, the Dirichlet Laplacian on $\sX$ has exactly one eigenvalue under the essential spectrum \cite[Prop.\,13]{Pa17} and Theorem \ref{thm:main} generalizes to the three-dimensional cross $\sY$ if we replace the guide $\Gamma$ by the two-dimensional cross $\sX$, see Section \ref{sec:ext}.

\subsection{Notations}
Let $d\in\{1,2,3\}$ and let $\bx=(x_1,\cdots,x_d)$ denote the Cartesian coordinates of $\R^d$. For $L>0$, we will make use of $\Box_L$, the box domains of $\mathbb{R}^d$ defined as
\begin{equation}
   \Box_L := \{\bx\in\mathbb{R}^d :\ \ \max_{j=1}^d |x_j+1| < L\}.
\label{eqn:box}
\end{equation}

\subsubsection{Three-dimensional domains}
In $\R^3$, the Fichera layer \eqref{eq:Fichera} is the unbounded layer domain $\Lambda$ that can alternatively be written as
\begin{equation}
\label{eq:L}
	\Lambda = \{\bx\in\mathbb{R}^3 :\ \  -1<\min\{x_1,x_2,x_3\} < 0\}.
\end{equation}
Bounded versions of $\Lambda$ are obtained as the intersection of $\Lambda$ with the three-dimensional boxes $\Box_L$ centered at the external vertex $(-1,-1,-1)$ of $\Lambda$. Set for $R>-1$:
\begin{equation}
\label{eq:LR}
   \Lambda_R = \Lambda\cap\Box_{R+1}\,.
\end{equation}
The residual parts ``at infinity'' of $\Lambda$ are denoted by $\Omega_R$
\begin{equation}
\label{eq:OR}
   \Omega_R = \Lambda \cap (\R^3\setminus\overline\Box_{R+1}) = \Lambda\setminus\overline\Lambda_R \,.
\end{equation}
The layer $\Lambda$ is invariant by the symmetries with respect to the three diagonal planes $x_j=x_k$ ($j<k$). It is natural to split $\Lambda$ into the three isometric parts $\Lambda^1$, $\Lambda^2$ and $\Lambda^3$, with $\Lambda^3$ defined as
\begin{equation}
\label{eq:L3}
   \Lambda^3 = \{\bx\in\Lambda :\ \  x_1<x_3 \ \ \mbox{and}\ \ x_2<x_3 \}
\end{equation}
and the other two by permutation of indices. We also introduce the finite and residual parts
\begin{equation}
\label{eq:LjR}
   \Lambda^j_R = \Lambda^j\cap\Lambda_R \quad\mbox{and}\quad \Omega^j_R = \Lambda^j\cap\Omega_R\,.
\end{equation}

\subsubsection{Two-dimensional domains} The two-dimensional domain corresponding to the Fichera layer is the broken guide $\Gamma$ of opening $\frac\pi2$ and width $1$
\begin{equation}
\label{eq:G}
	\Gamma = \{\bx\in\mathbb{R}^2 :\ \  -1<\min\{x_1,x_2\} < 0\}
\end{equation}
and the finite broken guide $\Gamma_R$ is defined accordingly
\begin{equation}
\label{eq:GR}
   \Gamma_R = \Gamma\cap\Box_{R+1} \quad\mbox{for}\quad R>-1\,,
\end{equation}
where $\Box_{R+1}$ is the two-dimensional box as defined in \eqref{eqn:box}.

These domains are symmetric with respect to the diagonal line $x_1=x_2$, which makes natural the definition of the subdomains $\Gamma^1 = \{\bx\in\Gamma :\ x_2<x_1\}$  and  $\Gamma^2 = \{\bx\in\Gamma :\  x_1<x_2\}$
together with their bounded analogues
$\Gamma^j_R = \Gamma^j\cap\Box_{R+1}$ for $j=1,2$.
We particularize the parts of their boundaries at ``distance'' $R$, i.e.
$\Sigma^j_R = \partial \Gamma^j_R \cap \partial \Box_{R+1}$.
Note that
\begin{equation}\label{eqn:dfnsigmaj}
   \Sigma^1_R = \{R\}\times\cI \quad\mbox{and}\quad \Sigma^2_R = \cI\times\{R\}
   \quad\mbox{with}\quad \cI=(-1,0)\,.
\end{equation}
Finally
\begin{equation}
   \Sigma_R = \Sigma^1_R\cup\Sigma^2_R = \partial \Gamma_R \cap \partial \Box_{R+1}\,.
	\label{eqn:defsigmaR}
\end{equation}
We will also make use of the rectangles $\cT^j_R$ and half-strips $\cS^j_R$
\[
   \cT^1_R = (0,R)\times\cI,\ \ \ \cT^2_R = \cI\times(0,R) \quad\mbox{and}\quad
   \cS^1_R = (R,\infty)\times\cI,\ \ \ \cS^2_R = \cI\times(R,\infty).
\]
These domains are represented in Figure \ref{fig:subdom_guide} and \ref{fig:fin_gui}.
\begin{figure}[h]
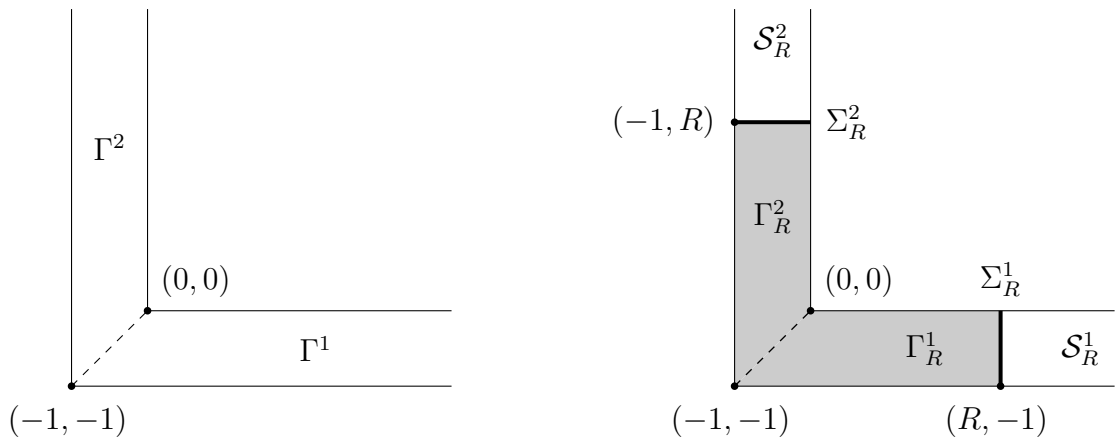

\begin{minipage}[l]{.49\linewidth}
\input fig4-1.tex
\end{minipage}
\begin{minipage}[l]{.49\linewidth}
\input fig4-2.tex
\end{minipage}
\caption{The guide $\Gamma$ and its associated subdomains.}
\label{fig:subdom_guide}
\end{figure}

In this paper, we will use in two occurrences the following simple uniform trace estimate:
\begin{lemma}
\label{lem:trace}
There exists a constant $C$ such that for all $R\ge1$ and all $u\in H^1(\Gamma_R)$ there holds
\begin{equation}
\label{eq:trace}
   \|u\|_{L^2(\Sigma_R)} \le C \|u\|_{H^1(\Gamma_R)}\,.
\end{equation}
\end{lemma}
\begin{proof}
We bound the $L^2$-norm of the trace of $u$ on $\Sigma_R$ by the $H^1$-norm on $\Gamma_R\setminus\overline\Gamma_{R-1}$.
\end{proof}

\subsubsection{Operators}
\label{subsub:operators}
The Laplace operator $\Delta$ in $\R^d$ is the partial differential operator
\[
	\Delta = \sum_{j=1}^d \partial_j^2.
\]
On a generic domain $\cO\subset\R^d$, the positive Laplacian $-\Delta$ is associated with the quadratic form
\[
   \cQ(u) = \int_{\cO} |\nabla u(\bx)|^2\,\rd\bx,
\]
defined for $u$ belonging to the Sobolev space $H^1(\cO)$. The boundary conditions considered in this paper are either Dirichlet or Neumann on parts of the boundary of $\cO$. Let us assume that $\cO$ is a Lipschitz domain and that $\partial_{\sD}\cO$ and $\partial_{\sN}\cO$ are two Lipschitz subdomains of the boundary $\partial\cO$ such that $\overline{\partial_{\sD}\cO}\cup\overline{\partial_{\sN}\cO} = \partial\cO$. Then we can introduce the variational space $\Dom(\cQ)$ (form domain for $\cQ$) with Dirichlet conditions on $\partial_{\sD}\cO$:
\[
   \Dom(\cQ) = \big\{u\in H^1(\cO):\ \ u\big|_{\partial_{\sD}\cO}=0\big\}.
\]
The associated self-adjoint operator is $\cL:=-\Delta$ with domain
\[
   \Dom(\cL) = \big\{u\in\Dom(\cQ):\ \ -\Delta u\in L^2(\Omega),\ \ \partial_n u\big|_{\partial_{\sN}\cO}=0\big\}.
\]
The spectrum of $\cL$ is denoted by $\sigma(\cL)$, its discrete and essential parts by $\sigma_{\dis}(\cL)$ and $\sigma_{\ess}(\cL)$, respectively. Likewise, we might also write $\sigma(\cQ)$, $\sigma_{\ess}(\cQ)$ and $\sigma_{\dis}(\cQ)$, respectively.
We particularize the notation of $\cL$ (respectively $\cQ$) on the domain $\cO$ as $\cL^\Dir_\cO$ (respectively $\cQ^\Dir_\cO$) if $\partial_\sN\cO$ has measure zero and $\cL_\cO^\Mix$ (respectively $\cQ_\cO^\Mix$) otherwise.

The $\ell$-th Rayleigh quotient of $\cQ_\cO^\Dir$ (respectively $\cQ_\cO^\Mix$) on its form domain 
is denoted by $\lambda_\ell^\Dir(\cO)$ (respectively $\lambda_\ell^\Mix(\cO)$).
Except possibly in Section \ref{sec:numerics}, there is no risk of confusion, thus, for the sake of readability, we omit the mention of the superscripts $\Dir$ and $\Mix$ using $\cL_\cO$, $\cQ_\cO$ and $\lambda_\ell(\cO)$ instead.

In particular $\lambda_1(\cO)$ is the bottom of the spectrum of $\cL_\cO$:
\[
   \lambda_1(\cO) = \min_{u\in\Dom(\cQ_\cO),\ u\neq0} \frac{\cQ_\cO(u)}{\|u\|^2}\,.
\]

In Table \ref{tab:recap}, we list the main geometrical domains we are interested in as well as the associated Laplace operators, their Rayleigh quotients and the main related results.

\renewcommand{\arraystretch}{1.4}

\begin{table}
\begin{tabular}{|c||c|c|c|}
\hline
Domain $\cO$ & $\Gamma$, \ {\em cf.}\ \eqref{eq:G} 
             & $\Gamma_R$, \ {\em cf.}\ \eqref{eq:GR} 
             & $\Lambda$, \ {\em cf.}\ \eqref{eq:L}\\
\hline
\hline
Neumann part $\partial_\sN\cO$ & $\partial_\sN\Gamma = \emptyset$ 
                               & $\partial_\sN \Gamma_R = \Sigma_R$, \ \em{cf.}\ \eqref{eqn:defsigmaR} 
                               & $\partial_\sN \Lambda = \emptyset$\\
\hline
Laplace operator & $\cL_\Gamma$& $\cL_{\Gamma_R}$ & $\cL_\Lambda$\\
\hline
Rayleigh quotients & $\lambda_\ell(\Gamma)$ & $\lambda_\ell(\Gamma_R)$ & $\lambda_\ell(\Lambda)$ \\
\hline
Condensed notations & $\lambda_1(\Gamma) := \lambda_\infty$& $\lambda_1(\Gamma_R) := \lambda_R$ & \textendash\\
\textendash & \textendash & for $R=x_3$\,: \ $\lambda_1(\Gamma_{x_3}) := \lambda(x_3)$ & \textendash\\
\hline
Main results & Th.~\ref{th:GR} & Cor.~\ref{cor:conv} \ and \ Prop.~\ref{prop:Dau-Hel} & Th.~\ref{thm:main} \ and \  Prop.~\ref{prop:L*}\\
\hline
\end{tabular}
\vglue 1.5ex
\caption{Notations for the main Laplace operators studied in this paper.}
\label{tab:recap}
\end{table}

\subsubsection{Domain partition}
\label{sss:partition}
We will often use a comparison principle of eigenvalues based on a domain partition. Let us introduce a finite partition $(\cO_j)_{j\in\cJ}$ of the Lipschitz domain $\cO$ in the sense that each $\cO_j$ is Lipschitz, that they are pairwise disjoint, and 
\[
   \cup_{j\in\cJ} \overline\cO_j = \overline\cO.
\]
Then we define the broken quadratic form $\cQ_{\cO}^\Bro$
\[
   \cQ_{\cO}^\Bro(u) = \sum_{j\in\cJ} \int_{\cO_j} |\nabla u(\bx)|^2\,\rd\bx,
\]
defined for $u$ in the domain
\[
   \Dom(\cQ_{\cO}^\Bro) = \big\{u\in L^2(\cO):\ \ 
   u\big|_{\cO_j} \in \Dom(\cQ_{\cO_j}),\ \ j\in\cJ\big\}\,,
\]
where $\Dom(\cQ_{\cO_j}) = \{u\in H^1(\cO_j),\ u\big|_{\partial_{\sD}\cO\cap\partial\cO_j}\!=0\}$. Let $\lambda_\ell^\Bro(\mathcal{O})$ denote the $\ell$-th Rayleigh quotient of $\cQ_\cO^\Bro$.
Since we have the obvious embedding between form domains $\Dom(\cQ_{\cO})\subset\Dom(\cQ_{\cO}^\Bro)$, the following inequalities between Rayleigh quotients hold
\begin{subequations}
\label{eq:part}
\begin{equation}
\label{eq:part1}
   \lambda_\ell^\Bro(\cO) \le \lambda_\ell(\cO),\quad \forall \ell\ge1\,,
\end{equation}
while
\begin{equation}
\label{eq:part2}
   \lambda_\ell^\Bro(\cO) \ \ \mbox{is the $\ell$-th smallest term in the set} \ \ 
   \bigcup_{j\in\cJ} \bigcup_{k=1}^\ell \{\lambda_k(\cO_j)\}\, 
\end{equation}
\end{subequations}
with multiplicities taken into account.

\subsection{Structure of the paper}
In Section \ref{sec:brok_guides} we study the auxiliary question of the two-dimensional broken guides of finite length $\Gamma_R$. We collect results regarding the first eigenvalue of the Dirichlet Laplacian in such domains, in particular its exponential convergence to the first eigenvalue of the infinite broken guide as the length $R$ goes to infinity (see Corollary \ref{cor:conv}) and a Dauge-Helffer type formula about its derivative with respect to the length of the finite guide (see Proposition \ref{prop:Dau-Hel}).

Section \ref{sec:spectrum} is devoted to the proof of Theorem \ref{thm:main}. First, we investigate the structure of the essential spectrum using the form decomposition method as well as constructing adapted Weyl sequences for the operator. Second, we prove finiteness of the number of bound states by a Born-Oppenheimer strategy: we compare the number of eigenvalues below the threshold of the essential spectrum to that of a one-dimensional operator obtained after projection on the lowest eigenfunction of a transverse operator. Third, we conclude this section by giving a lower bound on the spectrum that turns out to be numerically close to the threshold of the essential spectrum: there is very little room left for bound states to exist (see Proposition \ref{prop:L*} and Figure \ref{fig:lammu}).

In Section \ref{sec:numerics}, we illustrate some of our results thanks to computations performed with the finite element library XLiFE++ \cite{Xlifepp} and we exhibit the existence of exactly one isolated eigenvalue for the Dirichlet problem on the Fichera layer. We address in Section \ref{sec:ext} the extensions mentioned above about the Fichera layer with exterior rounded edges $\sLs$ and the cross $\sY$. We draw finally some conclusions about other possible generalizations of our results in Section \ref{sec:conc}, developing the discussion on the distinct spectral behaviors in the family of smooth conical layers and in a family of generalized Fichera layers. The two appendices \ref{sec:app1} and \ref{sec:app2} close the paper.

\section{Broken guides of finite length}\label{sec:brok_guides}
In this section, we address the finite plane broken guides $\Gamma_R$ \eqref{eq:GR} for $R\ge0$: we investigate the first eigenpair of the Laplacian $\cL_{\Gamma_R}$ with Dirichlet conditions on $\partial\Gamma_R\cap\partial\Gamma$ and Neumann conditions on the remaining part $\Sigma_R$ of the boundary of $\partial\Gamma_R$. We start with rough estimates on the first two eigenvalues of $\cL_{\Gamma_R}$.

\begin{lemma} With the notations summarized in Table \textup{\ref{tab:recap}}, there holds:
\begin{enumerate}[i)]
	\item\label{lem:GR12:i} For all $R\ge0$, $\lambda_1({\Gamma_R})\le\lambda_1({\Gamma})$.
	\item\label{lem:GR12:ii} For all $R\ge0$, $\lambda_2({\Gamma_R})\ge\pi^2$.
\end{enumerate}
\label{lem:GR12}
\end{lemma}

\begin{proof}
We use the domain partition as stated in \S\ref{sss:partition}.

{\em\ref{lem:GR12:i})} For a chosen $R\ge0$, we split the broken guide $\Gamma$ into three pieces: The finite guide $\Gamma_R$, and the two half-strips $\cS^1_R=(R,\infty)\times\cI$ and $\cS^2_R=\cI\times(R,\infty)$. On the half-strips $\cS^1_R$ and $\cS^2_R$, Dirichlet conditions are applied on the unbounded sides. By \eqref{eq:part1}-\eqref{eq:part2} 
we find
\[
   \lambda_1(\Gamma) \ge 
   \min\big\{\lambda_1({\Gamma_R}),\ \lambda_1({\cS^1_R}),\ \lambda_1({\cS^2_R})\big\}\,.
\] 
Since $\lambda_1({\cS^1_R})=\lambda_1({\cS^2_R})=\pi^2$, we deduce point i) thanks to Theorem \ref{th:GR}.

{\em\ref{lem:GR12:ii})} Now we split the finite guide $\Gamma_R$ into the square $\Gamma_0$ and the two finite rectangles $\cT^1_R=(0,R)\times\cI$ and $\cT^2_R=\cI\times(0,R)$ (see Figure \ref{fig:fin_gui}). On the rectangles $\cT^1_R$ and $\cT^2_R$, Dirichlet conditions are applied on the sides $(0,R)\times\{-1,0\}$ and $\{-1,0\}\times(0,R)$, respectively. By \eqref{eq:part1}-\eqref{eq:part2} we find
\begin{equation}
\label{eq:2dset}
   \lambda_2({\Gamma_R}) \ge \ \ \mbox{2d smallest element of}\ \ 
   \big\{\lambda_1({\Gamma_0}),\ \lambda_2({\Gamma_0}),\ 
   \lambda_1({\cT^j_R}),\ \lambda_2({\cT^j_R}),\ j=1,2\big\}\,.
\end{equation}
\begin{figure}[h]
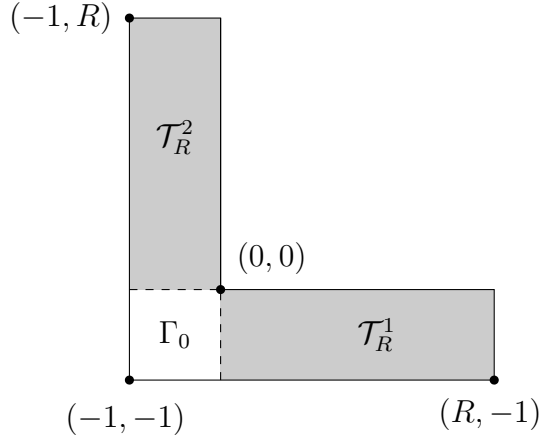

\input fig5.tex
\caption{The finite guide $\Gamma_R$, the square $\Gamma_0$ and the rectangles $\cT_R^j$, $j=1,2$.}
\label{fig:fin_gui}
\end{figure}

Recall that the boundary conditions on the square $\Gamma_0$ are Dirichlet on the sides $x_1=-1$, $x_2=-1$, and Neumann on the sides $x_1=0$, $x_2=0$. Therefore we have
\[
   \lambda_1({\Gamma_0}) = \Big(\frac14+\frac14\Big)\pi^2 = \frac{\pi^2}{2} \quad\mbox{and}\quad
   \lambda_2({\Gamma_0}) = \Big(\frac14+\frac94\Big)\pi^2 =\frac{5\pi^2}{2}.
\] 
Moreover, we get
\[
   \lambda_1({\cT^1_R})=\lambda_1({\cT^2_R})=\pi^2 \quad\mbox{and}\quad
   \lambda_2({\cT^1_R})=\lambda_2({\cT^2_R}) >\pi^2.
\]
Thus the second smallest element of the set in \eqref{eq:2dset} is $\pi^2$ and point {\em\ref{lem:GR12:ii})} of the lemma is proved.
\end{proof}

\begin{remark}
\label{rem:one}
A similar proof as that of point {\em\ref{lem:GR12:ii})} above yields that the second Rayleigh quotient of $\cL_\Gamma$ is greater than $\pi^2$. This is in the spirit of \cite{Pa17} and provides a more direct proof of the result \cite{NazSha14} that $\cL_\Gamma$ has at most one eigenvalue under its essential spectrum.
\end{remark}

\subsection{Exponential decay of eigenvectors}

In this section, we prove exponential decay estimates of the first eigenvector of $\cL_{\Gamma_R}$, uniformly as $R\to\infty$. For convenience,  
in the rest of Section \ref{sec:brok_guides} we use the following condensed notation for the first eigenvalue on $\Gamma_R$ and $\Gamma$, {\em cf.} Table \ref{tab:recap}, as well as the $L^2$-norm of their difference:
\begin{equation}
\label{not:sec2}
   \lambda_R := \lambda_1(\Gamma_R),\quad    \lambda_\infty := \lambda_1(\Gamma),
   \quad\mbox{and}\quad
      \omega = \sqrt{\pi^2-\lambda_\infty}\,.
\end{equation}

\begin{lemma}
\label{lem:expdec}
For $R\ge0$, let $v_R$ be the positive normalized eigenvector associated with the first eigenvalue $\lambda_R$ of $\cL_{\Gamma_R}$, i.e.
\begin{equation}
\label{eq:vR}
   \cL_{\Gamma_R} v_R = \lambda_R v_R,\quad v_R>0 \ \ \ \mbox{on}\ \ \ \Gamma_R,
   \quad\mbox{and}\quad
   \|v_R\|_{L^2(\Gamma_R)} = 1\,.
\end{equation}
Recall that $\Sigma_\rho=\partial\Gamma_\rho\cap\partial\Box_{\rho+1}$, {\em cf.} \eqref{eqn:defsigmaR}. Then, with $\omega$ given in \eqref{not:sec2}, for all integers $\ell$, $m\ge0$, there exists a constant $C_{\ell,m}$ such that
\[
   \forall R\ge1,\quad\forall\rho\in[1, R],\quad
   \|\partial^\ell_n\partial^m_\tau v_R\|_{L^2(\Sigma_\rho)} \le C_{\ell,m} \,e^{-\rho\omega}\,,
\]
where $\partial_n$ and $\partial_\tau$ are the normal and tangential derivatives on $\Sigma_\rho$\,, respectively.
\end{lemma}

\begin{proof}
We prove exponential decay for $v_R$ by a representation formula on the rectangle $\cT^1_R=(0,R)\times\cI$, the rectangle $\cT^2_R=\cI\times(0,R)$ being handled in a similar way. 
We note that $v_R$ is solution of the mixed problem in $\cT^1_R=(0,R)\times\cI$
\begin{equation}
	\left\{
		\begin{array}{rcll}
			-\Delta v_R &=& \lambda_R v_R & \text{in } (0,R)\times\cI,\\
			v_R(x_1,-1) = v_R(x_1,0) &=& 0 & \forall x_1\in (0,R),\\
			\partial_1 v_R(R,x_2) &=& 0 & \forall x_2\in\cI,\\
			v_R(0,x_2) &=& g_R &\forall x_2\in\cI,
		\end{array}
	\right.
\label{eqn:syst-eigfun}
\end{equation}
where $g_R$ is the trace of $v_R$ on the segment $\Sigma^1_0=\{0\}\times\cI$. Since $v_R$ belongs in particular to $H^1(\Gamma_0)$, its trace  belongs to $H^{1/2}(\Sigma^1_0)$  with the estimates
\[
   \|g_R\|_{H^{1/2}(\Sigma^1_0)} \le C_0 \|v_R\|_{H^1(\Gamma_0)}
\]
where $C_0$ does not depend on $R$. Now, there holds
\[
   \|v_R\|_{H^1(\Gamma_0)} \le \|v_R\|_{H^1(\Gamma_R)} = \sqrt{\lambda_R+1} \|v_R\|_{L^2(\Gamma_R)}
   = \sqrt{\lambda_R+1} \le \sqrt{\pi^2+1}\,.
\]
Thus we have obtained in particular
\begin{equation}
\label{eq:expdec1}
   \|g_R\|_{L^{2}(\Sigma^1_0)} \le C_0 \sqrt{\pi^2+1},\quad \forall R\ge0.
\end{equation}
We expand $v_R$ along the eigenvectors of the operator $-\partial_y^2$, self-adjoint on $(H^2\cap H_0^1)(\cI)$. Its normalized eigenvectors are
\[
	s_k(x_2) = \sqrt{2}\,\sin(k\pi x_2)
\]
and we write
\[
	v_R(x_1,x_2) = \sum_{k\geq 1} u_k(x_1)s_k(x_2), \ 
	\text{ where } \ u_k(x_1) = \int_{\cI} v_R(x_1,x_2)\,s_k(x_2) \,\rd x_2.
\]
Hence, \eqref{eqn:syst-eigfun} yields that for all $k\geq 1$ we have
\begin{equation}
	\left\{
		\begin{array}{rcll}
			-u_k'' + (k^2\pi^2 -\lambda_R)u_k &=& 0  \quad\text{ in } (0,R),\\
			u_k'(R) &=& 0, \\
			u_k(0) &=& g_{R,k} \\
		\end{array}
	\right. \quad\mbox{where}\quad
	g_{R,k} = \int_{\cI} g_R(x_2)\,s_k(x_2)\,\rd x_2\,.
\label{eqn:syst-eigfunfibr}
\end{equation}
The solution of \eqref{eqn:syst-eigfunfibr} is given by
\[
	u_k(x_1) = \frac{g_{R,k}}
	{\cosh\big(R\sqrt{k^2\pi^2 -\lambda_R}\,\big)}\ \cosh\big((R-x_1)\sqrt{k^2\pi^2 -\lambda_R}\,\big)
\]
which yields
\begin{equation}\label{eqn:rep-formula}
	v_R(x_1,x_2) = \sum_{k\geq 1} \frac{g_{R,k}}
	{\cosh\big(R\omega_{R,k}\big)} \ \cosh\big((R-x_1)\omega_{R,k}\big) \ s_k(x_2),
\end{equation}
where we have set $\omega_{R,k} = \sqrt{k^2\pi^2 -\lambda_R}$.
Thanks to the uniform convergence of this series and its derivatives on any subdomain of the form $[a,R]\times\cI$ with a positive $a$, we deduce the formulas for $\ell,m\in\mathbb{N}$
\[
\begin{cases}
	\partial^{\ell}_1\partial^{m}_2 v_R (x_1,x_2) = 
	\sum_{k\geq 1} (-1)^{m/2}\ \frac{\omega^{\ell}_{R,k} (k\pi)^{m}}
	{\cosh(R\omega_{R,k})} \ \, g_{R,k} \, \cosh\big((R-x_1)\omega_{R,k}\big) \   s_k(x_2) & m\ \ \mbox{even}\\[0.5ex]
	\partial^{\ell}_1\partial^{m}_2 v_R (x_1,x_2) = 
	\sum_{k\geq 1} (-1)^{(m-1)/2}\ \frac{\omega^{\ell}_{R,k} (k\pi)^{m}}
	{\cosh(R\omega_{R,k})}\ \, g_{R,k} \, \cosh\big((R-x_1)\omega_{R,k}\big) \   c_k(x_2) & m\ \ \mbox{odd}\\
\end{cases}
\]
if $\ell$ is even and
\[
\begin{cases}
	\partial^{\ell}_1\partial^{m}_2 v_R (x_1,x_2) = 
	\sum_{k\geq 1} (-1)^{m/2+1}\ \frac{\omega^{\ell}_{R,k} (k\pi)^{m}}
	{\cosh(R\omega_{R,k})} \ \, g_{R,k} \, \sinh\big((R-x_1)\omega_{R,k}\big) \   s_k(x_2) & m\ \ \mbox{even}\\[0.5ex]
	\partial^{\ell}_1\partial^{m}_2 v_R (x_1,x_2) = 
	\sum_{k\geq 1} (-1)^{(m+1)/2}\ \frac{\omega^{\ell}_{R,k} (k\pi)^{m}}
	{\cosh(R\omega_{R,k})}\ \, g_{R,k} \, \sinh\big((R-x_1)\omega_{R,k}\big) \   c_k(x_2) & m\ \ \mbox{odd}\\
\end{cases}
\]
if $\ell$ is odd, where we set $c_k=\sqrt{2}\,\cos(k\pi x_2)$ the cosine basis.
Thus we can calculate the $L^2$-norm of their trace on $\Sigma^1_\rho$, {\it i.e.} at $x_1=\rho$, for any $\rho\in [1,R]$:
\[
	\|\partial^{\ell}_1\partial^{m}_2 v_R\|_{L^2(\Sigma^1_\rho)}^2  =
	 	\begin{cases}
	\sum_{k\geq 1} \Big(\omega^{\ell}_{R,k} (k\pi)^{m}\, \frac{\cosh((R-\rho)\omega_{R,k}) }
	{\cosh(R\omega_{R,k})} \ g_{R,k} \Big)^2 & \ell\ \ \mbox{even}\\[0.5ex]
	\sum_{k\geq 1} \Big(\omega^{\ell}_{R,k} (k\pi)^{m}\, \frac{\sinh((R-\rho)\omega_{R,k}) }
	{\cosh(R\omega_{R,k})} \ g_{R,k} \Big)^2 & \ell\ \ \mbox{odd}.\\
										\end{cases}
\]
Since $\lambda_R\le \lambda_1(\Gamma)$ by point i) of Lemma \ref{lem:GR12}, we notice that $\omega_{R,k}$ is larger than $\sqrt{k^2\pi^2-\lambda_1(\Gamma)}$, itself larger than  $k\omega$. Thus we deduce 
\[
   \frac{\sinh\big((R-\rho)\omega_{R,k}\big) }
	{\cosh\big(R\omega_{R,k}\big)} \le \frac{\cosh\big((R-\rho)\omega_{R,k}\big) }
	{\cosh\big(R\omega_{R,k}\big)} \le 2 e^{-\rho\omega_{R,k}}  \le 2 e^{-\rho k\omega},
\] 
hence
\[
	\|\partial^{\ell}_1\partial^{m}_2 v_R\|_{L^2(\Sigma^1_\rho)}^2  \le 
	4\sum_{k\geq 1} \big(\omega^{\ell}_{R,k} (k\pi)^{m}\big)^2\,g^2_{R,k} \, e^{-2\rho k\omega}\,.
\]
Using the upper bound $\omega_{R,k}\le k\pi$, we find
\[
\begin{aligned}
	\|\partial^{\ell}_1\partial^{m}_2 v_R\|_{L^2(\Sigma^1_\rho)}^2  &\le 
	4\sum_{k\geq 1} (k\pi)^{2(\ell+m)}\,g^2_{R,k} \, e^{-2\rho k\omega}\\
	&\le 
	4 e^{-2\rho\omega} \Big(\sum_{k\geq 1} g^2_{R,k} \Big)
	\max_{k\geq 1} \Big\{ (k\pi)^{2(\ell+m)} \, e^{-2\rho(k-1)\omega} \Big\} \\
	&\le 
	4 e^{-2\rho\omega} \|g_R\|_{L^{2}(\Sigma^1_0)}^2 
	\max_{k\geq 1} \big\{ (k\pi)^{2(\ell+m)} \, e^{-2(k-1)\omega} \big\}	\quad\mbox{for}\quad \rho\ge1\,.
\end{aligned}
\]
Combining with \eqref{eq:expdec1}, we find for any $\ell,m\ge0$ the estimates $\|\partial^{\ell}_1\partial^{m}_2 v_R\|_{L^2(\Sigma^1_\rho)} \le C_{l,m} \, e^{-\rho\omega}$ with constants $C_{l,m}$ independent of $R$. The lemma is thus proved.
\end{proof}

\begin{corollary}
\label{cor:conv}
With the positive number $\omega=\sqrt{\pi^2-\lambda_\infty}$, {\em cf.\ \eqref{not:sec2}}, there exists a constant $C_1$ such that
\[
   \forall R\ge0,\quad 0\le\lambda_\infty - \lambda_R \le C_1\,e^{-2R\omega}.
\]
\end{corollary}

\begin{proof}
We know already that $\lambda_\infty \ge \lambda_R$. To prove the right estimate, we use the eigenvectors $v_R$ as quasi-modes for $\cL_\Gamma$ after cutting them off on the rectangular regions $(R-1,R)\times\cI$ and $\cI\times(R-1,R)$: assume $R\ge2$ for simplicity and introduce a smooth cut-off $\chi$ such that
\[
   \chi(t)=1 \ \ \mbox{for} \ \ t\le-1 \quad\mbox{and}\quad \chi(t)=0 \ \ \mbox{for} \ \ t\ge0.
\] 
We define $\widetilde v_R(\bx)$  by $\chi(x_1-R)v_R(\bx)$ if $\bx\in\Gamma^1_R$, by $\chi(x_2-R)v_R(\bx)$ if $\bx\in\Gamma^2_R$, and by $0$ if $\bx\in\Gamma\setminus\Gamma_R$. Then $\widetilde v_R$ belongs to the domain of $\cL_\Gamma$ and we can evaluate its Rayleigh quotient. Lemma \ref{lem:expdec} implies the following estimates with a constant $C$ independent of $R$:
\begin{equation}\label{eqn:expestimates1}
   \big| \|\widetilde v_R\|^2_{L^2(\Gamma)} - 1 \big| \le C\,e^{-2R\omega} \quad\mbox{and}\quad
   \big| \|\nabla\widetilde v_R\|^2_{L^2(\Gamma)} -  \|\nabla v_R\|^2_{L^2(\Gamma_R)}\big| \le C\,e^{-2R\omega}.
\end{equation}
Hence, the min-max principle and \eqref{eqn:expestimates1} give
\[
	\lambda_\infty(1-Ce^{-2R\omega})\leq\lambda_\infty \|\widetilde v_R\|_{L^2(\Gamma)}^2 \leq \|\nabla\widetilde v_R\|_{L^2(\Gamma)}^2 \leq \|\nabla v_R\|_{L^2(\Gamma_R)}^2 + Ce^{-2R\omega} = \lambda_R + Ce^{-2R\omega}.
\]
Consequently, we get
\[
	\lambda_\infty - \lambda_R \leq C_1e^{-2R\omega},
\]
with $C_1 = C(\lambda_\infty+1)$ and the corollary is proved.
\end{proof}

\subsection{Variation of eigenpairs}
\label{subsec:deriv}
Before proving differential formulas for the first eigenpair of the operator $\cL_{\Gamma_R}$, let us give an argument stating the regularity of this first eigenpair with respect to $R$.

\begin{lemma}
\label{lem:kato}
Recall the abbreviated notation $\lambda_R$ for the first eigenvalue of $\cL_{\Gamma_R}$ and let $v_R$ be the associated normalized and positive eigenvector, {\em cf.}\ \eqref{eq:vR}.
Then the function $R\mapsto\lambda_R$ is analytic on $(0,\infty)$ and the derivative $w_R:=\partial_Rv_R$ makes sense in $\Dom(\cQ_{\Gamma_R})$.
\end{lemma}

\begin{proof}For any $R>0$, let us introduce the following change of variables that transforms $\Gamma_R$ into $\Gamma_1$:
\begin{equation}
\label{eq:change}
	(x,y) = \left\{\begin{array}{lcl}
			(x_1,x_2)&\text{if}& (x_1,x_2)\in\Gamma_0,\\
			(R^{-1} x_1,x_2)&\text{if}& (x_1,x_2)\in\cT_R^1,\\
			(x_1,R^{-1}x_2)&\text{if}& (x_1,x_2)\in\cT_R^2.
		\end{array}\right.
\end{equation}
For $u\in \Dom(\cQ_{\Gamma_R})$, setting $\widehat{u}(x,y) = u(x_1,x_2)$ we get $\cQ_{\Gamma_R}(u) = \widehat{\cQ}_{R}(\widehat{u})$, where $\widehat{\cQ}_{R}$ is the parameter dependent $H^1$ semi-norm on $\Gamma_1$ defined as
\[
	\begin{gathered}
	\widehat{\cQ}_{R}(\widehat{u}) = \|\nabla\widehat{u}\|_{L^2(\Gamma_0)}^2 + R^{-1} \Big(\|\partial_x\widehat{u}\|_{L^2(\cT_1^1)}^2 + \|\partial_y\widehat{u}\|_{L^2(\cT_1^2)}^2\Big) + R\Big(\|\partial_x\widehat{u}\|_{L^2(\cT_1^2)}^2 + \|\partial_y\widehat{u}\|_{L^2(\cT_1^1)}^2\Big),\\
	\Dom (\widehat{\cQ}_{R}) = \{ \widehat{u} \in H^1(\Gamma_1) : \widehat{u}|_{\partial\Gamma_1\setminus\Sigma_1} = 0\}.
	\end{gathered}
\]
Remark that the $L^2$-norm becomes
\[
\|u\|_{L^2(\Gamma_R)}^2 = \widehat{\cN}_R(\widehat{u})^2 := \|\widehat{u}\|_{L^2(\Gamma_0)}^2 + R\Big(\|\widehat{u}\|_{L^2(\cT_1^1)}^2 + \|\widehat{u}\|_{L^2(\cT_1^2)}^2\Big),
\]
where the norm $\widehat{\cN}_R$ is equivalent to the usual norm $\|\cdot\|_{L^2(\Gamma_1)}$.

Now, as $\Dom (\widehat{\cQ}_{R})$ does not depend on $R$ and that for all $\widehat{u}\in\Dom (\widehat{\cQ}_{R})$ the application
\[
R\in(0,+\infty)\mapsto \frac{\widehat{\cQ}_{R}(\widehat{u})}{\widehat{\cN}_R(\widehat{u})^2}
\]
is analytic, using that $\widehat{\cQ}_{R}$ is bounded from below (non-negative), the strategy of Kato \cite[Chapt. 7 \S 4.2]{Kat95} can be adapted to obtain that the self-adjoint operators associated with the family $(\widehat{\cQ}_{R})_{R>0}$ {\it via} the first representation theorem is an analytic family of operators for $R\in(0,\infty)$. In particular, as for any $R>0$ the first eigenvalue $\lambda_R$ is simple, we obtain the lemma.
\end{proof}

Now we can prove a formula for the derivative with respect to $R$ of the first eigenvalue of $\cL_{\Gamma_R}$ in the spirit of \cite[VII, sect 6(5)]{Kat95} and \cite[Theorem (1.4)]{DH93-1}.

\begin{proposition}\label{prop:Dau-Hel}
With notations as in Lemma \textup{\ref{lem:kato}}, the derivative $\partial_R\lambda_R$ satisfies the formula (here $\tau$ is a tangential variable in the Neumann part $\Sigma_R$ of $\partial\Gamma_R$)
\begin{equation}
\label{eq:deriv}
   \partial_R\lambda_R = \int_{\Sigma_R} \big(|\partial_\tau v_R|^2 -\lambda_R |v_R|^2\big)\,\rd\tau,
   \quad R>0\,.
\end{equation}
\end{proposition}

\begin{proof} 
Since $\lambda_R$ is a simple eigenvalue, and since $\cL_{\Gamma_R}$ commutes with the symmetry with respect to the diagonal $\cD=\{\bx:\ x_1=x_2\}$, the eigenvector $v_R$ is also symmetric with respect to $\cD$. Indeed, it satisfies Neumann conditions on $\cD\cap\Gamma_0$, see \cite[Prop.2.2]{DaLaRa11}. Thus, we can reduce our analysis to the lower half $\Gamma^1_R$ of $\Gamma_R$. There holds $\|v_R\|_{L^2(\Gamma^1_R)}^2=\frac{1}{2}$ and we are going to prove
\begin{equation}
\label{eq:deriv2}
   \partial_R\lambda_R = 2\int_{\Sigma^1_R} \big(|\partial_2 v_R|^2 -\lambda_R |v_R|^2\big)\,\rd x_2,
   \quad R>0\,.
\end{equation}
We follow the steps of the proof of \cite[Theorem (1.4)]{DH93-1}.
Integrating by parts and using that $\partial_nv_R$ is zero on $\Sigma^1_R$ and $\cD\cap\Gamma_0$, we find for any chosen $h>0$
\[
   \int_{\Gamma^1_R} (-\Delta -\lambda_R)v_R\,v_{R+h}\,\rd x_1\rd x_2 =
   \int_{\Gamma^1_R} v_R\,(-\Delta -\lambda_{R})v_{R+h}\,\rd x_1\rd x_2 + 
   \int_{\Sigma^1_R} v_R\,\partial_1 v_{R+h}\,\rd x_2\,.
\]
Hence, using the eigen-equations for $v_R$ and $v_{R+h}$: 
\[
   (\lambda_{R+h} -\lambda_R)\int_{\Gamma^1_R} v_R\,v_{R+h}\,\rd x_1\rd x_2
    +\int_{\Sigma^1_R} v_R\,\partial_1v_{R+h}\,\rd x_2 = 0.  
\]
Taking advantage of the condition $\partial_1v_{R+h}\big|_{\Sigma^1_{R+h}}=0$ we can write
\[
   \int_{\Sigma^1_R} v_R\,\partial_1v_{R+h}\,\rd x_2 = 
   \int_{\cI} v_R(R,x_2)\,\Big(\partial_1v_{R+h}(R,x_2) - \partial_1v_{R+h}(R+h,x_2)\Big)\,\rd x_2.
\]
Putting together the last two identities, 
dividing  by $h$ and letting $h$ tend to $0$, we obtain the relation
\begin{equation*}
   \partial_R\lambda_R\int_{\Gamma^1_R} |v_R|^2 \,\rd x_1\rd x_2 =
   \int_{\Sigma^1_R} v_R\,(\partial^2_1v_R)\,\rd x_2\,.
\end{equation*}
Using the relation $(-\partial^2_1-\partial^2_2-\lambda_R)v_R=0$, 
we deduce formula \eqref{eq:deriv2}, hence formula \eqref{eq:deriv}.
\end{proof}

\begin{remark}
1) Formula  \eqref{eq:deriv} takes also the form
\begin{equation}
\label{eq:compat}
   \partial_R\lambda_R =
   \int_{\Sigma_R} v_R\,(\partial^2_nv_R)\,\rd \tau,
   \quad R>0\,.
\end{equation}
2) Since for any $R>0$, $v_R$ belongs to $H^1_0(\Sigma_R)$, thus satisfies $\|\partial_\tau v_R\|_{L^2(\Gamma_R)}^2\ge \pi^2 \|v_R\|_{L^2(\Gamma_R)}^2$, formula \eqref{eq:deriv} implies the inequality
\begin{equation}
\label{eq:deriv3}
   \partial_R\lambda_R \ge \int_{\Sigma_R} (\pi^2 -\lambda_R) |v_R|^2\,\rd\tau,
   \quad R>0\,.
\end{equation}
As $v_R$ is not identically $0$ on the Neumann boundary $\Sigma_R$, the above inequality implies:
\begin{equation}
\label{eq:increas}
   \mbox{the function \ $R\mapsto\lambda_R$ \ is increasing on \ $(0,\infty)$.}
\end{equation}
3) The function $R\mapsto\lambda_R$ is analytic on $(0,\infty)$, but has no extension as an analytic (nor even $\sC^1$) function on the closed interval $[0,\infty)$.
\label{rmk:pt_vp}
\end{remark}

Let $w_R$ be the derivative $\partial_Rv_R$. On $\Gamma_R$, the eigen-equations for $v_R$, $v_{R+h}$ yield
\[
   (-\Delta -\lambda_{R+h})(v_{R+h} - v_R) = 
   (\lambda_{R+h} -\lambda_R)v_R\,.
\]
The function $v_{R+h}-v_R$ satisfies the zero Dirichlet conditions on $\partial\Gamma_R\setminus\Sigma_R$. On $\Sigma_R$, we can write like in the proof above
\[
   \partial_n(v_{R+h} - v_R)\big|_{\Sigma_R} = 
   \partial_n v_{R+h}\big|_{\Sigma_R} - \partial_n v_{R+h}\big|_{\Sigma_{R+h}}
\]
Dividing by $h$ and letting $h\to0$, we deduce that $w_R$ is solution of the mixed problem
\begin{equation}
\label{eq:pbwR}
\left\{\begin{array}{rcll}
   (-\Delta -\lambda_R)w_R &=& (\partial_R\lambda_R)\,v_R \quad& \mbox{in}\ \ \Gamma_R,\\
   w_R &=& 0 & \mbox{on}\ \ \partial\Gamma_R\setminus\Sigma_R,\\
   \partial_n w_R &=& -\partial^2_nv_R & \mbox{on}\ \ \Sigma_R.
\end{array}\right.
\end{equation}
We note that formula \eqref{eq:compat} is the compatibility relation for the existence of a solution to the mixed problem \eqref{eq:pbwR}. Moreover, the normalization $\int_{\Gamma_R}|v_R|^2\rd x_1\rd x_2 = 1$ implies the relation
\begin{equation}
	\label{eqn:psnul}
	\int_{\Gamma_R}w_Rv_R \, \rd x_1\rd x_2  =  - \frac{1}{2}\int_{\Sigma_R}|v_R|^2 \,\rd\tau \,.
\end{equation}

\begin{lemma}
\label{lem:wR}
With the notations of Lemma \textup{\ref{lem:kato}} and $\omega=\sqrt{\pi^2-\lambda_1(\Gamma)}$, the derivative $w_R=\partial_Rv_R$ satisfies the estimates
\begin{equation}
\label{eq:wR}
   \|w_R\|_{H^1(\Gamma_R)} \le C\,e^{-R\omega },\quad R\ge1.
\end{equation}
\end{lemma}	

\begin{proof}
The solution $w_R$ of \eqref{eq:pbwR}-\eqref{eqn:psnul} is unique and can be written as
\[
	w_R = \mu_R v_R + w_R^\perp, \quad\text{with}\quad \int_{\Gamma_R}w_R^\perp v_R \, \rd x_1\rd x_2 = 0
	\ \ \mbox{and}\ \  \mu_R = - \frac12\int_{\Sigma_R}|v_R|^2 \,\rd\tau.
\]
Recall that the variational space associated with Problem \eqref{eq:pbwR} is the form domain $V=\Dom(\cQ_{\Gamma_R})$ and denote by $V'$ its dual space. Let $f_R$ denote the right hand side of \eqref{eq:pbwR}. Let us prove that 
\begin{equation}\label{eqn:lbvarfor1}
	\|w_R^\perp\|_{H^1(\Gamma_R)} \leq K \|f_R\|_{V'}\quad\mbox{for}\quad K = \frac{1+\pi^2}{\omega^2}\,.
\end{equation}
Indeed, if $\langle\cdot,\cdot\rangle_{V',V}$ denotes the duality pairing, for any $\varepsilon \in (0,1)$ we have
\begin{align*}
	\|w_R^\perp\|_{H^1(\Gamma_R)}\|f_R\|_{V'} \geq \langle f_R, w_R^\perp\rangle_{V',V} &= \|\nabla w_R^\perp\|_{L^2(\Gamma_R)}^2 - \lambda_1(\Gamma_R)\|w_R^\perp\|_{L^2(\Gamma_R)}^2\\
	&\geq \varepsilon \|\nabla w_R^\perp\|_{L^2(\Gamma_R)}^2 + \big((1-\varepsilon)\lambda_{2}(\Gamma_R) - \lambda_1(\Gamma_R)\big)\|w_R^\perp\|_{L^2(\Gamma_R)}^2\,.
\end{align*}
Using Lemma \ref{lem:GR12}, we get $\lambda_1(\Gamma_R) < \lambda_1(\Gamma)$ and $\lambda_2(\Gamma_R) - \lambda_1(\Gamma_R)\ge \pi^2-\lambda_1(\Gamma)=\omega^2$, hence
\[
	\varepsilon\|\nabla w_R^\perp\|_{L^2(\Gamma_R)}^2 + (\omega^2 - \varepsilon\pi^2)\|w_R^\perp\|_{L^2(\Gamma_R)}^2 \leq \|w_R^\perp\|_{H^1(\Gamma_R)} \|f_R\|_{V'}\,.
\]
Choosing $\varepsilon$ so that $\varepsilon=\omega^2 - \varepsilon\pi^2$, we obtain \eqref{eqn:lbvarfor1}.
From \eqref{eqn:lbvarfor1} we deduce (still using the condensed notation $\lambda_R$ for $\lambda_1(\Gamma_R)$)
\begin{align}
	\|w_R\|_{H^1(\Gamma_R)} &\le |\mu_R|\,\|v_R\|_{H^1(\Gamma_R)} + \|w_R^\perp\|_{H^1(\Gamma_R)} \nonumber\\
	&\leq \|v_R\|_{L^2(\Sigma_R)}^2 \,\|v_R\|_{H^1(\Gamma_R)} + K\,\|f_R\|_{V'}\nonumber\\
	&\leq \sqrt{1+\lambda_R}\,\|v_R\|_{L^2(\Sigma_R)}^2 \, + K\,\|f_R\|_{V'}\nonumber\\
	&\leq \sqrt{1+\pi^2}\,\|v_R\|_{L^2(\Sigma_R)}^2 + K\,\|f_R\|_{V'}\,.
\label{eqn:maj_wR}
\end{align}
As the duality pairing between $f_R\in V'$ and any $g\in V$ satisfies $\langle f_R,g\rangle_{V',V} = \int_{\Gamma_R}(\partial_R\lambda_R) v_R\,g\, \rd\bx - 
\int_{\Sigma_R} \partial^2_nv_R \,g\,\rd\tau$, we get
\begin{equation}
	|\langle f_R,g\rangle_{V',V}| \leq |\partial_R\lambda_R|\|g\|_{L^2(\Gamma_R)} + \|\partial_n^2v_R\|_{L^2(\Sigma_R)}\|g\|_{L^2(\Sigma_R)}.
\label{eqn:dual_pair}
\end{equation}
Using Lemma \ref{lem:trace} that provides a uniform estimate of $\|g\|_{L^2(\Sigma_R)}$ by $\|g\|_{H^1(\Gamma_R)}$ (here the assumption $R\ge1$ comes into play), we find that 
\eqref{eqn:dual_pair} becomes
\[
	|\langle f_R,g\rangle_{V',V}| \leq \Big(|\partial_R\lambda_R| + C\|\partial_n^2v_R\|_{L^2(\Sigma_R)}\Big)\|g\|_{H^1(\Gamma_R)}
\]
and we obtain
\begin{equation}
	\|f_R\|_{V'} \leq |\partial_R\lambda_R| + C\|\partial_n^2v_R\|_{L^2(\Sigma_R)}.
	 \label{eqn:maj_normdual}
\end{equation}
Combining estimates \eqref{eqn:maj_wR} and \eqref{eqn:maj_normdual} with formula \eqref{eq:compat} yields
\[
	\|w_R\|_{H^1(\Gamma_R)}  
	\leq \sqrt{1+\pi^2}\,\|v_R\|_{L^2(\Sigma_R)}^2  + K\,
	\Big( \|v_R\|_{L^2(\Sigma_R)}\|\partial^2_nv_R\|_{L^2(\Sigma_R)} + C\|\partial^2_nv_R\|_{L^2(\Sigma_R)} \Big)\,.
\]
The application of Lemma \ref{lem:expdec} ends the proof.
\end{proof}

\section{Fichera layer: Finiteness of discrete spectrum}\label{sec:spectrum}
This section is devoted to the proof of our main theoretical result, that is Theorem \ref{thm:main} that describes the essential spectrum of the Dirichlet Laplacian $\cL_\Lambda$ on the Fichera layer and states the finiteness of its discrete spectrum. We also exhibit a lower bound for the whole spectrum of $\cL_\Lambda$.

\subsection{Essential spectrum}
In this subsection we prove point i) of Theorem \ref{thm:main}, i.e.,
\begin{equation}
\label{eq:ess}
   \sigma_{\ess}(\cL_\Lambda) = [\lambda_1(\Gamma),+\infty).
\end{equation}
The proof of \eqref{eq:ess} is made in two steps: first we establish the inclusion $[\lambda_1(\Gamma),+\infty)\subset \sigma_{\ess}(\cL_\Lambda)$, and second we show the inequality $\min\sigma_{\ess}(\cL_\Lambda)\ge \lambda_1(\Gamma)$. For this, we make use of the following result \cite[Th.10.2.4]{BS87}: \begin{equation}
\label{eq:BS}
   \mbox{$\cQ_1 \leq \cQ_2$ in the sense of quadratic forms}\ \Longrightarrow \ 
   \min\sigma_{\ess}(\cQ_1)\le\min\sigma_{\ess}(\cQ_2).
\end{equation}

\begin{proof}[Proof of $[\lambda_1(\Gamma),+\infty)\subset \sigma_{\ess}(\cL_\Lambda)$]
To prove this, it suffices to consider suitable Weyl sequences for the operator $\cL_{\Lambda}$.
Let $v_\infty$ denote an eigenvector of $\cL_\Gamma$ associated with its first eigenvalue. Choose $\kappa\ge0$ and $\chi\in\mathscr{D}(\R)$ satisfying $\chi\equiv1$ on $[1,2]$ and $\supp(\chi) \subset [\frac12,\frac52]$. The sequence $(\psi_n)$ given by
\[
   \psi_n(\bx) = 
   \begin{cases}
   v_\infty(x_1,x_2)\,e^{i\kappa x_3}\, \frac{1}{\sqrt{n}}\,\chi\big(\frac{x_3}{n}\big),
   \quad &\bx=(x_1,x_2,x_3)\in\Gamma\times\R_+\,, \\
   0 \quad &\bx\in\Lambda\setminus(\Gamma\times\R_+)\,.
   \end{cases}
\]
is a suitable Weyl sequence for the value $\lambda_1(\Gamma)+\kappa^2$ and the operator $\cL_\Lambda$.
\end{proof}

\begin{proof}[Proof of $\min\sigma_{\ess}(\cL_\Lambda)\ge \lambda_1(\Gamma)$]
Choose $R>0$. Then, using subdomains introduced in \eqref{eq:LR}--\eqref{eq:LjR}, we see that
\[
   \Lambda_R \cup \Omega^1_R \cup \Omega^2_R \cup \Omega^3_R
\]
is a domain partition of $\Lambda$ in the sense of Section \ref{sss:partition} and its associated quadratic form $\cQ_\Lambda^\Bro$ satisfies $\cQ_\Lambda^\Bro\le\cQ_\Lambda$. So, by \eqref{eq:BS}
\[
   \min\sigma_{\ess}(\cQ_\Lambda^\Bro) \le \min\sigma_{\ess}(\cQ_\Lambda).
\]
As $\Lambda_R$ is bounded and $\Omega^1_R$, $\Omega^2_R$ and $\Omega^3_R$ are isometric, we find
\[
   \min\sigma_{\ess}(\cQ_\Lambda^\Bro) = \min\sigma_{\ess}(\cQ_{\Omega^3_R}) \ge \lambda_1(\cQ_{\Omega^3_R}),
\]
where we recall that the quadratic $\cQ_{\Omega^3_R}$ and its domain are defined according to the conventions of \S \ref{sss:partition}. Note that by definition, the domain $\Omega^3_R$ satisfies
\begin{equation}
\label{eq:O3R}
	\Omega^3_R = \{(x_1,x_2,x_3) \in \Lambda :\ \ x_3>R\text{ and } (x_1,x_2) \in \Gamma_{x_3}\},
\end{equation}
where $\Gamma_{x_3}$ is the finite waveguide of length $x_3$ defined in \eqref{eq:GR} taking $R=x_3$. Hence, for $u\in\Dom (\cQ_{\Omega^3_R})$ we have
\begin{align*}
	\cQ_{\Omega^3_R}(u) 	&= \int_{R}^{\infty}
	\bigg\{\int_{\Gamma_{x_3}}\big(|\partial_1 u|^2 + |\partial_2 u|^2 + |\partial_3 u|^2 \big)
	\,\rd x_1\rd x_2\bigg\}\rd x_3\\
				&\geq \int_{\Omega^3_R} \lambda_1(\Gamma_{x_3})\, |u|^2 \,\rd x_1\rd x_2 \rd x_3
				\geq \lambda_1(\Gamma_R)\|u\|_{L^2(\Omega^3_R)}^2,
\end{align*}
where we used the monotonicity property \eqref{eq:increas} of the first eigenvalue $\lambda_1(\Gamma_{x_3})=\lambda_{x_3}$ with respect to $x_3$. We have finally obtained for any $R>0$
\[
	\min\sigma_{\ess}(\cL_\Lambda) \geq \min \sigma(\cQ_{\Omega^3_R}) \geq \lambda_1(\Gamma_R).
\]
Combined with the convergence result $\lambda_1(\Gamma_R)\to\lambda_1(\Gamma)$ as $R\to\infty$ (Corollary \ref{cor:conv}), this yields the desired inequality.
\end{proof}

\subsection{Finiteness of the number of bound states} \label{subsec:finite}
The purpose of this part is to prove
\begin{equation}
\label{eq:finit1}
   \#\Big(\sigma(\cL_{\Lambda})\cap\big[0,\lambda_1(\Gamma)\big)\Big) < + \infty.
\end{equation}

\begin{proof}[Proof step 1: Reduction to the residual domain $\Omega^3_L$.]
Like in the previous proof, we use for $L\ge0$ the domain partition $\Lambda_L \cup \Omega^1_L \cup \Omega^2_L \cup \Omega^3_L$ of $\Lambda$. As a consequence of \eqref{eq:part1}-\eqref{eq:part2}, if the number of eigenvalues under $\lambda_1(\Gamma)$ is finite for each of the operators acting on $\Lambda_L$, $\Omega^1_L$, $\Omega^2_L$, and $\Omega^3_L$, the same holds for $\cL_{\Lambda}$. For each chosen $L$, this finiteness holds for the bounded domain $\Lambda_L$. Moreover the spectra of the three operators $\Omega^j_L$ are identical by symmetry. Thus \eqref{eq:finit1} will be proved if there holds
\begin{equation}
\label{eq:finit2}
   \#\Big(\sigma(\cQ_{\Omega^3_L})\cap\big [0,\lambda_1(\Gamma)\big)\Big) < + \infty.
\end{equation}
for some $L\ge0$.
\end{proof}

\begin{proof}[Proof step 2. A Born-Oppenheimer type lower bound]
Recall that $\Omega^3_L$ is the set of $\bx \in \Lambda$ such that $x_3>L$ and $(x_1,x_2) \in \Gamma_{x_3}$,  {\em cf.}\ \eqref{eq:O3R}. In order to prove \eqref{eq:finit2}, we establish a lower bound for the associated quadratic form $\cQ_{\Omega^3_L}$ by projection on the first eigenvector $v_{x_3}$ of $\cL_{\Gamma_{x_3}}$ for each $x_3>L$. This is in the spirit of the so-called Born-Oppenheimer approximation \cite{CDS81,KMSW92,Mar89}. Recall that for $R\ge0$, $v_R$ denotes the positive normalized eigenfunction associated with the first eigenvalue $\lambda_R$ of $\cL_{\Gamma_R}$. Now, $R$ is set as the third coordinate $x_3$ of $\bx\in\Omega^3_L$. For an improved readability we denote, {\em cf.}\ Table \ref{tab:recap},
\begin{equation}
\label{eq:not}
   \lambda(x_3) := \lambda_1(\Gamma_{x_3}) \equiv \lambda_{x_3},\quad\lambda_\infty := \lambda_1(\Gamma)
    \quad\mbox{and}\quad \omega=\sqrt{\pi^2-\lambda_\infty}.
\end{equation}

\begin{lemma}
For any $L\ge0$ and any $u\in L^2(\Omega^3_L)$ we introduce the orthogonal projections
\[
	\big(\Pi_{x_3} u\big)(x_1,x_2,x_3) = f(x_3)\,v_{x_3}(x_1,x_2),\quad \Pi_{x_3}^\perp u = u - \Pi_{x_3} u\,,
	\quad  \mbox{for}\ \ x_3\ge L,
\]
where for the sake of simplicity we set
\[
	f(x_3) = \int_{\Gamma_{x_3}} u(x_1,x_2,x_3)\,v_{x_3}(x_1,x_2)\,\rd x_1 \rd x_2.
\]
Then, for all $\varepsilon\in(0,1)$, there exists $L_0\ge1$, such that for all $L\ge L_0$ and $u\in\Dom (\cQ_{\Omega^3_{L}})$, we have, with notation \eqref{eq:not},
\[
	  \cQ_{\Omega^3_L}(u) \geq  
	  (1-\varepsilon) \|f'\|_{L^2(L,\infty)}^2 +
	  \int_{L}^\infty (1-e^{-2\omega (x_3-L_0)}) \,\lambda(x_3)\,|f(x_3)|^2\, \rd x_3
	  + \lambda_\infty\,\|\Pi_{x_3}^\perp u\|_{L^2(\Omega^3_L)}^2\,.
\]
\label{prop:lbfq1D}
\end{lemma}

\begin{proof}
Thanks to the orthogonality in $L^2(\Gamma_{x_3})$ of $\Pi_{x_3}v$ and $\Pi_{x_3}^\perp v$ for any $v\in L^2(\Gamma_{x_3})$ we get
\[
	  \|\partial_3 u\|_{L^2(\Omega^3_L)}^2 = 
	  \|\Pi_{x_3} (\partial_3 u)\|_{L^2(\Omega^3_L)}^2 + \|\Pi_{x_3}^\perp (\partial_3 u)\|_{L^2(\Omega^3_L)}^2 .
\]
Thus, for all $u\in\Dom (\cQ_{\Omega^3_L})$ we have
\begin{align}
	  \cQ_{\Omega^3_L}(u) &\nonumber=  
	  \|\partial_3 u\|_{L^2(\Omega^3_L)}^2 + \int_{L}^\infty \cQ_{\Gamma_{x_3}}(u)\, \rd x_3\\
	  &\geq \|\Pi_{x_3} (\partial_3 u)\|_{L^2(\Omega^3_L)}^2 
	   + \int_{L}^\infty \cQ_{\Gamma_{x_3}}(u)\, \rd x_3.
	\label{eqn:lbq2}
\end{align}
Now, since $u\in\Dom (\cQ_{\Omega^3_L})$, remark that $f \in H^1(L,+\infty)$ and that we have
\begin{equation}
\label{eq:fF}
	\Pi_{x_3}(\partial_3 u)(x_1,x_2,x_3) = f'(x_3)v_{x_3}(x_1,x_2) - F(x_3)\,v_{x_3}(x_1,x_2),
\end{equation}
where the commutator term $F$ is given by (using the notation $w_{x_3} = \partial_{3}v_{x_3}$, {\em cf.} Lemma \ref{lem:wR})
\begin{align*}
	  F(x_3) &:= \int_{\Sigma^2_{x_3}} u(x_1,x_3,x_3)\, v_{x_3}(x_1,x_3)\,\rd x_1 
	  + \int_{\Sigma^1_{x_3}} u(x_3,x_2,x_3)\,v_{x_3}(x_3,x_2)\,\rd x_2 \\
	  &\quad+ \int_{\Gamma_{x_3}} u(x_1,x_2,x_3)\, w_{x_3}(x_1,x_2)\,\rd x_1\rd x_2\,,\quad x_3\ge L.
\end{align*}
On one hand, using Lemmas \ref{lem:expdec} and \ref{lem:wR}, we deduce the exponentially decreasing upper bound for $F$
\begin{equation*}
   |F(x_3)| \le C\,e^{-\omega x_3}\big(\|u\|_{L^2(\Sigma_{x_3})} + \|u\|_{L^2(\Gamma_{x_3})}\big) \,,\quad\forall x_3\ge L.
\end{equation*}
Since $L\ge1$, we can use Lemma \ref{lem:trace} to bound the trace term $\|u\|_{L^2(\Sigma_{x_3})}$ by a multiple of $\|u\|_{H^1(\Gamma_{x_3})}$ uniformly in $x_3>L$. Hence, with other constants $C'$ and $C''$ independent of $x_3$
\begin{equation}
\label{eq:Fdecr}
   |F(x_3)| \le C'\,e^{-\omega x_3} \,\|u\|_{H^1(\Gamma_{x_3})} 
   \le C'' \,e^{-\omega x_3} \,\sqrt{\cQ_{\Gamma_{x_3}}(u)}
    \,,\quad\forall x_3\ge L,
\end{equation}
where we used the min-max principle and Point 2) in Remark \ref{rmk:pt_vp} to obtain
\[
	\lambda(\Gamma_1)\|u\|_{L^2(\Gamma_{x_3})}^2\le \lambda(\Gamma_{x_3})\|u\|_{L^2(\Gamma_{x_3})}^2\le \cQ_{\Gamma_{x_3}}(u).
\]
On the other hand, coming back to \eqref{eq:fF} we have for any $\varepsilon\in(0,1)$
\[
\begin{split}
   \|\Pi_{x_3}(\partial_3 u)\|_{L^2(\Omega^3_L)}^2 &\ge 
   (1-\varepsilon) \|f'(x_3)v_{x_3}(x_1,x_2)\|_{L^2(\Omega^3_L)}^2 + 
   (1-\varepsilon^{-1}) \|F(x_3)\,v_{x_3}(x_1,x_2)\|_{L^2(\Omega^3_L)}^2 \\
   &= 
   (1-\varepsilon) \|f'\|_{L^2(L,\infty)}^2 
   +(1-\varepsilon^{-1}) \|F\|_{L^2(L,\infty)}^2\\
   &\ge 
   (1-\varepsilon) \|f'\|_{L^2(L,\infty)}^2 
   -\varepsilon^{-1} \|F\|_{L^2(L,\infty)}^2\,.
\end{split}
\]
Now \eqref{eq:Fdecr} yields, with a new constant $C$ independent of $x_3$
\[
   \|F\|_{L^2(L,\infty)}^2 \le C \int_{L}^\infty e^{-2\omega x_3} \cQ_{\Gamma_{x_3}}(u) \,\rd x_3\,.
\]
Hence we get 
\[
   \|\Pi_{x_3}(\partial_3 u)\|_{L^2(\Omega^3_L)}^2 \ge
   (1-\varepsilon) \|f'\|_{L^2(L,\infty)}^2 -
   C\,\varepsilon^{-1}\int_{L}^\infty e^{-2\omega x_3} \cQ_{\Gamma_{x_3}}(u) \,\rd x_3\,.
\]
Combining this with \eqref{eqn:lbq2} we obtain
\[
	  \cQ_{\Omega^3_L}(u) \geq  
	 (1-\varepsilon) \|f'\|_{L^2(L,\infty)}^2 + \int_{L}^\infty \cQ_{\Gamma_{x_3}}(u)\, \rd x_3  -
   C\,\varepsilon^{-1}\int_{L}^\infty e^{-2\omega x_3} \cQ_{\Gamma_{x_3}}(u) \,\rd x_3\,.
\]
Let us fix $\varepsilon\in(0,1)$ and choose $L_1\geq 1$ such that $C\,\varepsilon^{-1} e^{-2\omega L_1}\le1$. Then we write for any $L\ge L_1$
\begin{equation}
\label{eq:QO}
	  \cQ_{\Omega^3_L}(u) \geq  
	  (1-\varepsilon) \|f'\|_{L^2(L,\infty)}^2 +
	  \int_{L}^\infty (1-e^{-2\omega (x_3-L_1)}) \cQ_{\Gamma_{x_3}}(u)\, \rd x_3  \,.
\end{equation}
But for $u\in\Dom (\cQ_{\Omega^3_L})$ we have for $x_3\ge L$:
\[
\begin{split}
	  \cQ_{\Gamma_{x_3}}(u) &= \cQ_{\Gamma_{x_3}}(\Pi_{x_3} u) + \cQ_{\Gamma_{x_3}}(\Pi_{x_3}^\perp u) \\
	  &\geq \lambda_1(\Gamma_{x_3})\|\Pi_{x_3} u\|_{L^2(\Gamma_{x_3})}^2 
	  + \lambda_2(\Gamma_{x_3})\|\Pi_{x_3}^\perp u\|_{L^2(\Gamma_{x_3})}^2\,.
\end{split}
\]
Noting that $\|\Pi_{x_3} u\|_{L^2(\Gamma_{x_3})}^2 = |f(x_3)|^2$ we get (with notation \eqref{eq:not})
\begin{equation}
	\cQ_{\Gamma_{x_3}}(u) \geq \lambda(x_3) \,|f(x_3)|^2 + \pi^2\|\Pi_{x_3}^\perp u\|_{L^2(\Gamma_{x_3})}^2,
	\label{eqn:minq_gamma}
\end{equation}
where we used point {\em\ref{lem:GR12:ii})} of Lemma \ref{lem:GR12}. Combining \eqref{eq:QO} and \eqref{eqn:minq_gamma} yields
\[
	  \cQ_{\Omega^3_L}(u) \geq  
	  (1-\varepsilon) \|f'\|_{L^2(L,\infty)}^2 + \!
	  \int_{L}^\infty \! (1-e^{-2\omega (x_3-L_1)}) \Big( \lambda(x_3)\,|f(x_3)|^2
	  + \pi^2\|\Pi_{x_3}^\perp u\|_{L^2(\Gamma_{x_3})}^2\Big)\, \rd x_3  
\]
Take $L_0>L_1$ such that $(1-e^{-2\omega (L_0-L_1)})\pi^2\ge \lambda_\infty$. For $L\ge L_0$ the previous estimate implies
\[
	  \cQ_{\Omega^3_L}(u) \geq  
	  (1-\varepsilon) \|f'\|_{L^2(L,\infty)}^2 + \!
	  \int_{L}^\infty \! (1-e^{-2\omega (x_3-L_1)}) \lambda(x_3)\,|f(x_3)|^2\, \rd x_3
	  + \lambda_\infty\,\|\Pi_{x_3}^\perp u\|_{L^2(\Omega^3_L)}^2,
\]
But since $(1-e^{-2\omega (x_3-L_1)})\ge(1-e^{-2\omega (x_3-L_0)})$, the lemma is proved.
\end{proof}

Taking advantage of the exponential convergence of $\lambda(x_3)$ toward $\lambda_\infty$, we deduce:

\begin{corollary}
\label{cor:W}
There exists $L_0\ge1$, such that for all $L\ge L_0$ and $u\in\Dom (\cQ_{\Omega^3_L})$, we have, 
\begin{equation}
\label{eq:W}
	  \cQ_{\Omega^3_L}(u) \geq  
   \int_{L}^\infty \Big(\tfrac12|f'(x_3)|^2 + \big(\lambda_\infty- V_0(x_3)\big) |f(x_3)|^2 \Big)\,\rd x_3 
   + \lambda_\infty \,\|\Pi_{x_3}^\perp u\|_{L^2(\Omega^3_{L})}^2
\end{equation}
with $V_0(x_3) = e^{-2\omega (x_3-L_0)}$ and with $\lambda_\infty$ defined in \eqref{eq:not}.
\end{corollary}

\begin{proof} Thanks to Lemma \ref{prop:lbfq1D} with $\varepsilon = \frac12$ we get the existence of $L^\flat_0\ge 1$ such that for all $L\ge L^\flat_0$:
\begin{equation}
	\cQ_{\Omega^3_L}(u) \geq  
	  \tfrac12 \|f'\|_{L^2(L,\infty)}^2 + \!
	  \int_{L}^\infty \! (1-e^{-2\omega (x_3-L^\flat_0)}) \lambda(x_3)\,|f(x_3)|^2\, \rd x_3
	  + \lambda_\infty\,\|\Pi_{x_3}^\perp u\|_{L^2(\Omega^3_L)}^2.
	\label{eqn:lbfq_cor}
\end{equation}
But:
\[
\begin{split}
   (1-e^{-2\omega (x_3-L^\flat_0)}) \lambda(x_3) 
   &= \lambda_\infty + (\lambda(x_3)-\lambda_\infty) - e^{-2\omega (x_3-L^\flat_0)} \lambda(x_3) \\
   &\ge \lambda_\infty - C_1 e^{-2\omega x_3} - e^{-2\omega (x_3-L^\flat_0)} \lambda(x_3) \\
   &\ge \lambda_\infty - C_1 e^{-2\omega x_3} - e^{-2\omega (x_3-L^\flat_0)} \lambda_\infty \\
\end{split}
\]
where we have used Corollary \ref{cor:conv} and the monotonicity of $\lambda$. Hence, to obtain the corollary, it suffices to choose $L_0\ge L^\flat_0$ such that $C_1 + e^{2\omega L^\flat_0} \lambda_\infty \le  e^{2\omega L_0}$.
\end{proof}

This concludes step 2 of the proof of \eqref{eq:finit1}.
\end{proof}

\begin{proof}[Proof step 3. Reduction to a one-dimensional Schr\"odinger operator]
Taking advantage of Corollary \ref{cor:W}, we extend the quadratic form in the right hand side of \eqref{eq:W} to a larger  functional space, defining the quadratic form $\cQ^{\tens}$ with tensor product domain:
\[
\begin{split}
   \cQ^{\tens}(f,v) &= \int_{L}^\infty \Big(\tfrac12|f'(x_3)|^2 
   +  \big(\lambda_\infty- V_0(x_3)\big) |f(x_3)|^2 \Big)\,\rd x_3 
   + \lambda_\infty \,\|v\|_{L^2(\Omega^3_{L})}^2\\
   \Dom (\cQ^{\tens}) &= H^1(L,+\infty)\times L^2(\Omega^3_{L})\,.
\end{split}
\]
Let $u\in \Dom(\cQ_{\Omega^3_L})$, thanks to \eqref{eq:W} we have
\[
	  \cQ_{\Omega^3_L}(u) \geq \cQ^{\tens}(f,\Pi_{x_3}u),\quad 
	  \|u\|_{L^2(\Omega^3_L)}^2 = \|f\|_{L^2(L,+\infty)}^2 + \|\Pi_{x_3}^\perp u\|_{L^2(\Omega^3_L)}^2,
\]
with $f(x_3) = \langle u,v_{x_3}\rangle_{L^2(\Gamma_{x_3})}$. This inequality and the natural embedding of domains
\[
   	\begin{array}{ccc}
		\Dom(\cQ_{\Omega^3_L}) & \longrightarrow & \Dom (\cQ^{\tens})\\
		u			&\longmapsto 	& \big(\langle u,v_{x_3}\rangle_{L^2(\Gamma_{x_3})}, \Pi_{x_3}^\perp u\big)
		\end{array}
\]
imply by the min-max principle that the number of eigenvalues of $\cQ_{\Omega^3_L}$ below $\lambda_\infty$ is not greater than the number of eigenvalues of $\cQ^{\tens}$ below $\lambda_\infty$. Moreover, by construction, any eigenstate of $\cQ^{\tens}$ below $\lambda_\infty$ is of the form $(f,0)$, with $f$ an eigenstate associated with a negative eigenvalue of the one-dimensional Schr\"odinger quadratic form
\[
	\cQ^{\red}(f) = \int_{L}^{+\infty}\tfrac12|f'(t)|^2 - V_0(t)|f(t)|^2 \,\rd t,\quad 
	\Dom (\cQ^{\red}) = H^1(L,+\infty).
\]
It remains to prove that the $\cQ^{\red}(f)$ has at most a finite number of negative eigenvalues.
\end{proof}

\begin{proof}[Proof step 4. Conclusion]
As the potential $V_0$ of $\cQ^{\red}$ satisfies
\[
	\int_{L}^{+\infty} t \,V_0(t) \,\rd t < +\infty,
\]
a Bargmann estimate (see \cite{Bar52} or \cite[Thm.~XIII.9 a)]{RS78}) gives the finiteness of negative eigenvalues of $\cQ^{\red}$. More precisely, in our case
\[
   \#\Big(\sigma(\cQ^{\red})\cap(-\infty,0)\Big) \le 1 + 2\int_{L}^{+\infty} t \,V_0(t) \,\rd t\,,
\]
where the shift of $1$ with respect to \cite[Thm.~XIII.9 a)]{RS78} comes from the fact that the operator with Neumann boundary condition in $x_3 = L$ is a perturbation of rank one of the same differential operator but with Dirichlet boundary condition in $x_3 = L$.
This ends the proof of \eqref{eq:finit1}, i.e.\ of point {\em\ref{thm:main:ii})} of Theorem \ref{thm:main}.
\end{proof}

\subsection{Bounds for the discrete spectrum} 

As a consequence of \eqref{eq:ess}, an upper bound for the discrete spectrum $\sigma_{\dis}(\cL_\Lambda)$ is $\lambda_1(\Gamma)$. We exhibit now a lower bound. First, we prove symmetry properties for possible eigenvectors.

\begin{lemma}
\label{lem:sym}
Denote by $\cP^1$ the diagonal plane $x_2=x_3$, and $\cP^2$, $\cP^3$ by permutation of indices.
Let $u$ be an eigenvector of $\cL_\Lambda$ associated with a discrete eigenvalue. Then $u$ satisfies Neumann conditions $\partial_n u=0$ on the three diagonal planes $\cP^\ell\cap\Lambda$, $\ell=1,2,3$.
\end{lemma}

\begin{proof}
Let us write the eigenvector $u$ as the sum of its even and odd parts $u_+$ and $u_-$ with respect to the plane $\cP^1$.
The operator $\Delta$ and the domain $\Lambda$ being invariant by the symmetry with respect to the plane $\cP^1$, the parts $u_+$ and $u_-$ satisfy the same eigenproblem as $u$. The odd part $u_-$ is zero on $\cP^1\cap\Lambda$, hence satisfies Dirichlet boundary conditions on the domain $\Pi^1\cap\Lambda$, where $\Pi^1$ denotes the half-space $x_2<x_3$. Let us notice that
\[
   \Pi^1\cap\Lambda = \Pi^1\cap(\Gamma\times\R).
\]
Indeed, $\Pi^1\cap\Lambda = \{x\in\R^3 : -1<\min\{x_1,x_2,x_3\} < 0 \;\mbox{ and }\;x_2<x_3\}$. But  if $x_2<x_3$, then $\min\{x_1,x_2,x_3\}=\min\{x_1,x_2\}$. Whence the above equality. Therefore we have 
$$\lambda^{\Dir}_1(\Pi^1\cap\Lambda) =\lambda^{\Dir}_1(\Pi^1\cap(\Gamma\times\R)),$$ 
where for a Lipschitz domain $\mathcal{O}$, as introduced in \S \ref{subsub:operators}, $\lambda_i^\Dir(\mathcal{O})$ denotes the $i$-th Rayleigh quotient of the Dirichlet Laplacian posed on $\mathcal{O}$. By Dirichlet bracketing $\lambda^{\Dir}_1(\Pi^1\cap(\Gamma\times\R)) \ge \lambda^{\Dir}_1(\Gamma\times\R)$. This last quantity coincides with $\lambda_1(\Gamma)$, the infimum of the essential spectrum of $\cL_\Lambda$. Thus we deduce that $u_-$ is zero and are left with $u=u_+$. It proves that $u$ is even with respect to the plane $\cP^1$, thus satisfies $\partial_nu=0$ on $\cP^1$, and similarly for $\cP^2$ and $\cP^3$. 
\end{proof}

\begin{corollary}
\label{cor:L3}
With the definition \eqref{eq:L3} of $\Lambda^3$ (see also Figure \textup{\ref{fig-Lambda3}}), we denote by $\cL_{\Lambda^3}$ the realization of $-\Delta$ in $\Lambda^3$ with Dirichlet boundary conditions on $\partial\Lambda^3\cap\partial\Lambda$ and Neumann boundary conditions on the remaining part of the boundary $\partial\Lambda^3\cap(\cP^1\cup\cP^2)$. Then 
\[
   \sigma_{\dis}(\cL_\Lambda) = \sigma_{\dis}(\cL_{\Lambda^3})\quad\mbox{with multiplicities}
\]
and the associated eigenvectors of $\cL_{\Lambda^3}$ are the restrictions to $\Lambda^3$ of the eigenvectors of $\cL_\Lambda$.
\end{corollary}

\input Fig-Lambda3.tex

\begin{figure}
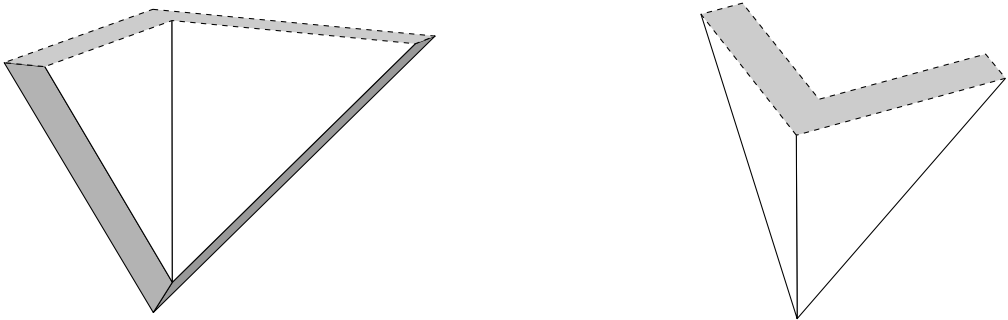

\begin{minipage}[l]{.49\linewidth}
\LambdaTroisInt{27}{12}{200}{7}
\end{minipage}
\begin{minipage}[l]{.49\linewidth}
\LambdaTroisExt{-115}{36}{200}{5}
\end{minipage}
\caption{Two views of $\Lambda^3$ (cut by a plane perpendicular to the $x_3$ axis).}
\label{fig-Lambda3}
\end{figure}

\begin{proof}
Lemma \ref{lem:sym} implies that if $u$ is an eigenvector associated with an eigenvalue below $\lambda_1(\Gamma)$, then its restriction to $\Lambda^3$ is an eigenvector of $\cL_{\Lambda^3}$ associated with the same eigenvalue. Conversely, if $u$ is an eigenvector of $\cL_{\Lambda^3}$ associated with an eigenvalue below $\lambda_1(\Gamma)$, we prove that it is symmetric with respect to the plane $\cP^3$ by an argument similar as above: the odd part $u_-$ of $u$ is zero on the boundary of $\Lambda^3\cap\Pi^3$ (with $\Pi^3$ the half-space $x_1<x_2$). But $\Lambda^3\cap\Pi^3=(\cI\times\R^2)\cap\Pi^3$, from which we deduce that $u_-$ is zero. Hence, $u$ can be extended to $\Lambda$ by symmetry through $\cP^1$ and $\cP^2$, defining an eigenvector of $\cL_\Lambda$.
\end{proof}

\begin{notation}
\label{not:VL}
According to Table \ref{tab:recap}, we use the condensed notation $\lambda(x_3)$ for $\lambda_1(\Gamma_{x_3})$
($x_3>-1$).
For any $L>-1$, let $\cV_L$ be the Sturm-Liouville operator with potential $\lambda:x_3\mapsto\lambda(x_3)$ on $(L,\infty)$:
\begin{equation}
\label{eq:VL}
   \cV_L: q\mapsto -q''+\lambda q,\quad\Dom (\cV_L) = \big\{q \in H^2(L,+\infty) : \ q'(L) = 0\big\}.
\end{equation}  
Let $\mu(L)$ be its lowest eigenvalue and $q_L$ be its normalized eigenfunction satisfying $q_L(L)>0$.
\end{notation}

Before stating the main result of this paragraph, we need the following two lemmas.
\begin{lemma} There holds:
\begin{enumerate}[i)]
	\item\label{lem:lambda1:i} $\lambda(x_3) = \frac{\pi^2}{2}(1+x_3)^{-2}$ for $x_3\in(-1,0]$, in particular $\lambda$ is analytic and decreasing on $(-1,0]$;
	\item\label{lem:lambda1:ii} $\lambda$ is analytic and increasing on $(0,+\infty)$;
	\item\label{lem:lambda1:iii} $\lambda$ has a right limit in $x_3=0$ satisfying $\lambda(0_+)\ge\lambda(0)=\frac{\pi^2}{2}$.
\end{enumerate}
\label{lem:lambda1}
\end{lemma}

\begin{proof}
For Point {\em\ref{lem:lambda1:i})} we remark that if $x_3\le0$, $\lambda(x_3)$ is the first eigenvalue of a Laplacian on a square of size $(1+x_3)$ with Dirichlet boundary condition on $\partial\Gamma\cap\Box_{1+x_3}$ and Neumann on the remaining part of the boundary. This yields Point i) immediately and we notice that 
$\lim_{x_3\rightarrow 0^-}\lambda(x_3) = \frac{\pi^2}{2}$.

Point {\em\ref{lem:lambda1:ii})} is a direct consequence of Lemma \ref{lem:kato} and \eqref{eq:increas}. 

Concerning Point {\em\ref{lem:lambda1:iii})}, for any positive $x_3$, we apply the argument \eqref{eq:part1}-\eqref{eq:part2} to the partition of $\Gamma_{x_3}$ into $\Gamma_0$, $\cT^1_{x_3}$ and $\cT^2_{x_3}$ and find
\[
   \lambda(x_3) \ge 
   \min\big\{\lambda(0),\ \lambda_1(\cT^1_{x_3}),\ \lambda_1(\cT^2_{x_3})\big\}\,.
\] 
Since $\lambda_1(\cT^j_{x_3})=\pi^2$ for $j=1,2$, which is larger than $\lambda(0)$, we find that $\lambda(x_3)\ge\lambda(0)$ for all positive $x_3$. As the function $\lambda$ is increasing on $(0,\infty)$, it has a right limit $\lambda(0_+)\ge\lambda(0)$ in $x_3=0$.
\end{proof}

\begin{lemma} 
The first eigenvalue $\mu(L)$ of the Sturm-Liouville operator $\cV_L$ defines a $\sC^0$ function $\mu$ on $(-1,+\infty)$ and the following holds:
\begin{enumerate}[i)]
	\item\label{lem:mu1:i} For all $L\neq0$, $\mu$ has a derivative that satisfies $\mu'(L) = \big(\mu(L) - \lambda(L)\big) \, q_L(L)^2$,
	\item\label{lem:mu1:ii} For all $L\geq 0$, $\mu(L) > \lambda(L)$,
	\item\label{lem:mu1:iii} For all $L\in(-1,+\infty)$, $\mu(L) \leq \lambda_\infty$ (where $\lambda_\infty$ is defined in \eqref{eq:not}).
\end{enumerate}
\label{lem:mu1}
\end{lemma}

\begin{proof}The continuity of $\mu(L)$ is obtained \textit{via} its characterization by the min-max principle and the continuity of its associated Rayleigh quotient.
Point {\em\ref{lem:mu1:i})} is straightforward application of the main result in \cite{DH93-1}.
About Point {\em\ref{lem:mu1:ii})}, using that $\lambda$ is increasing on $[0,+\infty)$ we know that for all $L\geq0$
\[
	\mu(L) = \int_L^{+\infty}\big(|q_L'|^2 + \lambda(x_3)|q_L|^2\big) dx_3 \geq \int_L^{+\infty}|q_L'|^2dx_3 + \lambda(L)>\lambda(L),
\]
where the strict inequality holds because $q_L'$ can not be identically zero otherwise $q_L$ would be constant, which is incompatible with the eigen-equation verified by $q_L$.

Finally Point {\em\ref{lem:mu1:iii})} is a consequence of the fact that the potential $\lambda$ of the Sturm-Liouville operator $\cV_L$ is smaller than $\lambda_\infty$ on the unbounded interval $(0,\infty)$. Therefore $\lambda_\infty$ belongs to the essential spectrum of $\cV_L$ and is an upper bound for $\mu(L)$.
\end{proof}

We are ready to prove the following proposition.

\begin{proposition}
\label{prop:L*}
With Notation \textup{\ref{not:VL}}, the first Rayleigh quotient of the Fichera layer $\lambda_1(\Lambda)$ admits the lower bound
\begin{equation}
\label{eq:lowb}
   \lambda_1(\Lambda) \ge \inf_{L\,>\,-1} \mu(L)\,.
\end{equation}
and there exists a unique element $L^*\in(-1,\infty)$ such that
\[
   \inf_{L\,>\,-1} \mu(L) = \mu(L^*).
\]
Moreover $L^*$ belongs to $(-1,0)$ and satisfies $\mu(L^*)=\lambda(L^*)$.
\end{proposition}

\begin{proof}
By Corollary \ref{cor:L3}, it suffices to consider the operator $\cL_{\Lambda^3}$.
Let $u\in \Dom(\cL_{\Lambda^3})$ such that $\|u\|_{L^2(\Lambda^3)}=1$. We bound its energy from below:
\[
\begin{split}
   \int_{\Lambda^3} |\nabla u|^2\,\rd x_1\rd x_2\rd x_3    & =
   \int_{-1}^\infty \Big(\int_{\Gamma_{x_3}} \big(|\partial_1 u|^2 + |\partial_2 u|^2 + |\partial_3 u|^2\big)
   \,\rd x_1\rd x_2\Big) \,\rd x_3 
   \\   & \ge
   \int_{-1}^\infty \Big(\int_{\Gamma_{x_3}} \big(\lambda(x_3) |u|^2+ |\partial_3 u|^2\big)\,\rd x_1\rd x_2\Big) \,\rd x_3    
   \\   & \ge
   \int_\Gamma \Big(\int_{\max\{x_1,x_2\}}^\infty \big(\lambda(x_3) |u|^2+ |\partial_3 u|^2\big)\,\rd x_3\Big) 
   \,\rd x_1\rd x_2    
   \\   & \ge
   \int_\Gamma \mu(\max\{x_1,x_2\}) \Big(\int_{\max\{x_1,x_2\}}^\infty  |u|^2\,\rd x_3\Big) 
   \,\rd x_1\rd x_2    
   \\   & \ge
   \inf_{L\,>\,-1} \mu(L) \int_{\Lambda^3}  |u|^2\,\rd x_1\rd x_2\rd x_3\,,
\end{split}
\]
which proves \eqref{eq:lowb}.
Thanks to  Lemma \ref{lem:lambda1} {\em\ref{lem:lambda1:i})} we know that $\lim_{x_3\,\rightarrow\, -1^+}\lambda(x_3) = +\infty$ and thanks to Lemma \ref{lem:mu1} {\em\ref{lem:mu1:iii})} we know that $\mu(L) \leq \lambda_\infty$ for all $L\in(-1,+\infty)$. In particular, in a neighborhood of $-1$ we have $\mu<\lambda$. Combining Lemma \ref{lem:mu1} {\em\ref{lem:mu1:ii})} and  Lemma \ref{lem:lambda1} {\em\ref{lem:lambda1:iii})}, we know that $\mu > \lambda$ in a neighborhood of $0$. Since $\mu - \lambda$ is continuous on $(-1,0)$, the functions $\mu$ and $\lambda$ cross at a lowest point $L^*\in (-1,0)$. They cannot cross again due to  Lemma \ref{lem:lambda1} {\em\ref{lem:lambda1:i})} and  Lemma \ref{lem:mu1} {\em\ref{lem:mu1:i})}: $\lambda$ is decreasing while $\mu$ is increasing on $(L^*,0)$.
\end{proof}

\begin{remark}
\label{rem:lammu}
In fact the function $\lambda$ is continuous in $0$, i.e. $\lambda(0_+)=\lambda(0)$. But the convergence to $\lambda(0)$ on the right is very slow, behaving as $1/|\log x_3|$, as in the case of a small Dirichlet hole \cite{Besson1985,MazNazPla2000}. This is illustrated by the plots in Figure \ref{fig:lammu} where $\lambda$ and $\mu$ are represented as functions of $x_3$ (computation details are provided in the next section). We can see that the function $\mu$ is piecewise $\sC^1$ ($\mu'$ has a discontinuity in $0$) and reaches its minimum at its intersection point with $\lambda$.
\end{remark}

\begin{figure}[ht]
\includegraphics[scale=0.52]{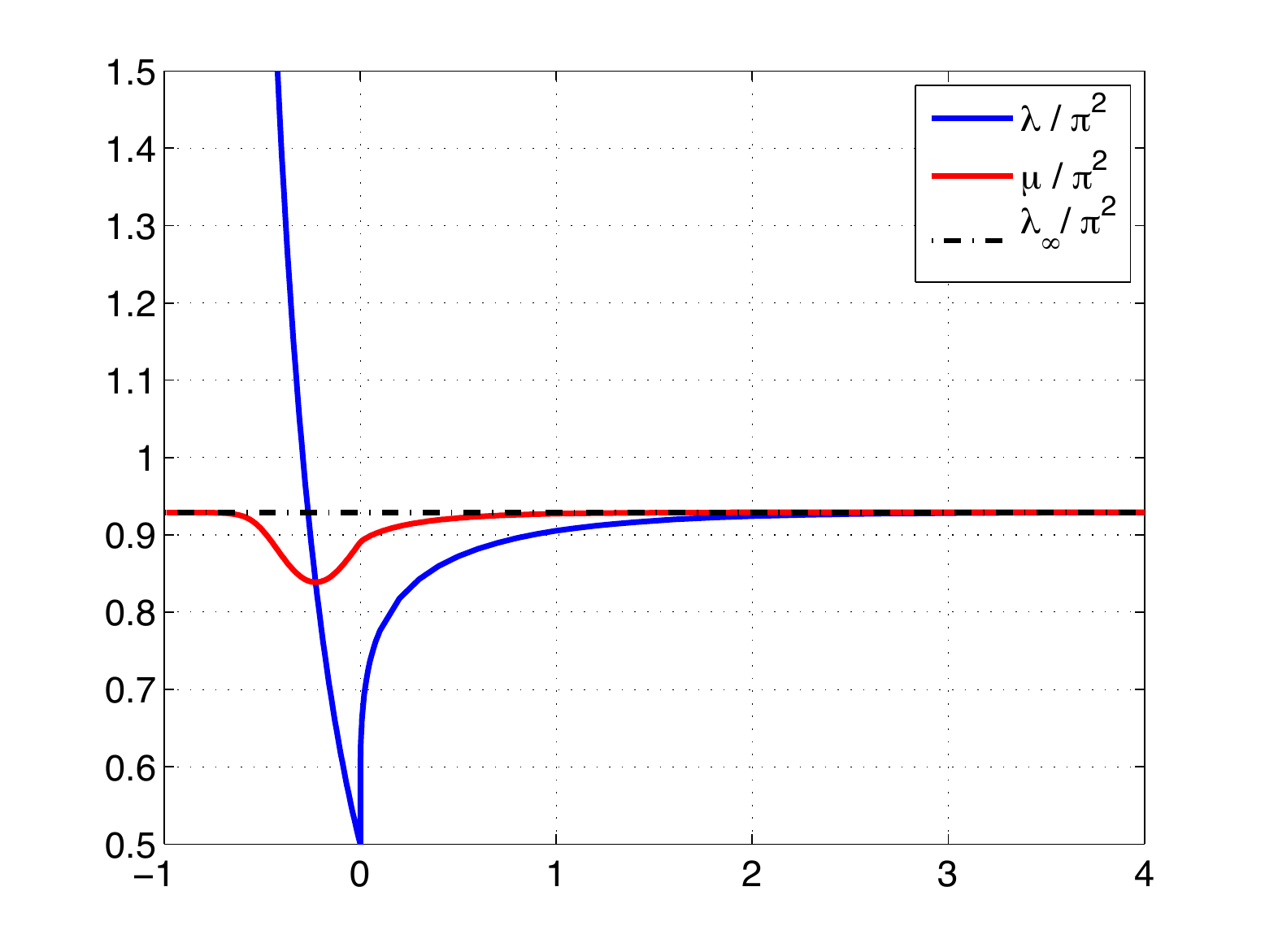}
\includegraphics[scale=0.52]{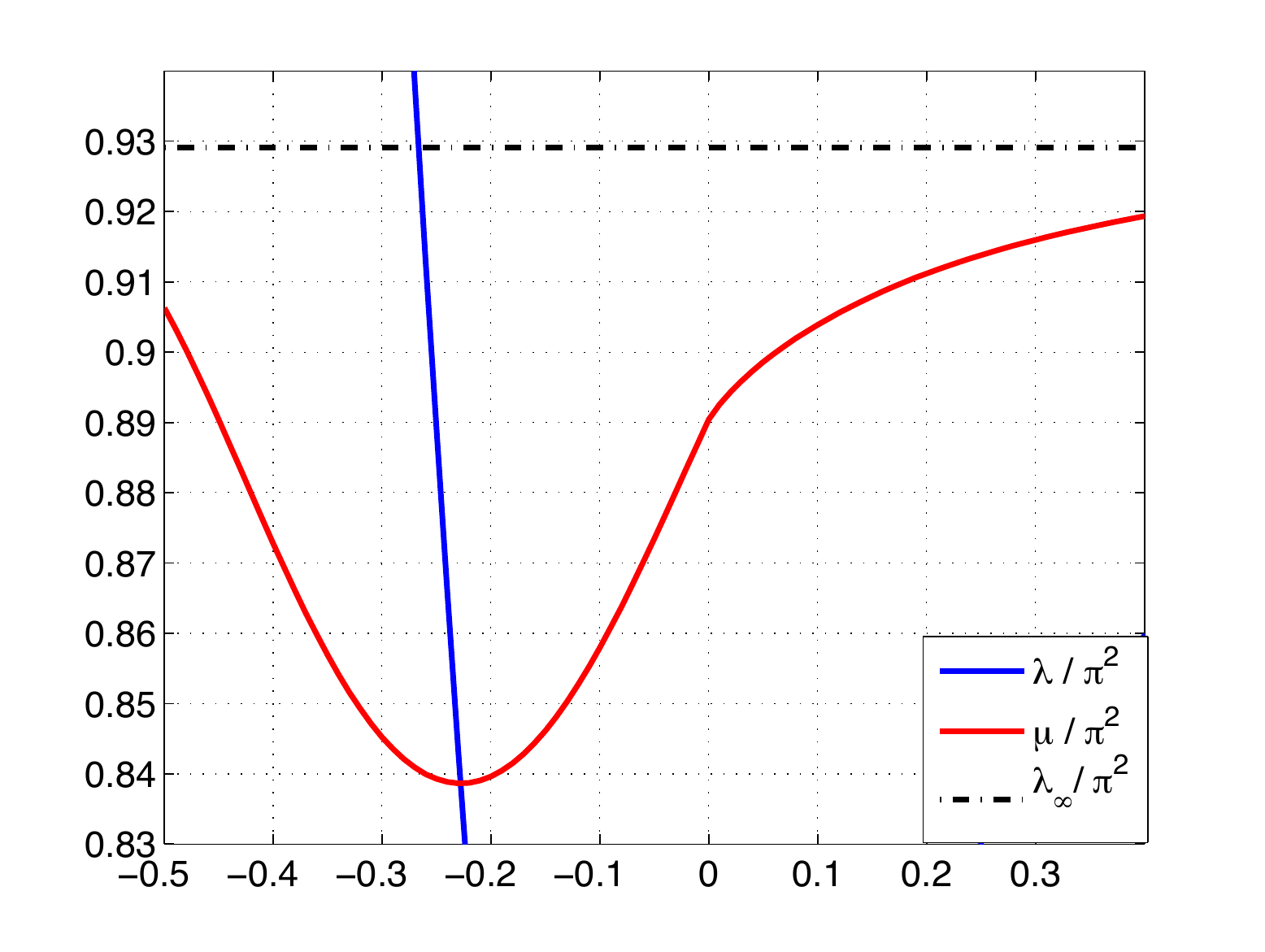}
   \caption{Normalized functions $x_3\mapsto\lambda(x_3)/\pi^2$ and $x_3\mapsto\mu(x_3)/\pi^2$ (with zoom).}
\label{fig:lammu}
\end{figure}

\section{Fichera layer: Numerical evidence of discrete spectrum}\label{sec:numerics}
Here we address the issue of existence of bound states for the Dirichlet Laplacian $\cL_\Lambda$ on the Fichera layer. In absence of a theoretical proof for the existence of discrete eigenvalues, we investigate the spectrum of $\cL_\Lambda$ by means of Galerkin projections using finite element approximations. 

It is a classical result that, in presence of exact geometric description and exact quadrature, the computed eigenvalues are {\em larger} than those of the original problem, see for instance the short survey \cite[Section 1]{DaLaRa11}. 

\subsection{Finite two-dimensional guides}\label{ss:GR}

Beforehand, we compute an approximation of the functions $\lambda$ and $\mu$, \emph{cf.}\ Figure \ref{fig:lammu}. Since $\lambda$ is explicit when $x_3\equiv R\le0$, we only need to compute it when $R$ is positive. In order to deal with a single domain and a single finite element mesh, we perform the change of variables \eqref{eq:change} that transforms $\Gamma_R$ into $\Gamma_1$. We use a quadrilateral strongly refined mesh with 4 layers and a refinement ratio of $0.1$ around the nonconvex corner $(0,0)$. For the results in Figure \ref{fig:lammu}, the interpolation degree $p$ is chosen equal to 16. Values of $R$ are sampled between 0.001 and 10. Varying the interpolation degree $p$ shows that we may expect 5 or 6 correct digits when $p=16$.

\begin{figure}[ht]
\includegraphics[scale=0.33]{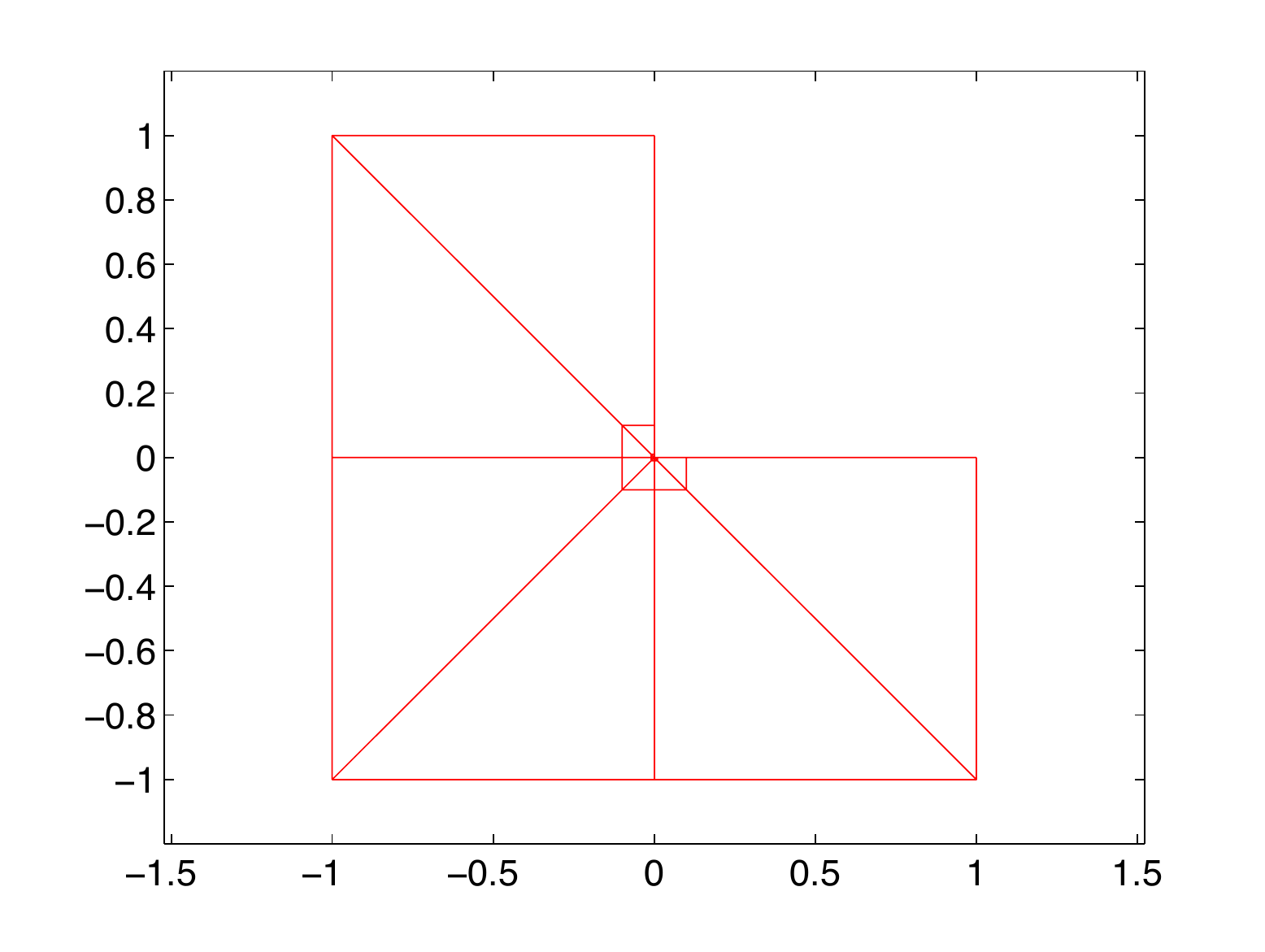}
\includegraphics[scale=0.33]{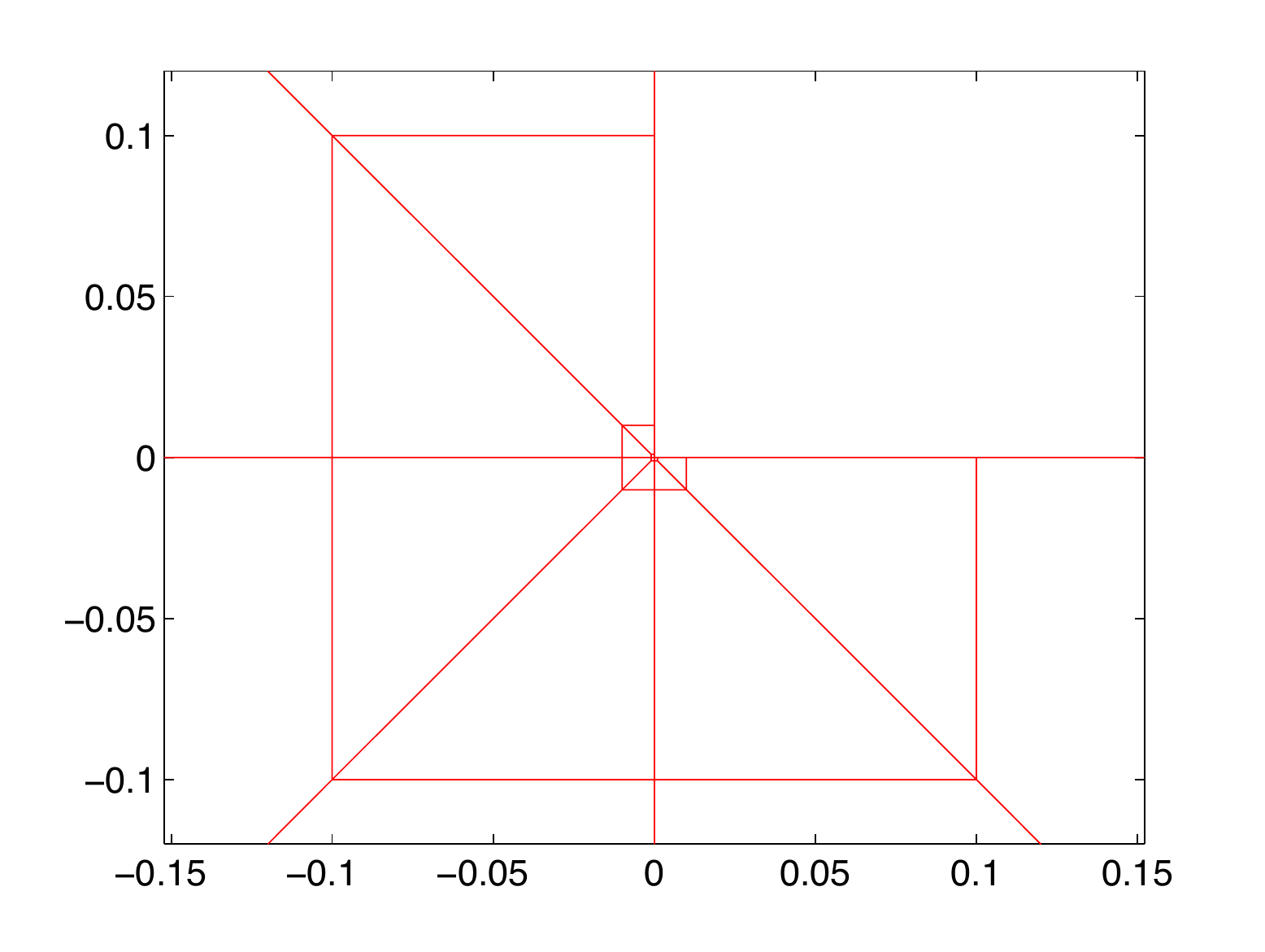}
\includegraphics[scale=0.33]{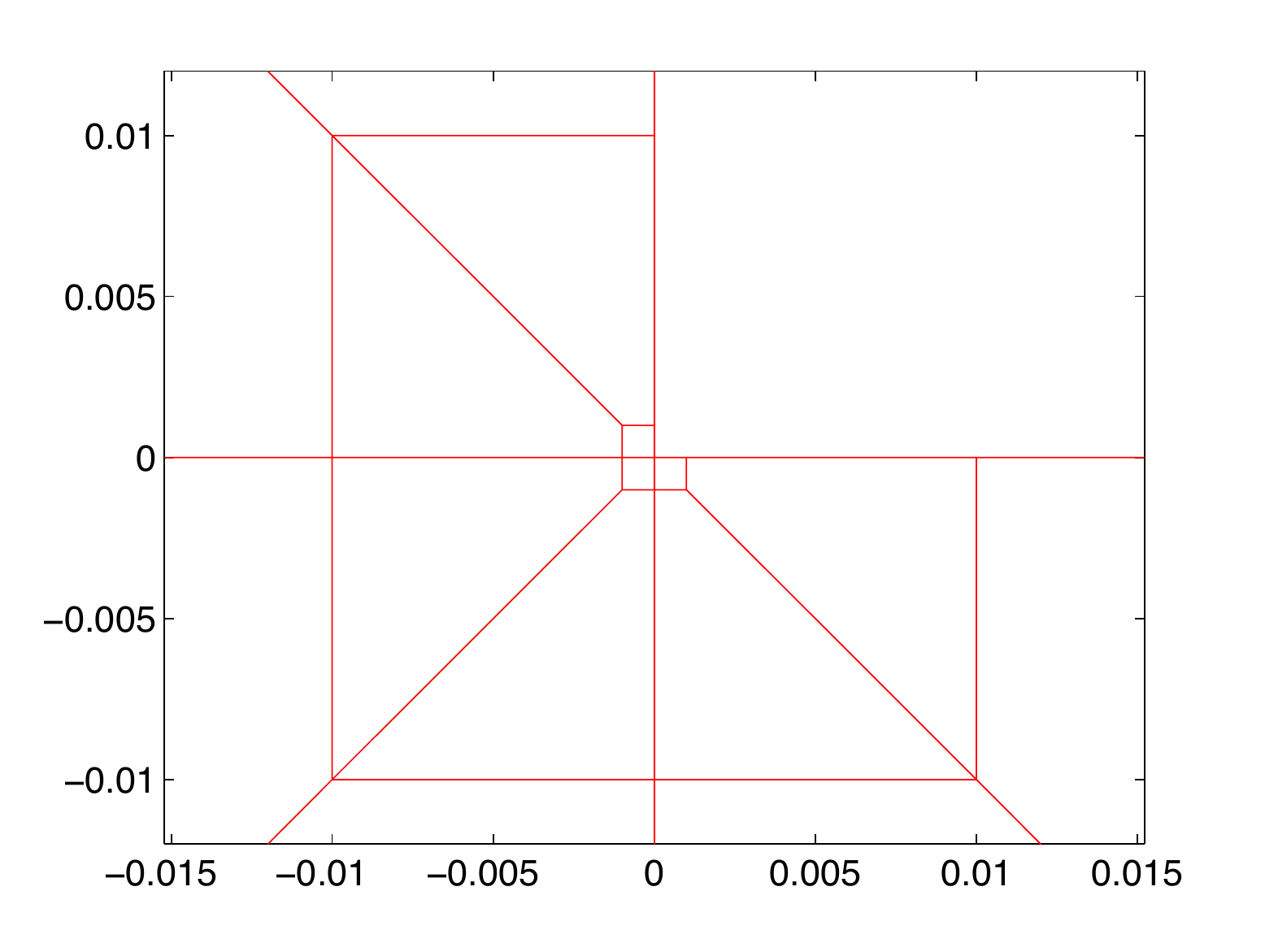}
   \caption{Mesh of $\Gamma_1$ with 4 layers and ratio $0.1$, and zoom by factors 10 and 100.}
\label{fig:mesh}
\end{figure}

These computations allow to find numerical upper and lower bounds for the first Rayleigh quotient $\lambda_1(\Lambda)$ of the Fichera layer.

\subsubsection{Upper bound}\label{subsub:ub} To evaluate $\lambda_\infty$, we compute the first eigenvalue of the Laplacian on $\Gamma_R$ either with Dirichlet or Neumann conditions on $\Sigma_R$, (and still Dirichlet conditions on $\partial\Gamma_R\setminus\Sigma_R$). Denote by $\lambda^\Dir_1(\Gamma_R)$ and $\lambda^\Mix_1(\Gamma_R)\equiv\lambda_R$, respectively, these eigenvalues.

We perform the computations for a sample of values of $R$. The $\log$ of their difference is plotted in Figure \ref{fig:conv}, together with its slope with respect to $R$. The slope converges to a number $-\alpha$, which means the exponential convergence
\begin{equation}
\label{eq:num}
   \lambda^\Dir_1(\Gamma_R)-\lambda^\Mix_1(\Gamma_R) \sim e^{-\alpha R}\quad\mbox{with}\quad \alpha\simeq1.672782.
\end{equation}
As an approximation of $\lambda_\infty$, we take the mean value of $\lambda^\Dir_1(\Gamma_R)$ and $\lambda^\Mix_1(\Gamma_R)$ for $R=10$
\[
   \tfrac12\big(\lambda^\Dir_1(\Gamma_{10})+\lambda^\Mix_1(\Gamma_{10})\big) \simeq 0.9291205\,\pi^2:=\widetilde\lambda_\infty.
\]
We notice that, with this numerical value, the theoretical exponential convergence $\simeq e^{-2R\omega}$ stated by Corollary \ref{cor:conv} becomes
\[
   e^{-2R\widetilde\omega}\quad\mbox{with}\quad 2\widetilde\omega = 2\pi\sqrt{1-0.9291205} \simeq 1.672785,
\]
which is very close to the observed $\alpha$ in \eqref{eq:num}.

\begin{figure}[ht]
\includegraphics[scale=0.52]{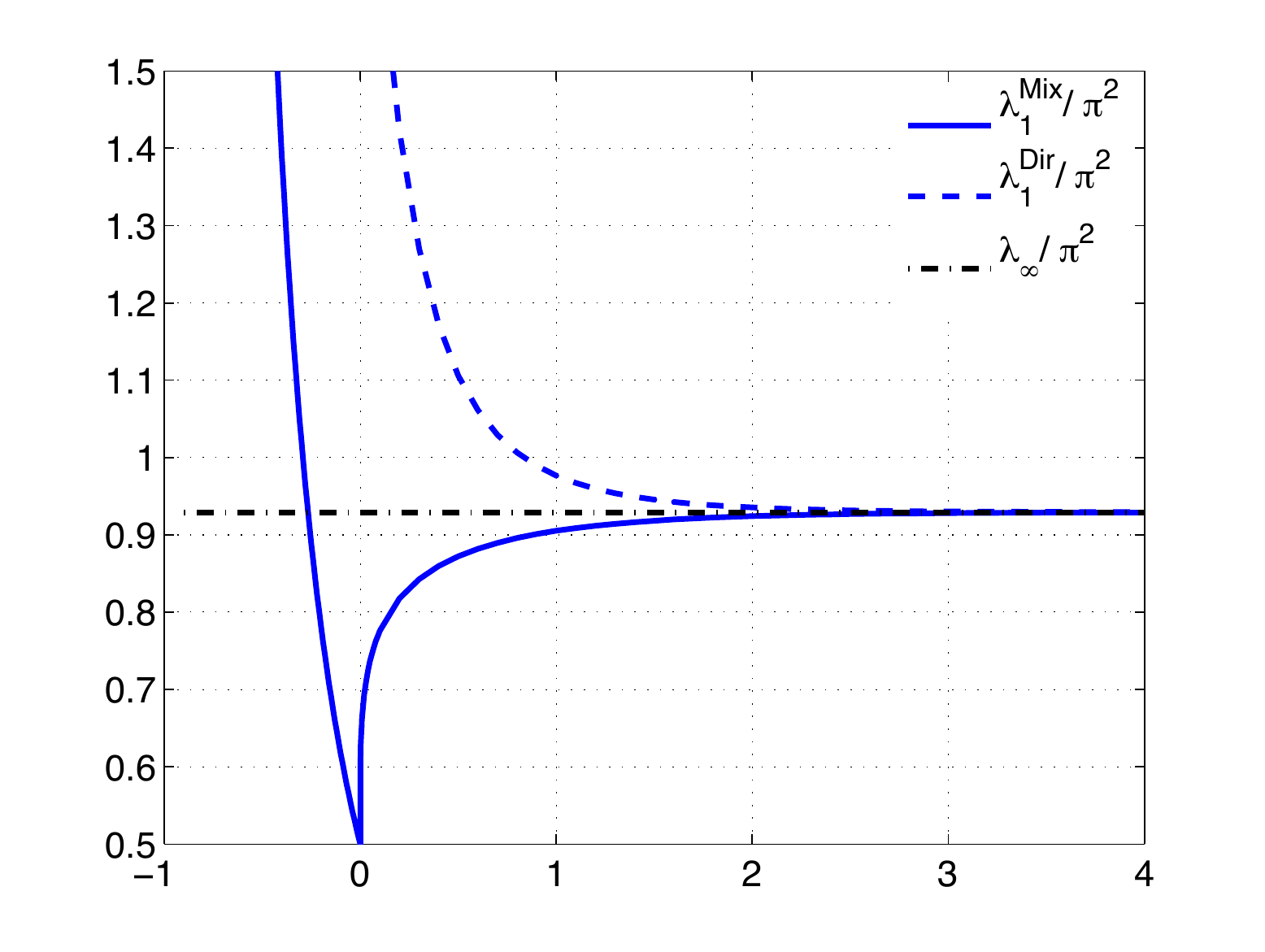}
\includegraphics[scale=0.52]{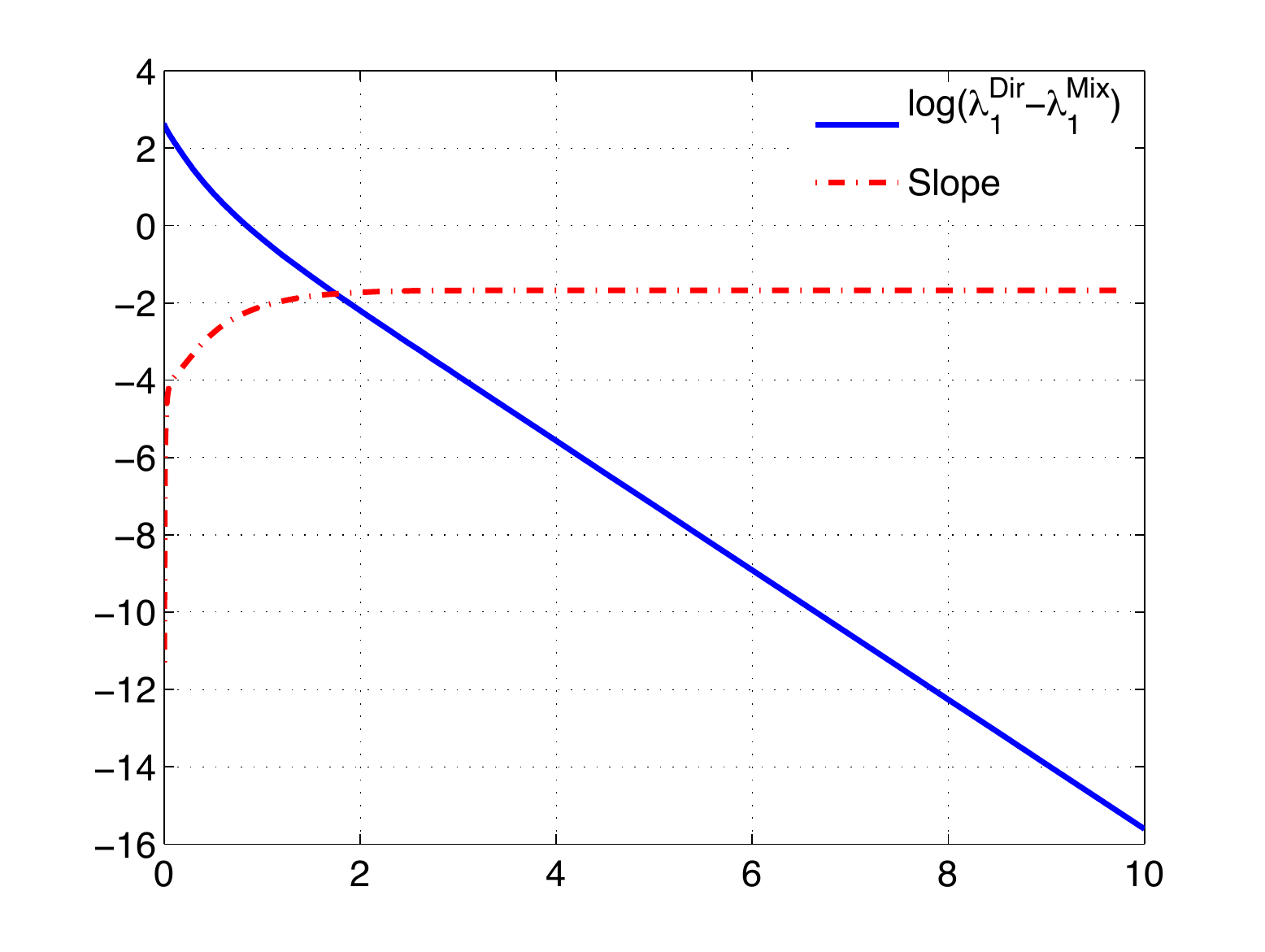}
   \caption{Computed eigenvalues $\lambda^\Dir_1(\Gamma_R)$ and $\lambda^\Mix_1(\Gamma_R)$ as functions of $R$ (left), $\log$ of difference $\lambda^\Dir_1(\Gamma_R)-\lambda^\Mix_1(\Gamma_R)$ and slope (right).}
\label{fig:conv}
\end{figure}

We also notice that the value $\widetilde\lambda_\infty/\pi^2=0.9291205$ is smaller, thus more precise, than the value $0.92934$ computed in \cite{DaLaRa11}: here the strongly refined mesh and the polynomials of degree 16 capture the corner singularity of the eigenvectors more effectively than the uniform triangular mesh with polynomials of degree 6 used in \cite{DaLaRa11}.

\subsubsection{Lower bound} Computing $\lambda^\Mix_1(\Gamma_R)\equiv\lambda(R)$ for a sufficiently dense sample of values of $R$ allows to evaluate in turn the first eigenvalue $\mu$ of the Sturm Liouville operator \eqref{eq:VL} $\cV_L:q\mapsto -q''+\lambda q$ for a sample of values of $L$, as shown in Figure \ref{fig:lammu}: here the computation is performed with polynomials of degree $10$ on the interval $[L,40]$ with Neumann conditions at the two ends. This yields an evaluation of the quantities $L^*$ and $\mu(L^*)$ defined in Proposition \ref{prop:L*}:
\[
   L^* \simeq -0.228\quad\mbox{and}\quad \mu(L^*) \simeq 0.838653\,\pi^2.
\]
Thus the first Rayleigh quotient $\lambda_1(\Lambda)$ of the Fichera layer satisfies
\begin{equation}
\label{eq:encad}
   0.838653\,\pi^2 \le \lambda_1(\Lambda) \le 0.9291205\,\pi^2.
\end{equation}

\subsection{Finite Fichera layers}\label{subsec:FFiclay}
We compute the first eigenvalues of the bilinear form $\cQ=\nabla\cdot\nabla$ on finite layers $\Lambda_R$ with Dirichlet or Neumann boundary conditions on the part $\Sigma_R$ of the boundary of $\Lambda_R$ (and Dirichlet on $\partial\Lambda_R\cap\partial\Lambda$ in both cases). Let us denote them by $\lambda^\Dir_\ell(\Lambda_R)$ and $\lambda^\Mix_\ell(\Lambda_R)$, respectively. We have the inequalities
\begin{equation}
\label{eq:ineq3d}
   \lambda^\Mix_\ell(\Lambda_R) \le \lambda^\Dir_\ell(\Lambda_R)\quad\mbox{and}\quad
   \lambda^\Dir_\ell(\Lambda) \le \lambda^\Dir_\ell(\Lambda_R),\quad
   \forall j\ge1,\quad\forall R>-1.
\end{equation}

\begin{figure}[ht]
\includegraphics[scale=0.52]{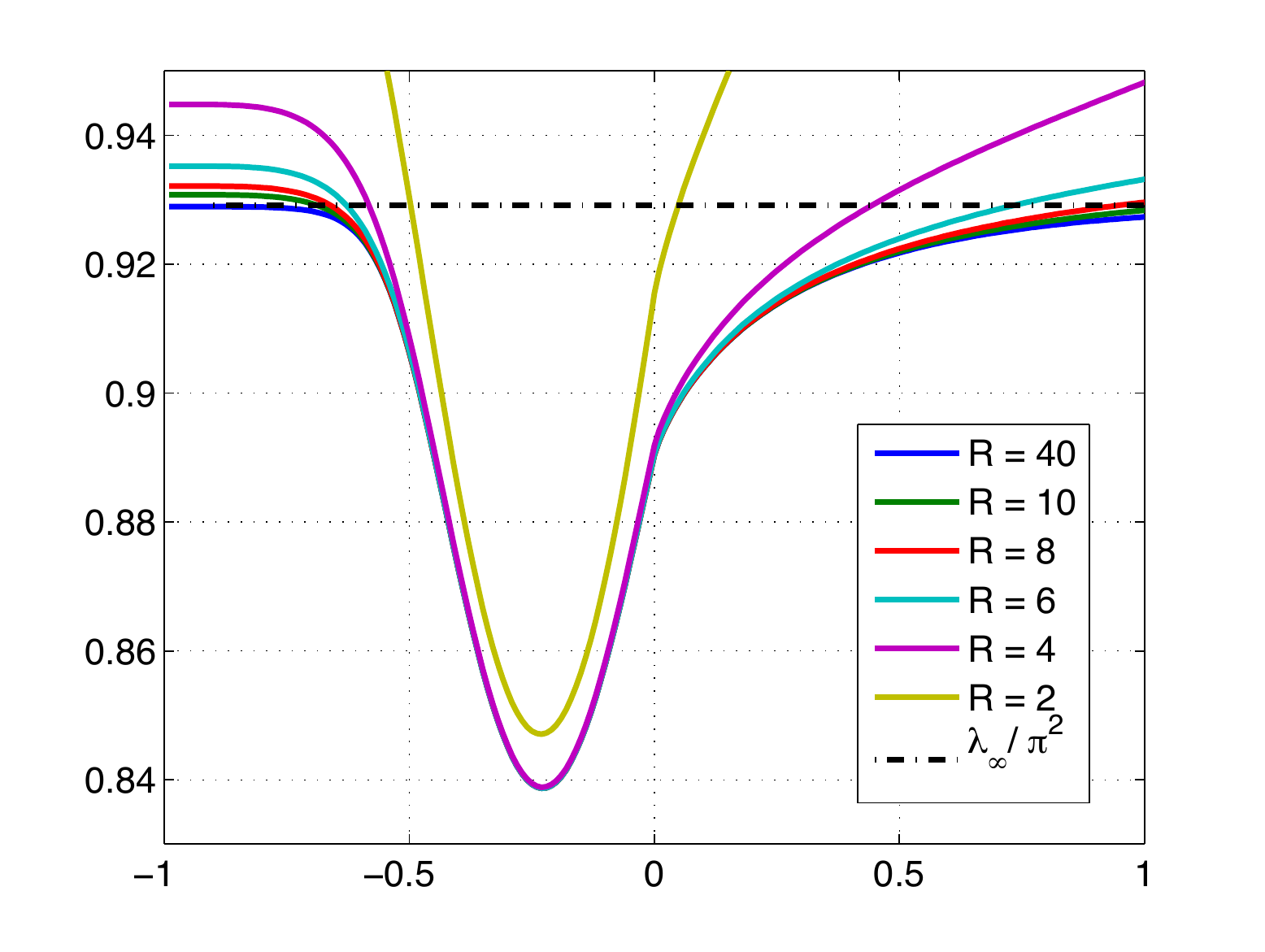}
   \caption{First eigenvalues $\mu(L,R)$ (divided by $\pi^2$) of Sturm-Liouville operators $\cV_{L,R}$ as functions of $L$ for $R=2,4,6,8,10$ and $40$.}
\label{fig:mu}
\end{figure}
Beforehand, we compute the first eigenvalue $\mu(L,R)$ of the Sturm-Liouville operators $\cV_{L,R}:\: q\mapsto -q''+\lambda q$ on the interval $(L,R)$ with Neumann condition at $L$ and Dirichlet condition at $R$. By a reasoning similar to Lemma \ref{lem:mu1} we find the lower bound for $\lambda^\Dir_1(\Lambda_R)$
\begin{equation}
\label{eq:muLR}
      \lambda^\Dir_1(\Lambda_R) \ge \min_{L\,\in\,(-1,R)} \mu(L,R)\,.
\end{equation}
From Figure \ref{fig:mu}, we see that 
the value of $R$ has a very slight influence on $\min_{L}\mu(L,R)$ as soon as $R\ge4$. Moreover, the values for $R=40$ are very close to those presented in Figure \ref{fig:lammu} for which the Neumann condition at $R=40$ was imposed.

We compute the first three eigenvalues $\lambda^\Dir_\ell(\Lambda_R)$, $j=1,2,3$, for $R$ ranging from $2$ to $10$, and $\lambda^\Mix_\ell(\Lambda_R)$ for the same values of $R$ for comparison sake, see Figure \ref{fig:3D}. We use three structured tensor hexahedral grids $\gG_1$, $\gG_2$, and $\gG_3$ on the finite Fichera layers $\Lambda_R$ (see details in Appendix \ref{sec:app1}) combined with the interpolation degrees $p=4$, $2$ and $1$, respectively. In all these configurations, the number of degrees of freedom is 20293 when $R=2$ and 36829 for larger values of $R$.
\begin{figure}[ht]
\includegraphics[scale=0.52]{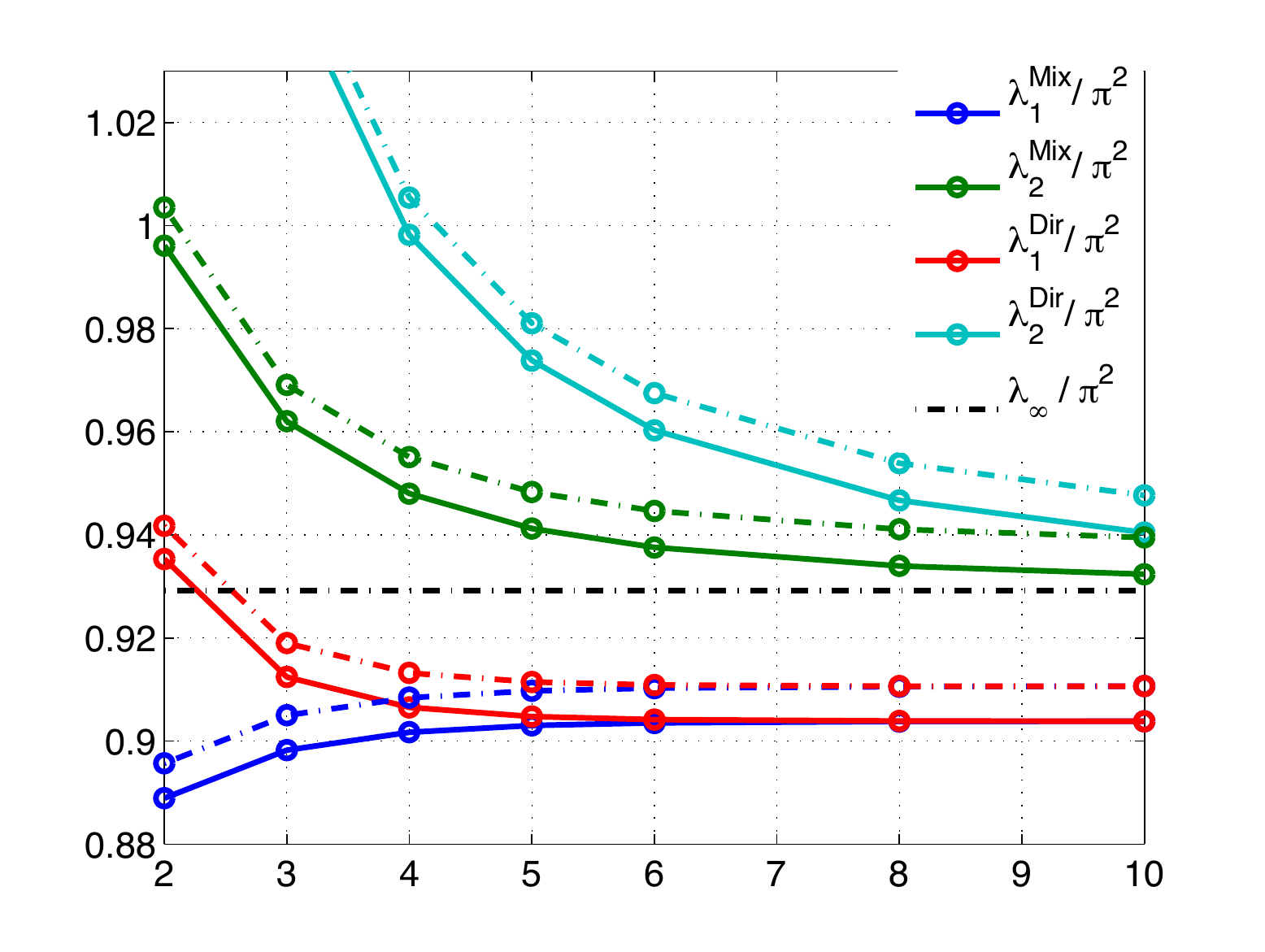}
\includegraphics[scale=0.52]{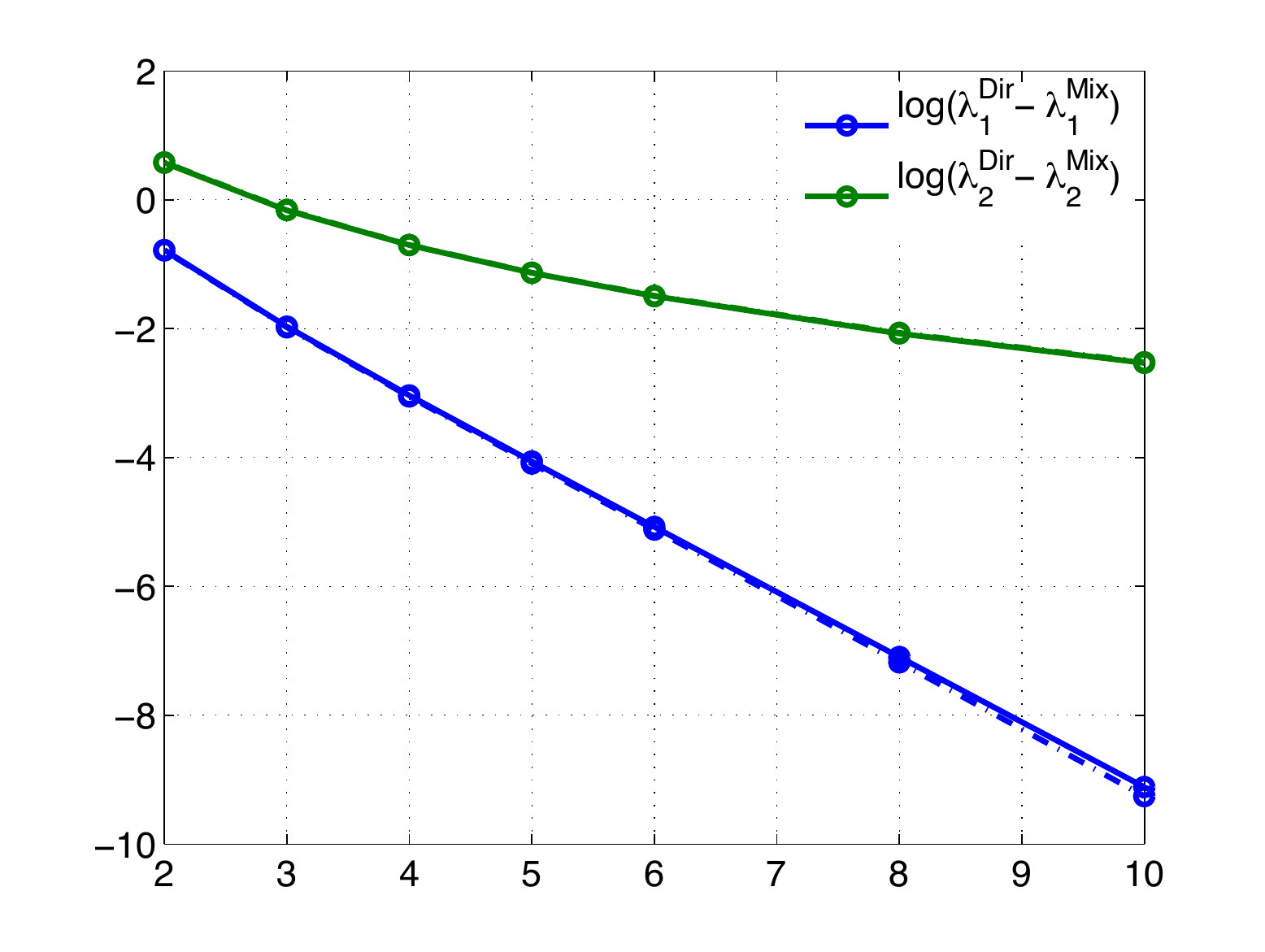}
   \caption{Computed eigenvalues $\lambda^\Dir_\ell(\Lambda_R)$ and $\lambda^\Mix_\ell(\Lambda_R)$ as functions of $R$ (left), $\log$ of difference $\lambda^\Dir_\ell(\Lambda_R)-\lambda^\Mix_\ell(\Lambda_R)$ (right). Solid lines: $p=4$ on $\gG_1$, 
dash and dots: $p=1$ on $\gG_3$. The results with $p=2$ on $\gG_2$ are very close to $p=4$ on $\gG_1$ and are not plotted.}
\label{fig:3D}
\end{figure}

We do not plot the third eigenvalue because it is identical (within 13 digits) to the second one. We notice the numerical evidence of exactly one discrete eigenvalue under the essential spectrum of $\cL_\Lambda$:
\begin{enumerate}[{\em i)}]
\item As soon as $R\ge3$, both $\lambda^\Dir_1(\Lambda_R)$ and $\lambda^\Mix_1(\Lambda_R)$ are smaller than $\lambda_\infty$
\item The difference $\lambda^\Dir_1(\Lambda_R)-\lambda^\Mix_1(\Lambda_R)$ converges exponentially to $0$ with respect to $R$.
\item $\lambda^\Dir_2(\Lambda_R)$ and $\lambda^\Mix_2(\Lambda_R)$ tend to $\lambda_\infty$ by superior values.
\item The difference $\lambda^\Dir_2(\Lambda_R)-\lambda^\Mix_2(\Lambda_R)$ converges slowly to $0$.
\end{enumerate}
The observed convergence rate $\beta$ such that $\lambda^\Dir_1(\Lambda_R)-\lambda^\Mix_1(\Lambda_R)\simeq e^{-\beta R}$ is slightly larger than $1$. Extrapolating from computations on the same grid with degrees $p$ from 1 to 8 yields the estimate
\[
   0.9031\pi^2 < \lambda_1(\Lambda) < 0.9033\pi^2.
\]
A theoretical Agmon-type exponential decay estimate is $e^{-2\gamma R}$ with $\gamma=\sqrt{\lambda_\infty-\lambda_1(\Lambda)}$. With the numerical values $0.9291\pi^2$ for $\lambda_\infty$ and $0.9033\pi^2$ for $\lambda_1(\Lambda)$, we find $\gamma\simeq0.5046$, which is coherent with $2\gamma\simeq\beta\simeq1$. 

The relative gap between the bound state and the bottom of the essential spectrum is
\begin{equation}
\label{eq:gapLD}
   g(\Lambda) := \frac{\lambda_1(\Gamma)-\lambda_1(\Lambda)}{\lambda_1(\Lambda)} \simeq  
   \frac{0.9291-0.9032}{0.9032} \simeq 0.029\,.
\end{equation}

We represent in Figure \ref{fig:eigv} slices of the eigenvectors associated with the first and second eigenvalues $\lambda^\Mix_1(\Lambda_R)$ and $\lambda^\Mix_2(\Lambda_R)$ for $R=4$. We slice the domain $\Lambda_R$ by the plane $\cP$ of equation $x_1=x_2$. The first eigenvector is concentrated near the origin and is close to the restriction of the first eigenvector on the infinite Fichera layer $\Lambda$. We can see that the decay is slower along the edge $x_1=x_2=0$ (vertical leg in the figure) than in the horizontal leg that goes away from the edges. The second eigenvector is a manifestation of the essential spectrum and concentrates along the edge.

\begin{figure}[ht]
\includegraphics[scale=0.64]{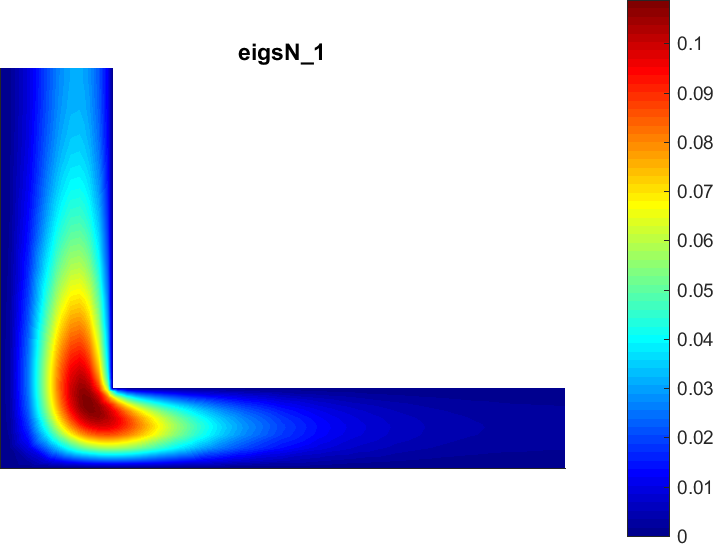}\qquad
\includegraphics[scale=0.64]{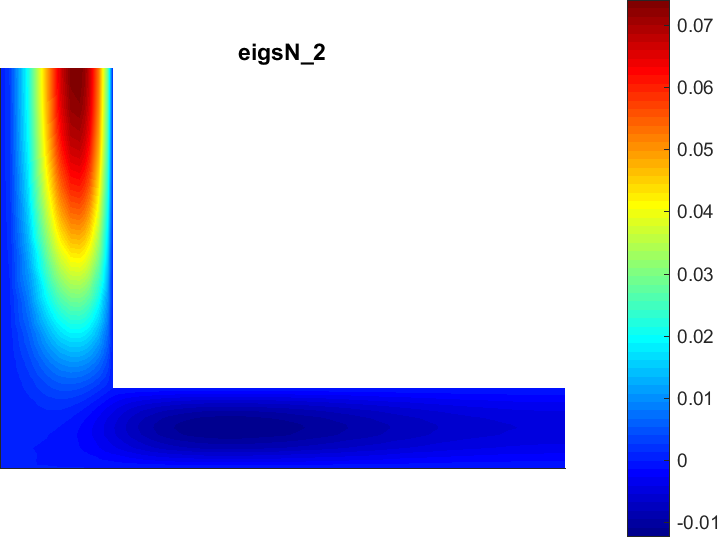}
   \caption{Slices of eigenvectors associated with $\lambda^\Mix_1(\Lambda_4)$ (left) and $\lambda^\Mix_2(\Lambda_4)$ (right) through the plane $x_1=x_2$. The vertical leg has length $4$ and width $\sqrt{2}$ whereas the horizontal leg has length $4\sqrt{2}$ and width $1$.}
\label{fig:eigv}
\end{figure}

\section{Two extensions}\label{sec:ext}

In this section we discuss two extensions of Theorem \ref{thm:main}. The first one is about the layer with rounded edges $\sLs$ and the second one concerns the three-dimensional cross $\sY$, both introduced in \S \ref{sec:mainres}. A similar strategy could be applied to investigate the layer $\sLf$ (also introduced in \S \ref{sec:mainres}).

The theoretical investigation of $\sLs$ and $\sY$ follows the same lines as for the Fichera layer $\Lambda$. Let us bring out the main points of the rationale of our proof:

\begin{enumerate}[(a)]
	\item There are two-dimensional waveguides canonically associated with $\sLs$ and $\sY$. The Dirichlet spectrum in these waveguides has the same structure as in Theorem \ref{th:GR}: the essential spectrum coincides with $[\pi^2,+\infty)$ and there is a (unique) bound state, as in Theorem \ref{th:GR}.
	\item\label{it:b} The eigenvalues of the truncated guides converge exponentially to the one on the entire waveguide, as in Corollary \ref{cor:conv}.
	\item If they exist, the eigenfunctions in $\sLs$ or $\sY$ have symmetry properties, which allows to reduce the problem to truncated layers around each edge as in Corollary \ref{cor:L3}. 
	\item The Born-Oppenheimer strategy of \S \ref{subsec:finite} reduces the study to a Sturm-Liouville operator with an adequate potential which has a finite number of bound states.
\end{enumerate}
To keep the discussion concise and avoid redundancies, we choose to skip proof details for $\sLs$ and $\sY$ and focus on numerical results.

\subsection{Fichera layer with exterior rounded edges}

Let us start by defining the geometrical objects we are interested in. Recall that the surface $\sS^0$ is defined as $\{\bx \in\R^3: \min\{x_1,x_2,x_3\} = 0\}$. Note that it coincides with the boundary $\partial\R_+^3$ of the first octant $\R^3_+$. The layer $\sLs$ is
$$
	\sLs = \{\bx \in \R^3\setminus\overline{\R^3_+}: 
	\mathsf{dist}(\bx,\partial\R_+^3) < 1\}.
$$
Its two-dimensional analogue is the guide $\sGs$ (see Figure \ref{fig:round}, right) that is defined as
$$
	\sGs = \{\bx \in \R^2\setminus\overline{\R_+^2}: \mathsf{dist}(\bx,\partial\R_+^2) < 1\}.
$$

\begin{theorem}\label{thm:guideround} With the broken guide $\sGs$ and $\cL_{\sGs}$ the positive Dirichlet Laplacian on $\sGs$, there holds:
\begin{enumerate}[i)]
	\item The essential spectrum of $\cL_{\sGs}$ coincides with $[\pi^2,+\infty)$;
	\item The operator $\cL_{\sGs}$ has exactly one eigenvalue under its essential spectrum denoted by $\lambda_1(\sGs)$.
\end{enumerate}
\end{theorem}
Remark that Theorem \ref{thm:guideround} is not a direct consequence of \cite{DE95} because the hypothesis of \emph{loc. cit.} on the curvature is not satisfied. For completeness, we present a proof in Appendix \ref{sec:app2}.

The analogue of Theorem \ref{thm:main}, stated for the layer $\sLs$, reads as follows:
\begin{theorem}\label{thm:mainround} Let $\cL_{\sLs}$ be the positive Dirichlet Laplacian on $\sLs$. There holds:
	\begin{enumerate}[i)]
		\item The essential spectrum of $\cL_{\sLs}$ coincides with $[\lambda_1(\sGs),+\infty)$;
		\item $\cL_{\sLs}$ has at most a finite number of eigenvalues under its essential spectrum.
	\end{enumerate}
\end{theorem}

We present now in \S \ref{subsub:2dround} computations supporting point (\ref{it:b}) above (exponential convergence in two-dimensional finite guides). Next in \S \ref{sss:3DRoundFiclay}, we give numerical evidence about the existence of exactly one bound state for $\sLs$.

\subsubsection{Two-dimensional rounded guides}\label{subsub:2dround}
For $R\geq0$, we define the finite broken guide $\sGs_R$ like we did in \eqref{eq:GR}
$$
	\sGs_R = \sGs \cap \Box_{R+1}.
$$
Note that the part of its boundary at ``distance" $R$, i.e.\  $\partial\sGs \cap \partial \Box_{R+1}$, coincides with $\Sigma_R$ in \eqref{eqn:defsigmaR}.

Like in Section \ref{ss:GR}, $\lambda_1^\Dir(\sGs_R)$ and $\lambda_1^\Mix(\sGs_R)$ denote the first eigenvalues of the Laplace operator in $\sGs_R$ with Dirichlet and mixed ($\partial_{\sN} \sGs_R = \Sigma_R$) boundary conditions, respectively.
The counterpart of Corollary \ref{cor:conv} for the guide $\sGs$ writes as follows.

\begin{corollary}\label{cor:conv2} With the positive number $\omega^\sharp = \sqrt{\pi^2 - \lambda_1(\sGs)}$, there exists a constant $C^\sharp$ such that
$$
	\forall R\geq 1,\quad 0\leq \lambda_1(\sGs) - \lambda_1^\Mix(\sGs_R) \leq C^\sharp e^{-2R\omega^\sharp}.
$$
\end{corollary}
Similarly as in \S \ref{subsub:ub}, to evaluate ${\lambda}_1(\sGs)$ we compute $\lambda_1^\Dir(\sGs_R)$ and $\lambda_1^\Mix(\sGs_R)$ for a sample of values of $R$ (see Figure \ref{fig:convcb}, left). 
We observe the numerical exponential convergence
\begin{equation}\label{eqn:slope2}
	\lambda_1^\Dir(\sGs_R) - 	\lambda_1^\Mix(\sGs_R) \sim e^{-\alpha R}\quad\text{with}\quad \alpha = 0.7293.
\end{equation}
As an approximation of $\lambda_1(\sGs)$, we take the mean value of $\lambda^\Dir_1(\sGs_R)$ and $\lambda^\Mix_1(\sGs_R)$ for $R=12$:
\begin{equation}
\label{eq:linftysharp}
	\lambda_1(\sGs) \simeq 0.9865\pi^2.
\end{equation}
With this numerical value, the convergence rate expected in Corollary \ref{cor:conv2} becomes
$e^{-2R\widetilde{\omega}^\sharp}$ with $2\widetilde{\omega}^\sharp = 0.7294$.
This is consistent with the observed slope given in \eqref{eqn:slope2}.

\begin{figure}[ht]
\includegraphics[scale=0.52]{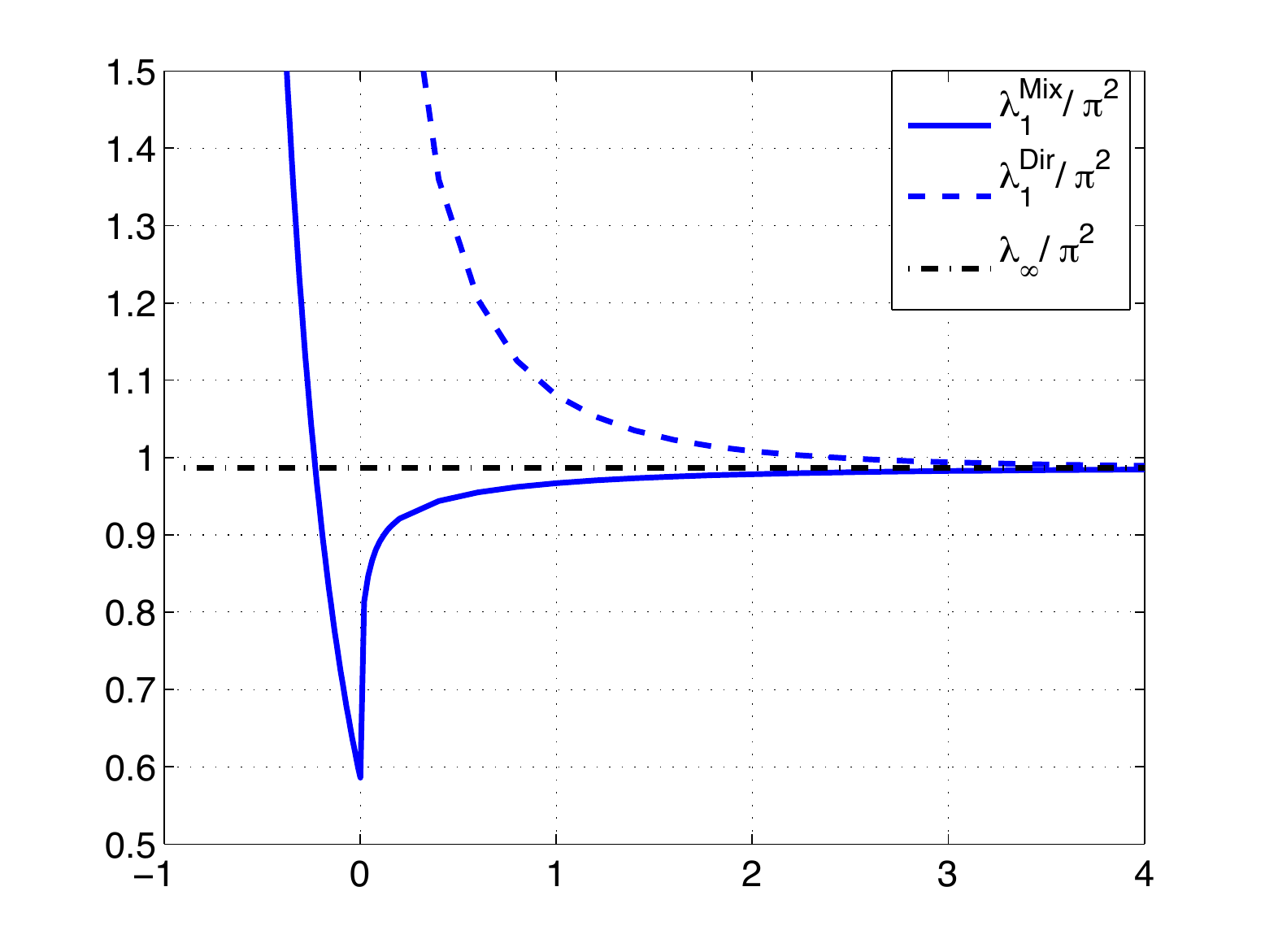}
\includegraphics[scale=0.52]{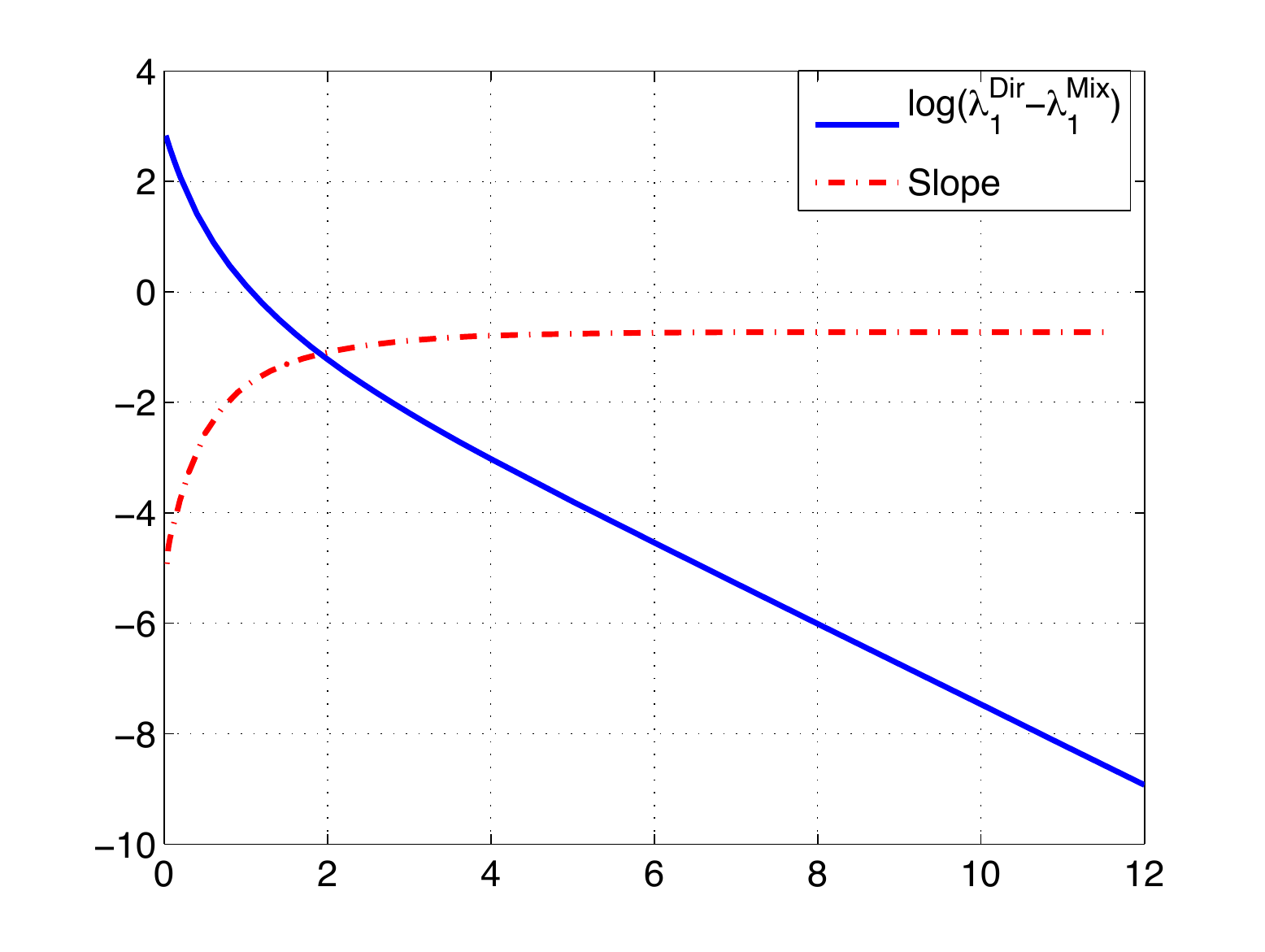}
   \caption{Computed eigenvalues $\lambda^\Dir_1(\sGs_R)$ and $\lambda^\Mix_1(\sGs_R)$ as functions of $R$ (left), $\log$ of difference $\lambda^\Dir_1(\sGs_R)-\lambda^\Mix_1(\sGs_R)$ and slope (right).}
\label{fig:convcb}
\end{figure}

\subsubsection{Three-dimensional rounded layers}\label{sss:3DRoundFiclay}
For $R\geq0$, we define the finite layers like in \eqref{eq:LR}
$$
	\sLs_R := \sLs \cap \Box_{R+1},
$$
To find approximate Rayleigh quotients of $\cL_{\sLs}$, we compute the first two eigenvalues $\lambda_\ell^\Dir(\sLs_R)$ and $\lambda_\ell^\Mix(\sLs_R)$ ($\ell=1,2$), of the Laplacian on $\sLs_R$, with Dirichlet boundary conditions on $\partial\sLs\cap\partial\sLs_R$ and Dirichlet or Neumann conditions in what remains of the boundary. 
In Figure \ref{fig:3Dcb}, we plot the results for $R$ ranging from $4$ to $16$.
We used a tetrahedral mesh, see Figure \ref{fig:mesh3DFL}, and an interpolation of degree $6$.
Like for the Fichera layer (see \S \ref{subsec:FFiclay}) the existence of a unique bound state is clearly exhibited. We find the value
$$
	\lambda_1(\sLs)\simeq 0.9817\pi^2.
$$
Using \eqref{eq:linftysharp}, we know that the threshold of the essential spectrum for the layer $\sLs$ is given by $\lambda_1(\sGs) \simeq 0.9865\pi^2$. Thus the relative gap between the bound state and the bottom of the essential spectrum is
\begin{equation}
\label{eq:gapLC}
   g(\sLs) := \frac{\lambda_1(\sGs)-\lambda_1(\sLs)}{\lambda_1(\sLs)} \simeq  
   \frac{0.9865-0.9817}{0.9817} \simeq 0.0049\,.
\end{equation}
Compared to \eqref{eq:gapLD}, there is less room for a bound state to exist than in the Fichera layer $\Lambda$.
\begin{figure}[ht]
\includegraphics[scale=0.52]{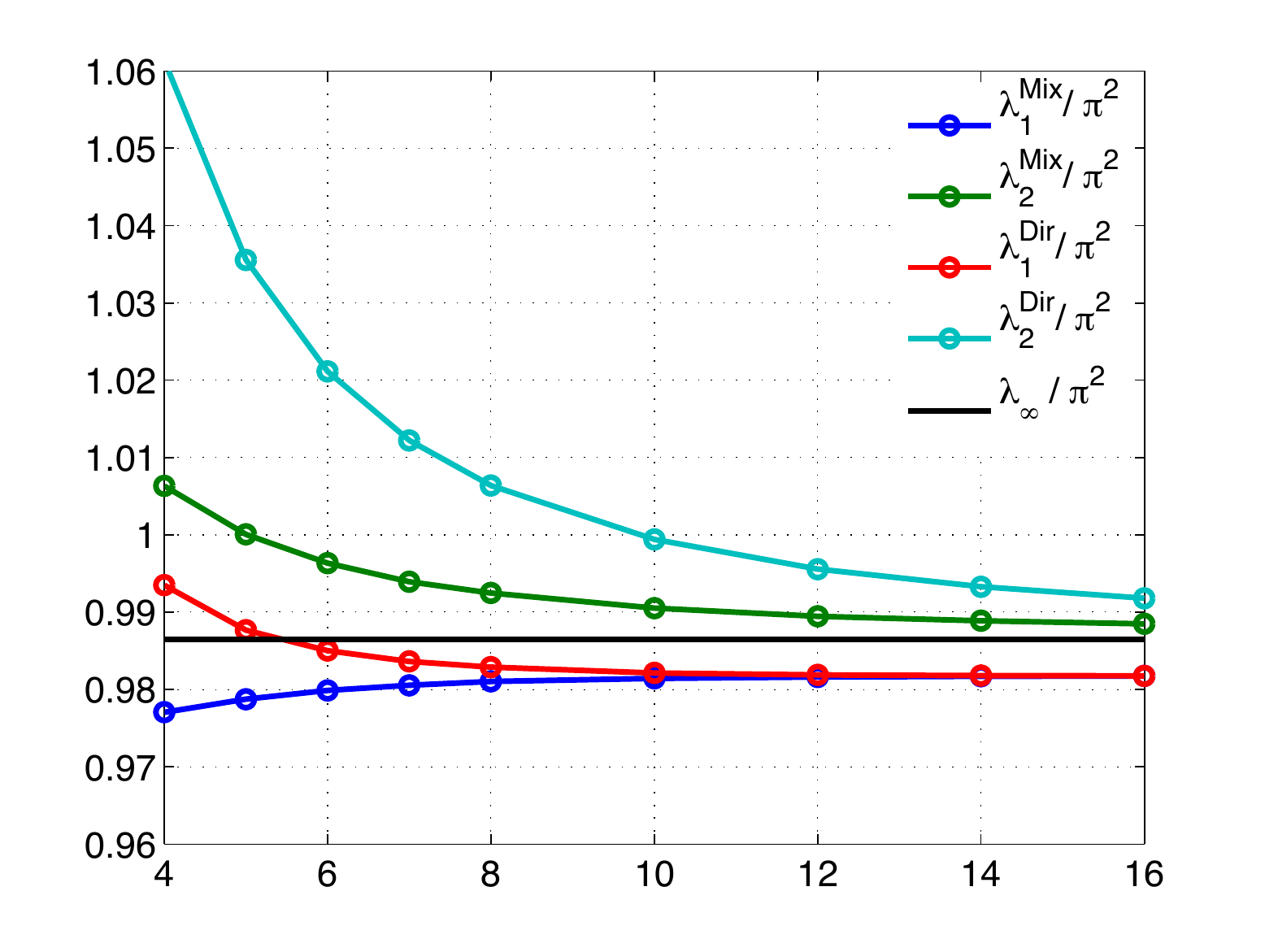}
\includegraphics[scale=0.52]{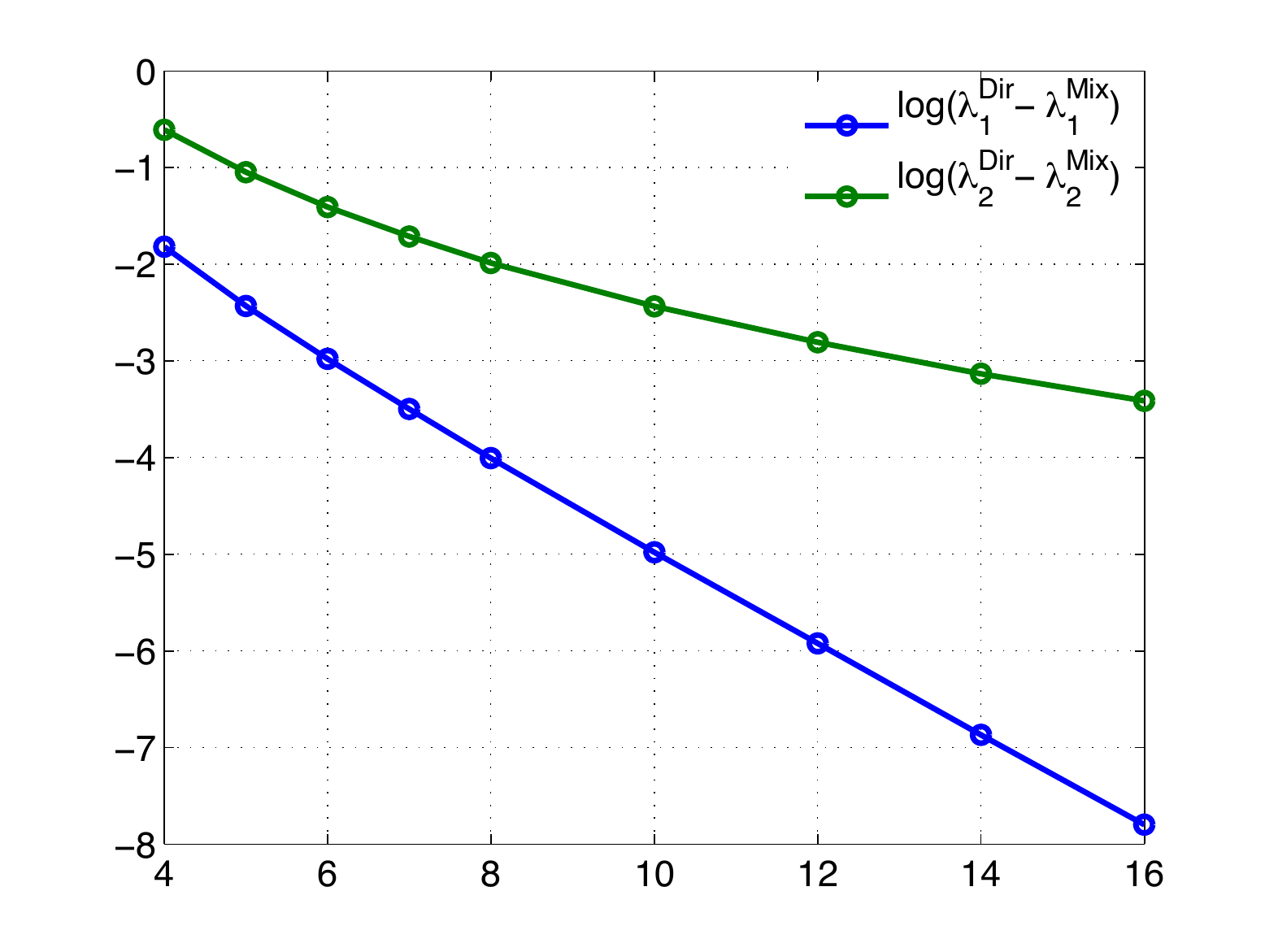}
   \caption{Computed eigenvalues $\lambda^\Dir_\ell(\sLs_R)$ and $\lambda^\Mix_\ell(\sLs_R)$ as functions of $R$ (left), $\log$ of difference $\lambda^\Dir_\ell(\sLs_R)-\lambda^\Mix_\ell(\sLs_R)$ (right), for $\ell=1,2$.}
\label{fig:3Dcb}
\end{figure}

\subsection{Three-dimensional cross}\label{subsec:3DX}
Recall that the cross waveguide $\sX$ and the three-dimensional analogue $\sY$ are defined in \eqref{eqn:defXY}. The following result is 
known, see \cite[\S 3.4]{Pa17}.
\begin{theorem}[\cite{Pa17}] Let $\cL_\sX$ be the Dirichlet Laplacian in $\sX$. There holds:
\begin{enumerate}[i)]
	\item The essential spectrum of $\cL_\sX$ coincides with $[\pi^2,+\infty)$;
	\item $\cL_\sX$ has exactly one eigenvalue under its essential spectrum, denoted by $\lambda_1(\sX)$.
\end{enumerate}
\end{theorem}

The analogue of Theorem \ref{thm:main} reads as follow.
\begin{theorem} Let $\cL_\sY$ be the Dirichlet Laplacian in $\sY$. There holds:
	\begin{enumerate}[i)]
		\item The essential spectrum of $\cL_\sY$ coincides with $[\lambda_1(\sX),+\infty)$;
		\item $\cL_\sY$ has at most a finite number of eigenvalues under its essential spectrum.
	\end{enumerate}
\end{theorem}

The following lemma can be proved adapting Lemma \ref{lem:sym}. It implies that, up to a scaling factor,  calculations for $\sX$ and $\sY$ reduce to calculations on $\Gamma$ and $\Lambda$, respectively, see Corollary \ref{cor:X}.

\begin{lemma}\label{lem:redu-DNcross} For $d=2,3$ and $j\in\{1,\dots,d\}$, we define $\sP^j:=\{\bx\in\R^d : x_j = -1/2\}$. Then:
\begin{enumerate}[i)]
\item For $d=2$, an eigenfunction associated with $\lambda_1(\sX)$ satisfies $\partial_nu = 0$ on $\sP^j$ ($j=1,2$),
\item For $d=3$, an eigenfunction associated with a bound state of $\cL_\sY$ satisfies $\partial_nu = 0$ on $\sP^j$ ($j=1,2,3$).
\end{enumerate}
\end{lemma}

Let us introduce the scaled versions of the broken guide $\Gamma$ and the Fichera layer $\Lambda$ as
$$
	\widehat{\Gamma} := \tfrac12\, \Gamma \quad\mbox{and}\quad \widehat{\Lambda} := \tfrac12\, \Lambda.
$$
We consider Dirichlet boundary conditions on $\partial\Gamma\cap\partial\widehat{\Gamma}$ and $\partial\Lambda\cap\partial\widehat{\Lambda}$, and Neumann on the remaining part of the boundary. The Rayleigh quotients of the corresponding positive two-dimensional and three-dimensional Laplacians $\cL_{\widehat{\Gamma}}$ and $\cL_{\widehat{\Lambda}}$  are denoted by $\lambda_\ell(\widehat{\Gamma})$ and $\lambda_\ell(\widehat{\Lambda})$, respectively.

A consequence of Lemma \ref{lem:redu-DNcross} is the following corollary, reminiscent of Corollary \ref{cor:L3}, that can be proved using symmetries of the eigenfunctions.
\begin{corollary} 
\label{cor:X}
The following holds:
	\begin{enumerate}[i)]
		\item We have $\lambda_1(\sX) = \lambda_1(\widehat{\Gamma})$.
		\item We have $\sigma_\dis(\cL_\sY) = \sigma_\dis(\cL_{\widehat{\Lambda}})$.
	\end{enumerate}
\end{corollary}

The bounded versions of $\widehat{\Gamma}$ and $\widehat{\Lambda}$ are defined as $\widehat{\Gamma}_R=\widehat{\Gamma}\cap\Box_{R+1}$ and $\widehat{\Lambda}_R=\widehat{\Lambda}\cap\Box_{R+1}$. Boundary conditions on $\partial\widehat{\Gamma}_R$ and $\partial\widehat{\Lambda}_R$ on the common part with $\widehat{\Gamma}$ and $\widehat{\Lambda}$ are the same as mentioned above, whereas on the remaining part of their boundaries, we take Dirichlet or Neumann, thus defining $\lambda_\ell^\Dir(\widehat{\Gamma}_R)$, $\lambda_\ell^\Mix(\widehat{\Gamma}_R)$, and $\lambda_\ell^\Dir(\widehat{\Lambda}_R)$, $\lambda_\ell^\Mix(\widehat{\Lambda}_R)$.

\begin{figure}[ht]
\includegraphics[scale=0.52]{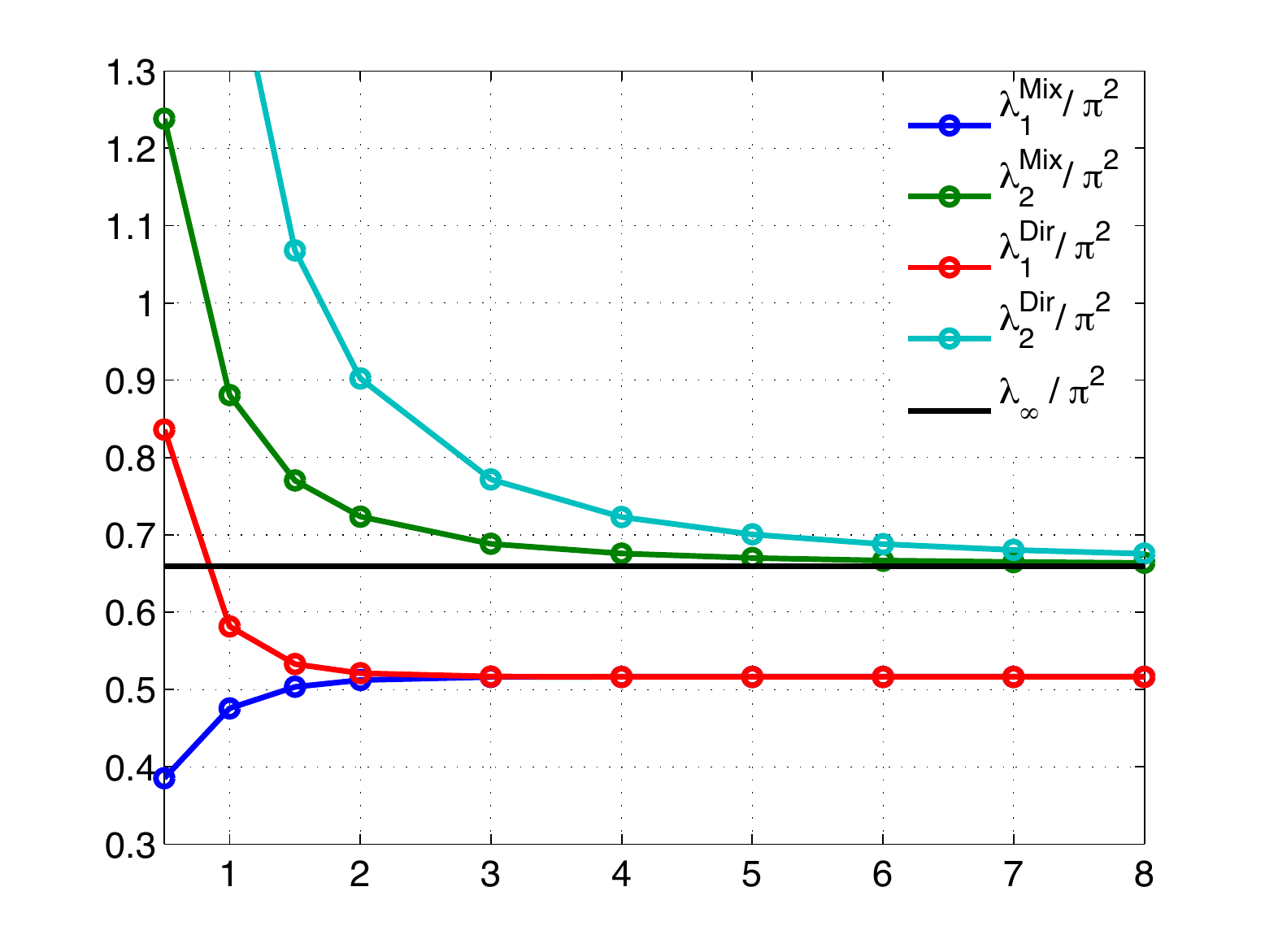}
\includegraphics[scale=0.52]{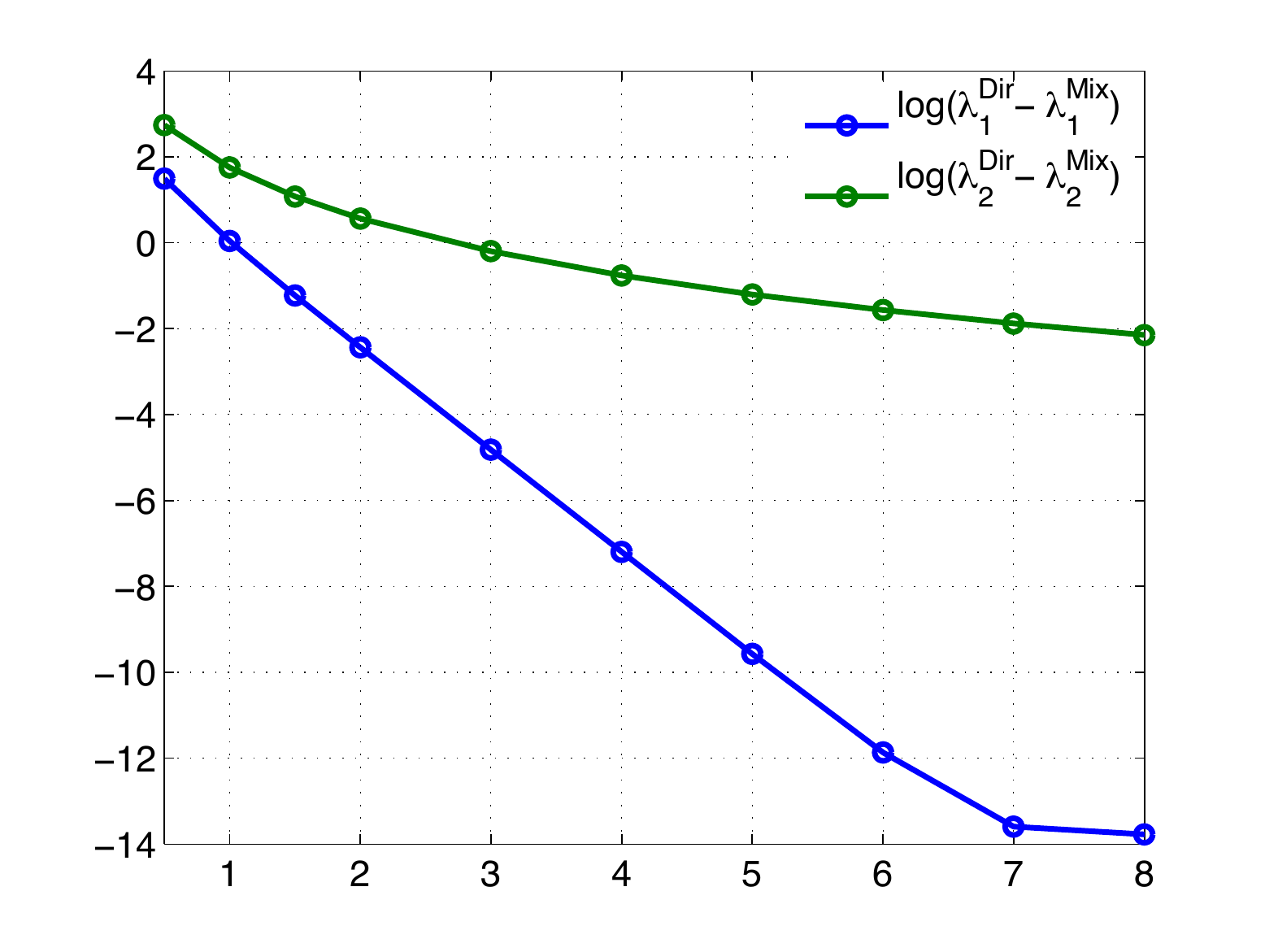}
   \caption{Computed eigenvalues $\lambda_\ell^\Dir(\widehat{\Lambda}_R)$, $\lambda_\ell^\Mix(\widehat{\Lambda}_R)$ as functions of $R$ (left), and $\log$ of their difference (right), for $\ell=1,2$.}
\label{fig:3DX}
\end{figure}

Then an approximation of $\lambda_1(\sX):=\lambda_\infty$ is given, for $R$ large enough, by the mean value of $\lambda_\ell^\Dir(\widehat{\Gamma}_R)$ and $\lambda_\ell^\Mix(\widehat{\Gamma}_R)$. For $R = 12$ we obtain
$$
	\lambda_1(\sX) \simeq 0.6596\pi^2.
$$
In the same way as stated in Corollary \ref{cor:conv} and Corollary \ref{cor:conv2}, an exponential convergence of the truncated problems toward $\lambda_1(\sX)$ can be exhibited.

For the three-dimensional domain, we compute the first two eigenvalues $\lambda_\ell^\Dir(\widehat{\Lambda}_R)$, $\lambda_\ell^\Mix(\widehat{\Lambda}_R)$, for $R$ ranging from $0.5$ to $8$ (see Figure \ref{fig:3DX}).
Computations have been performed on a rather coarse mesh with interpolation degree $4$, which is sufficient to exhibit the existence of a unique bound state with approximate value
$$
	\lambda_1(\sY) \simeq 0.5165\pi^2.
$$
The relative gap between the bound state and the bottom of the essential spectrum is
\begin{equation}
\label{eq:gapX}
   g(\sX) := \frac{\lambda_1(\sY)-\lambda_1(\sX)}{\lambda_1(\sX)} \simeq  
   \frac{0.6596-0.5165}{0.5165} \simeq 0.277\,.
\end{equation}

\section{Conclusion}\label{sec:conc}
We have investigated spectral properties of the Dirichlet Laplacian $\cL_\Lambda$ on the Fichera layer $\Lambda$ and we have given hints that these properties are shared with three-dimensional layers of a more general structure. This suggests the definition of a family $\gF$ of ``generalized Fichera layers'' in which the following main spectral features of $\cL_\Lambda$ subsist:
\begin{enumerate}[{\em i)}]
\item The bottom of the essential spectrum is driven by the first eigenvalue of associated two-dimen\-sional quantum wave guides;
\item The number of independent bound states is finite.
\end{enumerate}
This contrasts with the family of smooth conical layers, denoted here as $\gC$, investigated in \cite{OBP16}, in which the Dirichlet Laplacian satisfies:
\begin{enumerate}[{\em i)}]
\item The bottom of the essential spectrum is driven by the first eigenvalue of a one-dimensional problem;
\item The number of bound states is infinite, and their counting function satisfy a $\cO(|\log E|)$ estimate, with $E$ being the distance to the essential spectrum.
\end{enumerate}

\begin{figure}[ht]
\includegraphics[scale=0.15]{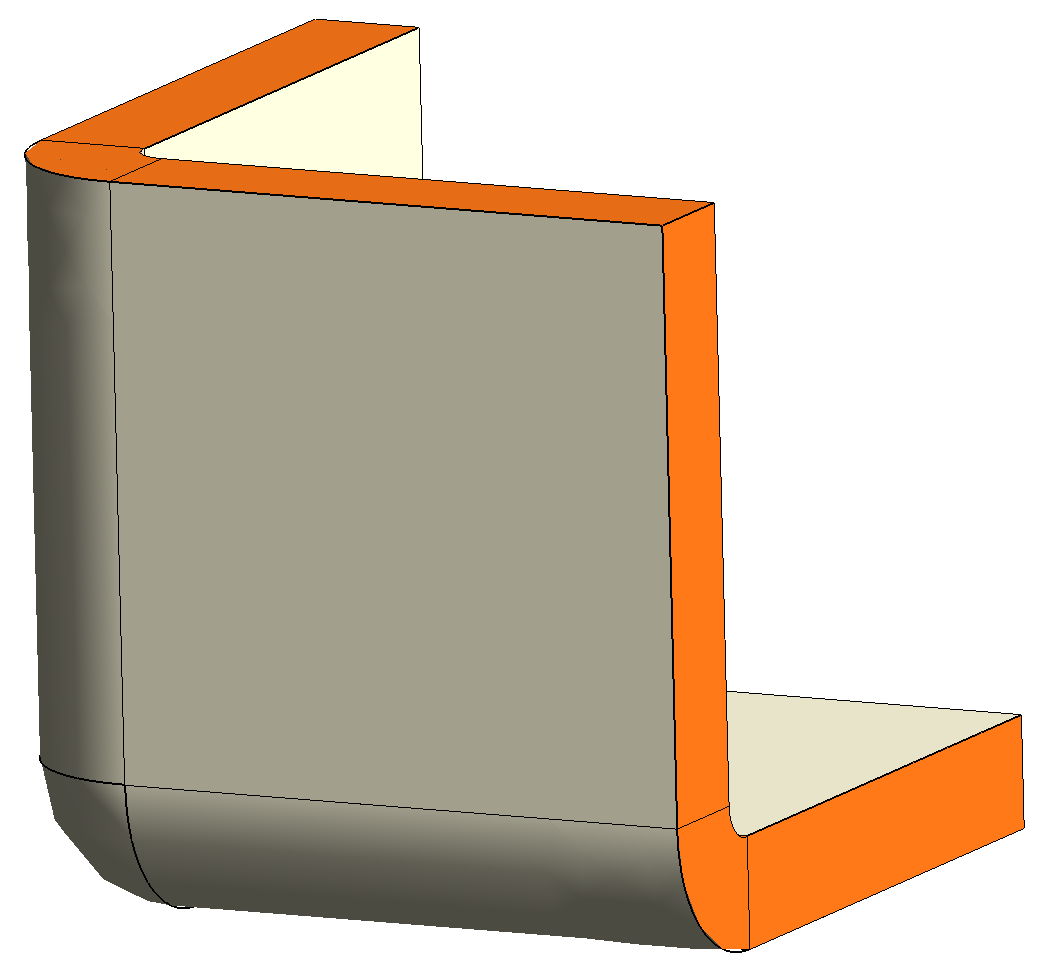}\hskip20pt
\includegraphics[scale=0.15]{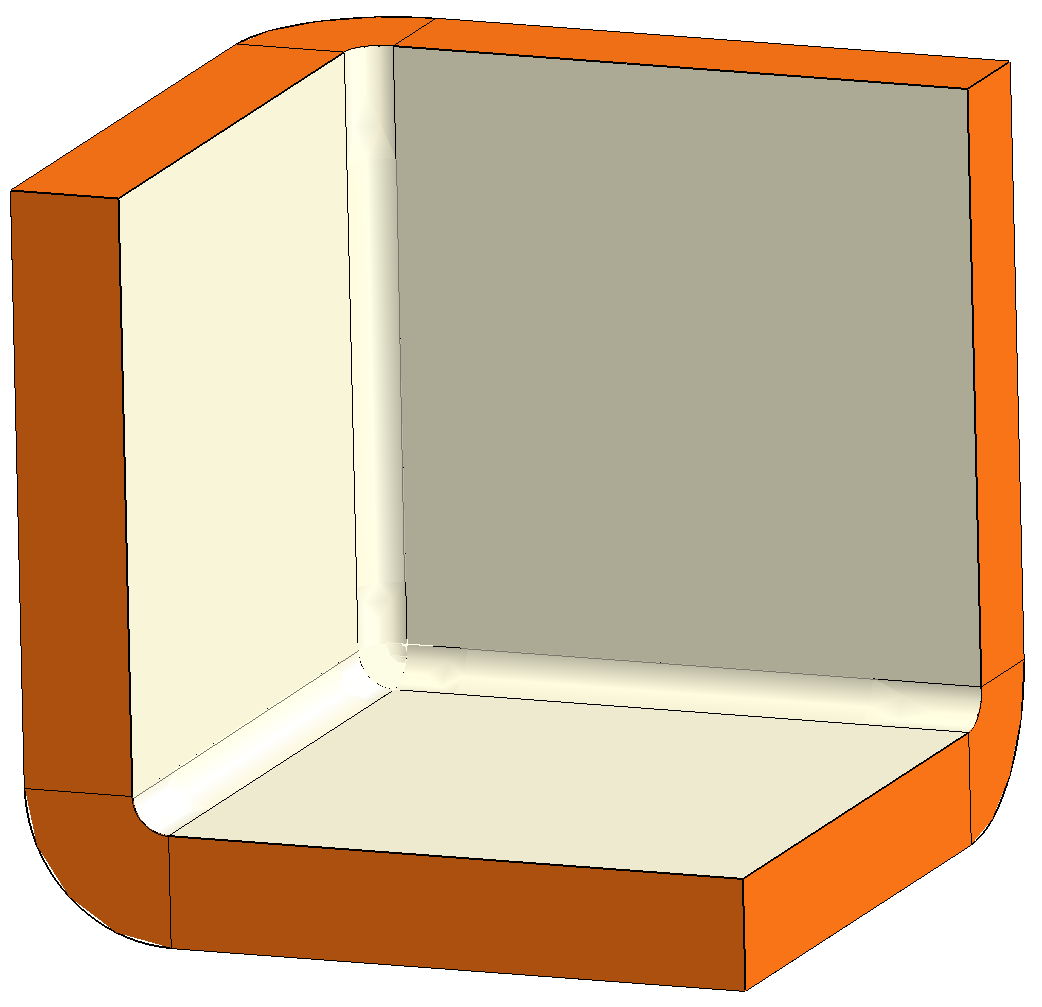}%
   \caption{Two views of the layer $\sL[\varepsilon]$ for $\varepsilon=\frac35$ for which the internal and external maximal radii of curvature are $\frac15$ and $\frac45$.}
\label{fig:FCC}
\end{figure}

Even though it was our initial motivation, it turns out that it is not particularly the existence of edges that generates these different spectral features between elements of $\gC$ and the Fichera layer $\Lambda$. Indeed, consider for instance the smooth 
surface $\sS^\sharp$ (a slice of which is drawn in Figure \ref{fig:round}) and define $\sL[\varepsilon]$ as the set of points at distance $\varepsilon/2$ from $\sS^\sharp$, see Figure \ref{fig:FCC}. Then, for any $\varepsilon\in(0,1)$, the layer $\sL[\varepsilon]$ has a smooth boundary but shares the spectral properties of the Fichera layer $\Lambda$. Actually, the discriminating feature between $\gF$ and $\gC$ is related to ``conical invariance properties'', which characterizes the structure of the layer at infinity. By this, we mean the following:
\begin{quote}
There exists a partition of $\R^3$ in a finite number of axisymmetric cones $\cC_j$ such that:
\begin{enumerate}[(a)]
\item If $\sL$ belongs to $\gF$, the sections of $\sL\cap\cC_j$ across the perpendicular planes $\Pi_j(R)$, $R>0$, to the axis of $\cC_j$, are translation invariant: This means that there exists a two-dimensional guide $\sG_j$ such that $(\sL\cap\cC_j)\cap\Pi_j(R)$ is isomorphic to a part $\sG_j(R)$ of $\sG_j$, and $\sG_j(R)$ tends to $\sG_j$ as $R\to\infty$.
\item If $\sL$ belongs to $\gC$, the sections of its midsurface $\sS\cap\cC_j$ across the spheres $R\,\mathbb{S}_j$, $R>0$, centered at the tip of $\cC_j$, are homothetic.
\end{enumerate}
\end{quote}

\section*{Acknowledgements}
The authors thank Pavel Exner who gave the impulse to write this paper, as well as Konstantin Pankrashkin for fruitful discussions and Adrien Semin for his help in mesh processing with Gmsh.

The first two authors belong to the Centre Henri Lebesgue ANR-11-LABX-0020-01 and the third author is supported by a public grant as part of the ``Investissement d'avenir" project, reference ANR-11-LABX-0056-LMH, LabEx LMH.

\appendix

\section{Three-dimensional meshes}\label{sec:app1}
We describe here the 3D meshes used for the computations presented in \S \ref{subsec:FFiclay} and Section \ref{sec:ext}.

\smallskip
For the finite Fichera layers (\S \ref{subsec:FFiclay}), the mesh of the domain $\Lambda_R$ is deduced from a 3D tensor product based on a specific subdivision of the interval $[-1, R]$, leading to a mesh made of hexahedra whose faces are parallel to the Cartesian planes. By a 3D tensor product, we get a mesh in $[-1, R]^3$, from which the cube $[0, R]^3$ is removed. This process allows to build the meshes, called hexahedral grids $\gG_1$, $\gG_2$, and $\gG_3$ in \S \ref{subsec:FFiclay}.

The choice of the subdivision allows the refinement of the mesh on some particular parts of the domain, namely the internal corner $(0,0,0)$ and the edges starting from this point. The subdivisions $S_k$ corresponding to the grids $\gG_k, k=1,2,3$ are defined as follows. Starting from $ S_1 = \{-1, -\tfrac12, -\tfrac14, 0, \tfrac14,  \tfrac12, 1, 2, \tfrac{(R+2)}2, R\}$, the subdivision $S_{k+1}$ is deduced from $S_k$ by adding the midpoint of each interval of $S_k$. As an example, we show on figure \ref{fig:mesh3DFL}, left, the grid $\gG_1$ for $R=4$.

\smallskip
For the 3D cross (\S \ref{subsec:3DX}), the computations have been performed on the same kind of mesh, based on the subdivision $\{-1, -\tfrac{1}{10}, 0, \tfrac{1}{10}, 1, a, b\}$ with $a=\min(4,R)$ and $b=\max(a,R)$ (with the convention that $a$ or $b$ should be removed if it is equal to the preceding abscissa in the list).

\smallskip
For the 3D Fichera layer with exterior rounded edges (\S \ref{sss:3DRoundFiclay}), the mesh of the domain $\sLs_R$ has been created with Gmsh \cite{GR09}. At the corner $(0,0,0)$, there is one eighth of sphere, extended across its three plane faces by quarters of cylinder; three parallelepipeds (the walls) complete the domain. The mesh is made of tetrahedra of order 2; moreover it is non uniform: elements are densified inside the spherical part and along the internal edges. An example is shown for $R=4$ on figure \ref{fig:mesh3DFL}, right.

\begin{figure}[ht]
\includegraphics[scale=0.80]{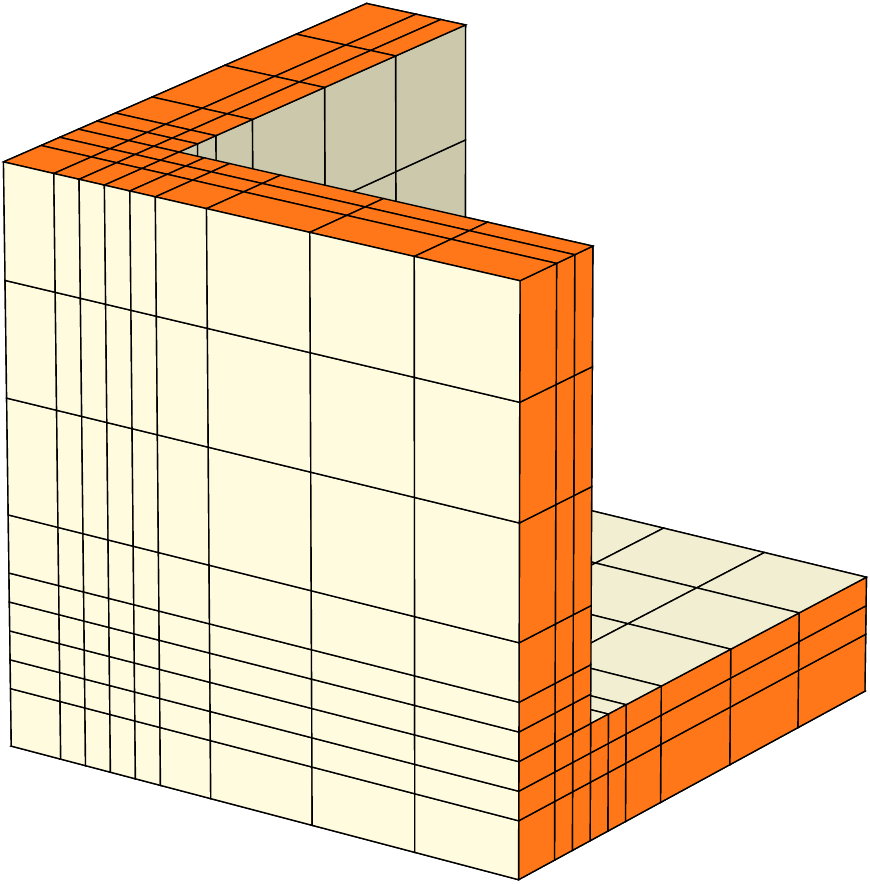}\hskip20pt
\includegraphics[scale=0.23]{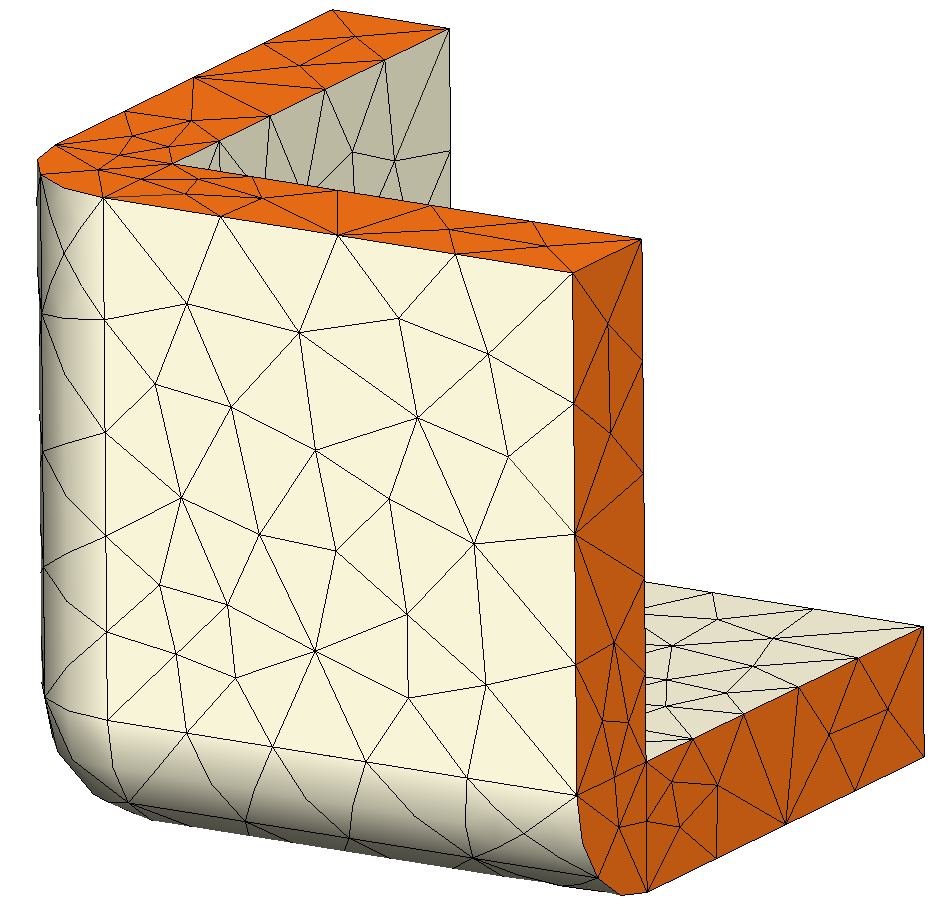}%
   \caption{Meshes of the 3D Fichera layer (left), with exterior rounded edges (right).}
\label{fig:mesh3DFL}
\end{figure}

\section{Existence of a bound state for the guide $\sGs$}\label{sec:app2}
The aim of this appendix is to prove Theorem \ref{thm:guideround}. 

Recall that the guide $\sGs$ (Figure \ref{fig:round}, right) is the union of the quarter disk of radius $1$, $\sGs_0 = \sGs\cap\Box_1$, and of the two infinite strips $[0,\infty)\times\cI$ and $\cI\times[0,\infty)$, with $\cI=(-1,0)$. The broken guide $\Gamma$ is the union of the square $\Gamma_0=\cI\times\cI$ and of the same strips. So we see that $\Gamma$ and $\sGs$ coincide outside $\Box_1$ and that we have the inclusion $\sGs\subset\Gamma$. As a consequence of this and Theorem \ref{th:GR}, we obtain immediately
\begin{enumerate}[{\em i)}]
\item The essential spectra of the Dirichlet Laplacian on $\sGs$ and $\Gamma$ coincide,
\item The number of bound states of the Dirichlet Laplacian in $\sGs$ is smaller than that of $\cL$ in $\Gamma$, thus is less than $1$.
\end{enumerate}
Hence, to prove Theorem \ref{thm:guideround}, it remains to prove that there exists at least one bound state. For this it is enough to construct a function $\psi\in H_0^1(\sGs)$ such that
\begin{equation}
\label{eq:psi}
	\|\nabla \psi\|_{L^2(\sGs)}^2 < \pi^2 \|\psi\|_{L^2(\sGs)}^2.
\end{equation}
Our proof is inspired by \cite[\S A]{GJ92}. In the following, for the sake of completeness, we check that the arguments {\em loc. cit.} apply to the guide $\sGs$. We start with properties of the Helmholtz problem on $\sGs_0$.

\begin{lemma}\label{lem:soledp1} 
Let $g\in H^{1/2}(\partial\sGs_0)$. There exists a unique $\psi_0 \in H^1(\sGs_0)$ such that:
\begin{equation}
\label{eq:pbpsi0}
	\left\{\begin{array}{ll}
		\Delta \psi_0 + \pi^2\psi_0 = 0 & \text{in } \sGs_0,\\
		\psi_0 = g & \text{on } \partial\sGs_0.
	\end{array}\right.
\end{equation}
Define, for $\psi\in H^1(\sGs_0)$, the energy functional of the above problem
\begin{equation}
\label{eq:J}
   J(\psi) = \|\nabla{\psi}\|_{L^2(\sGs_0)}^2 -\pi^2\|{\psi}\|_{L^2(\sGs_0)}^2.
\end{equation}
Then $J(\psi_0)$ is the unique minimizer of $J$ on the space of functions with trace $g$, namely:
\begin{equation}\label{eqn:charact_min}
	J(\psi_0) = \min_{\psi\in H^1(\sGs_0),\, \psi = g \text{ on }\partial\sGs_0} J({\psi}).
\end{equation}
\end{lemma}

\begin{proof} 
Consider a half-disk $\cH$ of radius $1$ containing $\sGs_0$. By monotonicity of Dirichlet eigenvalues
$$
	\lambda_1^\Dir(\cH) \leq \lambda_1^\Dir(\sGs_0).
$$
But it is known that $\lambda_1^\Dir(\cH) = j_{1,1}^2$, where $j_{1,1}$ is the first zero of the first Bessel function of first kind $J_1$. Remark that $j_{1,1}\simeq 3.8$ and in particular $j_{1,1} > \pi$, thus
\begin{equation}\label{eqn:majvp1sharp}
	\pi^2 < \lambda_1^\Dir(\sGs_0).
\end{equation}
Therefore, the operator $\Delta+\pi^2$ is an isomorphism from $H^1_0(\sGs_0)$ onto its dual space. This, combined with the fact that $H^{1/2}(\partial\sGs_0)$ is the trace space of $H^1(\sGs_0)$, provides existence and uniqueness for the solution of problem \eqref{eq:pbpsi0}.

The functional $J$ is the energy related with the bilinear form $a:(\psi,\widehat\psi)\mapsto\int_{\Gamma_0^\sharp}\big(\nabla\psi\cdot\nabla\widehat\psi -\pi^2\psi\widehat\psi\big)$ associated with problem \eqref{eq:pbpsi0}. The variational formulation of this problem is
\[
   \mbox{Find $\psi_0\in H^1(\sGs_0)$ with $\psi_0\big|_{\partial\sGs_0}=g$}\quad \mbox{s.t.}\quad
   \forall\widehat\psi\in H^1_0(\sGs_0),\quad
   a(\psi_0,\widehat\psi)=0.
\]
Let $\psi_1\in H^1(\sGs_0)$ such that $\psi_1=g$ on $\partial\sGs_0$. Since $\psi_0-\psi_1$ has zero trace on $\partial\sGs_0$, there holds $a(\psi_0,\psi_1-\psi_0)=0$. Hence
\[
   J(\psi_1) = J(\psi_0) + J(\psi_1-\psi_0).
\]
But, as a consequence of \eqref{eqn:majvp1sharp}, $J(\psi_1-\psi_0)$ is bounded from below by $\gamma\|{\psi_1-\psi_0}\|_{L^2(\sGs_0)}^2$ for a positive constant $\gamma$. This ends the proof of \eqref{eqn:charact_min}.
\end{proof}

We are ready to end the proof of Theorem \ref{thm:guideround}. 
To prove the existence of bound states, by the min-max principle, it is enough to construct a function $\psi\in H_0^1(\sGs)$ such that
$$
	\|\nabla \psi\|_{L^2(\sGs)}^2 < \pi^2 \|\psi\|_{L^2(\sGs)}^2.
$$
Let $\mu>0$ be a parameter that will be chosen thereafter. We set $\varphi(t) := \sqrt{2}\sin(\pi t)$ and define the following function
$$
	\psi(\bx) = \left\{\begin{array}{cl}
					e^{-\mu x_1}\varphi(x_2) & \text{ if }\ \bx \in \R_+\!\times\cI,\\
					e^{-\mu x_2}\varphi(x_1) & \text{ if }\ \bx \in \cI\times\R_+,\\
					\psi_0 & \text{ if }\ \bx \in \sGs_0,
				\end{array}\right.
$$
where $\psi_0$ is the solution of problem \eqref{eq:pbpsi0} with $g$ defined on $\partial\sGs_0$ as
\[
   g(\bx)=0 \ \mbox{ if }\ |\bx|=1,\quad 
   g(0,x_2) = \varphi(x_2),\ \forall x_2\in\cI\quad\mbox{and}\quad g(x_1,0) = \varphi(x_1),\ \forall x_1\in\cI.
\]
Note that $g$ is continuous on $\partial\sGs_0$, thus belongs to $H^{1/2}(\partial\sGs_0)$, and that $\psi$ belongs to $ H_0^1(\Gamma^\sharp)$.

A simple computation yields
\begin{equation}\label{eqn:rayleighround}
	\frac{\|\nabla\psi\|_{L^2(\sGs)}^2}{\|\psi\|_{L^2(\sGs)}^2} = \pi^2 +\frac{\mu}{1+\mu\|\psi_0\|_{L^2(\sGs_0)}^2}\big(\mu + J(\psi_0)\big).
\end{equation}

We claim that $J(\psi_0)<0$ and to prove this, we rely on the characterization \eqref{eqn:charact_min} of $J(\psi_0)$. So, it suffices to exhibit a specific function $\widehat{\psi}\neq\psi_0$, $\widehat{\psi}\in H^1(\sGs_0)$ such that $\widehat{\psi} = g$ on $\partial\sGs_0$, satisfying
\begin{equation}\label{eqn:negenergy}
	J(\widehat{\psi}) = 0.
\end{equation}
Consider the function $\widehat{\psi}(\bx) = -\sqrt{2}\sin\big(\pi\sqrt{x_1^2+x_2^2}\big)$. By definition $\widehat{\psi} \in H^1(\sGs_0)$ and coincide with $g$ on $\partial\sGs_0$. Using polar coordinates we get
\begin{align*}
	J(\widehat{\psi}) & = \pi\int_0^1 |\partial_r\big(\sin(\pi r)\big)|^2\,r \rd r 
	- \pi^3\int_0^1 |\sin(\pi r)|^2\, r\rd r\\
				& = \pi^3 \int_0^1\big(\cos^2(\pi r) - \sin^2(\pi r) \big)r \rd r\\
				& = \pi^3 \int_0^1\cos(2\pi r)\,r \rd r = 0.
\end{align*}
However, remark that 
$$
	(\Delta + \pi^2)\widehat{\psi}(\bx) = \frac{\pi}{|\bx|}\cos(\pi|\bx|)\neq 0.
$$
It proves \eqref{eqn:negenergy} because necessarily $\widehat{\psi} \neq \psi_0$ and consequently we get
$$
	J(\psi_0) < J(\widehat{\psi}) = 0.
$$
Set $\mu = \frac12 |J(\psi_0)|$, \eqref{eqn:rayleighround} becomes
$$
	\frac{\|\nabla \psi\|_{L^2(\sGs)}^2}{\|\psi\|_{L^2(\sGs)}^2} = \pi^2 - \frac{1}{2 + |J(\psi_0)|\|\psi_0\|_{L^2(\sGs_0)}^2}|J(\psi_0)| < \pi^2.
$$
In particular, by the min-max principle, we obtain
$$
	\lambda_1(\sGs) < \pi^2,
$$
which yields the existence of at least one bound state.



\end{document}

%% file: fig4tex.tex
\ifx\figforTeXisloaded\relax \else\global\let\figforTeXisloaded=\relax\fi
\catcode`\@=11
\ifx\ctr@ln@m\undefined\else%
    \immediate\write16{*** Fig4TeX WARNING : \string\ctr@ln@m\space already defined.}\fi
\def\ctr@ln@m#1{\ifx#1\undefined\else%
    \immediate\write16{*** Fig4TeX WARNING : \string#1 already defined.}\fi}
\ctr@ln@m\ctr@ld@f
\def\ctr@ld@f#1#2{\ctr@ln@m#2#1#2}
\ctr@ld@f\def\ctr@ln@w#1#2{\ctr@ln@m#2\csname#1\endcsname#2}
{\catcode`\/=0 \catcode`/\=12 /ctr@ld@f/gdef/BS@{\}}
\ctr@ld@f\def\ctr@lcsn@m#1{\expandafter\ifx\csname#1\endcsname\relax\else%
    \immediate\write16{*** Fig4TeX WARNING : \BS@\expandafter\string#1\space already defined.}\fi}
\ctr@ld@f\edef\colonc@tcode{\the\catcode`\:}
\ctr@ld@f\edef\semicolonc@tcode{\the\catcode`\;}
\ctr@ld@f\def\t@stc@tcodech@nge{{\let\c@tcodech@nged=\z@%
    \ifnum\colonc@tcode=\the\catcode`\:\else\let\c@tcodech@nged=\@ne\fi%
    \ifnum\semicolonc@tcode=\the\catcode`\;\else\let\c@tcodech@nged=\@ne\fi%
    \ifx\c@tcodech@nged\@ne%
    \immediate\write16{}
    \immediate\write16{!!!=============================================================!!!}
    \immediate\write16{ Fig4TeX WARNING:}
    \immediate\write16{ The category code of some characters has been changed, which will}
    \immediate\write16{ result in an error (message "Runaway argument?").}
    \immediate\write16{ This probably comes from another package that changed the category}
    \immediate\write16{ code after Fig4TeX was loaded. If that proves to be exact, the}
    \immediate\write16{ solution is to exchange the loading commands on top of your file}
    \immediate\write16{ so that Fig4TeX is loaded last. For example, in LaTeX, we should}
    \immediate\write16{ say :}
    \immediate\write16{\BS@ usepackage[french]{babel}}
    \immediate\write16{\BS@ usepackage{fig4tex}}
    \immediate\write16{!!!=============================================================!!!}
    \immediate\write16{}
    \fi}}
\ctr@ld@f\def\FigforTeX{F\kern-.05em i\kern-.05em g\kern-.1em\raise-.14em\hbox{4}\kern-.19em\TeX}
\ctr@ld@f\def\W@rnmesoldA#1{\W@rnmesold}
\ctr@ld@f\def\W@rnmesoldAB#1(#2){\W@rnmesold}
\ctr@ld@f\def\W@rnmesold{%
    \immediate\write16{}
    \immediate\write16{!!!=============================================================!!!}
    \immediate\write16{ Fig4TeX WARNING:}
    \immediate\write16{ The file to be compiled is not compatible with the current version}
    \immediate\write16{ of Fig4TeX. To fix that, upgrade the source file (mainly change \BS@ ps*}
    \immediate\write16{ macros by \BS@ fig* macros), or use fig4tex184.tex instead (\BS@ input fig4tex184}
    \immediate\write16{ or \BS@ usepackage{fig4tex184}).}
    \immediate\write16{!!!=============================================================!!!}
    \immediate\write16{}}
\ctr@ln@m\psbeginfig\let\psbeginfig\W@rnmesoldA
\ctr@ln@m\psset\let\psset\W@rnmesoldAB
\ctr@ln@m\pssetdefault\let\pssetdefault\W@rnmesoldAB
\ctr@ln@m\pssetupdate\let\pssetupdate\W@rnmesoldA
\ctr@ln@w{newdimen}\epsil@n\epsil@n=0.00005pt
\ctr@ln@w{newdimen}\Cepsil@n\Cepsil@n=0.005pt
\ctr@ln@w{newdimen}\dcq@\dcq@=254pt
\ctr@ln@w{newdimen}\PI@\PI@=3.141592pt
\ctr@ln@w{newdimen}\DemiPI@deg\DemiPI@deg=90pt
\ctr@ln@w{newdimen}\PI@deg\PI@deg=180pt
\ctr@ln@w{newdimen}\DePI@deg\DePI@deg=360pt
\ctr@ld@f\chardef\t@n=10
\ctr@ld@f\chardef\c@nt=100
\ctr@ld@f\chardef\@lxxiv=74
\ctr@ld@f\chardef\@xci=91
\ctr@ld@f\mathchardef\@nMnCQn=9949
\ctr@ld@f\chardef\@vi=6
\ctr@ld@f\chardef\@xxx=30
\ctr@ld@f\chardef\@lvi=56
\ctr@ld@f\chardef\@@lxxi=71
\ctr@ld@f\chardef\@lxxxv=85
\ctr@ld@f\mathchardef\@@mmmmlxviii=4068
\ctr@ld@f\mathchardef\@ccclx=360
\ctr@ld@f\mathchardef\@dccxx=720
\ctr@ln@w{newcount}\p@rtent \ctr@ln@w{newcount}\f@ctech \ctr@ln@w{newcount}\result@tent
\ctr@ln@w{newdimen}\v@lmin \ctr@ln@w{newdimen}\v@lmax \ctr@ln@w{newdimen}\v@leur
\ctr@ln@w{newdimen}\result@t\ctr@ln@w{newdimen}\result@@t
\ctr@ln@w{newdimen}\mili@u \ctr@ln@w{newdimen}\c@rre \ctr@ln@w{newdimen}\delt@
\ctr@ld@f\def\degT@rd{0.017453 }  
\ctr@ld@f\def\rdT@deg{57.295779 } 
\ctr@ln@m\v@leurseule
{\catcode`p=12 \catcode`t=12 \gdef\v@leurseule#1pt{#1}}
\ctr@ld@f\def\repdecn@mb#1{\expandafter\v@leurseule\the#1\space}
\ctr@ld@f\def\arct@n#1(#2,#3){{\v@lmin=#2\v@lmax=#3%
    \maxim@m{\mili@u}{-\v@lmin}{\v@lmin}\maxim@m{\c@rre}{-\v@lmax}{\v@lmax}%
    \delt@=\mili@u\m@ech\mili@u%
    \ifdim\c@rre>\@nMnCQn\mili@u\divide\v@lmax\tw@\c@lATAN\v@leur(\z@,\v@lmax)
    \else%
    \maxim@m{\mili@u}{-\v@lmin}{\v@lmin}\maxim@m{\c@rre}{-\v@lmax}{\v@lmax}%
    \m@ech\c@rre%
    \ifdim\mili@u>\@nMnCQn\c@rre\divide\v@lmin\tw@
    \maxim@m{\mili@u}{-\v@lmin}{\v@lmin}\c@lATAN\v@leur(\mili@u,\z@)%
    \else\c@lATAN\v@leur(\delt@,\v@lmax)\fi\fi%
    \ifdim\v@lmin<\z@\v@leur=-\v@leur\ifdim\v@lmax<\z@\advance\v@leur-\PI@%
    \else\advance\v@leur\PI@\fi\fi%
    \global\result@t=\v@leur}#1=\result@t}
\ctr@ld@f\def\m@ech#1{\ifdim#1>1.646pt\divide\mili@u\t@n\divide\c@rre\t@n\m@ech#1\fi}
\ctr@ld@f\def\c@lATAN#1(#2,#3){{\v@lmin=#2\v@lmax=#3\v@leur=\z@\delt@=\tw@ pt%
    \un@iter{0.785398}{\v@lmax<}%
    \un@iter{0.463648}{\v@lmax<}%
    \un@iter{0.244979}{\v@lmax<}%
    \un@iter{0.124355}{\v@lmax<}%
    \un@iter{0.062419}{\v@lmax<}%
    \un@iter{0.031240}{\v@lmax<}%
    \un@iter{0.015624}{\v@lmax<}%
    \un@iter{0.007812}{\v@lmax<}%
    \un@iter{0.003906}{\v@lmax<}%
    \un@iter{0.001953}{\v@lmax<}%
    \un@iter{0.000976}{\v@lmax<}%
    \un@iter{0.000488}{\v@lmax<}%
    \un@iter{0.000244}{\v@lmax<}%
    \un@iter{0.000122}{\v@lmax<}%
    \un@iter{0.000061}{\v@lmax<}%
    \un@iter{0.000030}{\v@lmax<}%
    \un@iter{0.000015}{\v@lmax<}%
    \global\result@t=\v@leur}#1=\result@t}
\ctr@ld@f\def\un@iter#1#2{%
    \divide\delt@\tw@\edef\dpmn@{\repdecn@mb{\delt@}}%
    \mili@u=\v@lmin%
    \ifdim#2\z@%
      \advance\v@lmin-\dpmn@\v@lmax\advance\v@lmax\dpmn@\mili@u%
      \advance\v@leur-#1pt%
    \else%
      \advance\v@lmin\dpmn@\v@lmax\advance\v@lmax-\dpmn@\mili@u%
      \advance\v@leur#1pt%
    \fi}
\ctr@ld@f\def\c@ssin#1#2#3{\expandafter\ifx\csname COS@\number#3\endcsname\relax\c@lCS{#3pt}%
    \expandafter\xdef\csname COS@\number#3\endcsname{\repdecn@mb\result@t}%
    \expandafter\xdef\csname SIN@\number#3\endcsname{\repdecn@mb\result@@t}\fi%
    \edef#1{\csname COS@\number#3\endcsname}\edef#2{\csname SIN@\number#3\endcsname}}
\ctr@ld@f\def\c@lCS#1{{\mili@u=#1\p@rtent=\@ne%
    \relax\ifdim\mili@u<\z@\red@ng<-\else\red@ng>+\fi\f@ctech=\p@rtent%
    \relax\ifdim\mili@u<\z@\mili@u=-\mili@u\f@ctech=-\f@ctech\fi\c@@lCS}}
\ctr@ld@f\def\c@@lCS{\v@lmin=\mili@u\c@rre=-\mili@u\advance\c@rre\DemiPI@deg\v@lmax=\c@rre%
    \mili@u\@@lxxi\mili@u\divide\mili@u\@@mmmmlxviii%
    \edef\v@larg{\repdecn@mb{\mili@u}}\mili@u=-\v@larg\mili@u%
    \edef\v@lmxde{\repdecn@mb{\mili@u}}%
    \c@rre\@@lxxi\c@rre\divide\c@rre\@@mmmmlxviii%
    \edef\v@largC{\repdecn@mb{\c@rre}}\c@rre=-\v@largC\c@rre%
    \edef\v@lmxdeC{\repdecn@mb{\c@rre}}%
    \fctc@s\mili@u\v@lmin\global\result@t\p@rtent\v@leur%
    \let\t@mp=\v@larg\let\v@larg=\v@largC\let\v@largC=\t@mp%
    \let\t@mp=\v@lmxde\let\v@lmxde=\v@lmxdeC\let\v@lmxdeC=\t@mp%
    \fctc@s\c@rre\v@lmax\global\result@@t\f@ctech\v@leur}
\ctr@ld@f\def\fctc@s#1#2{\v@leur=#1\relax\ifdim#2<\@lxxxv\p@\cosser@h\else\sinser@t\fi}
\ctr@ld@f\def\cosser@h{\advance\v@leur\@lvi\p@\divide\v@leur\@lvi%
    \v@leur=\v@lmxde\v@leur\advance\v@leur\@xxx\p@%
    \v@leur=\v@lmxde\v@leur\advance\v@leur\@ccclx\p@%
    \v@leur=\v@lmxde\v@leur\advance\v@leur\@dccxx\p@\divide\v@leur\@dccxx}
\ctr@ld@f\def\sinser@t{\v@leur=\v@lmxdeC\p@\advance\v@leur\@vi\p@%
    \v@leur=\v@largC\v@leur\divide\v@leur\@vi}
\ctr@ld@f\def\red@ng#1#2{\relax\ifdim\mili@u#1#2\DemiPI@deg\advance\mili@u#2-\PI@deg%
    \p@rtent=-\p@rtent\red@ng#1#2\fi}
\ctr@ld@f\def\pr@c@lCS#1#2#3{\ctr@lcsn@m{COS@\number#3 }%
    \expandafter\xdef\csname COS@\number#3\endcsname{#1}%
    \expandafter\xdef\csname SIN@\number#3\endcsname{#2}}
\pr@c@lCS{1}{0}{0}
\pr@c@lCS{0.7071}{0.7071}{45}\pr@c@lCS{0.7071}{-0.7071}{-45}
\pr@c@lCS{0}{1}{90}          \pr@c@lCS{0}{-1}{-90}
\pr@c@lCS{-1}{0}{180}        \pr@c@lCS{-1}{0}{-180}
\pr@c@lCS{0}{-1}{270}        \pr@c@lCS{0}{1}{-270}
\ctr@ld@f\def\invers@#1#2{{\v@leur=#2\maxim@m{\v@lmax}{-\v@leur}{\v@leur}%
    \f@ctech=\@ne\m@inv@rs%
    \multiply\v@leur\f@ctech\edef\v@lv@leur{\repdecn@mb{\v@leur}}%
    \p@rtentiere{\p@rtent}{\v@leur}\v@lmin=\p@\divide\v@lmin\p@rtent%
    \inv@rs@\multiply\v@lmax\f@ctech\global\result@t=\v@lmax}#1=\result@t}
\ctr@ld@f\def\m@inv@rs{\ifdim\v@lmax<\p@\multiply\v@lmax\t@n\multiply\f@ctech\t@n\m@inv@rs\fi}
\ctr@ld@f\def\inv@rs@{\v@lmax=-\v@lmin\v@lmax=\v@lv@leur\v@lmax%
    \advance\v@lmax\tw@ pt\v@lmax=\repdecn@mb{\v@lmin}\v@lmax%
    \delt@=\v@lmax\advance\delt@-\v@lmin\ifdim\delt@<\z@\delt@=-\delt@\fi%
    \ifdim\delt@>\epsil@n\v@lmin=\v@lmax\inv@rs@\fi}
\ctr@ld@f\def\minim@m#1#2#3{\relax\ifdim#2<#3#1=#2\else#1=#3\fi}
\ctr@ld@f\def\maxim@m#1#2#3{\relax\ifdim#2>#3#1=#2\else#1=#3\fi}
\ctr@ld@f\def\p@rtentiere#1#2{#1=#2\divide#1by65536 }
\ctr@ld@f\def\r@undint#1#2{{\v@leur=#2\divide\v@leur\t@n\p@rtentiere{\p@rtent}{\v@leur}%
    \v@leur=\p@rtent pt\global\result@t=\t@n\v@leur}#1=\result@t}
\ctr@ld@f\def\sqrt@#1#2{{\v@leur=#2%
    \minim@m{\v@lmin}{\p@}{\v@leur}\maxim@m{\v@lmax}{\p@}{\v@leur}%
    \f@ctech=\@ne\m@sqrt@\sqrt@@%
    \mili@u=\v@lmin\advance\mili@u\v@lmax\divide\mili@u\tw@\multiply\mili@u\f@ctech%
    \global\result@t=\mili@u}#1=\result@t}
\ctr@ld@f\def\m@sqrt@{\ifdim\v@leur>\dcq@\divide\v@leur\c@nt\v@lmax=\v@leur%
    \multiply\f@ctech\t@n\m@sqrt@\fi}
\ctr@ld@f\def\sqrt@@{\mili@u=\v@lmin\advance\mili@u\v@lmax\divide\mili@u\tw@%
    \c@rre=\repdecn@mb{\mili@u}\mili@u%
    \ifdim\c@rre<\v@leur\v@lmin=\mili@u\else\v@lmax=\mili@u\fi%
    \delt@=\v@lmax\advance\delt@-\v@lmin\ifdim\delt@>\epsil@n\sqrt@@\fi}
\ctr@ld@f\def\extrairelepremi@r#1\de#2{\expandafter\lepremi@r#2@#1#2}
\ctr@ld@f\def\lepremi@r#1,#2@#3#4{\def#3{#1}\def#4{#2}\ignorespaces}
\ctr@ld@f\def\@cfor#1:=#2\do#3{%
  \edef\@fortemp{#2}%
  \ifx\@fortemp\empty\else\@cforloop#2,\@nil,\@nil\@@#1{#3}\fi}
\ctr@ln@m\@nextwhile
\ctr@ld@f\def\@cforloop#1,#2\@@#3#4{%
  \def#3{#1}%
  \ifx#3\Fig@nnil\let\@nextwhile=\Fig@fornoop\else#4\relax\let\@nextwhile=\@cforloop\fi%
  \@nextwhile#2\@@#3{#4}}

\ctr@ld@f\def\@ecfor#1:=#2\do#3{%
  \def\@@cfor{\@cfor#1:=}%
  \edef\@@@cfor{#2}%
  \expandafter\@@cfor\@@@cfor\do{#3}}
\ctr@ld@f\def\Fig@nnil{\@nil}
\ctr@ld@f\def\Fig@fornoop#1\@@#2#3{}
\ctr@ln@m\list@@rg
\ctr@ld@f\def\trtlis@rg#1#2{\def\list@@rg{#1}%
    \@ecfor\p@rv@l:=\list@@rg\do{\expandafter#2\p@rv@l|}}
\ctr@ld@f\def\trtlis@rgtok#1{\let@xte={}\let\n@xt\addt@t@xt\addt@t@xt #1}
\ctr@ln@m\M@cro
\ctr@ln@m\n@xt
\ctr@ld@f\def\addt@t@xt#1{\if#1|\let\n@xt\relax\else%
    \if#1,\expandafter\M@cro\the\let@xte|\let@xte={}%
    \else\let@xte=\expandafter{\the\let@xte #1}\fi\fi\n@xt}
\ctr@ln@w{newbox}\b@xvisu
\ctr@ln@w{newtoks}\let@xte
\ctr@ln@w{newif}\ifitis@K
\ctr@ln@w{newcount}\s@mme
\ctr@ln@w{newcount}\l@mbd@un \ctr@ln@w{newcount}\l@mbd@de
\ctr@ln@w{newcount}\superc@ntr@l\superc@ntr@l=\@ne        
\ctr@ln@w{newcount}\typec@ntr@l\typec@ntr@l=\superc@ntr@l 
\ctr@ln@w{newdimen}\v@lX  \ctr@ln@w{newdimen}\v@lY  \ctr@ln@w{newdimen}\v@lZ
\ctr@ln@w{newdimen}\v@lXa \ctr@ln@w{newdimen}\v@lYa \ctr@ln@w{newdimen}\v@lZa
\ctr@ln@w{newdimen}\unit@\unit@=\p@ 
\ctr@ld@f\def\unit@util{pt}
\ctr@ld@f\def\ptT@ptps{0.996264 }
\ctr@ld@f\def\ptpsT@pt{1.00375 }
\ctr@ld@f\def\ptT@unit@{1} 
\ctr@ld@f\def\setunit@#1{\def\unit@util{#1}\setunit@@#1:\invers@{\result@t}{\unit@}%
    \edef\ptT@unit@{\repdecn@mb\result@t}}
\ctr@ld@f\def\setunit@@#1#2:{\ifcat#1a\unit@=\@ne#1#2\else\unit@=#1#2\fi}
\ctr@ld@f\def\d@fm@cdim#1#2{{\v@leur=#2\v@leur=\ptT@unit@\v@leur\xdef#1{\repdecn@mb\v@leur}}}
\ctr@ln@w{newif}\ifBdingB@x\BdingB@xtrue
\ctr@ln@w{newdimen}\c@@rdXmin \ctr@ln@w{newdimen}\c@@rdYmin  
\ctr@ln@w{newdimen}\c@@rdXmax \ctr@ln@w{newdimen}\c@@rdYmax
\ctr@ld@f\def\b@undb@x#1#2{\ifBdingB@x%
    \relax\ifdim#1<\c@@rdXmin\global\c@@rdXmin=#1\fi%
    \relax\ifdim#2<\c@@rdYmin\global\c@@rdYmin=#2\fi%
    \relax\ifdim#1>\c@@rdXmax\global\c@@rdXmax=#1\fi%
    \relax\ifdim#2>\c@@rdYmax\global\c@@rdYmax=#2\fi\fi}
\ctr@ld@f\def\b@undb@xP#1{{\Figg@tXY{#1}\b@undb@x{\v@lX}{\v@lY}}}
\ctr@ld@f\def\ellBB@x#1;#2,#3(#4,#5,#6){{\s@uvc@ntr@l\et@tellBB@x%
    \setc@ntr@l{2}\figptell-2::#1;#2,#3(#4,#6)\b@undb@xP{-2}%
    \figptell-2::#1;#2,#3(#5,#6)\b@undb@xP{-2}%
    \c@ssin{\C@}{\S@}{#6}\v@lmin=\C@ pt\v@lmax=\S@ pt%
    \mili@u=#3\v@lmin\delt@=#2\v@lmax\arct@n\v@leur(\delt@,\mili@u)%
    \mili@u=-#3\v@lmax\delt@=#2\v@lmin\arct@n\c@rre(\delt@,\mili@u)%
    \v@leur=\rdT@deg\v@leur\advance\v@leur-\DePI@deg%
    \c@rre=\rdT@deg\c@rre\advance\c@rre-\DePI@deg%
    \v@lmin=#4pt\v@lmax=#5pt%
    \loop\ifdim\v@leur<\v@lmax\ifdim\v@leur>\v@lmin%
    \edef\@ngle{\repdecn@mb\v@leur}\figptell-2::#1;#2,#3(\@ngle,#6)%
    \b@undb@xP{-2}\fi\advance\v@leur\PI@deg\repeat%
    \loop\ifdim\c@rre<\v@lmax\ifdim\c@rre>\v@lmin%
    \edef\@ngle{\repdecn@mb\c@rre}\figptell-2::#1;#2,#3(\@ngle,#6)%
    \b@undb@xP{-2}\fi\advance\c@rre\PI@deg\repeat%
    \resetc@ntr@l\et@tellBB@x}\ignorespaces}
\ctr@ld@f\def\initb@undb@x{\c@@rdXmin=\maxdimen\c@@rdYmin=\maxdimen%
    \c@@rdXmax=-\maxdimen\c@@rdYmax=-\maxdimen}
\ctr@ld@f\def\c@ntr@lnum#1{%
    \relax\ifnum\typec@ntr@l=\@ne%
    \ifnum#1<\z@%
    \immediate\write16{*** Forbidden point number (#1). Abort.}\end\fi\fi%
    \set@bjc@de{#1}}
\ctr@ln@m\objc@de
\ctr@ld@f\def\set@bjc@de#1{\edef\objc@de{@BJ\ifnum#1<\z@ M\romannumeral-#1\else\romannumeral#1\fi}}
\s@mme=\m@ne\loop\ifnum\s@mme>-19
  \set@bjc@de{\s@mme}\ctr@lcsn@m\objc@de\ctr@lcsn@m{\objc@de T}
\advance\s@mme\m@ne\repeat
\s@mme=\@ne\loop\ifnum\s@mme<6
  \set@bjc@de{\s@mme}\ctr@lcsn@m\objc@de\ctr@lcsn@m{\objc@de T}
\advance\s@mme\@ne\repeat
\ctr@ld@f\def\setc@ntr@l#1{\ifnum\superc@ntr@l>#1\typec@ntr@l=\superc@ntr@l%
    \else\typec@ntr@l=#1\fi}
\ctr@ld@f\def\resetc@ntr@l#1{\global\superc@ntr@l=#1\setc@ntr@l{#1}}
\ctr@ld@f\def\s@uvc@ntr@l#1{\edef#1{\the\superc@ntr@l}}
\ctr@ln@m\c@lproscal
\ctr@ld@f\def\c@lproscalDD#1[#2,#3]{{\Figg@tXY{#2}%
    \edef\Xu@{\repdecn@mb{\v@lX}}\edef\Yu@{\repdecn@mb{\v@lY}}\Figg@tXY{#3}%
    \global\result@t=\Xu@\v@lX\global\advance\result@t\Yu@\v@lY}#1=\result@t}
\ctr@ld@f\def\c@lproscalTD#1[#2,#3]{{\Figg@tXY{#2}\edef\Xu@{\repdecn@mb{\v@lX}}%
    \edef\Yu@{\repdecn@mb{\v@lY}}\edef\Zu@{\repdecn@mb{\v@lZ}}%
    \Figg@tXY{#3}\global\result@t=\Xu@\v@lX\global\advance\result@t\Yu@\v@lY%
    \global\advance\result@t\Zu@\v@lZ}#1=\result@t}
\ctr@ld@f\def\c@lprovec#1{%
    \det@rmC\v@lZa(\v@lX,\v@lY,\v@lmin,\v@lmax)%
    \det@rmC\v@lXa(\v@lY,\v@lZ,\v@lmax,\v@leur)%
    \det@rmC\v@lYa(\v@lZ,\v@lX,\v@leur,\v@lmin)%
    \Figv@ctCreg#1(\v@lXa,\v@lYa,\v@lZa)}
\ctr@ld@f\def\det@rm#1[#2,#3]{{\Figg@tXY{#2}\Figg@tXYa{#3}%
    \delt@=\repdecn@mb{\v@lX}\v@lYa\advance\delt@-\repdecn@mb{\v@lY}\v@lXa%
    \global\result@t=\delt@}#1=\result@t}
\ctr@ld@f\def\det@rmC#1(#2,#3,#4,#5){{\global\result@t=\repdecn@mb{#2}#5%
    \global\advance\result@t-\repdecn@mb{#3}#4}#1=\result@t}
\ctr@ld@f\def\getredf@ctDD#1(#2,#3){{\maxim@m{\v@lXa}{-#2}{#2}\maxim@m{\v@lYa}{-#3}{#3}%
    \maxim@m{\v@lXa}{\v@lXa}{\v@lYa}
    \ifdim\v@lXa>\@xci pt\divide\v@lXa\@xci%
    \p@rtentiere{\p@rtent}{\v@lXa}\advance\p@rtent\@ne\else\p@rtent=\@ne\fi%
    \global\result@tent=\p@rtent}#1=\result@tent\ignorespaces}
\ctr@ld@f\def\getredf@ctTD#1(#2,#3,#4){{\maxim@m{\v@lXa}{-#2}{#2}\maxim@m{\v@lYa}{-#3}{#3}%
    \maxim@m{\v@lZa}{-#4}{#4}\maxim@m{\v@lXa}{\v@lXa}{\v@lYa}%
    \maxim@m{\v@lXa}{\v@lXa}{\v@lZa}
    \ifdim\v@lXa>\@lxxiv pt\divide\v@lXa\@lxxiv%
    \p@rtentiere{\p@rtent}{\v@lXa}\advance\p@rtent\@ne\else\p@rtent=\@ne\fi%
    \global\result@tent=\p@rtent}#1=\result@tent\ignorespaces}
\ctr@ln@m\getredf@ctB
\ctr@ld@f\def\getredf@ctBDD#1{\getredf@ctDD#1(\v@lX,\v@lY)}
\ctr@ld@f\def\getredf@ctBTD#1{\getredf@ctTD#1(\v@lX,\v@lY,\v@lZ)}
\ctr@ld@f\def\FigptintercircB@zDD#1:#2:#3,#4[#5,#6,#7,#8]{{\s@uvc@ntr@l\et@tfigptintercircB@zDD%
    \setc@ntr@l{2}\figvectPDD-1[#5,#8]\Figg@tXY{-1}\getredf@ctDD\f@ctech(\v@lX,\v@lY)%
    \mili@u=#4\unit@\divide\mili@u\f@ctech\c@rre=\repdecn@mb{\mili@u}\mili@u%
    \figptBezierDD-5::#3[#5,#6,#7,#8]%
    \v@lmin=#3\p@\v@lmax=\v@lmin\advance\v@lmax0.1\p@%
    \loop\edef\T@{\repdecn@mb{\v@lmax}}\figptBezierDD-2::\T@[#5,#6,#7,#8]%
    \figvectPDD-1[-5,-2]\n@rmeucCDD{\delt@}{-1}\ifdim\delt@<\c@rre\v@lmin=\v@lmax%
    \advance\v@lmax0.1\p@\repeat%
    \loop\mili@u=\v@lmin\advance\mili@u\v@lmax%
    \divide\mili@u\tw@\edef\T@{\repdecn@mb{\mili@u}}\figptBezierDD-2::\T@[#5,#6,#7,#8]%
    \figvectPDD-1[-5,-2]\n@rmeucCDD{\delt@}{-1}\ifdim\delt@>\c@rre\v@lmax=\mili@u%
    \else\v@lmin=\mili@u\fi\v@leur=\v@lmax\advance\v@leur-\v@lmin%
    \ifdim\v@leur>\epsil@n\repeat\figptcopyDD#1:#2/-2/%
    \resetc@ntr@l\et@tfigptintercircB@zDD}\ignorespaces}
\ctr@ln@m\figptinterlines
\ctr@ld@f\def\inters@cDD#1:#2[#3,#4;#5,#6]{{\s@uvc@ntr@l\et@tinters@cDD%
    \setc@ntr@l{2}\vecunit@{-1}{#4}\vecunit@{-2}{#6}%
    \Figg@tXY{-1}\setc@ntr@l{1}\Figg@tXYa{#3}%
    \edef\A@{\repdecn@mb{\v@lX}}\edef\B@{\repdecn@mb{\v@lY}}%
    \v@lmin=\B@\v@lXa\advance\v@lmin-\A@\v@lYa%
    \Figg@tXYa{#5}\setc@ntr@l{2}\Figg@tXY{-2}%
    \edef\C@{\repdecn@mb{\v@lX}}\edef\D@{\repdecn@mb{\v@lY}}%
    \v@lmax=\D@\v@lXa\advance\v@lmax-\C@\v@lYa%
    \delt@=\A@\v@lY\advance\delt@-\B@\v@lX%
    \invers@{\v@leur}{\delt@}\edef\v@ldelta{\repdecn@mb{\v@leur}}%
    \v@lXa=\A@\v@lmax\advance\v@lXa-\C@\v@lmin%
    \v@lYa=\B@\v@lmax\advance\v@lYa-\D@\v@lmin%
    \v@lXa=\v@ldelta\v@lXa\v@lYa=\v@ldelta\v@lYa%
    \setc@ntr@l{1}\Figp@intregDD#1:{#2}(\v@lXa,\v@lYa)%
    \resetc@ntr@l\et@tinters@cDD}\ignorespaces}
\ctr@ld@f\def\inters@cTD#1:#2[#3,#4;#5,#6]{{\s@uvc@ntr@l\et@tinters@cTD%
    \setc@ntr@l{2}\figvectNVTD-1[#4,#6]\figvectNVTD-2[#6,-1]\figvectPTD-1[#3,#5]%
    \r@pPSTD\v@leur[-2,-1,#4]\edef\v@lcoef{\repdecn@mb{\v@leur}}%
    \figpttraTD#1:{#2}=#3/\v@lcoef,#4/\resetc@ntr@l\et@tinters@cTD}\ignorespaces}
\ctr@ld@f\def\r@pPSTD#1[#2,#3,#4]{{\Figg@tXY{#2}\edef\Xu@{\repdecn@mb{\v@lX}}%
    \edef\Yu@{\repdecn@mb{\v@lY}}\edef\Zu@{\repdecn@mb{\v@lZ}}%
    \Figg@tXY{#3}\v@lmin=\Xu@\v@lX\advance\v@lmin\Yu@\v@lY\advance\v@lmin\Zu@\v@lZ%
    \Figg@tXY{#4}\v@lmax=\Xu@\v@lX\advance\v@lmax\Yu@\v@lY\advance\v@lmax\Zu@\v@lZ%
    \invers@{\v@leur}{\v@lmax}\global\result@t=\repdecn@mb{\v@leur}\v@lmin}%
    #1=\result@t}
\ctr@ln@m\n@rminf
\ctr@ld@f\def\n@rminfDD#1#2{{\Figg@tXY{#2}\maxim@m{\v@lX}{\v@lX}{-\v@lX}%
    \maxim@m{\v@lY}{\v@lY}{-\v@lY}\maxim@m{\global\result@t}{\v@lX}{\v@lY}}%
    #1=\result@t}
\ctr@ld@f\def\n@rminfTD#1#2{{\Figg@tXY{#2}\maxim@m{\v@lX}{\v@lX}{-\v@lX}%
    \maxim@m{\v@lY}{\v@lY}{-\v@lY}\maxim@m{\v@lZ}{\v@lZ}{-\v@lZ}%
    \maxim@m{\v@lX}{\v@lX}{\v@lY}\maxim@m{\global\result@t}{\v@lX}{\v@lZ}}%
    #1=\result@t}
\ctr@ln@m\n@rmeucC
\ctr@ld@f\def\n@rmeucCDD#1#2{\Figg@tXY{#2}\divide\v@lX\f@ctech\divide\v@lY\f@ctech%
    #1=\repdecn@mb{\v@lX}\v@lX\v@lX=\repdecn@mb{\v@lY}\v@lY\advance#1\v@lX}
\ctr@ld@f\def\n@rmeucCTD#1#2{\Figg@tXY{#2}%
    \divide\v@lX\f@ctech\divide\v@lY\f@ctech\divide\v@lZ\f@ctech%
    #1=\repdecn@mb{\v@lX}\v@lX\v@lX=\repdecn@mb{\v@lY}\v@lY\advance#1\v@lX%
    \v@lX=\repdecn@mb{\v@lZ}\v@lZ\advance#1\v@lX}
\ctr@ln@m\n@rmeucSV
\ctr@ld@f\def\n@rmeucSVDD#1#2{{\Figg@tXY{#2}%
    \v@lXa=\repdecn@mb{\v@lX}\v@lX\v@lYa=\repdecn@mb{\v@lY}\v@lY%
    \advance\v@lXa\v@lYa\sqrt@{\global\result@t}{\v@lXa}}#1=\result@t}
\ctr@ld@f\def\n@rmeucSVTD#1#2{{\Figg@tXY{#2}\v@lXa=\repdecn@mb{\v@lX}\v@lX%
    \v@lYa=\repdecn@mb{\v@lY}\v@lY\v@lZa=\repdecn@mb{\v@lZ}\v@lZ%
    \advance\v@lXa\v@lYa\advance\v@lXa\v@lZa\sqrt@{\global\result@t}{\v@lXa}}#1=\result@t}
\ctr@ln@m\n@rmeuc
\ctr@ld@f\def\n@rmeucDD#1#2{{\Figg@tXY{#2}\getredf@ctDD\f@ctech(\v@lX,\v@lY)%
    \divide\v@lX\f@ctech\divide\v@lY\f@ctech%
    \v@lXa=\repdecn@mb{\v@lX}\v@lX\v@lYa=\repdecn@mb{\v@lY}\v@lY%
    \advance\v@lXa\v@lYa\sqrt@{\global\result@t}{\v@lXa}%
    \global\multiply\result@t\f@ctech}#1=\result@t}
\ctr@ld@f\def\n@rmeucTD#1#2{{\Figg@tXY{#2}\getredf@ctTD\f@ctech(\v@lX,\v@lY,\v@lZ)%
    \divide\v@lX\f@ctech\divide\v@lY\f@ctech\divide\v@lZ\f@ctech%
    \v@lXa=\repdecn@mb{\v@lX}\v@lX%
    \v@lYa=\repdecn@mb{\v@lY}\v@lY\v@lZa=\repdecn@mb{\v@lZ}\v@lZ%
    \advance\v@lXa\v@lYa\advance\v@lXa\v@lZa\sqrt@{\global\result@t}{\v@lXa}%
    \global\multiply\result@t\f@ctech}#1=\result@t}
\ctr@ln@m\vecunit@
\ctr@ld@f\def\vecunit@DD#1#2{{\Figg@tXY{#2}\getredf@ctDD\f@ctech(\v@lX,\v@lY)%
    \divide\v@lX\f@ctech\divide\v@lY\f@ctech%
    \Figv@ctCreg#1(\v@lX,\v@lY)\n@rmeucSV{\v@lYa}{#1}%
    \invers@{\v@lXa}{\v@lYa}\edef\v@lv@lXa{\repdecn@mb{\v@lXa}}%
    \v@lX=\v@lv@lXa\v@lX\v@lY=\v@lv@lXa\v@lY%
    \Figv@ctCreg#1(\v@lX,\v@lY)\multiply\v@lYa\f@ctech\global\result@t=\v@lYa}}
\ctr@ld@f\def\vecunit@TD#1#2{{\Figg@tXY{#2}\getredf@ctTD\f@ctech(\v@lX,\v@lY,\v@lZ)%
    \divide\v@lX\f@ctech\divide\v@lY\f@ctech\divide\v@lZ\f@ctech%
    \Figv@ctCreg#1(\v@lX,\v@lY,\v@lZ)\n@rmeucSV{\v@lYa}{#1}%
    \invers@{\v@lXa}{\v@lYa}\edef\v@lv@lXa{\repdecn@mb{\v@lXa}}%
    \v@lX=\v@lv@lXa\v@lX\v@lY=\v@lv@lXa\v@lY\v@lZ=\v@lv@lXa\v@lZ%
    \Figv@ctCreg#1(\v@lX,\v@lY,\v@lZ)\multiply\v@lYa\f@ctech\global\result@t=\v@lYa}}
\ctr@ld@f\def\vecunitC@TD[#1,#2]{\Figg@tXYa{#1}\Figg@tXY{#2}%
    \advance\v@lX-\v@lXa\advance\v@lY-\v@lYa\advance\v@lZ-\v@lZa\c@lvecunitTD}
\ctr@ld@f\def\vecunitCV@TD#1{\Figg@tXY{#1}\c@lvecunitTD}
\ctr@ld@f\def\c@lvecunitTD{\getredf@ctTD\f@ctech(\v@lX,\v@lY,\v@lZ)%
    \divide\v@lX\f@ctech\divide\v@lY\f@ctech\divide\v@lZ\f@ctech%
    \v@lXa=\repdecn@mb{\v@lX}\v@lX%
    \v@lYa=\repdecn@mb{\v@lY}\v@lY\v@lZa=\repdecn@mb{\v@lZ}\v@lZ%
    \advance\v@lXa\v@lYa\advance\v@lXa\v@lZa\sqrt@{\v@lYa}{\v@lXa}%
    \invers@{\v@lXa}{\v@lYa}\edef\v@lv@lXa{\repdecn@mb{\v@lXa}}%
    \v@lX=\v@lv@lXa\v@lX\v@lY=\v@lv@lXa\v@lY\v@lZ=\v@lv@lXa\v@lZ}
\ctr@ln@m\figgetangle
\ctr@ld@f\def\figgetangleDD#1[#2,#3,#4]{\ifGR@cri{\s@uvc@ntr@l\et@tfiggetangleDD\setc@ntr@l{2}%
    \figvectPDD-1[#2,#3]\figvectPDD-2[#2,#4]\vecunit@{-1}{-1}%
    \c@lproscalDD\delt@[-2,-1]\figvectNVDD-1[-1]\c@lproscalDD\v@leur[-2,-1]%
    \arct@n\v@lmax(\delt@,\v@leur)\v@lmax=\rdT@deg\v@lmax%
    \ifdim\v@lmax<\z@\advance\v@lmax\DePI@deg\fi\xdef#1{\repdecn@mb{\v@lmax}}%
    \resetc@ntr@l\et@tfiggetangleDD}\ignorespaces\fi}
\ctr@ld@f\def\figgetangleTD#1[#2,#3,#4,#5]{\ifGR@cri{\s@uvc@ntr@l\et@tfiggetangleTD\setc@ntr@l{2}%
    \figvectPTD-1[#2,#3]\figvectPTD-2[#2,#5]\figvectNVTD-3[-1,-2]%
    \figvectPTD-2[#2,#4]\figvectNVTD-4[-3,-1]%
    \vecunit@{-1}{-1}\c@lproscalTD\delt@[-2,-1]\c@lproscalTD\v@leur[-2,-4]%
    \arct@n\v@lmax(\delt@,\v@leur)\v@lmax=\rdT@deg\v@lmax%
    \ifdim\v@lmax<\z@\advance\v@lmax\DePI@deg\fi\xdef#1{\repdecn@mb{\v@lmax}}%
    \resetc@ntr@l\et@tfiggetangleTD}\ignorespaces\fi}    
\ctr@ld@f\def\figgetdist#1[#2,#3]{\ifGR@cri{\s@uvc@ntr@l\et@tfiggetdist\setc@ntr@l{2}%
    \figvectP-1[#2,#3]\n@rmeuc{\v@lX}{-1}\v@lX=\ptT@unit@\v@lX\xdef#1{\repdecn@mb{\v@lX}}%
    \resetc@ntr@l\et@tfiggetdist}\ignorespaces\fi}
\ctr@ld@f\def\figget#1=#2[#3]{\keln@mun#1|%
    \def\n@mref{a}\ifx\l@debut\n@mref\figgetangle#2[#3]\else
    \def\n@mref{d}\ifx\l@debut\n@mref\figgetdist#2[#3]\else
    \W@rnmeskwd{figget}{#1}\fi\fi\ignorespaces}
\ctr@ld@f\def\Figg@tT#1{\c@ntr@lnum{#1}%
    {\expandafter\expandafter\expandafter\extr@ctT\csname\objc@de\endcsname:%
     \ifnum\B@@ltxt=\z@\ptn@me{#1}\else\csname\objc@de T\endcsname\fi}}
\ctr@ld@f\def\extr@ctT#1,#2,#3/#4:{\def\B@@ltxt{#3}}
\ctr@ld@f\def\Figg@tXY#1{\c@ntr@lnum{#1}%
    \expandafter\expandafter\expandafter\extr@ctC\csname\objc@de\endcsname:}
\ctr@ln@m\extr@ctC
\ctr@ld@f\def\extr@ctCDD#1/#2,#3,#4:{\v@lX=#2\v@lY=#3}
\ctr@ld@f\def\extr@ctCTD#1/#2,#3,#4:{\v@lX=#2\v@lY=#3\v@lZ=#4}
\ctr@ld@f\def\Figg@tXYa#1{\c@ntr@lnum{#1}%
    \expandafter\expandafter\expandafter\extr@ctCa\csname\objc@de\endcsname:}
\ctr@ln@m\extr@ctCa
\ctr@ld@f\def\extr@ctCaDD#1/#2,#3,#4:{\v@lXa=#2\v@lYa=#3}
\ctr@ld@f\def\extr@ctCaTD#1/#2,#3,#4:{\v@lXa=#2\v@lYa=#3\v@lZa=#4}
\ctr@ln@m\t@xt@
\ctr@ld@f\def\figinit#1{\t@stc@tcodech@nge\initpr@lim\Figinit@#1,:\initpss@ttings\ignorespaces}
\ctr@ld@f\def\Figinit@#1,#2:{\setunit@{#1}\def\t@xt@{#2}\ifx\t@xt@\empty\else\Figinit@@#2:\fi}
\ctr@ld@f\def\Figinit@@#1#2:{\if#12 \else\Figs@tproj{#1}\initTD@\fi}
\ctr@ln@w{newif}\ifTr@isDim
\ctr@ld@f\def\UnD@fined{UNDEFINED}
\ctr@ln@m\@utoFN
\ctr@ln@m\@utoFInDone
\ctr@ln@m\disob@unit
\ctr@ld@f\def\initpr@lim{\initb@undb@x\figsetmark{}\figsetptname{$A_{##1}$}\def\Sc@leFact{1}%
    \initDD@\figsetroundcoord{yes}\GR@critrue\expandafter\setupd@te\D@FTupdate:%
    \edef\disob@unit{\UnD@fined}\edef\t@rgetpt{\UnD@fined}\gdef\@utoFInDone{1}\gdef\@utoFN{0}}
\ctr@ld@f\def\initDD@{\Tr@isDimfalse%
    \ifPDFm@ke%
     \let\Ps@rcerc=\Ps@rcercBz%
     \let\Ps@rell=\Ps@rellBz%
    \fi
    \let\c@lDCUn=\c@lDCUnDD%
    \let\c@lDCDeux=\c@lDCDeuxDD%
    \let\c@ldefproj=\relax%
    \let\c@lproscal=\c@lproscalDD%
    \let\c@lprojSP=\relax%
    \let\extr@ctC=\extr@ctCDD%
    \let\extr@ctCa=\extr@ctCaDD%
    \let\extr@ctCF=\extr@ctCFDD%
    \let\Figp@intreg=\Figp@intregDD%
    \let\Figpts@xes=\Figpts@xesDD%
    \let\getredf@ctB=\getredf@ctBDD%
    \let\n@rmeucSV=\n@rmeucSVDD\let\n@rmeuc=\n@rmeucDD\let\n@rmeucC\n@rmeucCDD\let\n@rminf=\n@rminfDD%
    \let\pr@dMatV=\pr@dMatVDD%
    \let\Q@@xes=\Q@@xesDD%
    \let\vecunit@=\vecunit@DD%
    \let\figcoord=\figcoordDD%
    \let\figgetangle=\figgetangleDD%
    \let\figpt=\figptDD%
    \let\figptBezier=\figptBezierDD%
    \let\figptbary=\figptbaryDD%
    \let\figptcirc=\figptcircDD%
    \let\figptcircumcenter=\figptcircumcenterDD%
    \let\figptcopy=\figptcopyDD%
    \let\figptcurvcenter=\figptcurvcenterDD%
    \let\figptell=\figptellDD%
    \let\figptendnormal=\figptendnormalDD%
    \let\figptinterlineplane=\figptinterlineplaneDD%
    \let\figptinterlines=\inters@cDD%
    \let\figptorthocenter=\figptorthocenterDD%
    \let\figptorthoprojline=\figptorthoprojlineDD%
    \let\figptorthoprojplane=\figptorthoprojplaneDD%
    \let\figptrot=\figptrotDD%
    \let\figptscontrol=\figptscontrolDD%
    \let\figptsintercirc=\figptsintercircDD%
    \let\figptsinterlinell=\figptsinterlinellDD%
    \let\figptsorthoprojline=\figptsorthoprojlineDD%
    \let\figptorthoprojplane=\figptorthoprojplaneDD%
    \let\figptsrot=\figptsrotDD%
    \let\figptssym=\figptssymDD%
    \let\figptstra=\figptstraDD%
    \let\figptsym=\figptsymDD%
    \let\figpttraC=\figpttraCDD%
    \let\figpttra=\figpttraDD%
    \let\figptvisilimSL=\figptvisilimSLDD%
    \let\figsetobdist=\figsetobdistDD%
    \let\figsettarget=\figsettargetDD%
    \let\figsetview=\figsetviewDD%
    \let\figvectDBezier=\figvectDBezierDD%
    \let\figvectN=\figvectNDD%
    \let\figvectNV=\figvectNVDD%
    \let\figvectP=\figvectPDD%
    \let\figvectU=\figvectUDD%
    \let\figdrawarccircP=\Q@arccircPDD%
    \let\figdrawarccirc=\Q@arccircDD%
    \let\figdrawarcell=\Q@arcellDD%
    \let\figdrawarcellPA=\Q@arcellPADD%
    \let\figdrawarrowBezier=\Q@arrowBezierDD%
    \let\figdrawarrowcircP=\Q@arrowcircPDD%
    \let\figdrawarrowcirc=\Q@arrowcircDD%
    \let\figdrawarrowhead=\Q@arrowheadDD%
    \let\figdrawarrow=\Q@arrowDD%
    \let\figdrawBezier=\Q@BezierDD%
    \let\figdrawcirc=\Q@circDD%
    \let\figdrawcurve=\Q@curveDD%
    \let\figdrawnormal=\Q@normalDD%
    }
\ctr@ld@f\def\initTD@{\Tr@isDimtrue\initb@undb@xTD\newt@rgetptfalse\newdis@bfalse%
    \let\c@lDCUn=\c@lDCUnTD%
    \let\c@lDCDeux=\c@lDCDeuxTD%
    \let\c@ldefproj=\c@ldefprojTD%
    \let\c@lproscal=\c@lproscalTD%
    \let\extr@ctC=\extr@ctCTD%
    \let\extr@ctCa=\extr@ctCaTD%
    \let\extr@ctCF=\extr@ctCFTD%
    \let\Figp@intreg=\Figp@intregTD%
    \let\Figpts@xes=\Figpts@xesTD%
    \let\getredf@ctB=\getredf@ctBTD%
    \let\n@rmeucSV=\n@rmeucSVTD\let\n@rmeuc=\n@rmeucTD\let\n@rmeucC\n@rmeucCTD\let\n@rminf=\n@rminfTD%
    \let\pr@dMatV=\pr@dMatVTD%
    \let\Q@@xes=\Q@@xesTD%
    \let\vecunit@=\vecunit@TD%
    \let\figcoord=\figcoordTD%
    \let\figgetangle=\figgetangleTD%
    \let\figpt=\figptTD%
    \let\figptBezier=\figptBezierTD%
    \let\figptbary=\figptbaryTD%
    \let\figptcirc=\figptcircTD%
    \let\figptcircumcenter=\figptcircumcenterTD%
    \let\figptcopy=\figptcopyTD%
    \let\figptcurvcenter=\figptcurvcenterTD%
    \let\figptinterlineplane=\figptinterlineplaneTD%
    \let\figptinterlines=\inters@cTD%
    \let\figptorthocenter=\figptorthocenterTD%
    \let\figptorthoprojline=\figptorthoprojlineTD%
    \let\figptorthoprojplane=\figptorthoprojplaneTD%
    \let\figptrot=\figptrotTD%
    \let\figptscontrol=\figptscontrolTD%
    \let\figptsintercirc=\figptsintercircTD%
    \let\figptsorthoprojline=\figptsorthoprojlineTD%
    \let\figptsorthoprojplane=\figptsorthoprojplaneTD%
    \let\figptsrot=\figptsrotTD%
    \let\figptssym=\figptssymTD%
    \let\figptstra=\figptstraTD%
    \let\figptsym=\figptsymTD%
    \let\figpttraC=\figpttraCTD%
    \let\figpttra=\figpttraTD%
    \let\figptvisilimSL=\figptvisilimSLTD%
    \let\figsetobdist=\figsetobdistTD%
    \let\figsettarget=\figsettargetTD%
    \let\figsetview=\figsetviewTD%
    \let\figvectDBezier=\figvectDBezierTD%
    \let\figvectN=\figvectNTD%
    \let\figvectNV=\figvectNVTD%
    \let\figvectP=\figvectPTD%
    \let\figvectU=\figvectUTD%
    \let\figdrawarccircP=\Q@arccircPTD%
    \let\figdrawarccirc=\Q@arccircTD%
    \let\figdrawarcell=\Q@arcellTD%
    \let\figdrawarcellPA=\Q@arcellPATD%
    \let\figdrawarrowBezier=\Q@arrowBezierTD%
    \let\figdrawarrowcircP=\Q@arrowcircPTD%
    \let\figdrawarrowcirc=\Q@arrowcircTD%
    \let\figdrawarrowhead=\Q@arrowheadTD%
    \let\figdrawarrow=\Q@arrowTD%
    \let\figdrawBezier=\Q@BezierTD%
    \let\figdrawcirc=\Q@circTD%
    \let\figdrawcurve=\Q@curveTD%
    }
\ctr@ld@f\def\un@v@ilable#1{\immediate\write16{*** The macro #1 is not available in the current context.}}
\ctr@ld@f\def\figinsert#1{{\def\t@xt@{#1}\relax%
    \ifx\t@xt@\empty\ifnum\@utoFInDone>\z@\Figinsert@\DefGIfilen@me,:\fi%
    \else\expandafter\FiginsertNu@#1 :\fi}\ignorespaces}
\ctr@ld@f\def\FiginsertNu@#1 #2:{\def\t@xt@{#1}\relax\ifx\t@xt@\empty\def\t@xt@{#2}%
    \ifx\t@xt@\empty\ifnum\@utoFInDone>\z@\Figinsert@\DefGIfilen@me,:\fi%
    \else\FiginsertNu@#2:\fi\else\expandafter\FiginsertNd@#1 #2:\fi}
\ctr@ld@f\def\FiginsertNd@#1#2:{\ifcat#1a\Figinsert@#1#2,:\else%
    \ifnum\@utoFInDone>\z@\Figinsert@\DefGIfilen@me,#1#2,:\fi\fi}
\ctr@ln@m\Sc@leFact
\ctr@ld@f\def\Figinsert@#1,#2:{\def\t@xt@{#2}\ifx\t@xt@\empty\xdef\Sc@leFact{1}\else%
    \X@rgdeux@#2\xdef\Sc@leFact{\@rgdeux}\fi%
    \Figdisc@rdLTS{#1}{\t@xt@}\@psfgetbb{\t@xt@}%
    \v@lX=\@psfllx\p@\v@lX=\ptpsT@pt\v@lX\v@lX=\Sc@leFact\v@lX%
    \v@lY=\@psflly\p@\v@lY=\ptpsT@pt\v@lY\v@lY=\Sc@leFact\v@lY%
    \b@undb@x{\v@lX}{\v@lY}%
    \v@lX=\@psfurx\p@\v@lX=\ptpsT@pt\v@lX\v@lX=\Sc@leFact\v@lX%
    \v@lY=\@psfury\p@\v@lY=\ptpsT@pt\v@lY\v@lY=\Sc@leFact\v@lY%
    \b@undb@x{\v@lX}{\v@lY}%
    \ifPDFm@ke\Figinclud@PDF{\t@xt@}{\Sc@leFact}\else%
    \v@lX=\c@nt pt\v@lX=\Sc@leFact\v@lX\edef\F@ct{\repdecn@mb{\v@lX}}%
    \ifx\TeXturesonMacOSltX\special{postscriptfile #1 vscale=\F@ct\space hscale=\F@ct}%
    \else\includegraphics{#1}\fi\fi%
    \message{[\t@xt@]}\ignorespaces}
\ctr@ld@f\def\Figdisc@rdLTS#1#2{\expandafter\Figdisc@rdLTS@#1 :#2}
\ctr@ld@f\def\Figdisc@rdLTS@#1 #2:#3{\def#3{#1}\relax\ifx#3\empty\expandafter\Figdisc@rdLTS@#2:#3\fi}
\ctr@ld@f\def\figinsertE#1{\FiginsertE@#1,:\ignorespaces}
\ctr@ld@f\def\FiginsertE@#1,#2:{{\def\t@xt@{#2}\ifx\t@xt@\empty\xdef\Sc@leFact{1}\else%
    \X@rgdeux@#2\xdef\Sc@leFact{\@rgdeux}\fi%
    \Figdisc@rdLTS{#1}{\t@xt@}\pdfximage{\t@xt@}%
    \setbox\Gb@x=\hbox{\pdfrefximage\pdflastximage}%
    \v@lX=\z@\v@lY=-\Sc@leFact\dp\Gb@x\b@undb@x{\v@lX}{\v@lY}%
    \advance\v@lX\Sc@leFact\wd\Gb@x\advance\v@lY\Sc@leFact\dp\Gb@x%
    \advance\v@lY\Sc@leFact\ht\Gb@x\b@undb@x{\v@lX}{\v@lY}%
    \v@lX=\Sc@leFact\wd\Gb@x\pdfximage width \v@lX {\t@xt@}%
    \rlap{\pdfrefximage\pdflastximage}\message{[\t@xt@]}}\ignorespaces}
\ctr@ld@f\def\X@rgdeux@#1,{\edef\@rgdeux{#1}}
\ctr@ln@m\figpt
\ctr@ld@f\def\figptDD#1:#2(#3,#4){\ifGR@cri\c@ntr@lnum{#1}%
    {\v@lX=#3\unit@\v@lY=#4\unit@\Fig@dmpt{#2}{\z@}}\ignorespaces\fi}
\ctr@ld@f\def\Fig@dmpt#1#2{\def\t@xt@{#1}\ifx\t@xt@\empty\def\B@@ltxt{\z@}%
    \else\expandafter\gdef\csname\objc@de T\endcsname{#1}\def\B@@ltxt{\@ne}\fi%
    \expandafter\xdef\csname\objc@de\endcsname{\ifitis@vect@r\C@dCl@svect%
    \else\C@dCl@spt\fi,\z@,\B@@ltxt/\the\v@lX,\the\v@lY,#2}}
\ctr@ld@f\def\C@dCl@spt{P}
\ctr@ld@f\def\C@dCl@svect{V}
\ctr@ln@m\c@@rdYZ
\ctr@ln@m\c@@rdY
\ctr@ld@f\def\figptTD#1:#2(#3,#4){\ifGR@cri\c@ntr@lnum{#1}%
    \def\c@@rdYZ{#4,0,0}\extrairelepremi@r\c@@rdY\de\c@@rdYZ%
    \extrairelepremi@r\c@@rdZ\de\c@@rdYZ%
    {\v@lX=#3\unit@\v@lY=\c@@rdY\unit@\v@lZ=\c@@rdZ\unit@\Fig@dmpt{#2}{\the\v@lZ}%
    \b@undb@xTD{\v@lX}{\v@lY}{\v@lZ}}\ignorespaces\fi}
\ctr@ln@m\Figp@intreg
\ctr@ld@f\def\Figp@intregDD#1:#2(#3,#4){\c@ntr@lnum{#1}%
    {\result@t=#4\v@lX=#3\v@lY=\result@t\Fig@dmpt{#2}{\z@}}\ignorespaces}
\ctr@ld@f\def\Figp@intregTD#1:#2(#3,#4){\c@ntr@lnum{#1}%
    \def\c@@rdYZ{#4,\z@,\z@}\extrairelepremi@r\c@@rdY\de\c@@rdYZ%
    \extrairelepremi@r\c@@rdZ\de\c@@rdYZ%
    {\v@lX=#3\v@lY=\c@@rdY\v@lZ=\c@@rdZ\Fig@dmpt{#2}{\the\v@lZ}%
    \b@undb@xTD{\v@lX}{\v@lY}{\v@lZ}}\ignorespaces}
\ctr@ln@m\figptBezier
\ctr@ld@f\def\figptBezierDD#1:#2:#3[#4,#5,#6,#7]{\ifGR@cri{\s@uvc@ntr@l\et@tfigptBezierDD%
    \FigptBezier@#3[#4,#5,#6,#7]\Figp@intregDD#1:{#2}(\v@lX,\v@lY)%
    \resetc@ntr@l\et@tfigptBezierDD}\ignorespaces\fi}
\ctr@ld@f\def\figptBezierTD#1:#2:#3[#4,#5,#6,#7]{\ifGR@cri{\s@uvc@ntr@l\et@tfigptBezierTD%
    \FigptBezier@#3[#4,#5,#6,#7]\Figp@intregTD#1:{#2}(\v@lX,\v@lY,\v@lZ)%
    \resetc@ntr@l\et@tfigptBezierTD}\ignorespaces\fi}
\ctr@ld@f\def\FigptBezier@#1[#2,#3,#4,#5]{\setc@ntr@l{2}%
    \edef\T@{#1}\v@leur=\p@\advance\v@leur-#1pt\edef\UNmT@{\repdecn@mb{\v@leur}}%
    \figptcopy-4:/#2/\figptcopy-3:/#3/\figptcopy-2:/#4/\figptcopy-1:/#5/%
    \l@mbd@un=-4 \l@mbd@de=-\thr@@\p@rtent=\m@ne\c@lDecast%
    \l@mbd@un=-4 \l@mbd@de=-\thr@@\p@rtent=-\tw@\c@lDecast%
    \l@mbd@un=-4 \l@mbd@de=-\thr@@\p@rtent=-\thr@@\c@lDecast\Figg@tXY{-4}}
\ctr@ln@m\c@lDCUn
\ctr@ld@f\def\c@lDCUnDD#1#2{\Figg@tXY{#1}\v@lX=\UNmT@\v@lX\v@lY=\UNmT@\v@lY%
    \Figg@tXYa{#2}\advance\v@lX\T@\v@lXa\advance\v@lY\T@\v@lYa%
    \Figp@intregDD#1:(\v@lX,\v@lY)}
\ctr@ld@f\def\c@lDCUnTD#1#2{\Figg@tXY{#1}\v@lX=\UNmT@\v@lX\v@lY=\UNmT@\v@lY\v@lZ=\UNmT@\v@lZ%
    \Figg@tXYa{#2}\advance\v@lX\T@\v@lXa\advance\v@lY\T@\v@lYa\advance\v@lZ\T@\v@lZa%
    \Figp@intregTD#1:(\v@lX,\v@lY,\v@lZ)}
\ctr@ld@f\def\c@lDecast{\relax\ifnum\l@mbd@un<\p@rtent\c@lDCUn{\l@mbd@un}{\l@mbd@de}%
    \advance\l@mbd@un\@ne\advance\l@mbd@de\@ne\c@lDecast\fi}
\ctr@ld@f\def\figptmap#1:#2=#3/#4/#5/{\ifGR@cri{\s@uvc@ntr@l\et@tfigptmap%
    \setc@ntr@l{2}\figvectP-1[#4,#3]\Figg@tXY{-1}%
    \pr@dMatV/#5/\figpttra#1:{#2}=#4/1,-1/%
    \resetc@ntr@l\et@tfigptmap}\ignorespaces\fi}
\ctr@ln@m\pr@dMatV
\ctr@ld@f\def\pr@dMatVDD/#1,#2;#3,#4/{\v@lXa=#1\v@lX\advance\v@lXa#2\v@lY%
    \v@lYa=#3\v@lX\advance\v@lYa#4\v@lY\Figv@ctCreg-1(\v@lXa,\v@lYa)}
\ctr@ld@f\def\pr@dMatVTD/#1,#2,#3;#4,#5,#6;#7,#8,#9/{%
    \v@lXa=#1\v@lX\advance\v@lXa#2\v@lY\advance\v@lXa#3\v@lZ%
    \v@lYa=#4\v@lX\advance\v@lYa#5\v@lY\advance\v@lYa#6\v@lZ%
    \v@lZa=#7\v@lX\advance\v@lZa#8\v@lY\advance\v@lZa#9\v@lZ%
    \Figv@ctCreg-1(\v@lXa,\v@lYa,\v@lZa)}
\ctr@ln@m\figptbary
\ctr@ld@f\def\figptbaryDD#1:#2[#3;#4]{\ifGR@cri{\edef\list@num{#3}\extrairelepremi@r\p@int\de\list@num%
    \s@mme=\z@\@ecfor\c@ef:=#4\do{\advance\s@mme\c@ef}%
    \edef\listec@ef{#4,0}\extrairelepremi@r\c@ef\de\listec@ef%
    \Figg@tXY{\p@int}\divide\v@lX\s@mme\divide\v@lY\s@mme%
    \multiply\v@lX\c@ef\multiply\v@lY\c@ef%
    \@ecfor\p@int:=\list@num\do{\extrairelepremi@r\c@ef\de\listec@ef%
           \Figg@tXYa{\p@int}\divide\v@lXa\s@mme\divide\v@lYa\s@mme%
           \multiply\v@lXa\c@ef\multiply\v@lYa\c@ef%
           \advance\v@lX\v@lXa\advance\v@lY\v@lYa}%
    \Figp@intregDD#1:{#2}(\v@lX,\v@lY)}\ignorespaces\fi}
\ctr@ld@f\def\figptbaryTD#1:#2[#3;#4]{\ifGR@cri{\edef\list@num{#3}\extrairelepremi@r\p@int\de\list@num%
    \s@mme=\z@\@ecfor\c@ef:=#4\do{\advance\s@mme\c@ef}%
    \edef\listec@ef{#4,0}\extrairelepremi@r\c@ef\de\listec@ef%
    \Figg@tXY{\p@int}\divide\v@lX\s@mme\divide\v@lY\s@mme\divide\v@lZ\s@mme%
    \multiply\v@lX\c@ef\multiply\v@lY\c@ef\multiply\v@lZ\c@ef%
    \@ecfor\p@int:=\list@num\do{\extrairelepremi@r\c@ef\de\listec@ef%
           \Figg@tXYa{\p@int}\divide\v@lXa\s@mme\divide\v@lYa\s@mme\divide\v@lZa\s@mme%
           \multiply\v@lXa\c@ef\multiply\v@lYa\c@ef\multiply\v@lZa\c@ef%
           \advance\v@lX\v@lXa\advance\v@lY\v@lYa\advance\v@lZ\v@lZa}%
    \Figp@intregTD#1:{#2}(\v@lX,\v@lY,\v@lZ)}\ignorespaces\fi}
\ctr@ld@f\def\figptbaryR#1:#2[#3;#4]{\ifGR@cri{%
    \v@leur=\z@\@ecfor\c@ef:=#4\do{\maxim@m{\v@lmax}{\c@ef pt}{-\c@ef pt}%
    \ifdim\v@lmax>\v@leur\v@leur=\v@lmax\fi}%
    \ifdim\v@leur<\p@\f@ctech=\@M\else\ifdim\v@leur<\t@n\p@\f@ctech=\@m\else%
    \ifdim\v@leur<\c@nt\p@\f@ctech=\c@nt\else\ifdim\v@leur<\@m\p@\f@ctech=\t@n\else%
    \f@ctech=\@ne\fi\fi\fi\fi%
    \def\listec@ef{0}%
    \@ecfor\c@ef:=#4\do{\sc@lec@nvRI{\c@ef pt}\edef\listec@ef{\listec@ef,\the\s@mme}}%
    \extrairelepremi@r\c@ef\de\listec@ef\figptbary#1:#2[#3;\listec@ef]}\ignorespaces\fi}
\ctr@ld@f\def\sc@lec@nvRI#1{\v@leur=#1\p@rtentiere{\s@mme}{\v@leur}\advance\v@leur-\s@mme\p@%
    \multiply\v@leur\f@ctech\p@rtentiere{\p@rtent}{\v@leur}%
    \multiply\s@mme\f@ctech\advance\s@mme\p@rtent}
\ctr@ln@m\figptcirc
\ctr@ld@f\def\figptcircDD#1:#2:#3;#4(#5){\ifGR@cri{\s@uvc@ntr@l\et@tfigptcircDD%
    \c@lptellDD#1:{#2}:#3;#4,#4(#5)\resetc@ntr@l\et@tfigptcircDD}\ignorespaces\fi}
\ctr@ld@f\def\figptcircTD#1:#2:#3,#4,#5;#6(#7){\ifGR@cri{\s@uvc@ntr@l\et@tfigptcircTD%
    \setc@ntr@l{2}\c@lExtAxes#3,#4,#5(#6)\figptellP#1:{#2}:#3,-4,-5(#7)%
    \resetc@ntr@l\et@tfigptcircTD}\ignorespaces\fi}
\ctr@ln@m\figptcircumcenter
\ctr@ld@f\def\figptcircumcenterDD#1:#2[#3,#4,#5]{\ifGR@cri{\s@uvc@ntr@l\et@tfigptcircumcenterDD%
    \setc@ntr@l{2}\figvectNDD-5[#3,#4]\figptbaryDD-3:[#3,#4;1,1]%
                  \figvectNDD-6[#4,#5]\figptbaryDD-4:[#4,#5;1,1]%
    \resetc@ntr@l{2}\inters@cDD#1:{#2}[-3,-5;-4,-6]%
    \resetc@ntr@l\et@tfigptcircumcenterDD}\ignorespaces\fi}
\ctr@ld@f\def\figptcircumcenterTD#1:#2[#3,#4,#5]{\ifGR@cri{\s@uvc@ntr@l\et@tfigptcircumcenterTD%
    \setc@ntr@l{2}\figvectNTD-1[#3,#4,#5]%
    \figvectPTD-3[#3,#4]\figvectNVTD-5[-1,-3]\figptbaryTD-3:[#3,#4;1,1]%
    \figvectPTD-4[#4,#5]\figvectNVTD-6[-1,-4]\figptbaryTD-4:[#4,#5;1,1]%
    \resetc@ntr@l{2}\inters@cTD#1:{#2}[-3,-5;-4,-6]%
    \resetc@ntr@l\et@tfigptcircumcenterTD}\ignorespaces\fi}
\ctr@ln@m\figptcopy
\ctr@ld@f\def\figptcopyDD#1:#2/#3/{\ifGR@cri{\Figg@tXY{#3}%
    \Figp@intregDD#1:{#2}(\v@lX,\v@lY)}\ignorespaces\fi}
\ctr@ld@f\def\figptcopyTD#1:#2/#3/{\ifGR@cri{\Figg@tXY{#3}%
    \Figp@intregTD#1:{#2}(\v@lX,\v@lY,\v@lZ)}\ignorespaces\fi}
\ctr@ln@m\figptcurvcenter
\ctr@ld@f\def\figptcurvcenterDD#1:#2:#3[#4,#5,#6,#7]{\ifGR@cri{\s@uvc@ntr@l\et@tfigptcurvcenterDD%
    \setc@ntr@l{2}\c@lcurvradDD#3[#4,#5,#6,#7]\edef\Sprim@{\repdecn@mb{\result@t}}%
    \figptBezierDD-1::#3[#4,#5,#6,#7]\figpttraDD#1:{#2}=-1/\Sprim@,-5/%
    \resetc@ntr@l\et@tfigptcurvcenterDD}\ignorespaces\fi}
\ctr@ld@f\def\figptcurvcenterTD#1:#2:#3[#4,#5,#6,#7]{\ifGR@cri{\s@uvc@ntr@l\et@tfigptcurvcenterTD%
    \setc@ntr@l{2}\figvectDBezierTD -5:1,#3[#4,#5,#6,#7]%
    \figvectDBezierTD -6:2,#3[#4,#5,#6,#7]\vecunit@TD{-5}{-5}%
    \edef\Sprim@{\repdecn@mb{\result@t}}\figvectNVTD-1[-6,-5]%
    \figvectNVTD-5[-5,-1]\c@lproscalTD\v@leur[-6,-5]%
    \invers@{\v@leur}{\v@leur}\v@leur=\Sprim@\v@leur\v@leur=\Sprim@\v@leur%
    \figptBezierTD-1::#3[#4,#5,#6,#7]\edef\Sprim@{\repdecn@mb{\v@leur}}%
    \figpttraTD#1:{#2}=-1/\Sprim@,-5/\resetc@ntr@l\et@tfigptcurvcenterTD}\ignorespaces\fi}
\ctr@ld@f\def\c@lcurvradDD#1[#2,#3,#4,#5]{{\figvectDBezierDD -5:1,#1[#2,#3,#4,#5]%
    \figvectDBezierDD -6:2,#1[#2,#3,#4,#5]\vecunit@DD{-5}{-5}%
    \edef\Sprim@{\repdecn@mb{\result@t}}\figvectNVDD-5[-5]\c@lproscalDD\v@leur[-6,-5]%
    \invers@{\v@leur}{\v@leur}\v@leur=\Sprim@\v@leur\v@leur=\Sprim@\v@leur%
    \global\result@t=\v@leur}}
\ctr@ln@m\figptell
\ctr@ld@f\def\figptellDD#1:#2:#3;#4,#5(#6,#7){\ifGR@cri{\s@uvc@ntr@l\et@tfigptell%
    \c@lptellDD#1::#3;#4,#5(#6)\figptrotDD#1:{#2}=#1/#3,#7/%
    \resetc@ntr@l\et@tfigptell}\ignorespaces\fi}
\ctr@ld@f\def\c@lptellDD#1:#2:#3;#4,#5(#6){\c@ssin{\C@}{\S@}{#6}\v@lmin=\C@ pt\v@lmax=\S@ pt%
    \v@lmin=#4\v@lmin\v@lmax=#5\v@lmax%
    \edef\Xc@mp{\repdecn@mb{\v@lmin}}\edef\Yc@mp{\repdecn@mb{\v@lmax}}%
    \setc@ntr@l{2}\figvectC-1(\Xc@mp,\Yc@mp)\figpttraDD#1:{#2}=#3/1,-1/}
\ctr@ld@f\def\figptellP#1:#2:#3,#4,#5(#6){\ifGR@cri{\s@uvc@ntr@l\et@tfigptellP%
    \setc@ntr@l{2}\figvectP-1[#3,#4]\figvectP-2[#3,#5]%
    \v@leur=#6pt\c@lptellP{#3}{-1}{-2}\figptcopy#1:{#2}/-3/%
    \resetc@ntr@l\et@tfigptellP}\ignorespaces\fi}
\ctr@ln@m\@ngle
\ctr@ld@f\def\c@lptellP#1#2#3{\edef\@ngle{\repdecn@mb\v@leur}\c@ssin{\C@}{\S@}{\@ngle}%
    \figpttra-3:=#1/\C@,#2/\figpttra-3:=-3/\S@,#3/}
\ctr@ln@m\figptendnormal
\ctr@ld@f\def\figptendnormalDD#1:#2:#3,#4[#5,#6]{\ifGR@cri{\s@uvc@ntr@l\et@tfigptendnormal%
    \Figg@tXYa{#5}\Figg@tXY{#6}%
    \advance\v@lX-\v@lXa\advance\v@lY-\v@lYa%
    \setc@ntr@l{2}\Figv@ctCreg-1(\v@lX,\v@lY)\vecunit@{-1}{-1}\Figg@tXY{-1}%
    \delt@=#3\unit@\maxim@m{\delt@}{\delt@}{-\delt@}\edef\l@ngueur{\repdecn@mb{\delt@}}%
    \v@lX=\l@ngueur\v@lX\v@lY=\l@ngueur\v@lY%
    \delt@=\p@\advance\delt@-#4pt\edef\l@ngueur{\repdecn@mb{\delt@}}%
    \figptbaryR-1:[#5,#6;#4,\l@ngueur]\Figg@tXYa{-1}%
    \advance\v@lXa\v@lY\advance\v@lYa-\v@lX%
    \setc@ntr@l{1}\Figp@intregDD#1:{#2}(\v@lXa,\v@lYa)\resetc@ntr@l\et@tfigptendnormal}%
    \ignorespaces\fi}
\ctr@ld@f\def\figptexcenter#1:#2[#3,#4,#5]{\ifGR@cri{\let@xte={-}
    \Figptexinsc@nter#1:#2[#3,#4,#5]}\ignorespaces\fi}
\ctr@ld@f\def\figptincenter#1:#2[#3,#4,#5]{\ifGR@cri{\let@xte={}
    \Figptexinsc@nter#1:#2[#3,#4,#5]}\ignorespaces\fi}
\ctr@ld@f
\ctr@ld@f\def\Figptexinsc@nter#1:#2[#3,#4,#5]{%
    \figgetdist\LA@[#4,#5]\figgetdist\LB@[#3,#5]\figgetdist\LC@[#3,#4]%
    \figptbaryR#1:{#2}[#3,#4,#5;\the\let@xte\LA@,\LB@,\LC@]}
\ctr@ln@m\figptinterlineplane
\ctr@ld@f\def\figptinterlineplaneDD{\un@v@ilable{figptinterlineplane}}
\ctr@ld@f\def\figptinterlineplaneTD#1:#2[#3,#4;#5,#6]{\ifGR@cri{\s@uvc@ntr@l\et@tfigptinterlineplane%
    \setc@ntr@l{2}\figvectPTD-1[#3,#5]\vecunit@TD{-2}{#6}%
    \r@pPSTD\v@leur[-2,-1,#4]\edef\v@lcoef{\repdecn@mb{\v@leur}}%
    \figpttraTD#1:{#2}=#3/\v@lcoef,#4/\resetc@ntr@l\et@tfigptinterlineplane}\ignorespaces\fi}
\ctr@ln@m\figptorthocenter
\ctr@ld@f\def\figptorthocenterDD#1:#2[#3,#4,#5]{\ifGR@cri{\s@uvc@ntr@l\et@tfigptorthocenterDD%
    \setc@ntr@l{2}\figvectNDD-3[#3,#4]\figvectNDD-4[#4,#5]%
    \resetc@ntr@l{2}\inters@cDD#1:{#2}[#5,-3;#3,-4]%
    \resetc@ntr@l\et@tfigptorthocenterDD}\ignorespaces\fi}
\ctr@ld@f\def\figptorthocenterTD#1:#2[#3,#4,#5]{\ifGR@cri{\s@uvc@ntr@l\et@tfigptorthocenterTD%
    \setc@ntr@l{2}\figvectNTD-1[#3,#4,#5]%
    \figvectPTD-2[#3,#4]\figvectNVTD-3[-1,-2]%
    \figvectPTD-2[#4,#5]\figvectNVTD-4[-1,-2]%
    \resetc@ntr@l{2}\inters@cTD#1:{#2}[#5,-3;#3,-4]%
    \resetc@ntr@l\et@tfigptorthocenterTD}\ignorespaces\fi}
\ctr@ln@m\figptorthoprojline
\ctr@ld@f\def\figptorthoprojlineDD#1:#2=#3/#4,#5/{\ifGR@cri{\s@uvc@ntr@l\et@tfigptorthoprojlineDD%
    \setc@ntr@l{2}\figvectPDD-3[#4,#5]\figvectNVDD-4[-3]\resetc@ntr@l{2}%
    \inters@cDD#1:{#2}[#3,-4;#4,-3]\resetc@ntr@l\et@tfigptorthoprojlineDD}\ignorespaces\fi}
\ctr@ld@f\def\figptorthoprojlineTD#1:#2=#3/#4,#5/{\ifGR@cri{\s@uvc@ntr@l\et@tfigptorthoprojlineTD%
    \setc@ntr@l{2}\figvectPTD-1[#4,#3]\figvectPTD-2[#4,#5]\vecunit@TD{-2}{-2}%
    \c@lproscalTD\v@leur[-1,-2]\edef\v@lcoef{\repdecn@mb{\v@leur}}%
    \figpttraTD#1:{#2}=#4/\v@lcoef,-2/\resetc@ntr@l\et@tfigptorthoprojlineTD}\ignorespaces\fi}
\ctr@ln@m\figptorthoprojplane
\ctr@ld@f\def\figptorthoprojplaneDD{\un@v@ilable{figptorthoprojplane}}
\ctr@ld@f\def\figptorthoprojplaneTD#1:#2=#3/#4,#5/{\ifGR@cri{\s@uvc@ntr@l\et@tfigptorthoprojplane%
    \setc@ntr@l{2}\figvectPTD-1[#3,#4]\vecunit@TD{-2}{#5}%
    \c@lproscalTD\v@leur[-1,-2]\edef\v@lcoef{\repdecn@mb{\v@leur}}%
    \figpttraTD#1:{#2}=#3/\v@lcoef,-2/\resetc@ntr@l\et@tfigptorthoprojplane}\ignorespaces\fi}
\ctr@ld@f\def\figpthom#1:#2=#3/#4,#5/{\ifGR@cri{\s@uvc@ntr@l\et@tfigpthom%
    \setc@ntr@l{2}\figvectP-1[#4,#3]\figpttra#1:{#2}=#4/#5,-1/%
    \resetc@ntr@l\et@tfigpthom}\ignorespaces\fi}
\ctr@ld@f\def\figptinv#1:#2=#3/#4,#5/{\ifGR@cri{\s@uvc@ntr@l\et@tfigptinv%
    \setc@ntr@l{2}\figvectP-1[#4,#3]\Figg@tXY{-1}%
    \getredf@ctB\f@ctech\n@rmeucC{\delt@}{-1}%
    \delt@=\ptT@unit@\delt@\delt@=\ptT@unit@\delt@%
    \invers@{\delt@}{\delt@}\multiply\f@ctech\f@ctech\divide\delt@\f@ctech%
    \delt@=#5\delt@\edef\v@lcoef{\repdecn@mb{\delt@}}\figpttra#1:{#2}=#4/\v@lcoef,-1/%
    \resetc@ntr@l\et@tfigptinv}\ignorespaces\fi}
\ctr@ln@m\figptrot
\ctr@ld@f\def\figptrotDD#1:#2=#3/#4,#5/{\ifGR@cri{\s@uvc@ntr@l\et@tfigptrotDD%
    \c@ssin{\C@}{\S@}{#5}\setc@ntr@l{2}\figvectPDD-1[#4,#3]\Figg@tXY{-1}%
    \v@lXa=\C@\v@lX\advance\v@lXa-\S@\v@lY%
    \v@lYa=\S@\v@lX\advance\v@lYa\C@\v@lY%
    \Figv@ctCreg-1(\v@lXa,\v@lYa)\figpttraDD#1:{#2}=#4/1,-1/%
    \resetc@ntr@l\et@tfigptrotDD}\ignorespaces\fi}
\ctr@ld@f\def\figptrotTD#1:#2=#3/#4,#5,#6/{\ifGR@cri{\s@uvc@ntr@l\et@tfigptrotTD%
    \c@ssin{\C@}{\S@}{#5}%
    \setc@ntr@l{2}\figptorthoprojplaneTD-3:=#4/#3,#6/\figvectPTD-2[-3,#3]%
    \n@rmeucTD\v@leur{-2}\ifdim\v@leur<\Cepsil@n\Figg@tXYa{#3}\else%
    \edef\v@lcoef{\repdecn@mb{\v@leur}}\figvectNVTD-1[#6,-2]%
    \Figg@tXYa{-1}\v@lXa=\v@lcoef\v@lXa\v@lYa=\v@lcoef\v@lYa\v@lZa=\v@lcoef\v@lZa%
    \v@lXa=\S@\v@lXa\v@lYa=\S@\v@lYa\v@lZa=\S@\v@lZa\Figg@tXY{-2}%
    \advance\v@lXa\C@\v@lX\advance\v@lYa\C@\v@lY\advance\v@lZa\C@\v@lZ%
    \Figg@tXY{-3}\advance\v@lXa\v@lX\advance\v@lYa\v@lY\advance\v@lZa\v@lZ\fi%
    \Figp@intregTD#1:{#2}(\v@lXa,\v@lYa,\v@lZa)\resetc@ntr@l\et@tfigptrotTD}\ignorespaces\fi}
\ctr@ln@m\figptsym
\ctr@ld@f\def\figptsymDD#1:#2=#3/#4,#5/{\ifGR@cri{\s@uvc@ntr@l\et@tfigptsymDD%
    \resetc@ntr@l{2}\figptorthoprojlineDD-5:=#3/#4,#5/\figvectPDD-2[#3,-5]%
    \figpttraDD#1:{#2}=#3/2,-2/\resetc@ntr@l\et@tfigptsymDD}\ignorespaces\fi}
\ctr@ld@f\def\figptsymTD#1:#2=#3/#4,#5/{\ifGR@cri{\s@uvc@ntr@l\et@tfigptsymTD%
    \resetc@ntr@l{2}\figptorthoprojplaneTD-3:=#3/#4,#5/\figvectPTD-2[#3,-3]%
    \figpttraTD#1:{#2}=#3/2,-2/\resetc@ntr@l\et@tfigptsymTD}\ignorespaces\fi}
\ctr@ln@m\figpttra
\ctr@ld@f\def\figpttraDD#1:#2=#3/#4,#5/{\ifGR@cri{\Figg@tXYa{#5}\v@lXa=#4\v@lXa\v@lYa=#4\v@lYa%
    \Figg@tXY{#3}\advance\v@lX\v@lXa\advance\v@lY\v@lYa%
    \Figp@intregDD#1:{#2}(\v@lX,\v@lY)}\ignorespaces\fi}
\ctr@ld@f\def\figpttraTD#1:#2=#3/#4,#5/{\ifGR@cri{\Figg@tXYa{#5}\v@lXa=#4\v@lXa\v@lYa=#4\v@lYa%
    \v@lZa=#4\v@lZa\Figg@tXY{#3}\advance\v@lX\v@lXa\advance\v@lY\v@lYa%
    \advance\v@lZ\v@lZa\Figp@intregTD#1:{#2}(\v@lX,\v@lY,\v@lZ)}\ignorespaces\fi}
\ctr@ln@m\figpttraC
\ctr@ld@f\def\figpttraCDD#1:#2=#3/#4,#5/{\ifGR@cri{\v@lXa=#4\unit@\v@lYa=#5\unit@%
    \Figg@tXY{#3}\advance\v@lX\v@lXa\advance\v@lY\v@lYa%
    \Figp@intregDD#1:{#2}(\v@lX,\v@lY)}\ignorespaces\fi}
\ctr@ld@f\def\figpttraCTD#1:#2=#3/#4,#5,#6/{\ifGR@cri{\v@lXa=#4\unit@\v@lYa=#5\unit@\v@lZa=#6\unit@%
    \Figg@tXY{#3}\advance\v@lX\v@lXa\advance\v@lY\v@lYa\advance\v@lZ\v@lZa%
    \Figp@intregTD#1:{#2}(\v@lX,\v@lY,\v@lZ)}\ignorespaces\fi}
\ctr@ld@f\def\figptsaxes#1:#2(#3){\ifGR@cri{\an@lys@xes#3,:\ifx\t@xt@\empty%
    \ifTr@isDim\Figpts@xes#1:#2(0,#3,0,#3,0,#3)\else\Figpts@xes#1:#2(0,#3,0,#3)\fi%
    \else\Figpts@xes#1:#2(#3)\fi}\ignorespaces\fi}
\ctr@ln@m\Figpts@xes
\ctr@ld@f\def\Figpts@xesDD#1:#2(#3,#4,#5,#6){%
    \s@mme=#1\figpttraC\the\s@mme:$x$=#2/#4,0/%
    \advance\s@mme\@ne\figpttraC\the\s@mme:$y$=#2/0,#6/}
\ctr@ld@f\def\Figpts@xesTD#1:#2(#3,#4,#5,#6,#7,#8){%
    \s@mme=#1\figpttraC\the\s@mme:$x$=#2/#4,0,0/%
    \advance\s@mme\@ne\figpttraC\the\s@mme:$y$=#2/0,#6,0/%
    \advance\s@mme\@ne\figpttraC\the\s@mme:$z$=#2/0,0,#8/}
\ctr@ld@f\def\figptsmap#1=#2/#3/#4/{\ifGR@cri{\s@uvc@ntr@l\et@tfigptsmap%
    \setc@ntr@l{2}\def\list@num{#2}\s@mme=#1%
    \@ecfor\p@int:=\list@num\do{\figvectP-1[#3,\p@int]\Figg@tXY{-1}%
    \pr@dMatV/#4/\figpttra\the\s@mme:=#3/1,-1/\advance\s@mme\@ne}%
    \resetc@ntr@l\et@tfigptsmap}\ignorespaces\fi}
\ctr@ln@m\figptscontrol
\ctr@ld@f\def\figptscontrolDD#1[#2,#3,#4,#5]{\ifGR@cri{\s@uvc@ntr@l\et@tfigptscontrolDD\setc@ntr@l{2}%
    \v@lX=\z@\v@lY=\z@\Figtr@nptDD{-5}{#2}\Figtr@nptDD{2}{#5}%
    \divide\v@lX\@vi\divide\v@lY\@vi%
    \Figtr@nptDD{3}{#3}\Figtr@nptDD{-1.5}{#4}\Figp@intregDD-1:(\v@lX,\v@lY)%
    \v@lX=\z@\v@lY=\z@\Figtr@nptDD{2}{#2}\Figtr@nptDD{-5}{#5}%
    \divide\v@lX\@vi\divide\v@lY\@vi\Figtr@nptDD{-1.5}{#3}\Figtr@nptDD{3}{#4}%
    \s@mme=#1\advance\s@mme\@ne\Figp@intregDD\the\s@mme:(\v@lX,\v@lY)%
    \figptcopyDD#1:/-1/\resetc@ntr@l\et@tfigptscontrolDD}\ignorespaces\fi}
\ctr@ld@f\def\figptscontrolTD#1[#2,#3,#4,#5]{\ifGR@cri{\s@uvc@ntr@l\et@tfigptscontrolTD\setc@ntr@l{2}%
    \v@lX=\z@\v@lY=\z@\v@lZ=\z@\Figtr@nptTD{-5}{#2}\Figtr@nptTD{2}{#5}%
    \divide\v@lX\@vi\divide\v@lY\@vi\divide\v@lZ\@vi%
    \Figtr@nptTD{3}{#3}\Figtr@nptTD{-1.5}{#4}\Figp@intregTD-1:(\v@lX,\v@lY,\v@lZ)%
    \v@lX=\z@\v@lY=\z@\v@lZ=\z@\Figtr@nptTD{2}{#2}\Figtr@nptTD{-5}{#5}%
    \divide\v@lX\@vi\divide\v@lY\@vi\divide\v@lZ\@vi\Figtr@nptTD{-1.5}{#3}\Figtr@nptTD{3}{#4}%
    \s@mme=#1\advance\s@mme\@ne\Figp@intregTD\the\s@mme:(\v@lX,\v@lY,\v@lZ)%
    \figptcopyTD#1:/-1/\resetc@ntr@l\et@tfigptscontrolTD}\ignorespaces\fi}
\ctr@ld@f\def\Figtr@nptDD#1#2{\Figg@tXYa{#2}\v@lXa=#1\v@lXa\v@lYa=#1\v@lYa%
    \advance\v@lX\v@lXa\advance\v@lY\v@lYa}
\ctr@ld@f\def\Figtr@nptTD#1#2{\Figg@tXYa{#2}\v@lXa=#1\v@lXa\v@lYa=#1\v@lYa\v@lZa=#1\v@lZa%
    \advance\v@lX\v@lXa\advance\v@lY\v@lYa\advance\v@lZ\v@lZa}
\ctr@ld@f\def\figptscontrolcurve#1,#2[#3]{\ifGR@cri{\s@uvc@ntr@l\et@tfigptscontrolcurve%
    \def\list@num{#3}\extrairelepremi@r\Ak@\de\list@num%
    \extrairelepremi@r\Ai@\de\list@num\extrairelepremi@r\Aj@\de\list@num%
    \s@mme=#1\figptcopy\the\s@mme:/\Ai@/%
    \setc@ntr@l{2}\figvectP -1[\Ak@,\Aj@]%
    \@ecfor\Ak@:=\list@num\do{\advance\s@mme\@ne\figpttra\the\s@mme:=\Ai@/\curv@roundness,-1/%
       \figvectP -1[\Ai@,\Ak@]\advance\s@mme\@ne\figpttra\the\s@mme:=\Aj@/-\curv@roundness,-1/%
       \advance\s@mme\@ne\figptcopy\the\s@mme:/\Aj@/%
       \edef\Ai@{\Aj@}\edef\Aj@{\Ak@}}\advance\s@mme-#1\divide\s@mme\thr@@%
       \xdef#2{\the\s@mme}%
    \resetc@ntr@l\et@tfigptscontrolcurve}\ignorespaces\fi}
\ctr@ln@m\figptsintercirc
\ctr@ld@f\def\figptsintercircDD#1[#2,#3;#4,#5]{\ifGR@cri{\s@uvc@ntr@l\et@tfigptsintercircDD%
    \setc@ntr@l{2}\let\c@lNVintc=\c@lNVintcDD\Figptsintercirc@#1[#2,#3;#4,#5]%
    \resetc@ntr@l\et@tfigptsintercircDD}\ignorespaces\fi}
\ctr@ld@f\def\figptsintercircTD#1[#2,#3;#4,#5;#6]{\ifGR@cri{\s@uvc@ntr@l\et@tfigptsintercircTD%
    \setc@ntr@l{2}\let\c@lNVintc=\c@lNVintcTD\vecunitC@TD[#2,#6]%
    \Figv@ctCreg-3(\v@lX,\v@lY,\v@lZ)\Figptsintercirc@#1[#2,#3;#4,#5]%
    \resetc@ntr@l\et@tfigptsintercircTD}\ignorespaces\fi}
\ctr@ld@f\def\Figptsintercirc@#1[#2,#3;#4,#5]{\figvectP-1[#2,#4]%
    \vecunit@{-1}{-1}\delt@=\result@t\f@ctech=\result@tent%
    \s@mme=#1\advance\s@mme\@ne\figptcopy#1:/#2/\figptcopy\the\s@mme:/#4/%
    \ifdim\delt@=\z@\else%
    \v@lmin=#3\unit@\v@lmax=#5\unit@\v@leur=\v@lmin\advance\v@leur\v@lmax%
    \ifdim\v@leur>\delt@%
    \v@leur=\v@lmin\advance\v@leur-\v@lmax\maxim@m{\v@leur}{\v@leur}{-\v@leur}%
    \ifdim\v@leur<\delt@%
    \divide\v@lmin\f@ctech\divide\v@lmax\f@ctech\divide\delt@\f@ctech%
    \v@lmin=\repdecn@mb{\v@lmin}\v@lmin\v@lmax=\repdecn@mb{\v@lmax}\v@lmax%
    \invers@{\v@leur}{\delt@}\advance\v@lmax-\v@lmin%
    \v@lmax=-\repdecn@mb{\v@leur}\v@lmax\advance\delt@\v@lmax\delt@=.5\delt@%
    \v@lmax=\delt@\multiply\v@lmax\f@ctech%
    \edef\t@ille{\repdecn@mb{\v@lmax}}\figpttra-2:=#2/\t@ille,-1/%
    \delt@=\repdecn@mb{\delt@}\delt@\advance\v@lmin-\delt@%
    \sqrt@{\v@leur}{\v@lmin}\multiply\v@leur\f@ctech\edef\t@ille{\repdecn@mb{\v@leur}}%
    \c@lNVintc\figpttra#1:=-2/-\t@ille,-1/\figpttra\the\s@mme:=-2/\t@ille,-1/\fi\fi\fi}
\ctr@ld@f\def\c@lNVintcDD{\Figg@tXY{-1}\Figv@ctCreg-1(-\v@lY,\v@lX)} 
\ctr@ld@f\def\c@lNVintcTD{{\Figg@tXY{-3}\v@lmin=\v@lX\v@lmax=\v@lY\v@leur=\v@lZ%
    \Figg@tXY{-1}\c@lprovec{-3}\vecunit@{-3}{-3}
    \Figg@tXY{-1}\v@lmin=\v@lX\v@lmax=\v@lY%
    \v@leur=\v@lZ\Figg@tXY{-3}\c@lprovec{-1}}} 
\ctr@ln@m\figptsinterlinell
\ctr@ld@f\def\figptsinterlinellDD#1[#2,#3,#4,#5;#6,#7]{\ifGR@cri{\s@uvc@ntr@l\et@tfigptsinterlinellDD%
    \figptcopy#1:/#6/\s@mme=#1\advance\s@mme\@ne\figptcopy\the\s@mme:/#7/%
    \v@lmin=#3\unit@\v@lmax=#4\unit@
    \setc@ntr@l{2}\figptbaryDD-4:[#6,#7;1,1]\figptsrotDD-3=-4,#7/#2,-#5/
    \Figg@tXY{-3}\Figg@tXYa{#2}\advance\v@lX-\v@lXa\advance\v@lY-\v@lYa
    \figvectP-1[-3,-2]\Figg@tXYa{-1}\figvectP-3[-4,#7]\Figptsint@rLE{#1}
    \resetc@ntr@l\et@tfigptsinterlinellDD}\ignorespaces\fi}
\ctr@ld@f\def\figptsinterlinellP#1[#2,#3,#4;#5,#6]{\ifGR@cri{\s@uvc@ntr@l\et@tfigptsinterlinellP%
    \figptcopy#1:/#5/\s@mme=#1\advance\s@mme\@ne\figptcopy\the\s@mme:/#6/\setc@ntr@l{2}%
    \figvectP-1[#2,#3]\vecunit@{-1}{-1}\v@lmin=\result@t
    \figvectP-2[#2,#4]\vecunit@{-2}{-2}\v@lmax=\result@t
    \figptbary-4:[#5,#6;1,1]
    \figvectP-3[#2,-4]\c@lproscal\v@lX[-3,-1]\c@lproscal\v@lY[-3,-2]
    \figvectP-3[-4,#6]\c@lproscal\v@lXa[-3,-1]\c@lproscal\v@lYa[-3,-2]
    \Figptsint@rLE{#1}\resetc@ntr@l\et@tfigptsinterlinellP}\ignorespaces\fi}
\ctr@ld@f\def\Figptsint@rLE#1{%
    \getredf@ctDD\f@ctech(\v@lmin,\v@lmax)%
    \getredf@ctDD\p@rtent(\v@lX,\v@lY)\ifnum\p@rtent>\f@ctech\f@ctech=\p@rtent\fi%
    \getredf@ctDD\p@rtent(\v@lXa,\v@lYa)\ifnum\p@rtent>\f@ctech\f@ctech=\p@rtent\fi%
    \divide\v@lmin\f@ctech\divide\v@lmax\f@ctech\divide\v@lX\f@ctech\divide\v@lY\f@ctech%
    \divide\v@lXa\f@ctech\divide\v@lYa\f@ctech%
    \c@rre=\repdecn@mb\v@lXa\v@lmax\mili@u=\repdecn@mb\v@lYa\v@lmin%
    \getredf@ctDD\f@ctech(\c@rre,\mili@u)%
    \c@rre=\repdecn@mb\v@lX\v@lmax\mili@u=\repdecn@mb\v@lY\v@lmin%
    \getredf@ctDD\p@rtent(\c@rre,\mili@u)\ifnum\p@rtent>\f@ctech\f@ctech=\p@rtent\fi%
    \divide\v@lmin\f@ctech\divide\v@lmax\f@ctech\divide\v@lX\f@ctech\divide\v@lY\f@ctech%
    \divide\v@lXa\f@ctech\divide\v@lYa\f@ctech%
    \v@lmin=\repdecn@mb{\v@lmin}\v@lmin\v@lmax=\repdecn@mb{\v@lmax}\v@lmax%
    \edef\G@xde{\repdecn@mb\v@lmin}\edef\P@xde{\repdecn@mb\v@lmax}%
    \c@rre=-\v@lmax\v@leur=\repdecn@mb\v@lY\v@lY\advance\c@rre\v@leur\c@rre=\G@xde\c@rre%
    \v@leur=\repdecn@mb\v@lX\v@lX\v@leur=\P@xde\v@leur\advance\c@rre\v@leur
    \v@lmin=\repdecn@mb\v@lYa\v@lmin\v@lmax=\repdecn@mb\v@lXa\v@lmax%
    \mili@u=\repdecn@mb\v@lX\v@lmax\advance\mili@u\repdecn@mb\v@lY\v@lmin
    \v@lmax=\repdecn@mb\v@lXa\v@lmax\advance\v@lmax\repdecn@mb\v@lYa\v@lmin
    \ifdim\v@lmax>\epsil@n%
    \maxim@m{\v@leur}{\c@rre}{-\c@rre}\maxim@m{\v@lmin}{\mili@u}{-\mili@u}%
    \maxim@m{\v@leur}{\v@leur}{\v@lmin}\maxim@m{\v@lmin}{\v@lmax}{-\v@lmax}%
    \maxim@m{\v@leur}{\v@leur}{\v@lmin}\p@rtentiere{\p@rtent}{\v@leur}\advance\p@rtent\@ne%
    \divide\c@rre\p@rtent\divide\mili@u\p@rtent\divide\v@lmax\p@rtent%
    \delt@=\repdecn@mb{\mili@u}\mili@u\v@leur=\repdecn@mb{\v@lmax}\c@rre%
    \advance\delt@-\v@leur\ifdim\delt@<\z@\else\sqrt@\delt@\delt@%
    \invers@\v@lmax\v@lmax\edef\Uns@rAp{\repdecn@mb\v@lmax}%
    \v@leur=-\mili@u\advance\v@leur-\delt@\v@leur=\Uns@rAp\v@leur%
    \edef\t@ille{\repdecn@mb\v@leur}\figpttra#1:=-4/\t@ille,-3/\s@mme=#1\advance\s@mme\@ne%
    \v@leur=-\mili@u\advance\v@leur\delt@\v@leur=\Uns@rAp\v@leur%
    \edef\t@ille{\repdecn@mb\v@leur}\figpttra\the\s@mme:=-4/\t@ille,-3/\fi\fi}
\ctr@ln@m\figptsorthoprojline
\ctr@ld@f\def\figptsorthoprojlineDD#1=#2/#3,#4/{\ifGR@cri{\s@uvc@ntr@l\et@tfigptsorthoprojlineDD%
    \setc@ntr@l{2}\figvectPDD-3[#3,#4]\figvectNVDD-4[-3]\resetc@ntr@l{2}%
    \def\list@num{#2}\s@mme=#1\@ecfor\p@int:=\list@num\do{%
    \inters@cDD\the\s@mme:[\p@int,-4;#3,-3]\advance\s@mme\@ne}%
    \resetc@ntr@l\et@tfigptsorthoprojlineDD}\ignorespaces\fi}
\ctr@ld@f\def\figptsorthoprojlineTD#1=#2/#3,#4/{\ifGR@cri{\s@uvc@ntr@l\et@tfigptsorthoprojlineTD%
    \setc@ntr@l{2}\figvectPTD-2[#3,#4]\vecunit@TD{-2}{-2}%
    \def\list@num{#2}\s@mme=#1\@ecfor\p@int:=\list@num\do{%
    \figvectPTD-1[#3,\p@int]\c@lproscalTD\v@leur[-1,-2]%
    \edef\v@lcoef{\repdecn@mb{\v@leur}}\figpttraTD\the\s@mme:=#3/\v@lcoef,-2/%
    \advance\s@mme\@ne}\resetc@ntr@l\et@tfigptsorthoprojlineTD}\ignorespaces\fi}
\ctr@ln@m\figptsorthoprojplane
\ctr@ld@f\def\figptsorthoprojplaneDD{\un@v@ilable{figptsorthoprojplane}}
\ctr@ld@f\def\figptsorthoprojplaneTD#1=#2/#3,#4/{\ifGR@cri{\s@uvc@ntr@l\et@tfigptsorthoprojplane%
    \setc@ntr@l{2}\vecunit@TD{-2}{#4}%
    \def\list@num{#2}\s@mme=#1\@ecfor\p@int:=\list@num\do{\figvectPTD-1[\p@int,#3]%
    \c@lproscalTD\v@leur[-1,-2]\edef\v@lcoef{\repdecn@mb{\v@leur}}%
    \figpttraTD\the\s@mme:=\p@int/\v@lcoef,-2/\advance\s@mme\@ne}%
    \resetc@ntr@l\et@tfigptsorthoprojplane}\ignorespaces\fi}
\ctr@ld@f\def\figptshom#1=#2/#3,#4/{\ifGR@cri{\s@uvc@ntr@l\et@tfigptshom%
    \setc@ntr@l{2}\def\list@num{#2}\s@mme=#1%
    \@ecfor\p@int:=\list@num\do{\figvectP-1[#3,\p@int]%
    \figpttra\the\s@mme:=#3/#4,-1/\advance\s@mme\@ne}%
    \resetc@ntr@l\et@tfigptshom}\ignorespaces\fi}
\ctr@ld@f\def\figptsinv#1=#2/#3,#4/{\ifGR@cri{\s@uvc@ntr@l\et@tfigptsinv%
    \setc@ntr@l{2}\def\list@num{#2}\s@mme=#1%
    \@ecfor\p@int:=\list@num\do{\figvectP-1[#3,\p@int]\Figg@tXY{-1}%
    \getredf@ctB\f@ctech\n@rmeucC{\delt@}{-1}%
    \delt@=\ptT@unit@\delt@\delt@=\ptT@unit@\delt@%
    \invers@{\delt@}{\delt@}\multiply\f@ctech\f@ctech\divide\delt@\f@ctech%
    \delt@=#4\delt@\edef\v@lcoef{\repdecn@mb{\delt@}}\figpttra\the\s@mme:=#3/\v@lcoef,-1/%
    \advance\s@mme\@ne}\resetc@ntr@l\et@tfigptsinv}\ignorespaces\fi}
\ctr@ln@m\figptsrot
\ctr@ld@f\def\figptsrotDD#1=#2/#3,#4/{\ifGR@cri{\s@uvc@ntr@l\et@tfigptsrotDD%
    \c@ssin{\C@}{\S@}{#4}\setc@ntr@l{2}\def\list@num{#2}\s@mme=#1%
    \@ecfor\p@int:=\list@num\do{\figvectPDD-1[#3,\p@int]\Figg@tXY{-1}%
    \v@lXa=\C@\v@lX\advance\v@lXa-\S@\v@lY%
    \v@lYa=\S@\v@lX\advance\v@lYa\C@\v@lY%
    \Figv@ctCreg-1(\v@lXa,\v@lYa)\figpttraDD\the\s@mme:=#3/1,-1/\advance\s@mme\@ne}%
    \resetc@ntr@l\et@tfigptsrotDD}\ignorespaces\fi}
\ctr@ld@f\def\figptsrotTD#1=#2/#3,#4,#5/{\ifGR@cri{\s@uvc@ntr@l\et@tfigptsrotTD%
    \c@ssin{\C@}{\S@}{#4}%
    \setc@ntr@l{2}\def\list@num{#2}\s@mme=#1%
    \@ecfor\p@int:=\list@num\do{\figptorthoprojplaneTD-3:=#3/\p@int,#5/%
    \figvectPTD-2[-3,\p@int]%
    \figvectNVTD-1[#5,-2]\n@rmeucTD\v@leur{-2}\edef\v@lcoef{\repdecn@mb{\v@leur}}%
    \Figg@tXYa{-1}\v@lXa=\v@lcoef\v@lXa\v@lYa=\v@lcoef\v@lYa\v@lZa=\v@lcoef\v@lZa%
    \v@lXa=\S@\v@lXa\v@lYa=\S@\v@lYa\v@lZa=\S@\v@lZa\Figg@tXY{-2}%
    \advance\v@lXa\C@\v@lX\advance\v@lYa\C@\v@lY\advance\v@lZa\C@\v@lZ%
    \Figg@tXY{-3}\advance\v@lXa\v@lX\advance\v@lYa\v@lY\advance\v@lZa\v@lZ%
    \Figp@intregTD\the\s@mme:(\v@lXa,\v@lYa,\v@lZa)\advance\s@mme\@ne}%
    \resetc@ntr@l\et@tfigptsrotTD}\ignorespaces\fi}
\ctr@ln@m\figptssym
\ctr@ld@f\def\figptssymDD#1=#2/#3,#4/{\ifGR@cri{\s@uvc@ntr@l\et@tfigptssymDD%
    \setc@ntr@l{2}\figvectPDD-3[#3,#4]\Figg@tXY{-3}\Figv@ctCreg-4(-\v@lY,\v@lX)%
    \resetc@ntr@l{2}\def\list@num{#2}\s@mme=#1%
    \@ecfor\p@int:=\list@num\do{\inters@cDD-5:[#3,-3;\p@int,-4]\figvectPDD-2[\p@int,-5]%
    \figpttraDD\the\s@mme:=\p@int/2,-2/\advance\s@mme\@ne}%
    \resetc@ntr@l\et@tfigptssymDD}\ignorespaces\fi}
\ctr@ld@f\def\figptssymTD#1=#2/#3,#4/{\ifGR@cri{\s@uvc@ntr@l\et@tfigptssymTD%
    \setc@ntr@l{2}\vecunit@TD{-2}{#4}\def\list@num{#2}\s@mme=#1%
    \@ecfor\p@int:=\list@num\do{\figvectPTD-1[\p@int,#3]%
    \c@lproscalTD\v@leur[-1,-2]\v@leur=2\v@leur\edef\v@lcoef{\repdecn@mb{\v@leur}}%
    \figpttraTD\the\s@mme:=\p@int/\v@lcoef,-2/\advance\s@mme\@ne}%
    \resetc@ntr@l\et@tfigptssymTD}\ignorespaces\fi}
\ctr@ln@m\figptstra
\ctr@ld@f\def\figptstraDD#1=#2/#3,#4/{\ifGR@cri{\Figg@tXYa{#4}\v@lXa=#3\v@lXa\v@lYa=#3\v@lYa%
    \def\list@num{#2}\s@mme=#1\@ecfor\p@int:=\list@num\do{\Figg@tXY{\p@int}%
    \advance\v@lX\v@lXa\advance\v@lY\v@lYa%
    \Figp@intregDD\the\s@mme:(\v@lX,\v@lY)\advance\s@mme\@ne}}\ignorespaces\fi}
\ctr@ld@f\def\figptstraTD#1=#2/#3,#4/{\ifGR@cri{\Figg@tXYa{#4}\v@lXa=#3\v@lXa\v@lYa=#3\v@lYa%
    \v@lZa=#3\v@lZa\def\list@num{#2}\s@mme=#1\@ecfor\p@int:=\list@num\do{\Figg@tXY{\p@int}%
    \advance\v@lX\v@lXa\advance\v@lY\v@lYa\advance\v@lZ\v@lZa%
    \Figp@intregTD\the\s@mme:(\v@lX,\v@lY,\v@lZ)\advance\s@mme\@ne}}\ignorespaces\fi}
\ctr@ln@m\figptvisilimSL
\ctr@ld@f\def\figptvisilimSLDD{\un@v@ilable{figptvisilimSL}}
\ctr@ld@f\def\figptvisilimSLTD#1:#2[#3,#4;#5,#6]{\ifGR@cri{\s@uvc@ntr@l\et@tfigptvisilimSLTD%
    \setc@ntr@l{2}\figvectP-1[#3,#4]\n@rminf{\delt@}{-1}%
    \ifcase\CUR@proj\v@lX=\cxa@\p@\v@lY=-\p@\v@lZ=\cxb@\p@
    \Figv@ctCreg-2(\v@lX,\v@lY,\v@lZ)\figvectP-3[#5,#6]\figvectNV-1[-2,-3]%
    \or\figvectP-1[#5,#6]\vecunitCV@TD{-1}\v@lmin=\v@lX\v@lmax=\v@lY
    \v@leur=\v@lZ\v@lX=\cza@\p@\v@lY=\czb@\p@\v@lZ=\czc@\p@\c@lprovec{-1}%
    \or\c@ley@pt{-2}\figvectN-1[#5,#6,-2]\fi
    \edef\Ai@{#3}\edef\Aj@{#4}\figvectP-2[#5,\Ai@]\c@lproscal\v@leur[-1,-2]%
    \ifdim\v@leur>\z@\p@rtent=\@ne\else\p@rtent=\m@ne\fi%
    \figvectP-2[#5,\Aj@]\c@lproscal\v@leur[-1,-2]%
    \ifdim\p@rtent\v@leur>\z@\figptcopy#1:#2/#3/%
    \message{*** \BS@ figptvisilimSL: points are on the same side.}\else%
    \figptcopy-3:/#3/\figptcopy-4:/#4/%
    \loop\figptbary-5:[-3,-4;1,1]\figvectP-2[#5,-5]\c@lproscal\v@leur[-1,-2]%
    \ifdim\p@rtent\v@leur>\z@\figptcopy-3:/-5/\else\figptcopy-4:/-5/\fi%
    \divide\delt@\tw@\ifdim\delt@>\epsil@n\repeat%
    \figptbary#1:#2[-3,-4;1,1]\fi\resetc@ntr@l\et@tfigptvisilimSLTD}\ignorespaces\fi}
\ctr@ld@f\def\c@ley@pt#1{\t@stp@r\ifitis@K\v@lX=\cza@\p@\v@lY=\czb@\p@\v@lZ=\czc@\p@%
    \Figv@ctCreg-1(\v@lX,\v@lY,\v@lZ)\Figp@intreg-2:(\wd\Bt@rget,\ht\Bt@rget,\dp\Bt@rget)%
    \figpttra#1:=-2/-\disob@intern,-1/\else\end\fi}
\ctr@ld@f\def\t@stp@r{\itis@Ktrue\ifnewt@rgetpt\else\itis@Kfalse%
    \message{*** \BS@ figptvisilimXX: target point undefined.}\fi\ifnewdis@b\else%
    \itis@Kfalse\message{*** \BS@ figptvisilimXX: observation distance undefined.}\fi%
    \ifitis@K\else\message{*** This macro must be called after \BS@ figdrawbegin or after
    having set the missing parameter(s) with \BS@ figset proj()}\fi}
\ctr@ld@f\def\figscan#1(#2,#3){{\s@uvc@ntr@l\et@tfigscan\@psfgetbb{#1}\if@psfbbfound\else%
    \def\@psfllx{0}\def\@psflly{20}\def\@psfurx{540}\def\@psfury{640}\fi\figscan@{#2}{#3}%
    \resetc@ntr@l\et@tfigscan}\ignorespaces}
\ctr@ld@f\def\figscan@#1#2{%
    \unit@=\@ne bp\setc@ntr@l{2}\figsetmark{}%
    \def\minst@p{20pt}%
    \v@lX=\@psfllx\p@\v@lX=\Sc@leFact\v@lX\r@undint\v@lX\v@lX%
    \v@lY=\@psflly\p@\v@lY=\Sc@leFact\v@lY\ifdim\v@lY>\z@\r@undint\v@lY\v@lY\fi%
    \delt@=\@psfury\p@\delt@=\Sc@leFact\delt@%
    \advance\delt@-\v@lY\v@lXa=\@psfurx\p@\v@lXa=\Sc@leFact\v@lXa\v@leur=\minst@p%
    \edef\valv@lY{\repdecn@mb{\v@lY}}\edef\LgTr@it{\the\delt@}%
    \loop\ifdim\v@lX<\v@lXa\edef\valv@lX{\repdecn@mb{\v@lX}}%
    \figptDD -1:(\valv@lX,\valv@lY)\figwriten -1:\hbox{\vrule height\LgTr@it}(0)%
    \ifdim\v@leur<\minst@p\else\figsetmark{\raise-8bp\hbox{$\scriptscriptstyle\triangle$}}%
    \figwrites -1:\@ffichnb{0}{\valv@lX}(6)\v@leur=\z@\figsetmark{}\fi%
    \advance\v@leur#1pt\advance\v@lX#1pt\repeat%
    \def\minst@p{10pt}%
    \v@lX=\@psfllx\p@\v@lX=\Sc@leFact\v@lX\ifdim\v@lX>\z@\r@undint\v@lX\v@lX\fi%
    \v@lY=\@psflly\p@\v@lY=\Sc@leFact\v@lY\r@undint\v@lY\v@lY%
    \delt@=\@psfurx\p@\delt@=\Sc@leFact\delt@%
    \advance\delt@-\v@lX\v@lYa=\@psfury\p@\v@lYa=\Sc@leFact\v@lYa\v@leur=\minst@p%
    \edef\valv@lX{\repdecn@mb{\v@lX}}\edef\LgTr@it{\the\delt@}%
    \loop\ifdim\v@lY<\v@lYa\edef\valv@lY{\repdecn@mb{\v@lY}}%
    \figptDD -1:(\valv@lX,\valv@lY)\figwritee -1:\vbox{\hrule width\LgTr@it}(0)%
    \ifdim\v@leur<\minst@p\else\figsetmark{$\triangleright$\kern4bp}%
    \figwritew -1:\@ffichnb{0}{\valv@lY}(6)\v@leur=\z@\figsetmark{}\fi%
    \advance\v@leur#2pt\advance\v@lY#2pt\repeat}
\ctr@ld@f
\ctr@ld@f\def\figscan@E#1(#2,#3){{\s@uvc@ntr@l\et@tfigscan@E%
    \Figdisc@rdLTS{#1}{\t@xt@}\pdfximage{\t@xt@}%
    \setbox\Gb@x=\hbox{\pdfrefximage\pdflastximage}%
    \edef\@psfllx{0}\v@lY=-\dp\Gb@x\edef\@psflly{\repdecn@mb{\v@lY}}%
    \edef\@psfurx{\repdecn@mb{\wd\Gb@x}}%
    \v@lY=\dp\Gb@x\advance\v@lY\ht\Gb@x\edef\@psfury{\repdecn@mb{\v@lY}}%
    \figscan@{#2}{#3}\resetc@ntr@l\et@tfigscan@E}\ignorespaces}
\ctr@ld@f\def\figshowpts[#1,#2]{{\figsetmark{$\bullet$}\figsetptname{\bf ##1}%
    \p@rtent=#2\relax\ifnum\p@rtent<\z@\p@rtent=\z@\fi%
    \s@mme=#1\relax\ifnum\s@mme<\z@\s@mme=\z@\fi%
    \loop\ifnum\s@mme<\p@rtent\pt@rvect{\s@mme}%
    \ifitis@K\figwriten{\the\s@mme}:(4pt)\fi\advance\s@mme\@ne\repeat%
    \pt@rvect{\s@mme}\ifitis@K\figwriten{\the\s@mme}:(4pt)\fi}\ignorespaces}
\ctr@ld@f\def\pt@rvect#1{\set@bjc@de{#1}%
    \expandafter\expandafter\expandafter\inqpt@rvec\csname\objc@de\endcsname:}
\ctr@ld@f\def\inqpt@rvec#1#2:{\if#1\C@dCl@spt\itis@Ktrue\else\itis@Kfalse\fi}
\ctr@ld@f\def\figshowsettings{{%
    \immediate\write16{====================================================================}%
    \immediate\write16{ Current settings are (DDV means "with dynamic default value"):}%
    \immediate\write16{ --- GENERAL ---}%
    \immediate\write16{Scale factor and Unit = \unit@util\space (\the\unit@)
     \space -> \BS@ figinit{ScaleFactorUnit}}%
    \immediate\write16{Update mode = \ifGRupdatem@de yes\else no\fi
     \space-> \BS@ figset(update=yes/no) or \BS@ figsetdefault(update=yes/no)}%
    \immediate\write16{ --- WRITING ---}%
    \immediate\write16{Implicit point name = \ptn@me{i} \space-> \BS@ figset write(ptname={Name})}%
    \immediate\write16{Point marker = \the\c@nsymb \space -> \BS@ figset write(mark=Mark)}%
    \immediate\write16{Print rounded coordinates = \ifr@undcoord yes\else no\fi
     \space-> \BS@ figset write(roundcoord=yes/no)}%
    \immediate\write16{ --- GRAPHICAL (general) ---}%
    \immediate\write16{Color = \CUR@color \space-> \BS@ figset(color=ColorDefinition)}%
    \immediate\write16{Filling mode = \iffillm@de yes\else no\fi
     \space-> \BS@ figset(fillmode=yes/no)}%
    \immediate\write16{Line join = \CUR@join \space-> \BS@ figset(join=miter/round/bevel)}%
    \immediate\write16{Line style = \CUR@dash \space-> \BS@ figset(dash=Index/Pattern)}%
    \immediate\write16{Line width = \CUR@width
     \space-> \BS@ figset(width=real in PostScript units)}%
    \immediate\write16{ --- GRAPHICAL (specific) ---}%
    \immediate\write16{Altitude (all the following attributes are DDV):}%
    \immediate\write16{ Base line color =
     \ifx\DDV@blcolor\D@FTref general color\else\DDV@blcolor\fi
     \space-> \BS@ figset altitude(blcolor=ColorDefinition)}%
    \immediate\write16{ Base line style =
     \ifx\DDV@bldash\D@FTref general style\else\DDV@bldash\fi
     \space-> \BS@ figset altitude(bldash=Index/Pattern)}%
    \immediate\write16{ Base line width =
     \ifx\DDV@blwidth\D@FTref general width\else\DDV@blwidth\fi
     \space-> \BS@ figset altitude(blwidth=real in PostScript units)}%
    \immediate\write16{ Square line color =
     \ifx\DDV@sqcolor\D@FTref general color\else\DDV@sqcolor\fi
     \space-> \BS@ figset altitude(sqcolor=ColorDefinition)}%
    \immediate\write16{ Square line style =
     \ifx\DDV@sqdash\D@FTref general style\else\DDV@sqdash\fi
     \space-> \BS@ figset altitude(sqdash=Index/Pattern)}%
    \immediate\write16{ Square line width =
     \ifx\DDV@sqwidth\D@FTref general width\else\DDV@sqwidth\fi
     \space-> \BS@ figset altitude(sqwidth=real in PostScript units)}%
    \immediate\write16{Arrowhead:}%
    \immediate\write16{ (half-)Angle = \@rrowheadangle
     \space-> \BS@ figset arrowhead(angle=real in degrees)}%
    \immediate\write16{ Filling mode = \if@rrowhfill yes\else no\fi
     \space-> \BS@ figset arrowhead(fillmode=yes/no)}%
    \immediate\write16{ "Outside" = \if@rrowhout yes\else no\fi
     \space-> \BS@ figset arrowhead(out=yes/no)}%
    \immediate\write16{ Length = \@rrowheadlength
     \if@rrowratio\space(not active)\else\space(active)\fi
     \space-> \BS@ figset arrowhead(length=real in user coord.)}%
    \immediate\write16{ Ratio = \@rrowheadratio
     \if@rrowratio\space(active)\else\space(not active)\fi
     \space-> \BS@ figset arrowhead(ratio=real in [0,1])}%
    \immediate\write16{Curve:}%
    \immediate\write16{ Roundness = \curv@roundness
     \space-> \BS@ figset curve(roundness=real in [0,0.5])}%
    \immediate\write16{Flow chart:}%
    \immediate\write16{ Arrow position = \@rrowp@s
     \space-> \BS@ figset flowchart(arrowposition=real in [0,1])}%
    \immediate\write16{ Arrow reference point = \ifcase\@rrowr@fpt start\else end\fi
     \space-> \BS@ figset flowchart(arrowrefpt = start/end)}%
    \immediate\write16{ Background color = \fcbgc@lor
     \space-> \BS@ figset flowchart(bgcolor=ColorDefinition)}%
    \immediate\write16{ Line type = \ifcase\fclin@typ@ curve\else polygon\fi
     \space-> \BS@ figset flowchart(line=polygon/curve)}%
    \immediate\write16{ Padding = (\Xp@dd, \Yp@dd)
     \space-> \BS@ figset flowchart(padding = real in user coord.)}%
    \immediate\write16{\space\space\space\space(or
     \BS@ figset flowchart(xpadding=real, ypadding=real) )}%
    \immediate\write16{ Radius = \fclin@r@d
     \space-> \BS@ figset flowchart(radius=positive real in user coord.)}%
    \immediate\write16{ Shape = \fcsh@pe
     \space-> \BS@ figset flowchart(shape = rectangle, ellipse or lozenge)}%
    \immediate\write16{ Thickness color (DDV) = 
     \ifx\DDV@thickcolor\D@FTref general color\else\DDV@thickcolor\fi
     \space-> \BS@ figset flowchart(thickcolor=ColorDefinition)}%
    \immediate\write16{ Thickness = \thickn@ss
     \space-> \BS@ figset flowchart(thickness = real in user coord.)}%
    \immediate\write16{Mesh:}%
    \immediate\write16{ Diagonal = \c@ntrolmesh
     \space-> \BS@ figset mesh(diag=integer in {-1,0,1})}%
    \immediate\write16{ Lines color (DDV) =
     \ifx\DDV@meshcolor\D@FTref general color\else\DDV@meshcolor\fi
     \space-> \BS@ figset mesh(color=ColorDefinition)}%
    \immediate\write16{ Lines style (DDV) =
     \ifx\DDV@meshdash\D@FTref general style\else\DDV@meshdash\fi
     \space-> \BS@ figset mesh(dash=Index/Pattern)}%
    \immediate\write16{ Lines width (DDV) =
     \ifx\DDV@meshwidth\D@FTref general width\else\DDV@meshwidth\fi
     \space-> \BS@ figset mesh(width=real in PostScript units)}%
    \immediate\write16{Trimesh:}%
    \immediate\write16{ Lines color (DDV) =
     \ifx\DDV@tmeshcolor\D@FTref general color\else\DDV@tmeshcolor\fi
     \space-> \BS@ figset trimesh(color=ColorDefinition)}%
    \immediate\write16{ Lines style (DDV) =
     \ifx\DDV@tmeshdash\D@FTref general style\else\DDV@tmeshdash\fi
     \space-> \BS@ figset trimesh(dash=Index/Pattern)}%
    \immediate\write16{ Lines width (DDV) =
     \ifx\DDV@tmeshwidth\D@FTref general width\else\DDV@tmeshwidth\fi
     \space-> \BS@ figset trimesh(width=real in PostScript units)}%
    \ifTr@isDim%
    \immediate\write16{ --- 3D to 2D PROJECTION ---}%
    \immediate\write16{Projection : \typ@proj \space-> \BS@ figinit{ScaleFactorUnit, ProjType}}%
    \immediate\write16{Longitude (psi) = \v@lPsi \space-> \BS@ figset proj(psi=real in degrees)}%
    \ifcase\CUR@proj\immediate\write16{Depth coeff. (Lambda)
     \space = \v@lTheta \space-> \BS@ figset proj(lambda=real in [0,1])}%
    \else\immediate\write16{Latitude (theta)
     \space = \v@lTheta \space-> \BS@ figset proj(theta=real in degrees)}%
    \fi%
    \ifnum\CUR@proj=\tw@%
    \immediate\write16{Observation distance = \disob@unit
     \space-> \BS@ figset proj(dist=real in user coord.)}%
    \immediate\write16{Target point = \t@rgetpt \space-> \BS@ figset proj(targetpt=pt number)}%
     \v@lX=\ptT@unit@\wd\Bt@rget\v@lY=\ptT@unit@\ht\Bt@rget\v@lZ=\ptT@unit@\dp\Bt@rget%
    \immediate\write16{ Its coordinates are
     (\repdecn@mb{\v@lX}, \repdecn@mb{\v@lY}, \repdecn@mb{\v@lZ})}%
    \fi%
    \fi%
    \immediate\write16{====================================================================}%
    \ignorespaces}}
\ctr@ln@w{newif}\ifitis@vect@r
\ctr@ld@f\def\figvectC#1(#2,#3){{\itis@vect@rtrue\figpt#1:(#2,#3)}\ignorespaces}
\ctr@ld@f\def\Figv@ctCreg#1(#2,#3){{\itis@vect@rtrue\Figp@intreg#1:(#2,#3)}\ignorespaces}
\ctr@ln@m\figvectDBezier
\ctr@ld@f\def\figvectDBezierDD#1:#2,#3[#4,#5,#6,#7]{\ifGR@cri{\s@uvc@ntr@l\et@tfigvectDBezierDD%
    \FigvectDBezier@#2,#3[#4,#5,#6,#7]\v@lX=\c@ef\v@lX\v@lY=\c@ef\v@lY%
    \Figv@ctCreg#1(\v@lX,\v@lY)\resetc@ntr@l\et@tfigvectDBezierDD}\ignorespaces\fi}
\ctr@ld@f\def\figvectDBezierTD#1:#2,#3[#4,#5,#6,#7]{\ifGR@cri{\s@uvc@ntr@l\et@tfigvectDBezierTD%
    \FigvectDBezier@#2,#3[#4,#5,#6,#7]\v@lX=\c@ef\v@lX\v@lY=\c@ef\v@lY\v@lZ=\c@ef\v@lZ%
    \Figv@ctCreg#1(\v@lX,\v@lY,\v@lZ)\resetc@ntr@l\et@tfigvectDBezierTD}\ignorespaces\fi}
\ctr@ld@f\def\FigvectDBezier@#1,#2[#3,#4,#5,#6]{\setc@ntr@l{2}%
    \edef\T@{#2}\v@leur=\p@\advance\v@leur-#2pt\edef\UNmT@{\repdecn@mb{\v@leur}}%
    \ifnum#1=\tw@\def\c@ef{6}\else\def\c@ef{3}\fi%
    \figptcopy-4:/#3/\figptcopy-3:/#4/\figptcopy-2:/#5/\figptcopy-1:/#6/%
    \l@mbd@un=-4 \l@mbd@de=-\thr@@\p@rtent=\m@ne\c@lDecast%
    \ifnum#1=\tw@\c@lDCDeux{-4}{-3}\c@lDCDeux{-3}{-2}\c@lDCDeux{-4}{-3}\else%
    \l@mbd@un=-4 \l@mbd@de=-\thr@@\p@rtent=-\tw@\c@lDecast%
    \c@lDCDeux{-4}{-3}\fi\Figg@tXY{-4}}
\ctr@ln@m\c@lDCDeux
\ctr@ld@f\def\c@lDCDeuxDD#1#2{\Figg@tXY{#2}\Figg@tXYa{#1}%
    \advance\v@lX-\v@lXa\advance\v@lY-\v@lYa\Figp@intregDD#1:(\v@lX,\v@lY)}
\ctr@ld@f\def\c@lDCDeuxTD#1#2{\Figg@tXY{#2}\Figg@tXYa{#1}\advance\v@lX-\v@lXa%
    \advance\v@lY-\v@lYa\advance\v@lZ-\v@lZa\Figp@intregTD#1:(\v@lX,\v@lY,\v@lZ)}
\ctr@ln@m\figvectN
\ctr@ld@f\def\figvectNDD#1[#2,#3]{\ifGR@cri{\Figg@tXYa{#2}\Figg@tXY{#3}%
    \advance\v@lX-\v@lXa\advance\v@lY-\v@lYa%
    \Figv@ctCreg#1(-\v@lY,\v@lX)}\ignorespaces\fi}
\ctr@ld@f\def\figvectNTD#1[#2,#3,#4]{\ifGR@cri{\vecunitC@TD[#2,#4]\v@lmin=\v@lX\v@lmax=\v@lY%
    \v@leur=\v@lZ\vecunitC@TD[#2,#3]\c@lprovec{#1}}\ignorespaces\fi}
\ctr@ln@m\figvectNV
\ctr@ld@f\def\figvectNVDD#1[#2]{\ifGR@cri{\Figg@tXY{#2}\Figv@ctCreg#1(-\v@lY,\v@lX)}\ignorespaces\fi}
\ctr@ld@f\def\figvectNVTD#1[#2,#3]{\ifGR@cri{\vecunitCV@TD{#3}\v@lmin=\v@lX\v@lmax=\v@lY%
    \v@leur=\v@lZ\vecunitCV@TD{#2}\c@lprovec{#1}}\ignorespaces\fi}
\ctr@ln@m\figvectP
\ctr@ld@f\def\figvectPDD#1[#2,#3]{\ifGR@cri{\Figg@tXYa{#2}\Figg@tXY{#3}%
    \advance\v@lX-\v@lXa\advance\v@lY-\v@lYa%
    \Figv@ctCreg#1(\v@lX,\v@lY)}\ignorespaces\fi}
\ctr@ld@f\def\figvectPTD#1[#2,#3]{\ifGR@cri{\Figg@tXYa{#2}\Figg@tXY{#3}%
    \advance\v@lX-\v@lXa\advance\v@lY-\v@lYa\advance\v@lZ-\v@lZa%
    \Figv@ctCreg#1(\v@lX,\v@lY,\v@lZ)}\ignorespaces\fi}
\ctr@ln@m\figvectU
\ctr@ld@f\def\figvectUDD#1[#2]{\ifGR@cri{\n@rmeuc\v@leur{#2}\invers@\v@leur\v@leur%
    \delt@=\repdecn@mb{\v@leur}\unit@\edef\v@ldelt@{\repdecn@mb{\delt@}}%
    \Figg@tXY{#2}\v@lX=\v@ldelt@\v@lX\v@lY=\v@ldelt@\v@lY%
    \Figv@ctCreg#1(\v@lX,\v@lY)}\ignorespaces\fi}
\ctr@ld@f\def\figvectUTD#1[#2]{\ifGR@cri{\n@rmeuc\v@leur{#2}\invers@\v@leur\v@leur%
    \delt@=\repdecn@mb{\v@leur}\unit@\edef\v@ldelt@{\repdecn@mb{\delt@}}%
    \Figg@tXY{#2}\v@lX=\v@ldelt@\v@lX\v@lY=\v@ldelt@\v@lY\v@lZ=\v@ldelt@\v@lZ%
    \Figv@ctCreg#1(\v@lX,\v@lY,\v@lZ)}\ignorespaces\fi}
\ctr@ld@f\def\figvisu#1#2#3{\c@ldefproj\initb@undb@x\xdef\figforTeXFigno{\figforTeXnextFigno}%
    \s@mme=\figforTeXnextFigno\advance\s@mme\@ne\xdef\figforTeXnextFigno{\number\s@mme}%
    \setbox\b@xvisu=\hbox{\ifnum\@utoFN>\z@\figinsert{}\gdef\@utoFInDone{0}\fi\ignorespaces#3}%
    \gdef\@utoFInDone{1}\gdef\@utoFN{0}%
    \v@lXa=-\c@@rdYmin\v@lYa=\c@@rdYmax\advance\v@lYa-\c@@rdYmin%
    \v@lX=\c@@rdXmax\advance\v@lX-\c@@rdXmin%
    \setbox#1=\hbox{#2}\v@lY=-\v@lX\maxim@m{\v@lX}{\v@lX}{\wd#1}%
    \advance\v@lY\v@lX\divide\v@lY\tw@\advance\v@lY-\c@@rdXmin%
    \setbox#1=\vbox{\parindent\z@\hsize=\v@lX\vskip\v@lYa%
    \rlap{\hskip\v@lY\smash{\raise\v@lXa\box\b@xvisu}}%
    \def\t@xt@{#2}\ifx\t@xt@\empty\else\medskip\centerline{#2}\fi}\wd#1=\v@lX}
\ctr@ld@f\def\figDecrementFigno{{\xdef\figforTeXnextFigno{\figforTeXFigno}%
    \s@mme=\figforTeXFigno\advance\s@mme\m@ne\xdef\figforTeXFigno{\number\s@mme}}}
\ctr@ln@w{newbox}\Bt@rget\setbox\Bt@rget=\null
\ctr@ln@w{newbox}\BminTD@\setbox\BminTD@=\null
\ctr@ln@w{newbox}\BmaxTD@\setbox\BmaxTD@=\null
\ctr@ln@w{newif}\ifnewt@rgetpt\ctr@ln@w{newif}\ifnewdis@b
\ctr@ld@f\def\b@undb@xTD#1#2#3{%
    \relax\ifdim#1<\wd\BminTD@\global\wd\BminTD@=#1\fi%
    \relax\ifdim#2<\ht\BminTD@\global\ht\BminTD@=#2\fi%
    \relax\ifdim#3<\dp\BminTD@\global\dp\BminTD@=#3\fi%
    \relax\ifdim#1>\wd\BmaxTD@\global\wd\BmaxTD@=#1\fi%
    \relax\ifdim#2>\ht\BmaxTD@\global\ht\BmaxTD@=#2\fi%
    \relax\ifdim#3>\dp\BmaxTD@\global\dp\BmaxTD@=#3\fi}
\ctr@ld@f\def\c@ldefdisob{{\ifdim\wd\BminTD@<\maxdimen\v@leur=\wd\BmaxTD@\advance\v@leur-\wd\BminTD@%
    \delt@=\ht\BmaxTD@\advance\delt@-\ht\BminTD@\maxim@m{\v@leur}{\v@leur}{\delt@}%
    \delt@=\dp\BmaxTD@\advance\delt@-\dp\BminTD@\maxim@m{\v@leur}{\v@leur}{\delt@}%
    \v@leur=5\v@leur\else\v@leur=800pt\fi\c@ldefdisob@{\v@leur}}}
\ctr@ln@m\disob@intern
\ctr@ln@m\disob@
\ctr@ln@m\divf@ctproj
\ctr@ld@f\def\c@ldefdisob@#1{{\v@leur=#1\ifdim\v@leur<\p@\v@leur=800pt\fi%
    \xdef\disob@intern{\repdecn@mb{\v@leur}}%
    \delt@=\ptT@unit@\v@leur\xdef\disob@unit{\repdecn@mb{\delt@}}%
    \f@ctech=\@ne\loop\ifdim\v@leur>\t@n pt\divide\v@leur\t@n\multiply\f@ctech\t@n\repeat%
    \xdef\disob@{\repdecn@mb{\v@leur}}\xdef\divf@ctproj{\the\f@ctech}}%
    \global\newdis@btrue}
\ctr@ln@m\t@rgetpt
\ctr@ld@f\def\c@ldeft@rgetpt{\newt@rgetpttrue\def\t@rgetpt{CenterBoundBox}{%
    \delt@=\wd\BmaxTD@\advance\delt@-\wd\BminTD@\divide\delt@\tw@%
    \v@leur=\wd\BminTD@\advance\v@leur\delt@\global\wd\Bt@rget=\v@leur%
    \delt@=\ht\BmaxTD@\advance\delt@-\ht\BminTD@\divide\delt@\tw@%
    \v@leur=\ht\BminTD@\advance\v@leur\delt@\global\ht\Bt@rget=\v@leur%
    \delt@=\dp\BmaxTD@\advance\delt@-\dp\BminTD@\divide\delt@\tw@%
    \v@leur=\dp\BminTD@\advance\v@leur\delt@\global\dp\Bt@rget=\v@leur}}
\ctr@ln@m\c@ldefproj
\ctr@ld@f\def\c@ldefprojTD{\ifnewt@rgetpt\else\c@ldeft@rgetpt\fi\ifnewdis@b\else\c@ldefdisob\fi}
\ctr@ld@f\def\c@lprojcav{
    \v@lZa=\cxa@\v@lY\advance\v@lX\v@lZa%
    \v@lZa=\cxb@\v@lY\v@lY=\v@lZ\advance\v@lY\v@lZa\ignorespaces}
\ctr@ln@m\v@lcoef
\ctr@ld@f\def\c@lprojrea{
    \advance\v@lX-\wd\Bt@rget\advance\v@lY-\ht\Bt@rget\advance\v@lZ-\dp\Bt@rget%
    \v@lZa=\cza@\v@lX\advance\v@lZa\czb@\v@lY\advance\v@lZa\czc@\v@lZ%
    \divide\v@lZa\divf@ctproj\advance\v@lZa\disob@ pt\invers@{\v@lZa}{\v@lZa}%
    \v@lZa=\disob@\v@lZa\edef\v@lcoef{\repdecn@mb{\v@lZa}}%
    \v@lXa=\cxa@\v@lX\advance\v@lXa\cxb@\v@lY\v@lXa=\v@lcoef\v@lXa%
    \v@lY=\cyb@\v@lY\advance\v@lY\cya@\v@lX\advance\v@lY\cyc@\v@lZ%
    \v@lY=\v@lcoef\v@lY\v@lX=\v@lXa\ignorespaces}
\ctr@ld@f\def\c@lprojort{
    \v@lXa=\cxa@\v@lX\advance\v@lXa\cxb@\v@lY%
    \v@lY=\cyb@\v@lY\advance\v@lY\cya@\v@lX\advance\v@lY\cyc@\v@lZ%
    \v@lX=\v@lXa\ignorespaces}
\ctr@ld@f\def\Figptpr@j#1:#2/#3/{{\Figg@tXY{#3}\superc@lprojSP%
    \Figp@intregDD#1:{#2}(\v@lX,\v@lY)}\ignorespaces}
\ctr@ln@m\figsetobdist
\ctr@ld@f\def\figsetobdistDD{\un@v@ilable{figsetobdist}}
\ctr@ld@f\def\figsetobdistTD(#1){{\ifCUR@PS\W@rnmesIgn{figset proj(dist=...)}%
    \else\v@leur=#1\unit@\c@ldefdisob@{\v@leur}\fi}\ignorespaces}
\ctr@ln@m\c@lprojSP
\ctr@ln@m\CUR@proj
\ctr@ln@m\typ@proj
\ctr@ln@m\superc@lprojSP
\ctr@ld@f\def\Figs@tproj#1{%
    \if#13 \def@ultproj\else\if#1c\def@ultproj%
    \else\if#1o\xdef\CUR@proj{1}\xdef\typ@proj{orthogonal}%
         \figsetviewTD(\def@ultpsi,\def@ulttheta)%
         \global\let\c@lprojSP=\c@lprojort\global\let\superc@lprojSP=\c@lprojort%
    \else\if#1r\xdef\CUR@proj{2}\xdef\typ@proj{realistic}%
         \figsetviewTD(\def@ultpsi,\def@ulttheta)%
         \global\let\c@lprojSP=\c@lprojrea\global\let\superc@lprojSP=\c@lprojrea%
    \else\def@ultproj\message{*** Unknown projection. Cavalier projection assumed.}%
    \fi\fi\fi\fi}
\ctr@ld@f\def\def@ultproj{\xdef\CUR@proj{0}\xdef\typ@proj{cavalier}\figsetviewTD(\def@ultpsi,0.5)%
         \global\let\c@lprojSP=\c@lprojcav\global\let\superc@lprojSP=\c@lprojcav}
\ctr@ln@m\figsettarget
\ctr@ld@f\def\figsettargetDD{\un@v@ilable{figsettarget}}
\ctr@ld@f\def\figsettargetTD[#1]{{\ifCUR@PS\W@rnmesIgn{figset proj(targetpt=...)}%
    \else\global\newt@rgetpttrue\xdef\t@rgetpt{#1}\Figg@tXY{#1}\global\wd\Bt@rget=\v@lX%
    \global\ht\Bt@rget=\v@lY\global\dp\Bt@rget=\v@lZ\fi}\ignorespaces}
\ctr@ln@m\figsetview
\ctr@ld@f\def\figsetviewDD{\un@v@ilable{figsetview}}
\ctr@ld@f\def\figsetviewTD(#1){\ifCUR@PS\W@rnmesIgn{figset proj(Psi|Theta|Lambda=...)}%
     \else\Figsetview@#1,:\fi\ignorespaces}
\ctr@ld@f\def\Figsetview@#1,#2:{{\xdef\v@lPsi{#1}\def\t@xt@{#2}%
    \ifx\t@xt@\empty\def\@rgdeux{\v@lTheta}\else\X@rgdeux@#2\fi%
    \c@ssin{\costhet@}{\sinthet@}{#1}\v@lmin=\costhet@ pt\v@lmax=\sinthet@ pt%
    \ifcase\CUR@proj%
    \v@leur=\@rgdeux\v@lmin\xdef\cxa@{\repdecn@mb{\v@leur}}%
    \v@leur=\@rgdeux\v@lmax\xdef\cxb@{\repdecn@mb{\v@leur}}\v@leur=\@rgdeux pt%
    \relax\ifdim\v@leur>\p@\message{*** Lambda too large ! See \BS@ figset proj() !}\fi%
    \else%
    \v@lmax=-\v@lmax\xdef\cxa@{\repdecn@mb{\v@lmax}}\xdef\cxb@{\costhet@}%
    \ifx\t@xt@\empty\edef\@rgdeux{\def@ulttheta}\fi\c@ssin{\C@}{\S@}{\@rgdeux}%
    \v@lmax=-\S@ pt%
    \v@leur=\v@lmax\v@leur=\costhet@\v@leur\xdef\cya@{\repdecn@mb{\v@leur}}%
    \v@leur=\v@lmax\v@leur=\sinthet@\v@leur\xdef\cyb@{\repdecn@mb{\v@leur}}%
    \xdef\cyc@{\C@}\v@lmin=-\C@ pt%
    \v@leur=\v@lmin\v@leur=\costhet@\v@leur\xdef\cza@{\repdecn@mb{\v@leur}}%
    \v@leur=\v@lmin\v@leur=\sinthet@\v@leur\xdef\czb@{\repdecn@mb{\v@leur}}%
    \xdef\czc@{\repdecn@mb{\v@lmax}}\fi%
    \xdef\v@lTheta{\@rgdeux}}}
\ctr@ld@f\def\def@ultpsi{40}
\ctr@ld@f\def\def@ulttheta{25}
\ctr@ln@m\l@debut
\ctr@ln@m\n@mref
\ctr@ld@f\def\Figsetpr@j#1=#2|{\keln@mtr#1|%
    \def\n@mref{dep}\ifx\l@debut\n@mref\Figsetd@p{#2}\else
    \def\n@mref{dis}\ifx\l@debut\n@mref%
     \ifnum\CUR@proj=\tw@\figsetobdist(#2)\else\Figset@rr\fi\else
    \def\n@mref{lam}\ifx\l@debut\n@mref\Figsetd@p{#2}\else
    \def\n@mref{lat}\ifx\l@debut\n@mref\Figsetth@{#2}\else
    \def\n@mref{lon}\ifx\l@debut\n@mref\figsetview(#2)\else
    \def\n@mref{psi}\ifx\l@debut\n@mref\figsetview(#2)\else
    \def\n@mref{tar}\ifx\l@debut\n@mref%
     \ifnum\CUR@proj=\tw@\figsettarget[#2]\else\Figset@rr\fi\else
    \def\n@mref{the}\ifx\l@debut\n@mref\Figsetth@{#2}\else
    \W@rnmesAttr{figset proj}{#1}\fi\fi\fi\fi\fi\fi\fi\fi}
\ctr@ld@f\def\Figsetd@p#1{\ifnum\CUR@proj=\z@\figsetview(\v@lPsi,#1)\else\Figset@rr\fi}
\ctr@ld@f\def\Figsetth@#1{\ifnum\CUR@proj=\z@\Figset@rr\else\figsetview(\v@lPsi,#1)\fi}
\ctr@ld@f\def\Figset@rr{\message{*** \BS@ figset proj(): Attribute "\n@mref" ignored, incompatible
    with current projection}}
\ctr@ld@f\def\initb@undb@xTD{\wd\BminTD@=\maxdimen\ht\BminTD@=\maxdimen\dp\BminTD@=\maxdimen%
    \wd\BmaxTD@=-\maxdimen\ht\BmaxTD@=-\maxdimen\dp\BmaxTD@=-\maxdimen}
\ctr@ln@w{newbox}\Gb@x      
\ctr@ln@w{newbox}\Gb@xSC    
\ctr@ln@w{newtoks}\c@nsymb  
\ctr@ln@w{newif}\ifr@undcoord\ctr@ln@w{newif}\ifunitpr@sent
\ctr@ld@f\def\unssqrttw@{0.707106 }
\ctr@ld@f\def\figAst{\raise-1.15ex\hbox{$\ast$}}
\ctr@ld@f\def\figBullet{\raise-1.15ex\hbox{$\bullet$}}
\ctr@ld@f\def\figCirc{\raise-1.15ex\hbox{$\circ$}}
\ctr@ld@f\def\figDiamond{\raise-1.15ex\hbox{$\diamond$}}%
\ctr@ld@f\def\boxit#1#2{\leavevmode\hbox{\vrule\vbox{\hrule\vglue#1%
    \vtop{\hbox{\kern#1{#2}\kern#1}\vglue#1\hrule}}\vrule}}
\ctr@ld@f
\ctr@ld@f
\ctr@ld@f\def\c@nterpt{\ignorespaces%
    \kern-.5\wd\Gb@xSC%
    \raise-.5\ht\Gb@xSC\rlap{\hbox{\raise.5\dp\Gb@xSC\hbox{\copy\Gb@xSC}}}%
    \kern .5\wd\Gb@xSC\ignorespaces}
\ctr@ld@f\def\b@undb@xSC#1#2{{\v@lXa=#1\v@lYa=#2%
    \v@leur=\ht\Gb@xSC\advance\v@leur\dp\Gb@xSC%
    \advance\v@lXa-.5\wd\Gb@xSC\advance\v@lYa-.5\v@leur\b@undb@x{\v@lXa}{\v@lYa}%
    \advance\v@lXa\wd\Gb@xSC\advance\v@lYa\v@leur\b@undb@x{\v@lXa}{\v@lYa}}}
\ctr@ln@m\Dist@n
\ctr@ln@m\l@suite
\ctr@ld@f\def\@keldist#1#2{\edef\Dist@n{#2}\y@tiunit{\Dist@n}%
    \ifunitpr@sent#1=\Dist@n\else#1=\Dist@n\unit@\fi}
\ctr@ld@f\def\y@tiunit#1{\unitpr@sentfalse\expandafter\y@tiunit@#1:}
\ctr@ld@f\def\y@tiunit@#1#2:{\ifcat#1a\unitpr@senttrue\else\def\l@suite{#2}%
    \ifx\l@suite\empty\else\y@tiunit@#2:\fi\fi}
\ctr@ln@m\figcoord
\ctr@ld@f\def\figcoordDD#1{{\v@lX=\ptT@unit@\v@lX\v@lY=\ptT@unit@\v@lY%
    \ifr@undcoord\ifcase#1\v@leur=0.5pt\or\v@leur=0.05pt\or\v@leur=0.005pt%
    \or\v@leur=0.0005pt\else\v@leur=\z@\fi%
    \ifdim\v@lX<\z@\advance\v@lX-\v@leur\else\advance\v@lX\v@leur\fi%
    \ifdim\v@lY<\z@\advance\v@lY-\v@leur\else\advance\v@lY\v@leur\fi\fi%
    (\@ffichnb{#1}{\repdecn@mb{\v@lX}},\ifmmode\else\thinspace\fi%
    \@ffichnb{#1}{\repdecn@mb{\v@lY}})}}
\ctr@ld@f\def\@ffichnb#1#2{{\def\@@ffich{\@ffich#1(}\edef\n@mbre{#2}%
    \expandafter\@@ffich\n@mbre)}}
\ctr@ld@f\def\@ffich#1(#2.#3){{#2\ifnum#1>\z@.\fi\def\dig@ts{#3}\s@mme=\z@%
    \loop\ifnum\s@mme<#1\expandafter\@ffichdec\dig@ts:\advance\s@mme\@ne\repeat}}
\ctr@ld@f\def\@ffichdec#1#2:{\relax#1\def\dig@ts{#20}}
\ctr@ld@f\def\figcoordTD#1{{\v@lX=\ptT@unit@\v@lX\v@lY=\ptT@unit@\v@lY\v@lZ=\ptT@unit@\v@lZ%
    \ifr@undcoord\ifcase#1\v@leur=0.5pt\or\v@leur=0.05pt\or\v@leur=0.005pt%
    \or\v@leur=0.0005pt\else\v@leur=\z@\fi%
    \ifdim\v@lX<\z@\advance\v@lX-\v@leur\else\advance\v@lX\v@leur\fi%
    \ifdim\v@lY<\z@\advance\v@lY-\v@leur\else\advance\v@lY\v@leur\fi%
    \ifdim\v@lZ<\z@\advance\v@lZ-\v@leur\else\advance\v@lZ\v@leur\fi\fi%
    (\@ffichnb{#1}{\repdecn@mb{\v@lX}},\ifmmode\else\thinspace\fi%
     \@ffichnb{#1}{\repdecn@mb{\v@lY}},\ifmmode\else\thinspace\fi%
     \@ffichnb{#1}{\repdecn@mb{\v@lZ}})}}
\ctr@ld@f\def\figsetroundcoord#1{\expandafter\Figsetr@undcoord#1:\ignorespaces}
\ctr@ld@f\def\Figsetr@undcoord#1#2:{\if#1n\r@undcoordfalse\else\r@undcoordtrue\fi}
\ctr@ld@f\def\Figsetwr@te#1=#2|{\keln@mun#1|%
    \def\n@mref{m}\ifx\l@debut\n@mref\figsetmark{#2}\else
    \def\n@mref{p}\ifx\l@debut\n@mref\figsetptname{#2}\else
    \def\n@mref{r}\ifx\l@debut\n@mref\figsetroundcoord{#2}\else
    \W@rnmesAttr{figset write}{#1}\fi\fi\fi}
\ctr@ld@f\def\figsetmark#1{\c@nsymb={#1}\setbox\Gb@xSC=\hbox{\the\c@nsymb}\ignorespaces}
\ctr@ln@m\ptn@me
\ctr@ld@f\def\figsetptname#1{\def\ptn@me##1{#1}\ignorespaces}
\ctr@ld@f\def\FigWrit@L#1:#2(#3,#4){\ignorespaces\@keldist\v@leur{#3}\@keldist\delt@{#4}%
    \C@rp@r@m\def\list@num{#1}\@ecfor\p@int:=\list@num\do{\FigWrit@pt{\p@int}{#2}}}
\ctr@ld@f\def\FigWrit@pt#1#2{\FigWp@r@m{#1}{#2}\Vc@rrect\figWp@si%
    \ifdim\wd\Gb@xSC>\z@\b@undb@xSC{\v@lX}{\v@lY}\fi\figWBB@x}
\ctr@ld@f\def\FigWp@r@m#1#2{\Figg@tXY{#1}%
    \setbox\Gb@x=\hbox{\def\t@xt@{#2}\ifx\t@xt@\empty\Figg@tT{#1}\else#2\fi}\c@lprojSP}
\ctr@ld@f\let\Vc@rrect=\relax
\ctr@ld@f\let\C@rp@r@m=\relax
\ctr@ld@f\def\figwrite[#1]#2{{\ignorespaces\def\list@num{#1}\@ecfor\p@int:=\list@num\do{%
    \setbox\Gb@x=\hbox{\def\t@xt@{#2}\ifx\t@xt@\empty\Figg@tT{\p@int}\else#2\fi}%
    \Figwrit@{\p@int}}}\ignorespaces}
\ctr@ld@f\def\Figwrit@#1{\Figg@tXY{#1}\c@lprojSP%
    \rlap{\kern\v@lX\raise\v@lY\hbox{\unhcopy\Gb@x}}\v@leur=\v@lY%
    \advance\v@lY\ht\Gb@x\b@undb@x{\v@lX}{\v@lY}\advance\v@lX\wd\Gb@x%
    \v@lY=\v@leur\advance\v@lY-\dp\Gb@x\b@undb@x{\v@lX}{\v@lY}}
\ctr@ld@f\def\figwritec[#1]#2{{\ignorespaces\def\list@num{#1}%
    \@ecfor\p@int:=\list@num\do{\Figwrit@c{\p@int}{#2}}}\ignorespaces}
\ctr@ld@f\def\Figwrit@c#1#2{\FigWp@r@m{#1}{#2}%
    \rlap{\kern\v@lX\raise\v@lY\hbox{\rlap{\kern-.5\wd\Gb@x%
    \raise-.5\ht\Gb@x\hbox{\raise.5\dp\Gb@x\hbox{\unhcopy\Gb@x}}}}}%
    \v@leur=\ht\Gb@x\advance\v@leur\dp\Gb@x%
    \advance\v@lX-.5\wd\Gb@x\advance\v@lY-.5\v@leur\b@undb@x{\v@lX}{\v@lY}%
    \advance\v@lX\wd\Gb@x\advance\v@lY\v@leur\b@undb@x{\v@lX}{\v@lY}}
\ctr@ld@f\def\figwritep[#1]{{\ignorespaces\def\list@num{#1}\setbox\Gb@x=\hbox{\c@nterpt}%
    \@ecfor\p@int:=\list@num\do{\Figwrit@{\p@int}}}\ignorespaces}
\ctr@ld@f\def\figwritew#1:#2(#3){\figwritegcw#1:{#2}(#3,0pt)}
\ctr@ld@f\def\figwritee#1:#2(#3){\figwritegce#1:{#2}(#3,0pt)}
\ctr@ld@f\def\figwriten#1:#2(#3){{\def\Vc@rrect{\v@lZ=\v@leur\advance\v@lZ\dp\Gb@x}%
    \Figwrit@NS#1:{#2}(#3)}\ignorespaces}
\ctr@ld@f\def\figwrites#1:#2(#3){{\def\Vc@rrect{\v@lZ=-\v@leur\advance\v@lZ-\ht\Gb@x}%
    \Figwrit@NS#1:{#2}(#3)}\ignorespaces}
\ctr@ld@f\def\Figwrit@NS#1:#2(#3){\let\figWp@si=\FigWp@siNS\let\figWBB@x=\FigWBB@xNS%
    \FigWrit@L#1:{#2}(#3,0pt)}
\ctr@ld@f\def\FigWp@siNS{\rlap{\kern\v@lX\raise\v@lY\hbox{\rlap{\kern-.5\wd\Gb@x%
    \raise\v@lZ\hbox{\unhcopy\Gb@x}}\c@nterpt}}}
\ctr@ld@f\def\FigWBB@xNS{\advance\v@lY\v@lZ%
    \advance\v@lY-\dp\Gb@x\advance\v@lX-.5\wd\Gb@x\b@undb@x{\v@lX}{\v@lY}%
    \advance\v@lY\ht\Gb@x\advance\v@lY\dp\Gb@x%
    \advance\v@lX\wd\Gb@x\b@undb@x{\v@lX}{\v@lY}}
\ctr@ld@f\def\figwritenw#1:#2(#3){{\let\figWp@si=\FigWp@sigW\let\figWBB@x=\FigWBB@xgWE%
    \def\C@rp@r@m{\v@leur=\unssqrttw@\v@leur\delt@=\v@leur%
    \ifdim\delt@=\z@\delt@=\epsil@n\fi}\let@xte={-}\FigWrit@L#1:{#2}(#3,0pt)}\ignorespaces}
\ctr@ld@f\def\figwritesw#1:#2(#3){{\let\figWp@si=\FigWp@sigW\let\figWBB@x=\FigWBB@xgWE%
    \def\C@rp@r@m{\v@leur=\unssqrttw@\v@leur\delt@=-\v@leur%
    \ifdim\delt@=\z@\delt@=-\epsil@n\fi}\let@xte={-}\FigWrit@L#1:{#2}(#3,0pt)}\ignorespaces}
\ctr@ld@f\def\figwritene#1:#2(#3){{\let\figWp@si=\FigWp@sigE\let\figWBB@x=\FigWBB@xgWE%
    \def\C@rp@r@m{\v@leur=\unssqrttw@\v@leur\delt@=\v@leur%
    \ifdim\delt@=\z@\delt@=\epsil@n\fi}\let@xte={}\FigWrit@L#1:{#2}(#3,0pt)}\ignorespaces}
\ctr@ld@f\def\figwritese#1:#2(#3){{\let\figWp@si=\FigWp@sigE\let\figWBB@x=\FigWBB@xgWE%
    \def\C@rp@r@m{\v@leur=\unssqrttw@\v@leur\delt@=-\v@leur%
    \ifdim\delt@=\z@\delt@=-\epsil@n\fi}\let@xte={}\FigWrit@L#1:{#2}(#3,0pt)}\ignorespaces}
\ctr@ld@f\def\figwritegw#1:#2(#3,#4){{\let\figWp@si=\FigWp@sigW\let\figWBB@x=\FigWBB@xgWE%
    \let@xte={-}\FigWrit@L#1:{#2}(#3,#4)}\ignorespaces}
\ctr@ld@f\def\figwritege#1:#2(#3,#4){{\let\figWp@si=\FigWp@sigE\let\figWBB@x=\FigWBB@xgWE%
    \let@xte={}\FigWrit@L#1:{#2}(#3,#4)}\ignorespaces}
\ctr@ld@f\def\FigWp@sigW{\v@lXa=\z@\v@lYa=\ht\Gb@x\advance\v@lYa\dp\Gb@x%
    \ifdim\delt@>\z@\relax%
    \rlap{\kern\v@lX\raise\v@lY\hbox{\rlap{\kern-\wd\Gb@x\kern-\v@leur%
          \raise\delt@\hbox{\raise\dp\Gb@x\hbox{\unhcopy\Gb@x}}}\c@nterpt}}%
    \else\ifdim\delt@<\z@\relax\v@lYa=-\v@lYa%
    \rlap{\kern\v@lX\raise\v@lY\hbox{\rlap{\kern-\wd\Gb@x\kern-\v@leur%
          \raise\delt@\hbox{\raise-\ht\Gb@x\hbox{\unhcopy\Gb@x}}}\c@nterpt}}%
    \else\v@lXa=-.5\v@lYa%
    \rlap{\kern\v@lX\raise\v@lY\hbox{\rlap{\kern-\wd\Gb@x\kern-\v@leur%
          \raise-.5\ht\Gb@x\hbox{\raise.5\dp\Gb@x\hbox{\unhcopy\Gb@x}}}\c@nterpt}}%
    \fi\fi}
\ctr@ld@f\def\FigWp@sigE{\v@lXa=\z@\v@lYa=\ht\Gb@x\advance\v@lYa\dp\Gb@x%
    \ifdim\delt@>\z@\relax%
    \rlap{\kern\v@lX\raise\v@lY\hbox{\c@nterpt\kern\v@leur%
          \raise\delt@\hbox{\raise\dp\Gb@x\hbox{\unhcopy\Gb@x}}}}%
    \else\ifdim\delt@<\z@\relax\v@lYa=-\v@lYa%
    \rlap{\kern\v@lX\raise\v@lY\hbox{\c@nterpt\kern\v@leur%
          \raise\delt@\hbox{\raise-\ht\Gb@x\hbox{\unhcopy\Gb@x}}}}%
    \else\v@lXa=-.5\v@lYa%
    \rlap{\kern\v@lX\raise\v@lY\hbox{\c@nterpt\kern\v@leur%
          \raise-.5\ht\Gb@x\hbox{\raise.5\dp\Gb@x\hbox{\unhcopy\Gb@x}}}}%
    \fi\fi}
\ctr@ld@f\def\FigWBB@xgWE{\advance\v@lY\delt@%
    \advance\v@lX\the\let@xte\v@leur\advance\v@lY\v@lXa\b@undb@x{\v@lX}{\v@lY}%
    \advance\v@lX\the\let@xte\wd\Gb@x\advance\v@lY\v@lYa\b@undb@x{\v@lX}{\v@lY}}
\ctr@ld@f\def\figwritegcw#1:#2(#3,#4){{\let\figWp@si=\FigWp@sigcW\let\figWBB@x=\FigWBB@xgcWE%
    \let@xte={-}\FigWrit@L#1:{#2}(#3,#4)}\ignorespaces}
\ctr@ld@f\def\figwritegce#1:#2(#3,#4){{\let\figWp@si=\FigWp@sigcE\let\figWBB@x=\FigWBB@xgcWE%
    \let@xte={}\FigWrit@L#1:{#2}(#3,#4)}\ignorespaces}
\ctr@ld@f\def\FigWp@sigcW{\rlap{\kern\v@lX\raise\v@lY\hbox{\rlap{\kern-\wd\Gb@x\kern-\v@leur%
     \raise-.5\ht\Gb@x\hbox{\raise\delt@\hbox{\raise.5\dp\Gb@x\hbox{\unhcopy\Gb@x}}}}%
     \c@nterpt}}}
\ctr@ld@f\def\FigWp@sigcE{\rlap{\kern\v@lX\raise\v@lY\hbox{\c@nterpt\kern\v@leur%
    \raise-.5\ht\Gb@x\hbox{\raise\delt@\hbox{\raise.5\dp\Gb@x\hbox{\unhcopy\Gb@x}}}}}}
\ctr@ld@f\def\FigWBB@xgcWE{\v@lZ=\ht\Gb@x\advance\v@lZ\dp\Gb@x%
    \advance\v@lX\the\let@xte\v@leur\advance\v@lY\delt@\advance\v@lY.5\v@lZ%
    \b@undb@x{\v@lX}{\v@lY}%
    \advance\v@lX\the\let@xte\wd\Gb@x\advance\v@lY-\v@lZ\b@undb@x{\v@lX}{\v@lY}}
\ctr@ld@f\def\figwritebn#1:#2(#3){{\def\Vc@rrect{\v@lZ=\v@leur}\Figwrit@NS#1:{#2}(#3)}\ignorespaces}
\ctr@ld@f\def\figwritebs#1:#2(#3){{\def\Vc@rrect{\v@lZ=-\v@leur}\Figwrit@NS#1:{#2}(#3)}\ignorespaces}
\ctr@ld@f\def\figwritebw#1:#2(#3){{\let\figWp@si=\FigWp@sibW\let\figWBB@x=\FigWBB@xbWE%
    \let@xte={-}\FigWrit@L#1:{#2}(#3,0pt)}\ignorespaces}
\ctr@ld@f\def\figwritebe#1:#2(#3){{\let\figWp@si=\FigWp@sibE\let\figWBB@x=\FigWBB@xbWE%
    \let@xte={}\FigWrit@L#1:{#2}(#3,0pt)}\ignorespaces}
\ctr@ld@f\def\FigWp@sibW{\rlap{\kern\v@lX\raise\v@lY\hbox{\rlap{\kern-\wd\Gb@x\kern-\v@leur%
          \hbox{\unhcopy\Gb@x}}\c@nterpt}}}
\ctr@ld@f\def\FigWp@sibE{\rlap{\kern\v@lX\raise\v@lY\hbox{\c@nterpt\kern\v@leur%
          \hbox{\unhcopy\Gb@x}}}}
\ctr@ld@f\def\FigWBB@xbWE{\v@lZ=\ht\Gb@x\advance\v@lZ\dp\Gb@x%
    \advance\v@lX\the\let@xte\v@leur\advance\v@lY\ht\Gb@x\b@undb@x{\v@lX}{\v@lY}%
    \advance\v@lX\the\let@xte\wd\Gb@x\advance\v@lY-\v@lZ\b@undb@x{\v@lX}{\v@lY}}
\ctr@ln@w{newread}\frf@g  \ctr@ln@w{newwrite}\fwf@g
\ctr@ln@w{newif}\ifCUR@PS
\ctr@ln@w{newif}\ifGR@cri
\ctr@ln@w{newif}\ifUse@llipse
\ctr@ln@w{newif}\ifGRdebugm@de \GRdebugm@defalse 
\ctr@ln@w{newif}\ifPDFm@ke
\ifx\pdfliteral\undefined\else\ifnum\pdfoutput>\z@\PDFm@ketrue\fi\fi
\ctr@ld@f\def\initPDF@rDVI{%
\ifPDFm@ke
 \let\figscan=\figscan@E
 \let\newGr@FN=\newGr@FNPDF
 \ctr@ld@f\def\c@mcurveto{c}
 \ctr@ld@f\def\c@mfill{f}
 \ctr@ld@f\def\c@mgsave{q}
 \ctr@ld@f\def\c@mgrestore{Q}
 \ctr@ld@f\def\c@mlineto{l}
 \ctr@ld@f\def\c@mmoveto{m}
 \ctr@ld@f\def\c@msetgray{g}     \ctr@ld@f\def\c@msetgrayStroke{G}
 \ctr@ld@f\def\c@msetcmykcolor{k}\ctr@ld@f\def\c@msetcmykcolorStroke{K}
 \ctr@ld@f\def\c@msetrgbcolor{rg}\ctr@ld@f\def\c@msetrgbcolorStroke{RG}
 \ctr@ld@f\def\d@fprimarC@lor{\CUR@color\space\CUR@colorc@md%
               \space\CUR@color\space\CUR@colorc@mdStroke}
 \ctr@ld@f\def\c@msetdash{d}
 \ctr@ld@f\def\c@msetlinejoin{j}
 \ctr@ld@f\def\c@msetlinewidth{w}
 \ctr@ld@f\def\f@gclosestroke{\immediate\write\fwf@g{s}}
 \ctr@ld@f\def\f@gfill{\immediate\write\fwf@g{\fillc@md}}
 \ctr@ld@f\def\f@gnewpath{}
 \ctr@ld@f\def\f@gstroke{\immediate\write\fwf@g{S}}
\else
 \let\figinsertE=\figinsert
 \let\newGr@FN=\newGr@FNDVI
 \ctr@ld@f\def\c@mcurveto{curveto}
 \ctr@ld@f\def\c@mfill{fill}
 \ctr@ld@f\def\c@mgsave{gsave}
 \ctr@ld@f\def\c@mgrestore{grestore}
 \ctr@ld@f\def\c@mlineto{lineto}
 \ctr@ld@f\def\c@mmoveto{moveto}
 \ctr@ld@f\def\c@msetgray{setgray}          \ctr@ld@f\def\c@msetgrayStroke{}
 \ctr@ld@f\def\c@msetcmykcolor{setcmykcolor}\ctr@ld@f\def\c@msetcmykcolorStroke{}
 \ctr@ld@f\def\c@msetrgbcolor{setrgbcolor}  \ctr@ld@f\def\c@msetrgbcolorStroke{}
 \ctr@ld@f\def\d@fprimarC@lor{\CUR@color\space\CUR@colorc@md}
 \ctr@ld@f\def\c@msetdash{setdash}
 \ctr@ld@f\def\c@msetlinejoin{setlinejoin}
 \ctr@ld@f\def\c@msetlinewidth{setlinewidth}
 \ctr@ld@f\def\f@gclosestroke{\immediate\write\fwf@g{closepath\space stroke}}
 \ctr@ld@f\def\f@gfill{\immediate\write\fwf@g{\fillc@md}}
 \ctr@ld@f\def\f@gnewpath{\immediate\write\fwf@g{newpath}}
 \ctr@ld@f\def\f@gstroke{\immediate\write\fwf@g{stroke}}
\fi}
\ctr@ld@f\def\c@pypsfile#1#2{\c@pyfil@{\immediate\write#1}{#2}}
\ctr@ld@f\def\Figinclud@PDF#1#2{\openin\frf@g=#1\pdfliteral{q #2 0 0 #2 0 0 cm}%
    \c@pyfil@{\pdfliteral}{\frf@g}\pdfliteral{Q}\closein\frf@g}
\ctr@ln@w{newif}\ifmored@ta
\ctr@ln@m\bl@nkline
\ctr@ld@f\def\c@pyfil@#1#2{\def\bl@nkline{\par}{\catcode`\%=12
    \loop\ifeof#2\mored@tafalse\else\mored@tatrue\immediate\read#2 to\tr@c
    \ifx\tr@c\bl@nkline\else#1{\tr@c}\fi\fi\ifmored@ta\repeat}}
\ctr@ld@f\def\keln@mun#1#2|{\def\l@debut{#1}\def\l@suite{#2}}
\ctr@ld@f\def\keln@mde#1#2#3|{\def\l@debut{#1#2}\def\l@suite{#3}}
\ctr@ld@f\def\keln@mtr#1#2#3#4|{\def\l@debut{#1#2#3}\def\l@suite{#4}}
\ctr@ld@f\def\keln@mqu#1#2#3#4#5|{\def\l@debut{#1#2#3#4}\def\l@suite{#5}}
\ctr@ld@f\let\@psffilein=\frf@g 
\ctr@ln@w{newif}\if@psffileok    
\ctr@ln@w{newif}\if@psfbbfound   
\ctr@ln@w{newif}\if@psfverbose   
\@psfverbosetrue
\ctr@ln@m\@psfllx \ctr@ln@m\@psflly
\ctr@ln@m\@psfurx \ctr@ln@m\@psfury
\ctr@ln@m\resetcolonc@tcode
\ctr@ld@f\def\@psfgetbb#1{\global\@psfbbfoundfalse%
\global\def\@psfllx{0}\global\def\@psflly{0}%
\global\def\@psfurx{30}\global\def\@psfury{30}%
\openin\@psffilein=#1\relax
\ifeof\@psffilein\errmessage{I couldn't open #1, will ignore it}\else
   \edef\resetcolonc@tcode{\catcode`\noexpand\:\the\catcode`\:\relax}%
   {\@psffileoktrue \chardef\other=12
    \def\do##1{\catcode`##1=\other}\dospecials \catcode`\ =10 \resetcolonc@tcode
    \loop
       \read\@psffilein to \@psffileline
       \ifeof\@psffilein\@psffileokfalse\else
          \expandafter\@psfaux\@psffileline:. \\%
       \fi
   \if@psffileok\repeat
   \if@psfbbfound\else
    \if@psfverbose\message{No bounding box comment in #1; using defaults}\fi\fi
   }\closein\@psffilein\fi}%
\ctr@ln@m\@psfbblit
\ctr@ln@m\@psfpercent
{\catcode`\%=12 \global\let\@psfpercent=
\ctr@ln@m\@psfaux
\long\def\@psfaux#1#2:#3\\{\ifx#1\@psfpercent
   \def\testit{#2}\ifx\testit\@psfbblit
      \@psfgrab #3 . . . \\%
      \@psffileokfalse
      \global\@psfbbfoundtrue
   \fi\else\ifx#1\par\else\@psffileokfalse\fi\fi}%
\ctr@ld@f\def\@psfempty{}%
\ctr@ld@f\def\@psfgrab #1 #2 #3 #4 #5\\{%
\global\def\@psfllx{#1}\ifx\@psfllx\@psfempty
      \@psfgrab #2 #3 #4 #5 .\\\else
   \global\def\@psflly{#2}%
   \global\def\@psfurx{#3}\global\def\@psfury{#4}\fi}%
\ctr@ld@f\def\PSwrit@cmd#1#2#3{{\Figg@tXY{#1}\c@lprojSP\b@undb@x{\v@lX}{\v@lY}%
    \v@lX=\ptT@ptps\v@lX\v@lY=\ptT@ptps\v@lY%
    \immediate\write#3{\repdecn@mb{\v@lX}\space\repdecn@mb{\v@lY}\space#2}}}
\ctr@ld@f\def\PSwrit@cmdS#1#2#3#4#5{{\Figg@tXY{#1}\c@lprojSP\b@undb@x{\v@lX}{\v@lY}%
    \global\result@t=\v@lX\global\result@@t=\v@lY%
    \v@lX=\ptT@ptps\v@lX\v@lY=\ptT@ptps\v@lY%
    \immediate\write#3{\repdecn@mb{\v@lX}\space\repdecn@mb{\v@lY}\space#2}}%
    \edef#4{\the\result@t}\edef#5{\the\result@@t}}
\ctr@ld@f\def\update@ttr#1#2#3{\Figdisc@rdLTS{#3}{\n@mref}%
    \ifx\n@mref\D@FTref#2{#1}\else#2{#3}\fi}
\ctr@ld@f\def\D@FTref{default}
\ctr@ld@f\def\W@rnmesAttr#1#2{%
    \immediate\write16{*** Unknown attribute: \BS@ #1(..., #2=...)}}
\ctr@ld@f\def\W@rnmeskwd#1#2{%
    \immediate\write16{*** Unknown keyword #2 in \BS@ #1}}
\ctr@ld@f\def\W@rnmesIgn#1{\immediate\write16{*** \BS@ #1 is ignored inside a
     \BS@ figdrawbegin-\BS@ figdrawend block.}}
\ctr@ld@f\def\Psset@lti#1=#2|{\keln@mtr#1|%
    \def\n@mref{blc}\ifx\l@debut\n@mref\update@ttr\D@FTref\P@setblcolor{#2}\else
    \def\n@mref{bld}\ifx\l@debut\n@mref\update@ttr\D@FTref\P@setbldash{#2}\else
    \def\n@mref{blw}\ifx\l@debut\n@mref\update@ttr\D@FTref\P@setblwidth{#2}\else
    \def\n@mref{sqc}\ifx\l@debut\n@mref\update@ttr\D@FTref\P@setsqcolor{#2}\else
    \def\n@mref{sqd}\ifx\l@debut\n@mref\update@ttr\D@FTref\P@setsqdash{#2}\else
    \def\n@mref{sqw}\ifx\l@debut\n@mref\update@ttr\D@FTref\P@setsqwidth{#2}\else
    \W@rnmesAttr{figset altitude}{#1}\fi\fi\fi\fi\fi\fi}
\ctr@ln@m\DDV@blcolor
\ctr@ld@f\def\P@setblcolor#1{\edef\DDV@blcolor{#1}}
\ctr@ln@m\DDV@bldash
\ctr@ld@f\def\P@setbldash#1{\edef\DDV@bldash{#1}}
\ctr@ln@m\DDV@blwidth
\ctr@ld@f\def\P@setblwidth#1{\edef\DDV@blwidth{#1}}
\ctr@ln@m\DDV@sqcolor
\ctr@ld@f\def\P@setsqcolor#1{\edef\DDV@sqcolor{#1}}
\ctr@ln@m\DDV@sqdash
\ctr@ld@f\def\P@setsqdash#1{\edef\DDV@sqdash{#1}}
\ctr@ln@m\DDV@sqwidth
\ctr@ld@f\def\P@setsqwidth#1{\edef\DDV@sqwidth{#1}}
\ctr@ld@f\def\figdrawaltitude#1[#2,#3,#4]{{\ifCUR@PS\ifGR@cri%
    \PSc@mment{altitude Square Dim=#1, Triangle=[#2 / #3,#4]}%
    \s@uvc@ntr@l\et@tpsaltitude\resetc@ntr@l{2}\figptorthoprojline-5:=#2/#3,#4/%
    \figvectP -1[#3,#4]\n@rminf{\v@leur}{-1}\vecunit@{-3}{-1}%
    \figvectP -1[-5,#3]\n@rminf{\v@lmin}{-1}\figvectP -2[-5,#4]\n@rminf{\v@lmax}{-2}%
    \ifdim\v@lmin<\v@lmax\s@mme=#3\else\v@lmax=\v@lmin\s@mme=#4\fi%
    \figvectP -4[-5,#2]\vecunit@{-4}{-4}\delt@=#1\unit@%
    \edef\t@ille{\repdecn@mb{\delt@}}\figpttra-1:=-5/\t@ille,-3/%
    \figptstra-3=-5,-1/\t@ille,-4/\figdrawline[#2,-5]%
    \Pss@tspecifSt{color=\DDV@sqcolor,dash=\DDV@sqdash,width=\DDV@sqwidth}%
    \figdrawline[-1,-2,-3]%
    \Psrest@reSt{color=\DDV@sqcolor,dash=\DDV@sqdash,width=\DDV@sqwidth}%
    \ifdim\v@leur<\v@lmax%
    \Pss@tspecifSt{color=\DDV@blcolor,dash=\DDV@bldash,width=\DDV@blwidth}%
    \figdrawline[-5,\the\s@mme]%
    \Psrest@reSt{color=\DDV@blcolor,dash=\DDV@bldash,width=\DDV@blwidth}%
    \fi\PSc@mment{End altitude}\resetc@ntr@l\et@tpsaltitude\fi\fi}}
\ctr@ld@f\def\Ps@rcerc#1;#2(#3,#4){\ellBB@x#1;#2,#2(#3,#4,0)%
    \f@gnewpath{\delt@=#2\unit@\delt@=\ptT@ptps\delt@%
    \BdingB@xfalse%
    \PSwrit@cmd{#1}{\repdecn@mb{\delt@}\space #3\space #4\space arc}{\fwf@g}}}
\ctr@ln@m\figdrawarccirc
\ctr@ld@f\def\Q@arccircDD#1;#2(#3,#4){\ifCUR@PS\ifGR@cri%
    \PSc@mment{arccircDD Center=#1 ; Radius=#2 (Ang1=#3, Ang2=#4)}%
    \iffillm@de\Ps@rcerc#1;#2(#3,#4)%
    \f@gfill%
    \else\Ps@rcerc#1;#2(#3,#4)\f@gstroke\fi%
    \PSc@mment{End arccircDD}\fi\fi}
\ctr@ld@f\def\Q@arccircTD#1,#2,#3;#4(#5,#6){{\ifCUR@PS\ifGR@cri\s@uvc@ntr@l\et@tpsarccircTD%
    \PSc@mment{arccircTD Center=#1,P1=#2,P2=#3 ; Radius=#4 (Ang1=#5, Ang2=#6)}%
    \setc@ntr@l{2}\c@lExtAxes#1,#2,#3(#4)\Q@arcellPATD#1,-4,-5(#5,#6)%
    \PSc@mment{End arccircTD}\resetc@ntr@l\et@tpsarccircTD\fi\fi}}
\ctr@ld@f\def\c@lExtAxes#1,#2,#3(#4){%
    \figvectPTD-5[#1,#2]\vecunit@{-5}{-5}\figvectNTD-4[#1,#2,#3]\vecunit@{-4}{-4}%
    \figvectNVTD-3[-4,-5]\delt@=#4\unit@\edef\r@yon{\repdecn@mb{\delt@}}%
    \figpttra-4:=#1/\r@yon,-5/\figpttra-5:=#1/\r@yon,-3/}
\ctr@ln@m\figdrawarccircP
\ctr@ld@f\def\Q@arccircPDD#1;#2[#3,#4]{{\ifCUR@PS\ifGR@cri\s@uvc@ntr@l\et@tpsarccircPDD%
    \PSc@mment{arccircPDD Center=#1; Radius=#2, [P1=#3, P2=#4]}%
    \Ps@ngleparam#1;#2[#3,#4]\ifdim\v@lmin>\v@lmax\advance\v@lmax\DePI@deg\fi%
    \edef\@ngdeb{\repdecn@mb{\v@lmin}}\edef\@ngfin{\repdecn@mb{\v@lmax}}%
    \figdrawarccirc#1;\r@dius(\@ngdeb,\@ngfin)%
    \PSc@mment{End arccircPDD}\resetc@ntr@l\et@tpsarccircPDD\fi\fi}}
\ctr@ld@f\def\Q@arccircPTD#1;#2[#3,#4,#5]{{\ifCUR@PS\ifGR@cri\s@uvc@ntr@l\et@tpsarccircPTD%
    \PSc@mment{arccircPTD Center=#1; Radius=#2, [P1=#3, P2=#4, P3=#5]}%
    \setc@ntr@l{2}\c@lExtAxes#1,#3,#5(#2)\figdrawarcellPP#1,-4,-5[#3,#4]%
    \PSc@mment{End arccircPTD}\resetc@ntr@l\et@tpsarccircPTD\fi\fi}}
\ctr@ld@f\def\Ps@ngleparam#1;#2[#3,#4]{\setc@ntr@l{2}%
    \figvectPDD-1[#1,#3]\vecunit@{-1}{-1}\Figg@tXY{-1}\arct@n\v@lmin(\v@lX,\v@lY)%
    \figvectPDD-2[#1,#4]\vecunit@{-2}{-2}\Figg@tXY{-2}\arct@n\v@lmax(\v@lX,\v@lY)%
    \v@lmin=\rdT@deg\v@lmin\v@lmax=\rdT@deg\v@lmax%
    \v@leur=#2pt\maxim@m{\mili@u}{-\v@leur}{\v@leur}%
    \edef\r@dius{\repdecn@mb{\mili@u}}}
\ctr@ld@f\def\Ps@rcercBz#1;#2(#3,#4){\Ps@rellBz#1;#2,#2(#3,#4,0)}
\ctr@ld@f\def\Ps@rellBz#1;#2,#3(#4,#5,#6){%
    \ellBB@x#1;#2,#3(#4,#5,#6)\BdingB@xfalse%
    \c@lNbarcs{#4}{#5}\v@leur=#4pt\setc@ntr@l{2}\figptell-13::#1;#2,#3(#4,#6)%
    \f@gnewpath\PSwrit@cmd{-13}{\c@mmoveto}{\fwf@g}%
    \s@mme=\z@\bcl@rellBz#1;#2,#3(#6)\BdingB@xtrue}
\ctr@ld@f\def\bcl@rellBz#1;#2,#3(#4){\relax%
    \ifnum\s@mme<\p@rtent\advance\s@mme\@ne%
    \advance\v@leur\delt@\edef\@ngle{\repdecn@mb\v@leur}\figptell-14::#1;#2,#3(\@ngle,#4)%
    \advance\v@leur\delt@\edef\@ngle{\repdecn@mb\v@leur}\figptell-15::#1;#2,#3(\@ngle,#4)%
    \advance\v@leur\delt@\edef\@ngle{\repdecn@mb\v@leur}\figptell-16::#1;#2,#3(\@ngle,#4)%
    \figptscontrolDD-18[-13,-14,-15,-16]%
    \PSwrit@cmd{-18}{}{\fwf@g}\PSwrit@cmd{-17}{}{\fwf@g}%
    \PSwrit@cmd{-16}{\c@mcurveto}{\fwf@g}%
    \figptcopyDD-13:/-16/\bcl@rellBz#1;#2,#3(#4)\fi}
\ctr@ld@f\def\Ps@rell#1;#2,#3(#4,#5,#6){\ellBB@x#1;#2,#3(#4,#5,#6)%
    \f@gnewpath{\v@lmin=#2\unit@\v@lmin=\ptT@ptps\v@lmin%
    \v@lmax=#3\unit@\v@lmax=\ptT@ptps\v@lmax\BdingB@xfalse%
    \PSwrit@cmd{#1}%
    {#6\space\repdecn@mb{\v@lmin}\space\repdecn@mb{\v@lmax}\space #4\space #5\space ellipse}{\fwf@g}}%
    \global\Use@llipsetrue}
\ctr@ln@m\figdrawarcell
\ctr@ld@f\def\Q@arcellDD#1;#2,#3(#4,#5,#6){{\ifCUR@PS\ifGR@cri%
    \PSc@mment{arcellDD Center=#1 ; XRad=#2, YRad=#3 (Ang1=#4, Ang2=#5, Inclination=#6)}%
    \iffillm@de\Ps@rell#1;#2,#3(#4,#5,#6)%
    \f@gfill%
    \else\Ps@rell#1;#2,#3(#4,#5,#6)\f@gstroke\fi%
    \PSc@mment{End arcellDD}\fi\fi}}
\ctr@ld@f\def\Q@arcellTD#1;#2,#3(#4,#5,#6){{\ifCUR@PS\ifGR@cri\s@uvc@ntr@l\et@tpsarcellTD%
    \PSc@mment{arcellTD Center=#1 ; XRad=#2, YRad=#3 (Ang1=#4, Ang2=#5, Inclination=#6)}%
    \setc@ntr@l{2}\figpttraC -8:=#1/#2,0,0/\figpttraC -7:=#1/0,#3,0/%
    \figvectC -4(0,0,1)\figptsrot -8=-8,-7/#1,#6,-4/\Q@arcellPATD#1,-8,-7(#4,#5)%
    \PSc@mment{End arcellTD}\resetc@ntr@l\et@tpsarcellTD\fi\fi}}
\ctr@ln@m\figdrawarcellPA
\ctr@ld@f\def\Q@arcellPADD#1,#2,#3(#4,#5){{\ifCUR@PS\ifGR@cri\s@uvc@ntr@l\et@tpsarcellPADD%
    \PSc@mment{arcellPADD Center=#1,PtAxis1=#2,PtAxis2=#3 (Ang1=#4, Ang2=#5)}%
    \setc@ntr@l{2}\figvectPDD-1[#1,#2]\vecunit@DD{-1}{-1}\v@lX=\ptT@unit@\result@t%
    \edef\XR@d{\repdecn@mb{\v@lX}}\Figg@tXY{-1}\arct@n\v@lmin(\v@lX,\v@lY)%
    \v@lmin=\rdT@deg\v@lmin\edef\Inclin@{\repdecn@mb{\v@lmin}}%
    \figgetdist\YR@d[#1,#3]\Q@arcellDD#1;\XR@d,\YR@d(#4,#5,\Inclin@)%
    \PSc@mment{End arcellPADD}\resetc@ntr@l\et@tpsarcellPADD\fi\fi}}
\ctr@ld@f\def\Q@arcellPATD#1,#2,#3(#4,#5){{\ifCUR@PS\ifGR@cri\s@uvc@ntr@l\et@tpsarcellPATD%
    \PSc@mment{arcellPATD Center=#1,PtAxis1=#2,PtAxis2=#3 (Ang1=#4, Ang2=#5)}%
    \iffillm@de\Ps@rellPATD#1,#2,#3(#4,#5)%
    \f@gfill%
    \else\Ps@rellPATD#1,#2,#3(#4,#5)\f@gstroke\fi%
    \PSc@mment{End arcellPATD}\resetc@ntr@l\et@tpsarcellPATD\fi\fi}}
\ctr@ld@f\def\Ps@rellPATD#1,#2,#3(#4,#5){\let\c@lprojSP=\relax%
    \setc@ntr@l{2}\figvectPTD-1[#1,#2]\figvectPTD-2[#1,#3]\c@lNbarcs{#4}{#5}%
    \v@leur=#4pt\c@lptellP{#1}{-1}{-2}\Figptpr@j-5:/-3/%
    \f@gnewpath\PSwrit@cmdS{-5}{\c@mmoveto}{\fwf@g}{\X@un}{\Y@un}%
    \edef\C@nt@r{#1}\s@mme=\z@\bcl@rellPATD}
\ctr@ld@f\def\bcl@rellPATD{\relax%
    \ifnum\s@mme<\p@rtent\advance\s@mme\@ne%
    \advance\v@leur\delt@\c@lptellP{\C@nt@r}{-1}{-2}\Figptpr@j-4:/-3/%
    \advance\v@leur\delt@\c@lptellP{\C@nt@r}{-1}{-2}\Figptpr@j-6:/-3/%
    \advance\v@leur\delt@\c@lptellP{\C@nt@r}{-1}{-2}\Figptpr@j-3:/-3/%
    \v@lX=\z@\v@lY=\z@\Figtr@nptDD{-5}{-5}\Figtr@nptDD{2}{-3}%
    \divide\v@lX\@vi\divide\v@lY\@vi%
    \Figtr@nptDD{3}{-4}\Figtr@nptDD{-1.5}{-6}\v@lmin=\v@lX\v@lmax=\v@lY%
    \v@lX=\z@\v@lY=\z@\Figtr@nptDD{2}{-5}\Figtr@nptDD{-5}{-3}%
    \divide\v@lX\@vi\divide\v@lY\@vi\Figtr@nptDD{-1.5}{-4}\Figtr@nptDD{3}{-6}%
    \BdingB@xfalse%
    \Figp@intregDD-4:(\v@lmin,\v@lmax)\PSwrit@cmdS{-4}{}{\fwf@g}{\X@de}{\Y@de}%
    \Figp@intregDD-4:(\v@lX,\v@lY)\PSwrit@cmdS{-4}{}{\fwf@g}{\X@tr}{\Y@tr}%
    \BdingB@xtrue\PSwrit@cmdS{-3}{\c@mcurveto}{\fwf@g}{\X@qu}{\Y@qu}%
    \B@zierBB@x{1}{\Y@un}(\X@un,\X@de,\X@tr,\X@qu)%
    \B@zierBB@x{2}{\X@un}(\Y@un,\Y@de,\Y@tr,\Y@qu)%
    \edef\X@un{\X@qu}\edef\Y@un{\Y@qu}\figptcopyDD-5:/-3/\bcl@rellPATD\fi}
\ctr@ld@f\def\c@lNbarcs#1#2{%
    \delt@=#2pt\advance\delt@-#1pt\maxim@m{\v@lmax}{\delt@}{-\delt@}%
    \v@leur=\v@lmax\divide\v@leur45 \p@rtentiere{\p@rtent}{\v@leur}\advance\p@rtent\@ne%
    \s@mme=\p@rtent\multiply\s@mme\thr@@\divide\delt@\s@mme}
\ctr@ld@f\def\figdrawarcellPP#1,#2,#3[#4,#5]{{\ifCUR@PS\ifGR@cri\s@uvc@ntr@l\et@tpsarcellPP%
    \PSc@mment{arcellPP Center=#1,PtAxis1=#2,PtAxis2=#3 [Point1=#4, Point2=#5]}%
    \setc@ntr@l{2}\figvectP-2[#1,#3]\vecunit@{-2}{-2}\v@lmin=\result@t%
    \invers@{\v@lmax}{\v@lmin}%
    \figvectP-1[#1,#2]\vecunit@{-1}{-1}\v@leur=\result@t%
    \v@leur=\repdecn@mb{\v@lmax}\v@leur\edef\AsB@{\repdecn@mb{\v@leur}}
    \c@lAngle{#1}{#4}{\v@lmin}\edef\@ngdeb{\repdecn@mb{\v@lmin}}%
    \c@lAngle{#1}{#5}{\v@lmax}\ifdim\v@lmin>\v@lmax\advance\v@lmax\DePI@deg\fi%
    \edef\@ngfin{\repdecn@mb{\v@lmax}}\figdrawarcellPA#1,#2,#3(\@ngdeb,\@ngfin)%
    \PSc@mment{End arcellPP}\resetc@ntr@l\et@tpsarcellPP\fi\fi}}
\ctr@ld@f\def\c@lAngle#1#2#3{\figvectP-3[#1,#2]%
    \c@lproscal\delt@[-3,-1]\c@lproscal\v@leur[-3,-2]%
    \v@leur=\AsB@\v@leur\arct@n#3(\delt@,\v@leur)#3=\rdT@deg#3}
\ctr@ln@w{newif}\if@rrowratio\@rrowratiotrue
\ctr@ln@w{newif}\if@rrowhfill
\ctr@ln@w{newif}\if@rrowhout
\ctr@ld@f\def\Psset@rrowhe@d#1=#2|{\keln@mun#1|%
    \def\n@mref{a}\ifx\l@debut\n@mref\update@ttr\D@FTarrowheadangle\Q@s@tarrowheadangle{#2}\else
    \def\n@mref{f}\ifx\l@debut\n@mref\update@ttr\D@FTarrowheadfill\Q@s@tarrowheadfill{#2}\else
    \def\n@mref{l}\ifx\l@debut\n@mref\update@ttr\D@FTarrowheadlength\Q@s@tarrowheadlength{#2}\else
    \def\n@mref{o}\ifx\l@debut\n@mref\update@ttr\D@FTarrowheadout\Q@s@tarrowheadout{#2}\else
    \def\n@mref{r}\ifx\l@debut\n@mref\update@ttr\D@FTarrowheadratio\Q@s@tarrowheadratio{#2}\else
    \W@rnmesAttr{figset arrowhead}{#1}\fi\fi\fi\fi\fi}
\ctr@ln@m\@rrowheadangle
\ctr@ln@m\C@AHANG \ctr@ln@m\S@AHANG \ctr@ln@m\UNSS@N
\ctr@ld@f\def\Q@s@tarrowheadangle#1{\edef\@rrowheadangle{#1}{\c@ssin{\C@}{\S@}{#1}%
    \xdef\C@AHANG{\C@}\xdef\S@AHANG{\S@}\v@lmax=\S@ pt%
    \invers@{\v@leur}{\v@lmax}\maxim@m{\v@leur}{\v@leur}{-\v@leur}%
    \xdef\UNSS@N{\the\v@leur}}}
\ctr@ld@f\def\Q@s@tarrowheadfill#1{\expandafter\set@rrowhfill#1:}
\ctr@ld@f\def\set@rrowhfill#1#2:{\if#1n\@rrowhfillfalse\else\@rrowhfilltrue\fi}
\ctr@ld@f\def\Q@s@tarrowheadout#1{\expandafter\set@rrowhout#1:}
\ctr@ld@f\def\set@rrowhout#1#2:{\if#1n\@rrowhoutfalse\else\@rrowhouttrue\fi}
\ctr@ln@m\@rrowheadlength
\ctr@ld@f\def\Q@s@tarrowheadlength#1{\edef\@rrowheadlength{#1}\@rrowratiofalse}
\ctr@ln@m\@rrowheadratio
\ctr@ld@f\def\Q@s@tarrowheadratio#1{\edef\@rrowheadratio{#1}\@rrowratiotrue}
\ctr@ln@m\D@FTarrowheadlength
\ctr@ld@f\def\figresetarrowhead{%
    \Q@s@tarrowheadangle{\D@FTarrowheadangle}%
    \Q@s@tarrowheadfill{\D@FTarrowheadfill}%
    \Q@s@tarrowheadout{\D@FTarrowheadout}%
    \Q@s@tarrowheadratio{\D@FTarrowheadratio}%
    \d@fm@cdim\D@FTarrowheadlength{\D@FTh@rdahlength}
    \Q@s@tarrowheadlength{\D@FTarrowheadlength}}
\ctr@ld@f\def\D@FTarrowheadratio{0.1}
\ctr@ld@f\def\D@FTarrowheadangle{20}
\ctr@ld@f\def\D@FTarrowheadfill{no}
\ctr@ld@f\def\D@FTarrowheadout{no}
\ctr@ld@f\def\D@FTh@rdahlength{8pt}
\ctr@ln@m\figdrawarrow
\ctr@ld@f\def\Q@arrowDD[#1,#2]{{\ifCUR@PS\ifGR@cri\s@uvc@ntr@l\et@tpsarrow%
    \PSc@mment{arrowDD [Pt1,Pt2]=[#1,#2]}\Q@s@tfillmode{no}%
    \Q@arrowheadDD[#1,#2]\setc@ntr@l{2}\figdrawline[#1,-3]%
    \PSc@mment{End arrowDD}\resetc@ntr@l\et@tpsarrow\fi\fi}}
\ctr@ld@f\def\Q@arrowTD[#1,#2]{{\ifCUR@PS\ifGR@cri\s@uvc@ntr@l\et@tpsarrowTD%
    \PSc@mment{arrowTD [Pt1,Pt2]=[#1,#2]}\resetc@ntr@l{2}%
    \Figptpr@j-5:/#1/\Figptpr@j-6:/#2/\let\c@lprojSP=\relax\Q@arrowDD[-5,-6]%
    \PSc@mment{End arrowTD}\resetc@ntr@l\et@tpsarrowTD\fi\fi}}
\ctr@ln@m\figdrawarrowhead
\ctr@ld@f\def\Q@arrowheadDD[#1,#2]{{\ifCUR@PS\ifGR@cri\s@uvc@ntr@l\et@tpsarrowheadDD%
    \if@rrowhfill\def\@hangle{-\@rrowheadangle}\else\def\@hangle{\@rrowheadangle}\fi%
    \if@rrowratio%
    \if@rrowhout\def\@hratio{-\@rrowheadratio}\else\def\@hratio{\@rrowheadratio}\fi%
    \PSc@mment{arrowheadDD Ratio=\@hratio, Angle=\@hangle, [Pt1,Pt2]=[#1,#2]}%
    \Ps@rrowhead\@hratio,\@hangle[#1,#2]%
    \else%
    \if@rrowhout\def\@hlength{-\@rrowheadlength}\else\def\@hlength{\@rrowheadlength}\fi%
    \PSc@mment{arrowheadDD Length=\@hlength, Angle=\@hangle, [Pt1,Pt2]=[#1,#2]}%
    \Ps@rrowheadfd\@hlength,\@hangle[#1,#2]%
    \fi%
    \PSc@mment{End arrowheadDD}\resetc@ntr@l\et@tpsarrowheadDD\fi\fi}}
\ctr@ld@f\def\Q@arrowheadTD[#1,#2]{{\ifCUR@PS\ifGR@cri\s@uvc@ntr@l\et@tpsarrowheadTD%
    \PSc@mment{arrowheadTD [Pt1,Pt2]=[#1,#2]}\resetc@ntr@l{2}%
    \Figptpr@j-5:/#1/\Figptpr@j-6:/#2/\let\c@lprojSP=\relax\Q@arrowheadDD[-5,-6]%
    \PSc@mment{End arrowheadTD}\resetc@ntr@l\et@tpsarrowheadTD\fi\fi}}
\ctr@ld@f\def\Ps@rrowhead#1,#2[#3,#4]{\v@leur=#1\p@\maxim@m{\v@leur}{\v@leur}{-\v@leur}%
    \ifdim\v@leur>\Cepsil@n{
    \PSc@mment{@rrowhead Ratio=#1, Angle=#2, [Pt1,Pt2]=[#3,#4]}\v@leur=\UNSS@N%
    \v@leur=\CUR@width\v@leur\v@leur=\ptpsT@pt\v@leur\delt@=.5\v@leur
    \setc@ntr@l{2}\figvectPDD-3[#4,#3]%
    \Figg@tXY{-3}\v@lX=#1\v@lX\v@lY=#1\v@lY\Figv@ctCreg-3(\v@lX,\v@lY)%
    \vecunit@{-4}{-3}\mili@u=\result@t%
    \ifdim#2pt>\z@\v@lXa=-\C@AHANG\delt@%
     \edef\c@ef{\repdecn@mb{\v@lXa}}\figpttraDD-3:=-3/\c@ef,-4/\fi%
    \edef\c@ef{\repdecn@mb{\delt@}}%
    \v@lXa=\mili@u\v@lXa=\C@AHANG\v@lXa%
    \v@lYa=\ptpsT@pt\p@\v@lYa=\CUR@width\v@lYa\v@lYa=\sDcc@ngle\v@lYa%
    \advance\v@lXa-\v@lYa\gdef\sDcc@ngle{0}%
    \ifdim\v@lXa>\v@leur\edef\c@efendpt{\repdecn@mb{\v@leur}}%
    \else\edef\c@efendpt{\repdecn@mb{\v@lXa}}\fi%
    \Figg@tXY{-3}\v@lmin=\v@lX\v@lmax=\v@lY%
    \v@lXa=\C@AHANG\v@lmin\v@lYa=\S@AHANG\v@lmax\advance\v@lXa\v@lYa%
    \v@lYa=-\S@AHANG\v@lmin\v@lX=\C@AHANG\v@lmax\advance\v@lYa\v@lX%
    \setc@ntr@l{1}\Figg@tXY{#4}\advance\v@lX\v@lXa\advance\v@lY\v@lYa%
    \setc@ntr@l{2}\Figp@intregDD-2:(\v@lX,\v@lY)%
    \v@lXa=\C@AHANG\v@lmin\v@lYa=-\S@AHANG\v@lmax\advance\v@lXa\v@lYa%
    \v@lYa=\S@AHANG\v@lmin\v@lX=\C@AHANG\v@lmax\advance\v@lYa\v@lX%
    \setc@ntr@l{1}\Figg@tXY{#4}\advance\v@lX\v@lXa\advance\v@lY\v@lYa%
    \setc@ntr@l{2}\Figp@intregDD-1:(\v@lX,\v@lY)%
    \ifdim#2pt<\z@\fillm@detrue\figdrawline[-2,#4,-1]
    \else\figptstraDD-3=#4,-2,-1/\c@ef,-4/\s@uvdash{\typ@dash}\Q@s@tdash{\D@FTdash}%
    \figdrawline[-2,-3,-1]\Q@s@tdash{\typ@dash}\fi
    \ifdim#1pt>\z@\figpttraDD-3:=#4/\c@efendpt,-4/\else\figptcopyDD-3:/#4/\fi%
    \PSc@mment{End @rrowhead}}\fi}
\ctr@ld@f\def\sDcc@ngle{0}
\ctr@ld@f\def\Ps@rrowheadfd#1,#2[#3,#4]{{%
    \PSc@mment{@rrowheadfd Length=#1, Angle=#2, [Pt1,Pt2]=[#3,#4]}%
    \setc@ntr@l{2}\figvectPDD-1[#3,#4]\n@rmeucDD{\v@leur}{-1}\v@leur=\ptT@unit@\v@leur%
    \invers@{\v@leur}{\v@leur}\v@leur=#1\v@leur\edef\R@tio{\repdecn@mb{\v@leur}}%
    \Ps@rrowhead\R@tio,#2[#3,#4]\PSc@mment{End @rrowheadfd}}}
\ctr@ln@m\figdrawarrowBezier
\ctr@ld@f\def\Q@arrowBezierDD[#1,#2,#3,#4]{{\ifCUR@PS\ifGR@cri\s@uvc@ntr@l\et@tpsarrowBezierDD%
    \PSc@mment{arrowBezierDD Control points=#1,#2,#3,#4}\setc@ntr@l{2}%
    \if@rrowratio\c@larclengthDD\v@leur,10[#1,#2,#3,#4]\else\v@leur=\z@\fi%
    \Ps@rrowB@zDD\v@leur[#1,#2,#3,#4]%
    \PSc@mment{End arrowBezierDD}\resetc@ntr@l\et@tpsarrowBezierDD\fi\fi}}
\ctr@ld@f\def\Q@arrowBezierTD[#1,#2,#3,#4]{{\ifCUR@PS\ifGR@cri\s@uvc@ntr@l\et@tpsarrowBezierTD%
    \PSc@mment{arrowBezierTD Control points=#1,#2,#3,#4}\resetc@ntr@l{2}%
    \Figptpr@j-7:/#1/\Figptpr@j-8:/#2/\Figptpr@j-9:/#3/\Figptpr@j-10:/#4/%
    \let\c@lprojSP=\relax\ifnum\CUR@proj<\tw@\Q@arrowBezierDD[-7,-8,-9,-10]%
    \else\f@gnewpath\PSwrit@cmd{-7}{\c@mmoveto}{\fwf@g}%
    \if@rrowratio\c@larclengthDD\mili@u,10[-7,-8,-9,-10]\else\mili@u=\z@\fi%
    \p@rtent=\NBz@rcs\advance\p@rtent\m@ne\subB@zierTD\p@rtent[#1,#2,#3,#4]%
    \f@gstroke%
    \advance\v@lmin\p@rtent\delt@
    \v@leur=\v@lmin\advance\v@leur0.33333 \delt@\edef\unti@rs{\repdecn@mb{\v@leur}}%
    \v@leur=\v@lmin\advance\v@leur0.66666 \delt@\edef\deti@rs{\repdecn@mb{\v@leur}}%
    \figptcopyDD-8:/-10/\c@lsubBzarc\unti@rs,\deti@rs[#1,#2,#3,#4]%
    \figptcopyDD-8:/-4/\figptcopyDD-9:/-3/\Ps@rrowB@zDD\mili@u[-7,-8,-9,-10]\fi%
    \PSc@mment{End arrowBezierTD}\resetc@ntr@l\et@tpsarrowBezierTD\fi\fi}}
\ctr@ld@f\def\c@larclengthDD#1,#2[#3,#4,#5,#6]{{\p@rtent=#2\figptcopyDD-5:/#3/%
    \delt@=\p@\divide\delt@\p@rtent\c@rre=\z@\v@leur=\z@\s@mme=\z@%
    \loop\ifnum\s@mme<\p@rtent\advance\s@mme\@ne\advance\v@leur\delt@%
    \edef\T@{\repdecn@mb{\v@leur}}\figptBezierDD-6::\T@[#3,#4,#5,#6]%
    \figvectPDD-1[-5,-6]\n@rmeucDD{\mili@u}{-1}\advance\c@rre\mili@u%
    \figptcopyDD-5:/-6/\repeat\global\result@t=\ptT@unit@\c@rre}#1=\result@t}
\ctr@ld@f\def\Ps@rrowB@zDD#1[#2,#3,#4,#5]{{\Q@s@tfillmode{no}%
    \if@rrowratio\delt@=\@rrowheadratio#1\else\delt@=\@rrowheadlength pt\fi%
    \v@leur=\C@AHANG\delt@\edef\R@dius{\repdecn@mb{\v@leur}}%
    \FigptintercircB@zDD-5::0,\R@dius[#5,#4,#3,#2]%
    \Q@s@tarrowheadlength{\repdecn@mb{\delt@}}\Q@arrowheadDD[-5,#5]%
    \let\n@rmeuc=\n@rmeucDD\figgetdist\R@dius[#5,-3]%
    \FigptintercircB@zDD-6::0,\R@dius[#5,#4,#3,#2]%
    \figptBezierDD-5::0.33333[#5,#4,#3,#2]\figptBezierDD-3::0.66666[#5,#4,#3,#2]%
    \figptscontrolDD-5[-6,-5,-3,#2]\Q@BezierDD1[-6,-5,-4,#2]}}
\ctr@ln@m\figdrawarrowcirc
\ctr@ld@f\def\Q@arrowcircDD#1;#2(#3,#4){{\ifCUR@PS\ifGR@cri\s@uvc@ntr@l\et@tpsarrowcircDD%
    \PSc@mment{arrowcircDD Center=#1 ; Radius=#2 (Ang1=#3,Ang2=#4)}%
    \Q@s@tfillmode{no}\Pscirc@rrowhead#1;#2(#3,#4)%
    \setc@ntr@l{2}\figvectPDD -4[#1,-3]\vecunit@{-4}{-4}%
    \Figg@tXY{-4}\arct@n\v@lmin(\v@lX,\v@lY)%
    \v@lmin=\rdT@deg\v@lmin\v@leur=#4pt\advance\v@leur-\v@lmin%
    \maxim@m{\v@leur}{\v@leur}{-\v@leur}%
    \ifdim\v@leur>\DemiPI@deg\relax\ifdim\v@lmin<#4pt\advance\v@lmin\DePI@deg%
    \else\advance\v@lmin-\DePI@deg\fi\fi\edef\ar@ngle{\repdecn@mb{\v@lmin}}%
    \ifdim#3pt<#4pt\figdrawarccirc#1;#2(#3,\ar@ngle)\else\figdrawarccirc#1;#2(\ar@ngle,#3)\fi%
    \PSc@mment{End arrowcircDD}\resetc@ntr@l\et@tpsarrowcircDD\fi\fi}}
\ctr@ld@f\def\Q@arrowcircTD#1,#2,#3;#4(#5,#6){{\ifCUR@PS\ifGR@cri\s@uvc@ntr@l\et@tpsarrowcircTD%
    \PSc@mment{arrowcircTD Center=#1,P1=#2,P2=#3 ; Radius=#4 (Ang1=#5, Ang2=#6)}%
    \resetc@ntr@l{2}\c@lExtAxes#1,#2,#3(#4)\let\c@lprojSP=\relax%
    \figvectPTD-11[#1,-4]\figvectPTD-12[#1,-5]\c@lNbarcs{#5}{#6}%
    \if@rrowratio\v@lmax=\degT@rd\v@lmax\edef\D@lpha{\repdecn@mb{\v@lmax}}\fi%
    \advance\p@rtent\m@ne\mili@u=\z@%
    \v@leur=#5pt\c@lptellP{#1}{-11}{-12}\Figptpr@j-9:/-3/%
    \f@gnewpath\PSwrit@cmdS{-9}{\c@mmoveto}{\fwf@g}{\X@un}{\Y@un}%
    \edef\C@nt@r{#1}\s@mme=\z@\bcl@rcircTD\f@gstroke%
    \advance\v@leur\delt@\c@lptellP{#1}{-11}{-12}\Figptpr@j-5:/-3/%
    \advance\v@leur\delt@\c@lptellP{#1}{-11}{-12}\Figptpr@j-6:/-3/%
    \advance\v@leur\delt@\c@lptellP{#1}{-11}{-12}\Figptpr@j-10:/-3/%
    \figptscontrolDD-8[-9,-5,-6,-10]%
    \if@rrowratio\c@lcurvradDD0.5[-9,-8,-7,-10]\advance\mili@u\result@t%
    \maxim@m{\mili@u}{\mili@u}{-\mili@u}\mili@u=\ptT@unit@\mili@u%
    \mili@u=\D@lpha\mili@u\advance\p@rtent\@ne\divide\mili@u\p@rtent\fi%
    \Ps@rrowB@zDD\mili@u[-9,-8,-7,-10]%
    \PSc@mment{End arrowcircTD}\resetc@ntr@l\et@tpsarrowcircTD\fi\fi}}
\ctr@ld@f\def\bcl@rcircTD{\relax%
    \ifnum\s@mme<\p@rtent\advance\s@mme\@ne%
    \advance\v@leur\delt@\c@lptellP{\C@nt@r}{-11}{-12}\Figptpr@j-5:/-3/%
    \advance\v@leur\delt@\c@lptellP{\C@nt@r}{-11}{-12}\Figptpr@j-6:/-3/%
    \advance\v@leur\delt@\c@lptellP{\C@nt@r}{-11}{-12}\Figptpr@j-10:/-3/%
    \figptscontrolDD-8[-9,-5,-6,-10]\BdingB@xfalse%
    \PSwrit@cmdS{-8}{}{\fwf@g}{\X@de}{\Y@de}\PSwrit@cmdS{-7}{}{\fwf@g}{\X@tr}{\Y@tr}%
    \BdingB@xtrue\PSwrit@cmdS{-10}{\c@mcurveto}{\fwf@g}{\X@qu}{\Y@qu}%
    \if@rrowratio\c@lcurvradDD0.5[-9,-8,-7,-10]\advance\mili@u\result@t\fi%
    \B@zierBB@x{1}{\Y@un}(\X@un,\X@de,\X@tr,\X@qu)%
    \B@zierBB@x{2}{\X@un}(\Y@un,\Y@de,\Y@tr,\Y@qu)%
    \edef\X@un{\X@qu}\edef\Y@un{\Y@qu}\figptcopyDD-9:/-10/\bcl@rcircTD\fi}
\ctr@ld@f\def\Pscirc@rrowhead#1;#2(#3,#4){{%
    \PSc@mment{circ@rrowhead Center=#1 ; Radius=#2 (Ang1=#3,Ang2=#4)}%
    \v@leur=#2\unit@\edef\s@glen{\repdecn@mb{\v@leur}}\v@lY=\z@\v@lX=\v@leur%
    \resetc@ntr@l{2}\Figv@ctCreg-3(\v@lX,\v@lY)\figpttraDD-5:=#1/1,-3/%
    \figptrotDD-5:=-5/#1,#4/%
    \figvectPDD-3[#1,-5]\Figg@tXY{-3}\v@leur=\v@lX%
    \ifdim#3pt<#4pt\v@lX=\v@lY\v@lY=-\v@leur\else\v@lX=-\v@lY\v@lY=\v@leur\fi%
    \Figv@ctCreg-3(\v@lX,\v@lY)\vecunit@{-3}{-3}%
    \if@rrowratio\v@leur=#4pt\advance\v@leur-#3pt\maxim@m{\mili@u}{-\v@leur}{\v@leur}%
    \mili@u=\degT@rd\mili@u\v@leur=\s@glen\mili@u\edef\s@glen{\repdecn@mb{\v@leur}}%
    \mili@u=#2\mili@u\mili@u=\@rrowheadratio\mili@u\else\mili@u=\@rrowheadlength pt\fi%
    \figpttraDD-6:=-5/\s@glen,-3/\v@leur=#2pt\v@leur=2\v@leur%
    \invers@{\v@leur}{\v@leur}\c@rre=\repdecn@mb{\v@leur}\mili@u
    \mili@u=\c@rre\mili@u=\repdecn@mb{\c@rre}\mili@u%
    \v@leur=\p@\advance\v@leur-\mili@u
    \invers@{\mili@u}{2\v@leur}\delt@=\c@rre\delt@=\repdecn@mb{\mili@u}\delt@%
    \xdef\sDcc@ngle{\repdecn@mb{\delt@}}
    \sqrt@{\mili@u}{\v@leur}\arct@n\v@leur(\mili@u,\c@rre)%
    \v@leur=\rdT@deg\v@leur
    \ifdim#3pt<#4pt\v@leur=-\v@leur\fi%
    \if@rrowhout\v@leur=-\v@leur\fi\edef\cor@ngle{\repdecn@mb{\v@leur}}%
    \figptrotDD-6:=-6/-5,\cor@ngle/\Q@arrowheadDD[-6,-5]%
    \PSc@mment{End circ@rrowhead}}}
\ctr@ln@m\figdrawarrowcircP
\ctr@ld@f\def\Q@arrowcircPDD#1;#2[#3,#4]{{\ifCUR@PS\ifGR@cri%
    \PSc@mment{arrowcircPDD Center=#1; Radius=#2, [P1=#3,P2=#4]}%
    \s@uvc@ntr@l\et@tpsarrowcircPDD\Ps@ngleparam#1;#2[#3,#4]%
    \ifdim\v@leur>\z@\ifdim\v@lmin>\v@lmax\advance\v@lmax\DePI@deg\fi%
    \else\ifdim\v@lmin<\v@lmax\advance\v@lmin\DePI@deg\fi\fi%
    \edef\@ngdeb{\repdecn@mb{\v@lmin}}\edef\@ngfin{\repdecn@mb{\v@lmax}}%
    \figdrawarrowcirc#1;\r@dius(\@ngdeb,\@ngfin)%
    \PSc@mment{End arrowcircPDD}\resetc@ntr@l\et@tpsarrowcircPDD\fi\fi}}
\ctr@ld@f\def\Q@arrowcircPTD#1;#2[#3,#4,#5]{{\ifCUR@PS\ifGR@cri\s@uvc@ntr@l\et@tpsarrowcircPTD%
    \PSc@mment{arrowcircPTD Center=#1; Radius=#2, [P1=#3,P2=#4,P3=#5]}%
    \figgetangleTD\@ngfin[#1,#3,#4,#5]\v@leur=#2pt%
    \maxim@m{\mili@u}{-\v@leur}{\v@leur}\edef\r@dius{\repdecn@mb{\mili@u}}%
    \ifdim\v@leur<\z@\v@lmax=\@ngfin pt\advance\v@lmax-\DePI@deg%
    \edef\@ngfin{\repdecn@mb{\v@lmax}}\fi\Q@arrowcircTD#1,#3,#5;\r@dius(0,\@ngfin)%
    \PSc@mment{End arrowcircPTD}\resetc@ntr@l\et@tpsarrowcircPTD\fi\fi}}
\ctr@ld@f\def\figdrawaxes#1(#2){{\ifCUR@PS\ifGR@cri\s@uvc@ntr@l\et@tpsaxes%
    \PSc@mment{axes Origin=#1 Range=(#2)}\an@lys@xes#2,:\resetc@ntr@l{2}%
    \ifx\t@xt@\empty\ifTr@isDim\Q@@xes#1(0,#2,0,#2,0,#2)\else\Q@@xes#1(0,#2,0,#2)\fi%
    \else\Q@@xes#1(#2)\fi\PSc@mment{End axes}\resetc@ntr@l\et@tpsaxes\fi\fi}}
\ctr@ld@f\def\an@lys@xes#1,#2:{\def\t@xt@{#2}}
\ctr@ln@m\Q@@xes
\ctr@ld@f\def\Q@@xesDD#1(#2,#3,#4,#5){%
    \figpttraC-5:=#1/#2,0/\figpttraC-6:=#1/#3,0/\Q@arrowDD[-5,-6]%
    \figpttraC-5:=#1/0,#4/\figpttraC-6:=#1/0,#5/\Q@arrowDD[-5,-6]}
\ctr@ld@f\def\Q@@xesTD#1(#2,#3,#4,#5,#6,#7){%
    \figpttraC-7:=#1/#2,0,0/\figpttraC-8:=#1/#3,0,0/\Q@arrowTD[-7,-8]%
    \figpttraC-7:=#1/0,#4,0/\figpttraC-8:=#1/0,#5,0/\Q@arrowTD[-7,-8]%
    \figpttraC-7:=#1/0,0,#6/\figpttraC-8:=#1/0,0,#7/\Q@arrowTD[-7,-8]}
\ctr@ln@m\newGr@FN
\ctr@ld@f\def\newGr@FNPDF#1{\s@mme=\Gr@FNb\advance\s@mme\@ne\xdef\Gr@FNb{\number\s@mme}}
\ctr@ld@f\def\newGr@FNDVI#1{\newGr@FNPDF{}\xdef#1{\jobname GI\Gr@FNb.anx}}
\ctr@ld@f\def\figdrawbegin#1{\newGr@FN\DefGIfilen@me\gdef\@utoFN{0}%
    \def\t@xt@{#1}\relax\ifx\t@xt@\empty\GRupdatem@detrue%
    \gdef\@utoFN{1}\Psb@ginfig\DefGIfilen@me\else\expandafter\Psb@ginfigNu@#1 :\fi}
\ctr@ld@f\def\Psb@ginfigNu@#1 #2:{\def\t@xt@{#1}\relax\ifx\t@xt@\empty\def\t@xt@{#2}%
    \ifx\t@xt@\empty\GRupdatem@detrue\gdef\@utoFN{1}\Psb@ginfig\DefGIfilen@me%
    \else\Psb@ginfigNu@#2:\fi\else\Psb@ginfig{#1}\fi}
\ctr@ln@m\PSfilen@me \ctr@ln@m\auxfilen@me
\ctr@ld@f\def\Psb@ginfig#1{\ifCUR@PS\else%
    \edef\PSfilen@me{#1}\edef\auxfilen@me{\jobname.anx}%
    \ifGRupdatem@de\GR@critrue\else\openin\frf@g=\PSfilen@me\relax%
    \ifeof\frf@g\GR@critrue\else\GR@crifalse\fi\closein\frf@g\fi%
    \CUR@PStrue\c@ldefproj\expandafter\setupd@te\D@FTupdate:%
    \ifGR@cri\initb@undb@x%
    \immediate\openout\fwf@g=\auxfilen@me\initpss@ttings\fi%
    \fi}
\ctr@ld@f\def\Gr@FNb{0}
\ctr@ld@f\def\figforTeXFileno{\Gr@FNb}
\ctr@ld@f\def\figforTeXFigno{0 }
\ctr@ld@f\def\figforTeXnextFigno{1 }
\ctr@ld@f\edef\DefGIfilen@me{\jobname GI.anx}
\ctr@ld@f\def\initpss@ttings{\figreset{altitude,arrowhead,curve,general,flowchart,mesh,trimesh}%
    \Use@llipsefalse}
\ctr@ld@f\def\B@zierBB@x#1#2(#3,#4,#5,#6){{\c@rre=\t@n\epsil@n
    \v@lmax=#4\advance\v@lmax-#5\v@lmax=\thr@@\v@lmax\advance\v@lmax#6\advance\v@lmax-#3%
    \mili@u=#4\mili@u=-\tw@\mili@u\advance\mili@u#3\advance\mili@u#5%
    \v@lmin=#4\advance\v@lmin-#3\maxim@m{\v@leur}{-\v@lmax}{\v@lmax}%
    \maxim@m{\delt@}{-\mili@u}{\mili@u}\maxim@m{\v@leur}{\v@leur}{\delt@}%
    \maxim@m{\delt@}{-\v@lmin}{\v@lmin}\maxim@m{\v@leur}{\v@leur}{\delt@}%
    \ifdim\v@leur>\c@rre\invers@{\v@leur}{\v@leur}\edef\Uns@rM@x{\repdecn@mb{\v@leur}}%
    \v@lmax=\Uns@rM@x\v@lmax\mili@u=\Uns@rM@x\mili@u\v@lmin=\Uns@rM@x\v@lmin%
    \maxim@m{\v@leur}{-\v@lmax}{\v@lmax}\ifdim\v@leur<\c@rre%
    \maxim@m{\v@leur}{-\mili@u}{\mili@u}\ifdim\v@leur<\c@rre\else%
    \invers@{\mili@u}{\mili@u}\v@leur=-0.5\v@lmin%
    \v@leur=\repdecn@mb{\mili@u}\v@leur\m@jBBB@x{\v@leur}{#1}{#2}(#3,#4,#5,#6)\fi%
    \else\delt@=\repdecn@mb{\mili@u}\mili@u\v@leur=\repdecn@mb{\v@lmax}\v@lmin%
    \advance\delt@-\v@leur\ifdim\delt@<\z@\else\invers@{\v@lmax}{\v@lmax}%
    \edef\Uns@rAp{\repdecn@mb{\v@lmax}}\sqrt@{\delt@}{\delt@}%
    \v@leur=-\mili@u\advance\v@leur\delt@\v@leur=\Uns@rAp\v@leur%
    \m@jBBB@x{\v@leur}{#1}{#2}(#3,#4,#5,#6)%
    \v@leur=-\mili@u\advance\v@leur-\delt@\v@leur=\Uns@rAp\v@leur%
    \m@jBBB@x{\v@leur}{#1}{#2}(#3,#4,#5,#6)\fi\fi\fi}}
\ctr@ld@f\def\m@jBBB@x#1#2#3(#4,#5,#6,#7){{\relax\ifdim#1>\z@\ifdim#1<\p@%
    \edef\T@{\repdecn@mb{#1}}\v@lX=\p@\advance\v@lX-#1\edef\UNmT@{\repdecn@mb{\v@lX}}%
    \v@lX=#4\v@lY=#5\v@lZ=#6\v@lXa=#7\v@lX=\UNmT@\v@lX\advance\v@lX\T@\v@lY%
    \v@lY=\UNmT@\v@lY\advance\v@lY\T@\v@lZ\v@lZ=\UNmT@\v@lZ\advance\v@lZ\T@\v@lXa%
    \v@lX=\UNmT@\v@lX\advance\v@lX\T@\v@lY\v@lY=\UNmT@\v@lY\advance\v@lY\T@\v@lZ%
    \v@lX=\UNmT@\v@lX\advance\v@lX\T@\v@lY%
    \ifcase#2\or\v@lY=#3\or\v@lY=\v@lX\v@lX=#3\fi\b@undb@x{\v@lX}{\v@lY}\fi\fi}}
\ctr@ld@f\def\PsB@zier#1[#2]{{\f@gnewpath%
    \s@mme=\z@\def\list@num{#2,0}\extrairelepremi@r\p@int\de\list@num%
    \PSwrit@cmdS{\p@int}{\c@mmoveto}{\fwf@g}{\X@un}{\Y@un}\p@rtent=#1\bclB@zier}}
\ctr@ld@f\def\bclB@zier{\relax%
    \ifnum\s@mme<\p@rtent\advance\s@mme\@ne\BdingB@xfalse%
    \extrairelepremi@r\p@int\de\list@num\PSwrit@cmdS{\p@int}{}{\fwf@g}{\X@de}{\Y@de}%
    \extrairelepremi@r\p@int\de\list@num\PSwrit@cmdS{\p@int}{}{\fwf@g}{\X@tr}{\Y@tr}%
    \BdingB@xtrue%
    \extrairelepremi@r\p@int\de\list@num\PSwrit@cmdS{\p@int}{\c@mcurveto}{\fwf@g}{\X@qu}{\Y@qu}%
    \B@zierBB@x{1}{\Y@un}(\X@un,\X@de,\X@tr,\X@qu)%
    \B@zierBB@x{2}{\X@un}(\Y@un,\Y@de,\Y@tr,\Y@qu)%
    \edef\X@un{\X@qu}\edef\Y@un{\Y@qu}\bclB@zier\fi}
\ctr@ln@m\figdrawBezier
\ctr@ld@f\def\Q@BezierDD#1[#2]{\ifCUR@PS\ifGR@cri%
    \PSc@mment{BezierDD N arcs=#1, Control points=#2}%
    \iffillm@de\PsB@zier#1[#2]%
    \f@gfill%
    \else\PsB@zier#1[#2]\f@gstroke\fi%
    \PSc@mment{End BezierDD}\fi\fi}
\ctr@ln@m\et@tpsBezierTD
\ctr@ld@f\def\Q@BezierTD#1[#2]{\ifCUR@PS\ifGR@cri\s@uvc@ntr@l\et@tpsBezierTD%
    \PSc@mment{BezierTD N arcs=#1, Control points=#2}%
    \iffillm@de\PsB@zierTD#1[#2]%
    \f@gfill%
    \else\PsB@zierTD#1[#2]\f@gstroke\fi%
    \PSc@mment{End BezierTD}\resetc@ntr@l\et@tpsBezierTD\fi\fi}
\ctr@ld@f\def\PsB@zierTD#1[#2]{\ifnum\CUR@proj<\tw@\PsB@zier#1[#2]\else\PsB@zier@TD#1[#2]\fi}
\ctr@ld@f\def\PsB@zier@TD#1[#2]{{\f@gnewpath%
    \s@mme=\z@\def\list@num{#2,0}\extrairelepremi@r\p@int\de\list@num%
    \let\c@lprojSP=\relax\setc@ntr@l{2}\Figptpr@j-7:/\p@int/%
    \PSwrit@cmd{-7}{\c@mmoveto}{\fwf@g}%
    \loop\ifnum\s@mme<#1\advance\s@mme\@ne\extrairelepremi@r\p@intun\de\list@num%
    \extrairelepremi@r\p@intde\de\list@num\extrairelepremi@r\p@inttr\de\list@num%
    \subB@zierTD\NBz@rcs[\p@int,\p@intun,\p@intde,\p@inttr]\edef\p@int{\p@inttr}\repeat}}
\ctr@ld@f\def\subB@zierTD#1[#2,#3,#4,#5]{\delt@=\p@\divide\delt@\NBz@rcs\v@lmin=\z@%
    {\Figg@tXY{-7}\edef\X@un{\the\v@lX}\edef\Y@un{\the\v@lY}%
    \s@mme=\z@\loop\ifnum\s@mme<#1\advance\s@mme\@ne%
    \v@leur=\v@lmin\advance\v@leur0.33333 \delt@\edef\unti@rs{\repdecn@mb{\v@leur}}%
    \v@leur=\v@lmin\advance\v@leur0.66666 \delt@\edef\deti@rs{\repdecn@mb{\v@leur}}%
    \advance\v@lmin\delt@\edef\trti@rs{\repdecn@mb{\v@lmin}}%
    \figptBezierTD-8::\trti@rs[#2,#3,#4,#5]\Figptpr@j-8:/-8/%
    \c@lsubBzarc\unti@rs,\deti@rs[#2,#3,#4,#5]\BdingB@xfalse%
    \PSwrit@cmdS{-4}{}{\fwf@g}{\X@de}{\Y@de}\PSwrit@cmdS{-3}{}{\fwf@g}{\X@tr}{\Y@tr}%
    \BdingB@xtrue\PSwrit@cmdS{-8}{\c@mcurveto}{\fwf@g}{\X@qu}{\Y@qu}%
    \B@zierBB@x{1}{\Y@un}(\X@un,\X@de,\X@tr,\X@qu)%
    \B@zierBB@x{2}{\X@un}(\Y@un,\Y@de,\Y@tr,\Y@qu)%
    \edef\X@un{\X@qu}\edef\Y@un{\Y@qu}\figptcopyDD-7:/-8/\repeat}}
\ctr@ld@f\def\NBz@rcs{2}
\ctr@ld@f\def\c@lsubBzarc#1,#2[#3,#4,#5,#6]{\figptBezierTD-5::#1[#3,#4,#5,#6]%
    \figptBezierTD-6::#2[#3,#4,#5,#6]\Figptpr@j-4:/-5/\Figptpr@j-5:/-6/%
    \figptscontrolDD-4[-7,-4,-5,-8]}
\ctr@ln@m\figdrawcirc
\ctr@ld@f\def\Q@circDD#1(#2){\ifCUR@PS\ifGR@cri\PSc@mment{circDD Center=#1 (Radius=#2)}%
    \Q@arccircDD#1;#2(0,360)\PSc@mment{End circDD}\fi\fi}
\ctr@ld@f\def\Q@circTD#1,#2,#3(#4){\ifCUR@PS\ifGR@cri%
    \PSc@mment{circTD Center=#1,P1=#2,P2=#3 (Radius=#4)}%
    \Q@arccircTD#1,#2,#3;#4(0,360)\PSc@mment{End circTD}\fi\fi}
\ctr@ln@m\p@urcent
{\catcode`\%=12\gdef\p@urcent{
\ctr@ld@f\def\PSc@mment#1{\ifGRdebugm@de\immediate\write\fwf@g{\p@urcent\space#1}\fi}
\ctr@ln@m\acc@louv \ctr@ln@m\acc@lfer
{\catcode`\[=1\catcode`\{=12\gdef\acc@louv[{}}
{\catcode`\]=2\catcode`\}=12\gdef\acc@lfer{}]]
\ctr@ld@f\def\PSdict@{\ifUse@llipse%
    \immediate\write\fwf@g{/ellipsedict 9 dict def ellipsedict /mtrx matrix put}%
    \immediate\write\fwf@g{/ellipse \acc@louv ellipsedict begin}%
    \immediate\write\fwf@g{ /endangle exch def /startangle exch def}%
    \immediate\write\fwf@g{ /yrad exch def /xrad exch def}%
    \immediate\write\fwf@g{ /rotangle exch def /y exch def /x exch def}%
    \immediate\write\fwf@g{ /savematrix mtrx currentmatrix def}%
    \immediate\write\fwf@g{ x y translate rotangle rotate xrad yrad scale}%
    \immediate\write\fwf@g{ 0 0 1 startangle endangle arc}%
    \immediate\write\fwf@g{ savematrix setmatrix end\acc@lfer def}%
    \fi\PShe@der{EndProlog}}
\ctr@ld@f\def\Pssetc@rve#1=#2|{\keln@mun#1|%
    \def\n@mref{r}\ifx\l@debut\n@mref\update@ttr\D@FTroundness\Q@s@troundness{#2}\else
    \W@rnmesAttr{figset curve}{#1}\fi}
\ctr@ln@m\curv@roundness
\ctr@ld@f\def\Q@s@troundness#1{\edef\curv@roundness{#1}}
\ctr@ld@f\def\D@FTroundness{0.2} 
\ctr@ln@m\figdrawcurve
\ctr@ld@f\def\Q@curveDD[#1]{{\ifCUR@PS\ifGR@cri\PSc@mment{curveDD Points=#1}%
    \s@uvc@ntr@l\et@tpscurveDD%
    \iffillm@de\Psc@rveDD\curv@roundness[#1]%
    \f@gfill%
    \else\Psc@rveDD\curv@roundness[#1]\f@gstroke\fi%
    \PSc@mment{End curveDD}\resetc@ntr@l\et@tpscurveDD\fi\fi}}
\ctr@ld@f\def\Q@curveTD[#1]{{\ifCUR@PS\ifGR@cri%
    \PSc@mment{curveTD Points=#1}\s@uvc@ntr@l\et@tpscurveTD\let\c@lprojSP=\relax%
    \iffillm@de\Psc@rveTD\curv@roundness[#1]%
    \f@gfill%
    \else\Psc@rveTD\curv@roundness[#1]\f@gstroke\fi%
    \PSc@mment{End curveTD}\resetc@ntr@l\et@tpscurveTD\fi\fi}}
\ctr@ld@f\def\Psc@rveDD#1[#2]{%
    \def\list@num{#2}\extrairelepremi@r\Ak@\de\list@num%
    \extrairelepremi@r\Ai@\de\list@num\extrairelepremi@r\Aj@\de\list@num%
    \f@gnewpath\PSwrit@cmdS{\Ai@}{\c@mmoveto}{\fwf@g}{\X@un}{\Y@un}%
    \setc@ntr@l{2}\figvectPDD -1[\Ak@,\Aj@]%
    \@ecfor\Ak@:=\list@num\do{\figpttraDD-2:=\Ai@/#1,-1/\BdingB@xfalse%
       \PSwrit@cmdS{-2}{}{\fwf@g}{\X@de}{\Y@de}%
       \figvectPDD -1[\Ai@,\Ak@]\figpttraDD-2:=\Aj@/-#1,-1/%
       \PSwrit@cmdS{-2}{}{\fwf@g}{\X@tr}{\Y@tr}\BdingB@xtrue%
       \PSwrit@cmdS{\Aj@}{\c@mcurveto}{\fwf@g}{\X@qu}{\Y@qu}%
       \B@zierBB@x{1}{\Y@un}(\X@un,\X@de,\X@tr,\X@qu)%
       \B@zierBB@x{2}{\X@un}(\Y@un,\Y@de,\Y@tr,\Y@qu)%
       \edef\X@un{\X@qu}\edef\Y@un{\Y@qu}\edef\Ai@{\Aj@}\edef\Aj@{\Ak@}}}
\ctr@ld@f\def\Psc@rveTD#1[#2]{\ifnum\CUR@proj<\tw@\Psc@rvePPTD#1[#2]\else\Psc@rveCPTD#1[#2]\fi}
\ctr@ld@f\def\Psc@rvePPTD#1[#2]{\setc@ntr@l{2}%
    \def\list@num{#2}\extrairelepremi@r\Ak@\de\list@num\Figptpr@j-5:/\Ak@/%
    \extrairelepremi@r\Ai@\de\list@num\Figptpr@j-3:/\Ai@/%
    \extrairelepremi@r\Aj@\de\list@num\Figptpr@j-4:/\Aj@/%
    \f@gnewpath\PSwrit@cmdS{-3}{\c@mmoveto}{\fwf@g}{\X@un}{\Y@un}%
    \figvectPDD -1[-5,-4]%
    \@ecfor\Ak@:=\list@num\do{\Figptpr@j-5:/\Ak@/\figpttraDD-2:=-3/#1,-1/%
       \BdingB@xfalse\PSwrit@cmdS{-2}{}{\fwf@g}{\X@de}{\Y@de}%
       \figvectPDD -1[-3,-5]\figpttraDD-2:=-4/-#1,-1/%
       \PSwrit@cmdS{-2}{}{\fwf@g}{\X@tr}{\Y@tr}\BdingB@xtrue%
       \PSwrit@cmdS{-4}{\c@mcurveto}{\fwf@g}{\X@qu}{\Y@qu}%
       \B@zierBB@x{1}{\Y@un}(\X@un,\X@de,\X@tr,\X@qu)%
       \B@zierBB@x{2}{\X@un}(\Y@un,\Y@de,\Y@tr,\Y@qu)%
       \edef\X@un{\X@qu}\edef\Y@un{\Y@qu}\figptcopyDD-3:/-4/\figptcopyDD-4:/-5/}}
\ctr@ld@f\def\Psc@rveCPTD#1[#2]{\setc@ntr@l{2}%
    \def\list@num{#2}\extrairelepremi@r\Ak@\de\list@num%
    \extrairelepremi@r\Ai@\de\list@num\extrairelepremi@r\Aj@\de\list@num%
    \Figptpr@j-7:/\Ai@/%
    \f@gnewpath\PSwrit@cmd{-7}{\c@mmoveto}{\fwf@g}%
    \figvectPTD -9[\Ak@,\Aj@]%
    \@ecfor\Ak@:=\list@num\do{\figpttraTD-10:=\Ai@/#1,-9/%
       \figvectPTD -9[\Ai@,\Ak@]\figpttraTD-11:=\Aj@/-#1,-9/%
       \subB@zierTD\NBz@rcs[\Ai@,-10,-11,\Aj@]\edef\Ai@{\Aj@}\edef\Aj@{\Ak@}}}
\ctr@ld@f\def\figdrawend{\ifCUR@PS\ifGR@cri\immediate\closeout\fwf@g%
    \immediate\openout\fwf@g=\PSfilen@me\relax%
    \ifPDFm@ke\PSBdingB@x\else%
    \immediate\write\fwf@g{\p@urcent\string!PS-Adobe-2.0 EPSF-2.0}%
    \PShe@der{Creator\string: TeX (fig4tex.tex)}%
    \PShe@der{Title\string: \PSfilen@me}%
    \PShe@der{CreationDate\string: \the\day/\the\month/\the\year}%
    \PSBdingB@x%
    \PShe@der{EndComments}\PSdict@\fi%
    \immediate\write\fwf@g{\c@mgsave}%
    \openin\frf@g=\auxfilen@me\c@pypsfile\fwf@g\frf@g\closein\frf@g%
    \immediate\write\fwf@g{\c@mgrestore}%
    \PSc@mment{End of file.}\immediate\closeout\fwf@g%
    \immediate\openout\fwf@g=\auxfilen@me\immediate\closeout\fwf@g%
    \immediate\write16{File \PSfilen@me\space created.}\fi\fi\CUR@PSfalse\GR@critrue}
\ctr@ld@f\def\PShe@der#1{\immediate\write\fwf@g{\p@urcent\p@urcent#1}}
\ctr@ld@f\def\PSBdingB@x{{\v@lX=\ptT@ptps\c@@rdXmin\v@lY=\ptT@ptps\c@@rdYmin%
     \v@lXa=\ptT@ptps\c@@rdXmax\v@lYa=\ptT@ptps\c@@rdYmax%
     \PShe@der{BoundingBox\string: \repdecn@mb{\v@lX}\space\repdecn@mb{\v@lY}%
     \space\repdecn@mb{\v@lXa}\space\repdecn@mb{\v@lYa}}}}
\ctr@ld@f\def\figdrawfcconnect[#1]{{\ifCUR@PS\ifGR@cri\PSc@mment{fcconnect Points=#1}%
    \Q@s@tfillmode{no}\s@uvc@ntr@l\et@tpsfcconnect\resetc@ntr@l{2}%
    \fcc@nnect@[#1]\resetc@ntr@l\et@tpsfcconnect\PSc@mment{End fcconnect}\fi\fi}}
\ctr@ld@f\def\fcc@nnect@[#1]{\let\N@rm=\n@rmeucDD\def\list@num{#1}%
    \extrairelepremi@r\Ai@\de\list@num\edef\pr@m{\Ai@}\v@leur=\z@\p@rtent=\@ne\c@llgtot%
    \ifcase\fclin@typ@\edef\list@num{[\pr@m,#1,\Ai@}\expandafter\figdrawcurve\list@num]%
    \else\ifdim\fclin@r@d\p@>\z@\Pslin@conge[#1]\else\figdrawline[#1]\fi\fi%
    \v@leur=\@rrowp@s\v@leur\edef\list@num{#1,\Ai@,0}%
    \extrairelepremi@r\Ai@\de\list@num\mili@u=\epsil@n\c@llgpart%
    \advance\mili@u-\epsil@n\advance\mili@u-\delt@\advance\v@leur-\mili@u%
    \ifcase\fclin@typ@\invers@\mili@u\delt@%
    \ifnum\@rrowr@fpt>\z@\advance\delt@-\v@leur\v@leur=\delt@\fi%
    \v@leur=\repdecn@mb\v@leur\mili@u\edef\v@lt{\repdecn@mb\v@leur}%
    \extrairelepremi@r\Ak@\de\list@num%
    \figvectPDD-1[\pr@m,\Aj@]\figpttraDD-6:=\Ai@/\curv@roundness,-1/%
    \figvectPDD-1[\Ak@,\Ai@]\figpttraDD-7:=\Aj@/\curv@roundness,-1/%
    \delt@=\@rrowheadlength\p@\delt@=\C@AHANG\delt@\edef\R@dius{\repdecn@mb{\delt@}}%
    \ifcase\@rrowr@fpt%
    \FigptintercircB@zDD-8::\v@lt,\R@dius[\Ai@,-6,-7,\Aj@]\Q@arrowheadDD[-5,-8]\else%
    \FigptintercircB@zDD-8::\v@lt,\R@dius[\Aj@,-7,-6,\Ai@]\Q@arrowheadDD[-8,-5]\fi%
    \else\advance\delt@-\v@leur%
    \p@rtentiere{\p@rtent}{\delt@}\edef\C@efun{\the\p@rtent}%
    \p@rtentiere{\p@rtent}{\v@leur}\edef\C@efde{\the\p@rtent}%
    \figptbaryDD-5:[\Ai@,\Aj@;\C@efun,\C@efde]\ifcase\@rrowr@fpt%
    \delt@=\@rrowheadlength\unit@\delt@=\C@AHANG\delt@\edef\t@ille{\repdecn@mb{\delt@}}%
    \figvectPDD-2[\Ai@,\Aj@]\vecunit@{-2}{-2}\figpttraDD-5:=-5/\t@ille,-2/\fi%
    \Q@arrowheadDD[\Ai@,-5]\fi}
\ctr@ld@f\def\c@llgtot{\@ecfor\Aj@:=\list@num\do{\figvectP-1[\Ai@,\Aj@]\N@rm\delt@{-1}%
    \advance\v@leur\delt@\advance\p@rtent\@ne\edef\Ai@{\Aj@}}}
\ctr@ld@f\def\c@llgpart{\extrairelepremi@r\Aj@\de\list@num\figvectP-1[\Ai@,\Aj@]\N@rm\delt@{-1}%
    \advance\mili@u\delt@\ifdim\mili@u<\v@leur\edef\pr@m{\Ai@}\edef\Ai@{\Aj@}\c@llgpart\fi}
\ctr@ld@f\def\Pslin@conge[#1]{\ifnum\p@rtent>\tw@{\def\list@num{#1}%
    \extrairelepremi@r\Ai@\de\list@num\extrairelepremi@r\Aj@\de\list@num%
    \figptcopy-6:/\Ai@/\figvectP-3[\Ai@,\Aj@]\vecunit@{-3}{-3}\v@lmax=\result@t%
    \@ecfor\Ak@:=\list@num\do{\figvectP-4[\Aj@,\Ak@]\vecunit@{-4}{-4}%
    \minim@m\v@lmin\v@lmax\result@t\v@lmax=\result@t%
    \det@rm\delt@[-3,-4]\maxim@m\mili@u{\delt@}{-\delt@}\ifdim\mili@u>\Cepsil@n%
    \ifdim\delt@>\z@\figgetangleDD\Angl@[\Aj@,\Ak@,\Ai@]\else%
    \figgetangleDD\Angl@[\Aj@,\Ai@,\Ak@]\fi%
    \v@leur=\PI@deg\advance\v@leur-\Angl@\p@\divide\v@leur\tw@%
    \edef\Angl@{\repdecn@mb\v@leur}\c@ssin{\C@}{\S@}{\Angl@}\v@leur=\fclin@r@d\unit@%
    \v@leur=\S@\v@leur\mili@u=\C@\p@\invers@\mili@u\mili@u%
    \v@leur=\repdecn@mb{\mili@u}\v@leur%
    \minim@m\v@leur\v@leur\v@lmin\edef\t@ille{\repdecn@mb{\v@leur}}%
    \figpttra-5:=\Aj@/-\t@ille,-3/\figdrawline[-6,-5]\figpttra-6:=\Aj@/\t@ille,-4/%
    \figvectNVDD-3[-3]\figvectNVDD-8[-4]\inters@cDD-7:[-5,-3;-6,-8]%
    \ifdim\delt@>\z@\figdrawarccircP-7;\fclin@r@d[-5,-6]\else\figdrawarccircP-7;\fclin@r@d[-6,-5]\fi%
    \else\figdrawline[-6,\Aj@]\figptcopy-6:/\Aj@/\fi
    \edef\Ai@{\Aj@}\edef\Aj@{\Ak@}\figptcopy-3:/-4/}\figdrawline[-6,\Aj@]}\else\figdrawline[#1]\fi}
\ctr@ld@f\def\figdrawfcnode[#1]#2{{\ifCUR@PS\ifGR@cri\PSc@mment{fcnode Points=#1}%
    \s@uvc@ntr@l\et@tpsfcnode\resetc@ntr@l{2}%
    \def\t@xt@{#2}\ifx\t@xt@\empty\def\g@tt@xt{\setbox\Gb@x=\hbox{\Figg@tT{\p@int}}}%
    \else\def\g@tt@xt{\setbox\Gb@x=\hbox{#2}}\fi%
    \v@lmin=\h@rdfcXp@dd\advance\v@lmin\Xp@dd\unit@\multiply\v@lmin\tw@%
    \v@lmax=\h@rdfcYp@dd\advance\v@lmax\Yp@dd\unit@\multiply\v@lmax\tw@%
    \Figv@ctCreg-8(\unit@,-\unit@)\def\list@num{#1}%
    \delt@=\CUR@width bp\divide\delt@\tw@%
    \fcn@de\PSc@mment{End fcnode}\resetc@ntr@l\et@tpsfcnode\fi\fi}}
\ctr@ld@f\def\d@butn@de{\g@tt@xt\v@lX=\wd\Gb@x%
    \v@lY=\ht\Gb@x\advance\v@lY\dp\Gb@x\advance\v@lX\v@lmin\advance\v@lY\v@lmax}
\ctr@ld@f\def\fcn@deE{%
    \@ecfor\p@int:=\list@num\do{\d@butn@de\v@lX=\unssqrttw@\v@lX\v@lY=\unssqrttw@\v@lY%
    \ifdim\thickn@ss\p@>\z@
    \v@lXa=\v@lX\advance\v@lXa\delt@\v@lXa=\ptT@unit@\v@lXa\edef\XR@d{\repdecn@mb\v@lXa}%
    \v@lYa=\v@lY\advance\v@lYa\delt@\v@lYa=\ptT@unit@\v@lYa\edef\YR@d{\repdecn@mb\v@lYa}%
    \arct@n\v@leur(\v@lXa,\v@lYa)\v@leur=\rdT@deg\v@leur\edef\@nglde{\repdecn@mb\v@leur}%
    {\c@lptellDD-2::\p@int;\XR@d,\YR@d(\@nglde)}
    \advance\v@leur-\PI@deg\edef\@nglun{\repdecn@mb\v@leur}%
    {\c@lptellDD-3::\p@int;\XR@d,\YR@d(\@nglun)}%
    \figptstra-6=-3,-2,\p@int/\thickn@ss,-8/\Q@s@tfillmode{yes}%
    \Pss@tspecifSt{color=\DDV@thickcolor}%
    \figdrawline[-2,-3,-6,-5]\figdrawarcell-4;\XR@d,\YR@d(\@nglun,\@nglde,0)%
    \Psrest@reSt{color=\DDV@thickcolor}\fi
    \v@lX=\ptT@unit@\v@lX\v@lY=\ptT@unit@\v@lY%
    \edef\XR@d{\repdecn@mb\v@lX}\edef\YR@d{\repdecn@mb\v@lY}%
    \Q@s@tfillmode{yes}\Pss@tspecifSt{color=\fcbgc@lor}%
    \figdrawarcell\p@int;\XR@d,\YR@d(0,360,0)%
    \Q@s@tfillmode{no}\Psrest@reSt{color=\fcbgc@lor}\figdrawarcell\p@int;\XR@d,\YR@d(0,360,0)}}
\ctr@ld@f\def\fcn@deL{\delt@=\ptT@unit@\delt@\edef\t@ille{\repdecn@mb\delt@}%
    \@ecfor\p@int:=\list@num\do{\Figg@tXYa{\p@int}\d@butn@de%
    \ifdim\v@lX>\v@lY\itis@Ktrue\else\itis@Kfalse\fi%
    \advance\v@lXa-\v@lX\Figp@intreg-1:(\v@lXa,\v@lYa)%
    \advance\v@lXa\v@lX\advance\v@lYa-\v@lY\Figp@intreg-2:(\v@lXa,\v@lYa)%
    \advance\v@lXa\v@lX\advance\v@lYa\v@lY\Figp@intreg-3:(\v@lXa,\v@lYa)%
    \advance\v@lXa-\v@lX\advance\v@lYa\v@lY\Figp@intreg-4:(\v@lXa,\v@lYa)%
    \ifdim\thickn@ss\p@>\z@
    \Figg@tXYa{\p@int}\Q@s@tfillmode{yes}\Pss@tspecifSt{color=\DDV@thickcolor}%
    \c@lpt@xt{-1}{-4}\c@lpt@xt@\v@lXa\v@lYa\v@lX\v@lY\c@rre\delt@%
    \Figp@intregDD-9:(\v@lZ,\v@lYa)\Figp@intregDD-11:(\v@lZa,\v@lYa)%
    \c@lpt@xt{-4}{-3}\c@lpt@xt@\v@lYa\v@lXa\v@lY\v@lX\delt@\c@rre%
    \Figp@intregDD-12:(\v@lXa,\v@lZ)\Figp@intregDD-10:(\v@lXa,\v@lZa)%
    \ifitis@K\figptstra-7=-9,-10,-11/\thickn@ss,-8/\figdrawline[-9,-11,-5,-6,-7]\else%
    \figptstra-7=-10,-11,-12/\thickn@ss,-8/\figdrawline[-10,-12,-5,-6,-7]\fi%
    \Psrest@reSt{color=\DDV@thickcolor}\fi
    \Q@s@tfillmode{yes}\Pss@tspecifSt{color=\fcbgc@lor}\figdrawline[-1,-2,-3,-4]%
    \Q@s@tfillmode{no}\Psrest@reSt{color=\fcbgc@lor}\figdrawline[-1,-2,-3,-4,-1]}}
\ctr@ld@f\def\c@lpt@xt#1#2{\figvectN-7[#1,#2]\vecunit@{-7}{-7}\figpttra-5:=#1/\t@ille,-7/%
    \figvectP-7[#1,#2]\Figg@tXY{-7}\c@rre=\v@lX\delt@=\v@lY\Figg@tXY{-5}}
\ctr@ld@f\def\c@lpt@xt@#1#2#3#4#5#6{\v@lZ=#6\invers@{\v@lZ}{\v@lZ}\v@leur=\repdecn@mb{#5}\v@lZ%
    \v@lZ=#2\advance\v@lZ-#4\mili@u=\repdecn@mb{\v@leur}\v@lZ%
    \v@lZ=#3\advance\v@lZ\mili@u\v@lZa=-\v@lZ\advance\v@lZa\tw@#1}
\ctr@ld@f\def\fcn@deR{\@ecfor\p@int:=\list@num\do{\Figg@tXYa{\p@int}\d@butn@de%
    \advance\v@lXa-0.5\v@lX\advance\v@lYa-0.5\v@lY\Figp@intreg-1:(\v@lXa,\v@lYa)%
    \advance\v@lXa\v@lX\Figp@intreg-2:(\v@lXa,\v@lYa)%
    \advance\v@lYa\v@lY\Figp@intreg-3:(\v@lXa,\v@lYa)%
    \advance\v@lXa-\v@lX\Figp@intreg-4:(\v@lXa,\v@lYa)%
    \ifdim\thickn@ss\p@>\z@
    \Q@s@tfillmode{yes}\Pss@tspecifSt{color=\DDV@thickcolor}%
    \Figv@ctCreg-5(-\delt@,-\delt@)\figpttra-9:=-1/1,-5/%
    \Figv@ctCreg-5(\delt@,-\delt@)\figpttra-10:=-2/1,-5/%
    \Figv@ctCreg-5(\delt@,\delt@)\figpttra-11:=-3/1,-5/%
    \figptstra-7=-9,-10,-11/\thickn@ss,-8/\figdrawline[-9,-11,-5,-6,-7]%
    \Psrest@reSt{color=\DDV@thickcolor}\fi
    \Q@s@tfillmode{yes}\Pss@tspecifSt{color=\fcbgc@lor}\figdrawline[-1,-2,-3,-4]%
    \Q@s@tfillmode{no}\Psrest@reSt{color=\fcbgc@lor}\figdrawline[-1,-2,-3,-4,-1]}}
\ctr@ld@f\def\Pssetfl@wchart#1=#2|{\keln@mtr#1|%
    \def\n@mref{arr}\ifx\l@debut\n@mref\expandafter\keln@mtr\l@suite|%
     \def\n@mref{owp}\ifx\l@debut\n@mref\update@ttr\D@FTfcarrowposition\P@setfcarrowposition{#2}\else
     \def\n@mref{owr}\ifx\l@debut\n@mref\update@ttr\D@FTfcarrowrefpt\P@setfcarrowrefpt{#2}\else
     \W@rnmesAttr{figset flowchart}{#1}\fi\fi\else%
    \def\n@mref{bgc}\ifx\l@debut\n@mref\update@ttr\D@FTfcbgcolor\P@setfcbgcolor{#2}\else
    \def\n@mref{lin}\ifx\l@debut\n@mref\update@ttr\D@FTfcline\P@setfcline{#2}\else
    \def\n@mref{pad}\ifx\l@debut\n@mref\update@ttr\D@FTfcxpadding\P@setfcxpadding{#2}%
                                       \update@ttr\D@FTfcypadding\P@setfcypadding{#2}\else
    \def\n@mref{rad}\ifx\l@debut\n@mref\update@ttr\D@FTfcradius\P@setfcradius{#2}\else
    \def\n@mref{sha}\ifx\l@debut\n@mref\update@ttr\D@FTfcshape\P@setfcshape{#2}\else
    \def\n@mref{thi}\ifx\l@debut\n@mref\expandafter\keln@mtr\l@suite|%
     \def\n@mref{ckc}\ifx\l@debut\n@mref\update@ttr\D@FTref\P@setfcthickcolor{#2}\else
     \def\n@mref{ckn}\ifx\l@debut\n@mref\update@ttr\D@FTfcthickness\P@setfcthickness{#2}\else
     \W@rnmesAttr{figset flowchart}{#1}\fi\fi\else%
    \def\n@mref{xpa}\ifx\l@debut\n@mref\update@ttr\D@FTfcxpadding\P@setfcxpadding{#2}\else
    \def\n@mref{ypa}\ifx\l@debut\n@mref\update@ttr\D@FTfcypadding\P@setfcypadding{#2}\else
    \W@rnmesAttr{figset flowchart}{#1}\fi\fi\fi\fi\fi\fi\fi\fi\fi}
\ctr@ln@m\@rrowp@s
\ctr@ld@f\def\P@setfcarrowposition#1{\edef\@rrowp@s{#1}}
\ctr@ln@m\@rrowr@fpt
\ctr@ld@f\def\P@setfcarrowrefpt#1{\setfcr@fpt#1|}
\ctr@ld@f\def\setfcr@fpt#1#2|{\if#1e\def\@rrowr@fpt{1}\else\def\@rrowr@fpt{0}\fi}
\ctr@ln@m\fcbgc@lor
\ctr@ld@f\def\P@setfcbgcolor#1{\edef\fcbgc@lor{#1}}
\ctr@ln@m\fclin@typ@
\ctr@ld@f\def\P@setfcline#1{\setfccurv@#1|}
\ctr@ld@f\def\setfccurv@#1#2|{\if#1c\def\fclin@typ@{0}\else\def\fclin@typ@{1}\fi}
\ctr@ln@m\fclin@r@d
\ctr@ld@f\def\P@setfcradius#1{\edef\fclin@r@d{#1}}
\ctr@ln@m\fcn@de \ctr@ln@m\fcsh@pe
\ctr@ln@m\h@rdfcXp@dd \ctr@ln@m\h@rdfcYp@dd
\ctr@ld@f\def\P@setfcshape#1{\setfcshap@#1|}
\ctr@ld@f\def\setfcshap@#1#2|{%
    \if#1e\let\fcn@de=\fcn@deE\def\h@rdfcXp@dd{4pt}\def\h@rdfcYp@dd{4pt}%
     \edef\fcsh@pe{ellipse}\else%
    \if#1l\let\fcn@de=\fcn@deL\def\h@rdfcXp@dd{4pt}\def\h@rdfcYp@dd{4pt}%
     \edef\fcsh@pe{lozenge}\else%
          \let\fcn@de=\fcn@deR\def\h@rdfcXp@dd{6pt}\def\h@rdfcYp@dd{6pt}%
     \edef\fcsh@pe{rectangle}\fi\fi}
\ctr@ln@m\DDV@thickcolor
\ctr@ld@f\def\P@setfcthickcolor#1{\edef\DDV@thickcolor{#1}}
\ctr@ln@m\thickn@ss
\ctr@ld@f\def\P@setfcthickness#1{\edef\thickn@ss{#1}}
\ctr@ln@m\Xp@dd
\ctr@ld@f\def\P@setfcxpadding#1{\edef\Xp@dd{#1}}
\ctr@ln@m\Yp@dd
\ctr@ld@f\def\P@setfcypadding#1{\edef\Yp@dd{#1}}
\ctr@ld@f\def\figdrawline[#1]{{\ifCUR@PS\ifGR@cri\PSc@mment{line Points=#1}%
    \let\figdrawlign@=\Pslign@P\Pslin@{#1}\PSc@mment{End line}\fi\fi}}
\ctr@ld@f\def\figdrawlineF#1{{\ifCUR@PS\ifGR@cri\PSc@mment{lineF Filename=#1}%
    \let\figdrawlign@=\Pslign@F\Pslin@{#1}\PSc@mment{End lineF}\fi\fi}}
\ctr@ld@f\def\figdrawlineC(#1){{\ifCUR@PS\ifGR@cri\PSc@mment{lineC}%
    \let\figdrawlign@=\Pslign@C\Pslin@{#1}\PSc@mment{End lineC}\fi\fi}}
\ctr@ld@f\def\Pslin@#1{\iffillm@de\figdrawlign@{#1}%
    \f@gfill%
    \else\figdrawlign@{#1}\ifx\derp@int\premp@int%
    \f@gclosestroke%
    \else\f@gstroke\fi\fi}
\ctr@ld@f\def\Pslign@P#1{\def\list@num{#1}\extrairelepremi@r\p@int\de\list@num%
    \edef\premp@int{\p@int}\f@gnewpath%
    \PSwrit@cmd{\p@int}{\c@mmoveto}{\fwf@g}%
    \@ecfor\p@int:=\list@num\do{\PSwrit@cmd{\p@int}{\c@mlineto}{\fwf@g}%
    \edef\derp@int{\p@int}}}
\ctr@ld@f\def\Pslign@F#1{\s@uvc@ntr@l\et@tPslign@F\setc@ntr@l{2}\openin\frf@g=#1\relax%
    \ifeof\frf@g\message{*** File #1 not found !}\end\else%
    \read\frf@g to\tr@c\edef\premp@int{\tr@c}\expandafter\extr@ctCF\tr@c:%
    \f@gnewpath\PSwrit@cmd{-1}{\c@mmoveto}{\fwf@g}%
    \loop\read\frf@g to\tr@c\ifeof\frf@g\mored@tafalse\else\mored@tatrue\fi%
    \ifmored@ta\expandafter\extr@ctCF\tr@c:\PSwrit@cmd{-1}{\c@mlineto}{\fwf@g}%
    \edef\derp@int{\tr@c}\repeat\fi\closein\frf@g\resetc@ntr@l\et@tPslign@F}
\ctr@ln@m\extr@ctCF
\ctr@ld@f\def\extr@ctCFDD#1 #2:{\v@lX=#1\unit@\v@lY=#2\unit@\Figp@intregDD-1:(\v@lX,\v@lY)}
\ctr@ld@f\def\extr@ctCFTD#1 #2 #3:{\v@lX=#1\unit@\v@lY=#2\unit@\v@lZ=#3\unit@%
    \Figp@intregTD-1:(\v@lX,\v@lY,\v@lZ)}
\ctr@ld@f\def\Pslign@C#1{\s@uvc@ntr@l\et@tPslign@C\setc@ntr@l{2}%
    \def\list@num{#1}\extrairelepremi@r\p@int\de\list@num%
    \edef\premp@int{\p@int}\f@gnewpath%
    \expandafter\Pslign@C@\p@int:\PSwrit@cmd{-1}{\c@mmoveto}{\fwf@g}%
    \@ecfor\p@int:=\list@num\do{\expandafter\Pslign@C@\p@int:%
    \PSwrit@cmd{-1}{\c@mlineto}{\fwf@g}\edef\derp@int{\p@int}}%
    \resetc@ntr@l\et@tPslign@C}
\ctr@ld@f\def\Pslign@C@#1 #2:{{\def\t@xt@{#1}\ifx\t@xt@\empty\Pslign@C@#2:
    \else\extr@ctCF#1 #2:\fi}}
\ctr@ld@f\def\Pssetm@sh#1=#2|{\keln@mde#1|%
    \def\n@mref{co}\ifx\l@debut\n@mref\update@ttr\D@FTref\P@setmeshcolor{#2}\else
    \def\n@mref{da}\ifx\l@debut\n@mref\update@ttr\D@FTref\P@setmeshdash{#2}\else
    \def\n@mref{di}\ifx\l@debut\n@mref\update@ttr\D@FTmeshdiag\Q@s@tmeshdiag{#2}\else
    \def\n@mref{wi}\ifx\l@debut\n@mref\update@ttr\D@FTref\P@setmeshwidth{#2}\else
    \W@rnmesAttr{figset mesh}{#1}\fi\fi\fi\fi}
\ctr@ln@m\c@ntrolmesh
\ctr@ld@f\def\Q@s@tmeshdiag#1{\edef\c@ntrolmesh{#1}}
\ctr@ld@f\def\D@FTmeshdiag{0}    
\ctr@ln@m\DDV@meshcolor
\ctr@ld@f\def\P@setmeshcolor#1{\edef\DDV@meshcolor{#1}}
\ctr@ln@m\DDV@meshdash
\ctr@ld@f\def\P@setmeshdash#1{\edef\DDV@meshdash{#1}}
\ctr@ln@m\DDV@meshwidth
\ctr@ld@f\def\P@setmeshwidth#1{\edef\DDV@meshwidth{#1}}
\ctr@ld@f\def\figdrawmesh#1,#2[#3,#4,#5,#6]{{\ifCUR@PS\ifGR@cri%
    \PSc@mment{mesh N1=#1, N2=#2, Quadrangle=[#3,#4,#5,#6]}\s@uvc@ntr@l\et@tpsmesh%
    \Pss@tspecifSt{color=\DDV@meshcolor,dash=\DDV@meshdash,width=\DDV@meshwidth}%
    \setc@ntr@l{2}%
    \ifnum#1>\@ne\Psmeshp@rt#1[#3,#4,#5,#6]\fi%
    \ifnum#2>\@ne\Psmeshp@rt#2[#4,#5,#6,#3]\fi%
    \ifnum\c@ntrolmesh>\z@\Psmeshdi@g#1,#2[#3,#4,#5,#6]\fi%
    \ifnum\c@ntrolmesh<\z@\Psmeshdi@g#2,#1[#4,#5,#6,#3]\fi%
    \Psrest@reSt{color=\DDV@meshcolor,dash=\DDV@meshdash,width=\DDV@meshwidth}%
    \figdrawline[#3,#4,#5,#6,#3]\PSc@mment{End mesh}\resetc@ntr@l\et@tpsmesh\fi\fi}}
\ctr@ld@f\def\Psmeshp@rt#1[#2,#3,#4,#5]{{\l@mbd@un=\@ne\l@mbd@de=#1\loop%
    \ifnum\l@mbd@un<#1\advance\l@mbd@de\m@ne\figptbary-1:[#2,#3;\l@mbd@de,\l@mbd@un]%
    \figptbary-2:[#5,#4;\l@mbd@de,\l@mbd@un]\figdrawline[-1,-2]\advance\l@mbd@un\@ne\repeat}}
\ctr@ld@f\def\Psmeshdi@g#1,#2[#3,#4,#5,#6]{\figptcopy-2:/#3/\figptcopy-3:/#6/%
    \l@mbd@un=\z@\l@mbd@de=#1\loop\ifnum\l@mbd@un<#1%
    \advance\l@mbd@un\@ne\advance\l@mbd@de\m@ne\figptcopy-1:/-2/\figptcopy-4:/-3/%
    \figptbary-2:[#3,#4;\l@mbd@de,\l@mbd@un]%
    \figptbary-3:[#6,#5;\l@mbd@de,\l@mbd@un]\Psmeshdi@gp@rt#2[-1,-2,-3,-4]\repeat}
\ctr@ld@f\def\Psmeshdi@gp@rt#1[#2,#3,#4,#5]{{\l@mbd@un=\z@\l@mbd@de=#1\loop%
    \ifnum\l@mbd@un<#1\figptbary-5:[#2,#5;\l@mbd@de,\l@mbd@un]%
    \advance\l@mbd@de\m@ne\advance\l@mbd@un\@ne%
    \figptbary-6:[#3,#4;\l@mbd@de,\l@mbd@un]\figdrawline[-5,-6]\repeat}}
\ctr@ln@m\figdrawnormal
\ctr@ld@f\def\Q@normalDD#1,#2[#3,#4]{{\ifCUR@PS\ifGR@cri%
    \PSc@mment{normal Length=#1, Lambda=#2 [Pt1,Pt2]=[#3,#4]}%
    \s@uvc@ntr@l\et@tpsnormal\resetc@ntr@l{2}\figptendnormal-6::#1,#2[#3,#4]%
    \figptcopyDD-5:/-1/\figdrawarrow[-5,-6]%
    \PSc@mment{End normal}\resetc@ntr@l\et@tpsnormal\fi\fi}}
\ctr@ld@f\def\figreset#1{\trtlis@rg{#1}{\Psreset@}}
\ctr@ld@f\def\Psreset@#1|{\def\t@xt@{#1}\ifx\t@xt@\empty\P@resetg@n
    \else\keln@mde#1|%
    \def\n@mref{al}\ifx\l@debut\n@mref%
        \figset altitude(blcolor=default,bldash=default,blwidth=default,%
        sqcolor=default,sqdash=default,sqwidth=default)\else
    \def\n@mref{ar}\ifx\l@debut\n@mref\figresetarrowhead\else
    \def\n@mref{cu}\ifx\l@debut\n@mref\figset curve(roundness=\D@FTroundness)\else
    \def\n@mref{ge}\ifx\l@debut\n@mref\P@resetg@n\else
    \def\n@mref{fl}\ifx\l@debut\n@mref%
        \figset flowchart(arrowp=\D@FTfcarrowposition,arrowr=\D@FTfcarrowrefpt,%
	bgcolor=\D@FTfcbgcolor,line=\D@FTfcline,radius=\D@FTfcradius,%
	shape=\D@FTfcshape,thickcolor=default,thickness=\D@FTfcthickness,%
	xpadd=\D@FTfcxpadding,ypadd=\D@FTfcypadding)\else
    \def\n@mref{me}\ifx\l@debut\n@mref\figset mesh(diag=\D@FTmeshdiag,%
        color=default,dash=default,width=default)\else
    \def\n@mref{tr}\ifx\l@debut\n@mref%
        \figset trimesh(color=default,dash=default,width=default)\else
    \W@rnmeskwd{figreset}{#1}\fi\fi\fi\fi\fi\fi\fi\fi}
\ctr@ld@f\def\P@resetg@n{\figset (color=\D@FTcolor,dash=\D@FTdash,fill=\D@FTfill,%
    join=\D@FTjoin,width=\D@FTwidth)}
\ctr@ld@f\def\figset#1(#2){\def\t@xt@{#1}\ifx\t@xt@\empty\trtlis@rg{#2}{\Pssetg@n}
    \else\keln@mde#1|%
    \def\n@mref{al}\ifx\l@debut\n@mref\trtlis@rg{#2}{\Psset@lti}\else
    \def\n@mref{ar}\ifx\l@debut\n@mref\trtlis@rg{#2}{\Psset@rrowhe@d}\else
    \def\n@mref{cu}\ifx\l@debut\n@mref\trtlis@rg{#2}{\Pssetc@rve}\else
    \def\n@mref{fl}\ifx\l@debut\n@mref\trtlis@rg{#2}{\Pssetfl@wchart}\else
    \def\n@mref{ge}\ifx\l@debut\n@mref\trtlis@rg{#2}{\Pssetg@n}\else
    \def\n@mref{me}\ifx\l@debut\n@mref\trtlis@rg{#2}{\Pssetm@sh}\else
    \def\n@mref{pr}\ifx\l@debut\n@mref\ifCUR@PS\W@rnmesIgn{figset proj(...)}%
     \else\trtlis@rg{#2}{\Figsetpr@j}\fi\else
    \def\n@mref{tr}\ifx\l@debut\n@mref\trtlis@rg{#2}{\Pssettrim@sh}\else
    \def\n@mref{wr}\ifx\l@debut\n@mref\let\M@cro=\Figsetwr@te\trtlis@rgtok{#2,|}\else
    \W@rnmeskwd{figset}{#1}\fi\fi\fi\fi\fi\fi\fi\fi\fi\fi\ignorespaces}
\ctr@ld@f\def\figsetdefault#1(#2){\ifCUR@PS\W@rnmesIgn{figsetdefault}\else%
    \def\t@xt@{#1}\ifx\t@xt@\empty\trtlis@rg{#2}{\Pssd@g@n}\else\keln@mun#1|
    \def\n@mref{a}\ifx\l@debut\n@mref\trtlis@rg{#2}{\Pssd@@rrowhe@d}\else
    \def\n@mref{c}\ifx\l@debut\n@mref\trtlis@rg{#2}{\Pssd@c@rve}\else
    \def\n@mref{g}\ifx\l@debut\n@mref\trtlis@rg{#2}{\Pssd@g@n}\else
    \def\n@mref{f}\ifx\l@debut\n@mref\trtlis@rg{#2}{\Pssd@fl@wchart}\else
    \def\n@mref{m}\ifx\l@debut\n@mref\trtlis@rg{#2}{\Pssd@m@sh}\else
    \W@rnmeskwd{figsetdefault}{#1}\fi\fi\fi\fi\fi\fi\initpss@ttings\fi}
\ctr@ld@f\def\Pssd@g@n#1=#2|{\keln@mun#1|%
    \def\n@mref{c}\ifx\l@debut\n@mref\edef\D@FTcolor{#2}\else
    \def\n@mref{d}\ifx\l@debut\n@mref\edef\D@FTdash{#2}\else
    \def\n@mref{f}\ifx\l@debut\n@mref\edef\D@FTfill{#2}\else
    \def\n@mref{j}\ifx\l@debut\n@mref\edef\D@FTjoin{#2}\else
    \def\n@mref{u}\ifx\l@debut\n@mref\edef\D@FTupdate{#2}\Q@s@tupdate{#2}\else
    \def\n@mref{w}\ifx\l@debut\n@mref\edef\D@FTwidth{#2}\else
    \W@rnmesAttr{figsetdefault}{#1}\fi\fi\fi\fi\fi\fi}
\ctr@ld@f\def\Pssd@@rrowhe@d#1=#2|{\keln@mun#1|%
    \def\n@mref{a}\ifx\l@debut\n@mref\edef\D@FTarrowheadangle{#2}\else
    \def\n@mref{f}\ifx\l@debut\n@mref\edef\D@FTarrowheadfill{#2}\else
    \def\n@mref{l}\ifx\l@debut\n@mref\y@tiunit{#2}\ifunitpr@sent%
     \edef\D@FTh@rdahlength{#2}\else\edef\D@FTh@rdahlength{#2pt}%
     \message{*** \BS@ figsetdefault (..., #1=#2, ...) : unit is missing, pt is assumed.}%
     \fi\else
    \def\n@mref{o}\ifx\l@debut\n@mref\edef\D@FTarrowheadout{#2}\else
    \def\n@mref{r}\ifx\l@debut\n@mref\edef\D@FTarrowheadratio{#2}\else
    \W@rnmesAttr{figsetdefault arrowhead}{#1}\fi\fi\fi\fi\fi}
\ctr@ld@f\def\Pssd@c@rve#1=#2|{\keln@mun#1|%
    \def\n@mref{r}\ifx\l@debut\n@mref\edef\D@FTroundness{#2}\else%
    \W@rnmesAttr{figsetdefault curve}{#1}\fi}
\ctr@ld@f\def\Pssd@fl@wchart#1=#2|{\keln@mtr#1|%
    \def\n@mref{arr}\ifx\l@debut\n@mref\expandafter\keln@mtr\l@suite|%
     \def\n@mref{owp}\ifx\l@debut\n@mref\edef\D@FTfcarrowposition{#2}\else
     \def\n@mref{owr}\ifx\l@debut\n@mref\edef\D@FTfcarrowrefpt{#2}\else
                     \W@rnmesAttr{figsetdefault flowchart}{#1}\fi\fi\else%
    \def\n@mref{bgc}\ifx\l@debut\n@mref\edef\D@FTfcbgcolor{#2}\else
    \def\n@mref{lin}\ifx\l@debut\n@mref\edef\D@FTfcline{#2}\else
    \def\n@mref{pad}\ifx\l@debut\n@mref\edef\D@FTfcxpadding{#2}%
                    \edef\D@FTfcypadding{#2}\else
    \def\n@mref{rad}\ifx\l@debut\n@mref\edef\D@FTfcradius{#2}\else
    \def\n@mref{sha}\ifx\l@debut\n@mref\edef\D@FTfcshape{#2}\else
    \def\n@mref{thi}\ifx\l@debut\n@mref\expandafter\keln@mtr\l@suite|%
     \def\n@mref{ckn}\ifx\l@debut\n@mref\edef\D@FTfcthickness{#2}\else
                     \W@rnmesAttr{figsetdefault flowchart}{#1}\fi\else%
    \def\n@mref{xpa}\ifx\l@debut\n@mref\edef\D@FTfcxpadding{#2}\else
    \def\n@mref{ypa}\ifx\l@debut\n@mref\edef\D@FTfcypadding{#2}\else
    \W@rnmesAttr{figsetdefault flowchart}{#1}\fi\fi\fi\fi\fi\fi\fi\fi\fi}
\ctr@ld@f\def\D@FTfcarrowposition{0.5}
\ctr@ld@f\def\D@FTfcarrowrefpt{start}
\ctr@ld@f\def\D@FTfcbgcolor{1}
\ctr@ld@f\def\D@FTfcline{polygon}
\ctr@ld@f\def\D@FTfcradius{0}
\ctr@ld@f\def\D@FTfcshape{rectangle}
\ctr@ld@f\def\D@FTfcthickness{0}
\ctr@ld@f\def\D@FTfcxpadding{0}
\ctr@ld@f\def\D@FTfcypadding{0}
\ctr@ld@f\def\Pssd@m@sh#1=#2|{\keln@mun#1|%
    \def\n@mref{d}\ifx\l@debut\n@mref\edef\D@FTmeshdiag{#2}\else%
    \W@rnmesAttr{figsetdefault mesh}{#1}\fi}
\ctr@ln@w{newif}\iffillm@de
\ctr@ld@f\def\Q@s@tfillmode#1{\expandafter\setfillm@de#1:}
\ctr@ld@f\def\setfillm@de#1#2:{\if#1n\fillm@defalse\else\fillm@detrue\fi}
\ctr@ld@f\def\D@FTfill{no}     
\ctr@ln@w{newif}\ifGRupdatem@de
\ctr@ld@f\def\Q@s@tupdate#1{\ifCUR@PS\W@rnmesIgn{figset (update=...)}%
    \else\expandafter\setupd@te#1:\fi}
\ctr@ld@f\def\setupd@te#1#2:{\if#1n\GRupdatem@defalse\else\GRupdatem@detrue\fi}
\ctr@ld@f\def\D@FTupdate{no}     
\ctr@ln@m\CUR@color \ctr@ln@m\CUR@colorc@md
\ctr@ld@f\def\s@uvcolor#1{\edef#1{\CUR@color}}
\ctr@ld@f\def\D@FTcolor{0}       
\ctr@ld@f\def\Pssetc@lor#1{\ifGR@cri\result@tent=\@ne\expandafter\c@lnbV@l#1 :%
    \def\CUR@color{}\def\CUR@colorc@md{}%
    \ifcase\result@tent\or\Q@s@tgray{#1}\or\or\Q@s@trgb{#1}\or\Q@s@tcmyk{#1}\fi\fi}
\ctr@ln@m\CUR@colorc@mdStroke
\ctr@ld@f\def\Q@s@tcmyk#1{\ifGR@cri\def\CUR@color{#1}\def\CUR@colorc@md{\c@msetcmykcolor}%
    \def\CUR@colorc@mdStroke{\c@msetcmykcolorStroke}%
    \ifCUR@PS\PSc@mment{setcmyk Color=#1}\us@primarC@lor\fi\fi}
\ctr@ld@f\def\Q@s@trgb#1{\ifGR@cri\def\CUR@color{#1}\def\CUR@colorc@md{\c@msetrgbcolor}%
    \def\CUR@colorc@mdStroke{\c@msetrgbcolorStroke}%
    \ifCUR@PS\PSc@mment{setrgb Color=#1}\us@primarC@lor\fi\fi}
\ctr@ld@f\def\Q@s@tgray#1{\ifGR@cri\def\CUR@color{#1}\def\CUR@colorc@md{\c@msetgray}%
    \def\CUR@colorc@mdStroke{\c@msetgrayStroke}%
    \ifCUR@PS\PSc@mment{setgray Gray level=#1}\us@primarC@lor\fi\fi}
\ctr@ln@m\fillc@md
\ctr@ld@f\def\us@primarC@lor{\immediate\write\fwf@g{\d@fprimarC@lor}%
    \let\fillc@md=\prfillc@md}
\ctr@ld@f\def\prfillc@md{\d@fprimarC@lor\space\c@mfill}
\ctr@ld@f\def\c@lnbV@l#1 #2:{\def\t@xt@{#1}\relax\ifx\t@xt@\empty\c@lnbV@l#2:
    \else\c@lnbV@l@#1 #2:\fi}
\ctr@ld@f\def\c@lnbV@l@#1 #2:{\def\t@xt@{#2}\ifx\t@xt@\empty%
    \def\t@xt@{#1}\ifx\t@xt@\empty\advance\result@tent\m@ne\fi
    \else\advance\result@tent\@ne\c@lnbV@l@#2:\fi}
\ctr@ld@f\def\Blackcmyk{0 0 0 1}
\ctr@ld@f\def\Whitecmyk{0 0 0 0}
\ctr@ld@f\def\Cyancmyk{1 0 0 0}
\ctr@ld@f\def\Magentacmyk{0 1 0 0}
\ctr@ld@f\def\Yellowcmyk{0 0 1 0}
\ctr@ld@f\def\Redcmyk{0 1 1 0}
\ctr@ld@f\def\Greencmyk{1 0 1 0}
\ctr@ld@f\def\Bluecmyk{1 1 0 0}
\ctr@ld@f\def\Graycmyk{0 0 0 0.50}
\ctr@ld@f\def\BrickRedcmyk{0 0.89 0.94 0.28} 
\ctr@ld@f\def\Browncmyk{0 0.81 1 0.60} 
\ctr@ld@f\def\ForestGreencmyk{0.91 0 0.88 0.12} 
\ctr@ld@f\def\Goldenrodcmyk{ 0 0.10 0.84 0} 
\ctr@ld@f\def\Marooncmyk{0 0.87 0.68 0.32} 
\ctr@ld@f\def\Orangecmyk{0 0.61 0.87 0} 
\ctr@ld@f\def\Purplecmyk{0.45 0.86 0 0} 
\ctr@ld@f\def\RoyalBluecmyk{1. 0.50 0 0} 
\ctr@ld@f\def\Violetcmyk{0.79 0.88 0 0} 
\ctr@ld@f\def\Blackrgb{0 0 0}
\ctr@ld@f\def\Whitergb{1 1 1}
\ctr@ld@f\def\Redrgb{1 0 0}
\ctr@ld@f\def\Greenrgb{0 1 0}
\ctr@ld@f\def\Bluergb{0 0 1}
\ctr@ld@f\def\Cyanrgb{0 1 1}
\ctr@ld@f\def\Magentargb{1 0 1}
\ctr@ld@f\def\Yellowrgb{1 1 0}
\ctr@ld@f\def\Grayrgb{0.5 0.5 0.5}
\ctr@ld@f\def\Chocolatergb{0.824 0.412 0.118}
\ctr@ld@f\def\DarkGoldenrodrgb{0.722 0.525 0.043}
\ctr@ld@f\def\DarkOrangergb{1 0.549 0}
\ctr@ld@f\def\Firebrickrgb{0.698 0.133 0.133}
\ctr@ld@f\def\ForestGreenrgb{0.133 0.545 0.133}
\ctr@ld@f\def\Goldrgb{1 0.843 0}
\ctr@ld@f\def\HotPinkrgb{1 0.412 0.706}
\ctr@ld@f\def\Maroonrgb{0.690 0.188 0.376}
\ctr@ld@f\def\Pinkrgb{1 0.753 0.796}
\ctr@ld@f\def\RoyalBluergb{0.255 0.412 0.882}
\ctr@ld@f\def\Pssetg@n#1=#2|{\keln@mun#1|%
    \def\n@mref{c}\ifx\l@debut\n@mref\update@ttr\D@FTcolor\Pssetc@lor{#2}\else
    \def\n@mref{d}\ifx\l@debut\n@mref\update@ttr\D@FTdash\Q@s@tdash{#2}\else
    \def\n@mref{f}\ifx\l@debut\n@mref\update@ttr\D@FTfill\Q@s@tfillmode{#2}\else
    \def\n@mref{j}\ifx\l@debut\n@mref\update@ttr\D@FTjoin\Q@s@tjoin{#2}\else
    \def\n@mref{u}\ifx\l@debut\n@mref\update@ttr\D@FTupdate\Q@s@tupdate{#2}\else
    \def\n@mref{w}\ifx\l@debut\n@mref\update@ttr\D@FTwidth\Q@s@twidth{#2}\else
    \W@rnmesAttr{figset}{#1}\fi\fi\fi\fi\fi\fi}
\ctr@ln@m\CUR@dash
\ctr@ld@f\def\s@uvdash#1{\edef#1{\CUR@dash}}
\ctr@ld@f\def\D@FTdash{1}        
\ctr@ld@f\def\Q@s@tdash#1{\ifGR@cri\edef\CUR@dash{#1}\ifCUR@PS\expandafter\Pssetd@sh#1 :\fi\fi}
\ctr@ld@f\def\Pssetd@shI#1{\PSc@mment{setdash Index=#1}\ifcase#1%
    \or\immediate\write\fwf@g{[] 0 \c@msetdash}
    \or\immediate\write\fwf@g{[6 2] 0 \c@msetdash}
    \or\immediate\write\fwf@g{[4 2] 0 \c@msetdash}
    \or\immediate\write\fwf@g{[2 2] 0 \c@msetdash}
    \or\immediate\write\fwf@g{[1 2] 0 \c@msetdash}
    \or\immediate\write\fwf@g{[2 4] 0 \c@msetdash}
    \or\immediate\write\fwf@g{[3 5] 0 \c@msetdash}
    \or\immediate\write\fwf@g{[3 3] 0 \c@msetdash}
    \or\immediate\write\fwf@g{[3 5 1 5] 0 \c@msetdash}
    \or\immediate\write\fwf@g{[6 4 2 4] 0 \c@msetdash}
    \fi}
\ctr@ld@f\def\Pssetd@sh#1 #2:{{\def\t@xt@{#1}\ifx\t@xt@\empty\Pssetd@sh#2:
    \else\def\t@xt@{#2}\ifx\t@xt@\empty\Pssetd@shI{#1}\else\s@mme=\@ne\def\debutp@t{#1}%
    \an@lysd@sh#2:\ifodd\s@mme\edef\debutp@t{\debutp@t\space\finp@t}\def\finp@t{0}\fi%
    \PSc@mment{setdash Pattern=#1 #2}%
    \immediate\write\fwf@g{[\debutp@t] \finp@t\space\c@msetdash}\fi\fi}}
\ctr@ld@f\def\an@lysd@sh#1 #2:{\def\t@xt@{#2}\ifx\t@xt@\empty\def\finp@t{#1}\else%
    \edef\debutp@t{\debutp@t\space#1}\advance\s@mme\@ne\an@lysd@sh#2:\fi}
\ctr@ln@m\CUR@width
\ctr@ld@f\def\s@uvwidth#1{\edef#1{\CUR@width}}
\ctr@ld@f\def\D@FTwidth{0.4}     
\ctr@ld@f\def\Q@s@twidth#1{\ifGR@cri\edef\CUR@width{#1}\ifCUR@PS%
    \PSc@mment{setwidth Width=#1}\immediate\write\fwf@g{#1 \c@msetlinewidth}\fi\fi}
\ctr@ln@m\CUR@join
\ctr@ld@f\def\s@uvjoin#1{\edef#1{\CUR@join}}
\ctr@ld@f\def\D@FTjoin{miter}   
\ctr@ld@f\def\Q@s@tjoin#1{\ifGR@cri\edef\CUR@join{#1}\ifCUR@PS\expandafter\Pssetj@in#1:\fi\fi}
\ctr@ld@f\def\Pssetj@in#1#2:{\PSc@mment{setjoin join=#1}%
    \if#1r\def\t@xt@{1}\else\if#1b\def\t@xt@{2}\else\def\t@xt@{0}\fi\fi%
    \immediate\write\fwf@g{\t@xt@\space\c@msetlinejoin}}
\ctr@ld@f\def\Pss@tspecifSt#1{\trtlis@rg{#1}{\Pss@tspecifSt@}}
\ctr@ld@f\def\Pss@tspecifSt@#1=#2|{\keln@mun#1|%
    \def\n@mref{c}\ifx\l@debut\n@mref\def\n@mref{#2}\ifx\n@mref\D@FTref\else%
     \s@uvcolor{\typ@color}\Pssetc@lor{#2}\fi\else
    \def\n@mref{d}\ifx\l@debut\n@mref\def\n@mref{#2}\ifx\n@mref\D@FTref\else%
     \s@uvdash{\typ@dash}\Q@s@tdash{#2}\fi\else
    \def\n@mref{j}\ifx\l@debut\n@mref\def\n@mref{#2}\ifx\n@mref\D@FTref\else%
     \s@uvjoin{\typ@join}\Q@s@tjoin{#2}\fi\else
    \def\n@mref{w}\ifx\l@debut\n@mref\def\n@mref{#2}\ifx\n@mref\D@FTref\else%
     \s@uvwidth{\typ@width}\Q@s@twidth{#2}\fi\else
    \W@rnmeskwd{Pss@tspecifSt}{#1}\fi\fi\fi\fi}
\ctr@ld@f\def\Psrest@reSt#1{\trtlis@rg{#1}{\Psrest@reSt@}}
\ctr@ld@f\def\Psrest@reSt@#1=#2|{\keln@mun#1|%
    \def\n@mref{c}\ifx\l@debut\n@mref\def\n@mref{#2}\ifx\n@mref\D@FTref\else%
     \Pssetc@lor{\typ@color}\fi\else
    \def\n@mref{d}\ifx\l@debut\n@mref\def\n@mref{#2}\ifx\n@mref\D@FTref\else%
     \Q@s@tdash{\typ@dash}\fi\else
    \def\n@mref{j}\ifx\l@debut\n@mref\def\n@mref{#2}\ifx\n@mref\D@FTref\else%
     \Q@s@tjoin{\typ@join}\fi\else
    \def\n@mref{w}\ifx\l@debut\n@mref\def\n@mref{#2}\ifx\n@mref\D@FTref\else%
     \Q@s@twidth{\typ@width}\fi\else
    \W@rnmeskwd{Psrest@reSt}{#1}\fi\fi\fi\fi}
\ctr@ld@f\def\Pssettrim@sh#1=#2|{\keln@mde#1|%
    \def\n@mref{co}\ifx\l@debut\n@mref\update@ttr\D@FTref\P@settmeshcolor{#2}\else
    \def\n@mref{da}\ifx\l@debut\n@mref\update@ttr\D@FTref\P@settmeshdash{#2}\else
    \def\n@mref{wi}\ifx\l@debut\n@mref\update@ttr\D@FTref\P@settmeshwidth{#2}\else
    \W@rnmesAttr{figset trimesh}{#1}\fi\fi\fi}
\ctr@ln@m\DDV@tmeshcolor
\ctr@ld@f\def\P@settmeshcolor#1{\edef\DDV@tmeshcolor{#1}}
\ctr@ln@m\DDV@tmeshdash
\ctr@ld@f\def\P@settmeshdash#1{\edef\DDV@tmeshdash{#1}}
\ctr@ln@m\DDV@tmeshwidth
\ctr@ld@f\def\P@settmeshwidth#1{\edef\DDV@tmeshwidth{#1}}
\ctr@ld@f\def\figdrawtrimesh#1[#2,#3,#4]{{\ifCUR@PS\ifGR@cri%
    \PSc@mment{trimesh Type=#1, Triangle=[#2,#3,#4]}%
    \s@uvc@ntr@l\et@tpstrimesh\ifnum#1>\@ne%
    \Pss@tspecifSt{color=\DDV@tmeshcolor,dash=\DDV@tmeshdash,width=\DDV@tmeshwidth}%
    \setc@ntr@l{2}%
    \Pstrimeshp@rt#1[#2,#3,#4]\Pstrimeshp@rt#1[#3,#4,#2]\Pstrimeshp@rt#1[#4,#2,#3]%
    \Psrest@reSt{color=\DDV@tmeshcolor,dash=\DDV@tmeshdash,width=\DDV@tmeshwidth}%
    \fi\figdrawline[#2,#3,#4,#2]%
    \PSc@mment{End trimesh}\resetc@ntr@l\et@tpstrimesh\fi\fi}}
\ctr@ld@f\def\Pstrimeshp@rt#1[#2,#3,#4]{{\l@mbd@un=\@ne\l@mbd@de=#1\loop\ifnum\l@mbd@de>\@ne%
    \advance\l@mbd@de\m@ne\figptbary-1:[#2,#3;\l@mbd@de,\l@mbd@un]%
    \figptbary-2:[#2,#4;\l@mbd@de,\l@mbd@un]\figdrawline[-1,-2]%
    \advance\l@mbd@un\@ne\repeat}}
\initpr@lim\initpss@ttings\initPDF@rDVI
\ctr@ln@w{newbox}\figBoxA
\ctr@ln@w{newbox}\figBoxB
\ctr@ln@w{newbox}\figBoxC
\catcode`\@=12

%% file: fig1-2d.tex
\figinit{1cm}
\def\Lmax{3}
\figpt 0:(0,0)
\figpt 1:(0,\Lmax)
\figpt 2:(\Lmax,0)
\figpt 4:(-1,-1)
\figpt 5:(-1,\Lmax)
\figpt 6:(\Lmax,-1)

\figdrawbegin{}
\figset (color=\GrisClair,fill=yes)
\figdrawline [5,4,6,2,0,1]
\figset (color=default,fill=no)
\figdrawline [1,0,2]
\figdrawline [5,4,6]
\figset (fill=yes)
\figdrawcirc 0 (\RadBullet)
\figdrawcirc 4 (\RadBullet)
\figset (fill=no)
\figdrawend

\figvisu{\figBoxA}{}{%
\figptbary 7:[2,6; 1,1]\figwritew 7:$\Gamma$(0.5)
\figwrites 4:$\figcoord{0}$(0.25)
\figwritene 0:$\figcoord{0}$(0.25)
}
\centerline{\box\figBoxA}

%% file: fig1-3d.tex

\def\drawpatch#1[#2]{%
\figset (color=#1,fill=yes)\figdrawline[#2]
\figset (color=default,fill=no)\figdrawline[#2]}

\figinit{4cm,realistic}
\def\ep{0.2}
\figpt 1:(1,1)
\figpt 2:(0,1)
\figpt 3:(0,0)
\figpt 4:(1,0)
\figpt 5:(0,1,1)
\figpt 6:(0,0,1)
\figpt 7:(1,0,1)
\figpt 11:(1,1,\ep)
\figpt 12:(\ep,1,\ep)
\figpt 13:(\ep,\ep,\ep)
\figpt 14:(1,\ep,\ep)
\figpt 15:(\ep,1,1)
\figpt 16:(\ep,\ep,1)
\figpt 17:(1,\ep,1)
%
\figset projection(longitude=27, latitude=20, dist=7)
\figdrawbegin{}
\drawpatch 1[16,13,14,17,16]\drawpatch 1[16,13,12,15,16]\drawpatch 1[12,13,14,11,12]
\figset (dash=4)
\drawpatch \GrisClair[7,17,14,11,1,4,7]
\drawpatch 0.6[15,5,2,1,11,12,15]
\drawpatch 0.5[6,5,15,16,17,7,6]
\figdrawend
%
\figvisu{\figBoxA}{}{
\figptbary 0:[11,14; 1,1]
\figwrites 0:{$\Lambda$} (0.04)
\figset write (mark=$\figBullet$)
\figwritene 13:{$(0,0,0)$} (0.04)
}
\centerline{\box\figBoxA}

%% file: fig2-round.tex
\figinit{1.2cm}
\def\Lmax{3}
\figpt 0:(0,0)
\figpt 1:(0,\Lmax)
\figpt 2:(\Lmax,0)
\figpt 5:(-1,\Lmax)
\figpt 6:(\Lmax,-1)

\figpt 20:(-0.5,-0.5)
\figpt 21:(-0.5,\Lmax)
\figpt 22:(\Lmax,-0.5)
\figpt 25:(-1,-0.5)
\figpt 26:(-0.5,-1)

\figdrawbegin{}
\figdrawline [1,0,2]
\figdrawline [25,5]
\figdrawline [26,6]
\figdrawarccirc 20 ; 0.5 (180,270)
\figset (dash=4)
\figdrawline [20,21]
\figdrawline [20,22]
\figptbary 31:[20,21;1,1]\figpttraC 31:=31/0.1,0/
\figptbary 30:[1,2;3,1]
\figptbary 33:[30,31;1,1]
\figptrot 32:=33/30,30/
\figptrot 33:=33/31,-30/
\figdrawarrowBezier [30,32,33,31]
\figset (fill=yes)
\figdrawcirc 0 (\RadBullet)
\figdrawcirc 20 (\RadBullet)
\figset (fill=no)
\figdrawend

\figvisu{\figBoxA}{}{
\figwritenw 22:$\sGf$(0.2)
\figwritege 30:$\sSf$(-0.08,0.1)
\figwritene 0:$\figcoord{0}$(0.05)
\figwritege 20:{$(-\tfrac12,-\tfrac12)$} (-0.3,-0.03)
}
\centerline{\box\figBoxA}

%% file: fig2-roundsharp.tex
\figinit{1.2cm}
\def\Lmax{3}
\figpt 0:(0,0)
\figpt 1:(0,\Lmax)
\figpt 2:(\Lmax,0)
\figpt 5:(-1,\Lmax)
\figpt 6:(\Lmax,-1)

\figpt 20:(-0.5,-0.5)
\figpt 21:(-0.5,\Lmax)
\figpt 22:(\Lmax,-0.5)
\figpt 23:(-0.5,0)
\figpt 24:(0,-0.5)
\figpt 25:(-1,0)
\figpt 26:(0,-1)

\figpt 27:(-1,-1)

\figdrawbegin{}
\figdrawline [1,0,2]
\figdrawline [25,5]
\figdrawline [26,6]
\figdrawarccirc 0 ; 1 (180,270)
\figset (dash=4)
\figdrawline [23,21]
\figdrawline [24,22]
\figdrawarccirc 0 ; 0.5 (180,270)
\figptbary 31:[20,21;1,1]\figpttraC 31:=31/0.1,0/
\figptbary 30:[1,2;3,1]
\figptbary 33:[30,31;1,1]
\figptrot 32:=33/30,30/
\figptrot 33:=33/31,-30/
\figdrawarrowBezier [30,32,33,31]
\figset (fill=yes)
\figdrawcirc 0 (\RadBullet)
\figdrawcirc 27 (\RadBullet)
\figset (fill=no)
\figdrawend

\figvisu{\figBoxA}{}{
\figwritenw 22:$\sGs$(0.2)
\figwritege 30:$\sSs$(-0.08,0.1)
\figwritene 0:$\figcoord{0}$(0.05)
\figwritenw 27:$\figcoord{0}$(0.05)
}
\centerline{\box\figBoxA}

%% file: fig3-cross.tex
\figinit{1cm}
\def\Lmin{0.5}\def\Lmax{3.5}
\figpt 0:(0,0)\figpt 101:(1,0)\figpt 102:(0,1)
\figpt 1:(\Lmax,\Lmin)
\figpt 2:(\Lmin,\Lmin)
\figpt 3:(\Lmin,\Lmax)
\figptssym 4=1,2,3/0,102/
\figptssym 7=1,2,3,4,5,6/0,101/

\figdrawbegin{}
\figset (color=\GrisClair,fill=yes)
\figdrawline [3,6,12,9]
\figdrawline [4,10,7,1]
\figset (color=default,fill=no)
\figdrawline [1,2,3]
\figdrawline [4,5,6]
\figdrawline [7,8,9]
\figdrawline [10,11,12]
\figdrawend

\figvisu{\figBoxA}{}{
\figwritec [0]{$\sX$}
}
\centerline{\box\figBoxA}

%% file: fig3-murs.tex

\def\drawpatch#1[#2]{%
\figset (color=#1,fill=yes)\figdrawline[#2]
\figset (color=default,fill=no)\figdrawline[#2]}

\figinit{3cm,realistic}
\def\ep{0.1}
\figpt 0:(0,0)
\figpt 1:(1,1)
\figpt 2:(-1,1)
\figpt 3:(-1,-1)
\figpt 4:(1,-1)
\figpt 71:(1,\ep,\ep)
\figpt 72:(1,-\ep,\ep)
\figpt 73:(1,-\ep,-\ep)
\figpt 74:(1,\ep,-\ep)
\figpt 81:(\ep,1,\ep)
\figpt 82:(-\ep,1,\ep)
\figpt 83:(-\ep,1,-\ep)
\figpt 84:(\ep,1,-\ep)
\figpt 91:(\ep,\ep,1)
\figpt 92:(-\ep,\ep,1)
\figpt 93:(-\ep,-\ep,1)
\figpt 94:(\ep,-\ep,1)
%
\figset projection(longitude=27, latitude=20, dist=14)
\figdrawbegin{}
\figvectC 103 (0,0,1)
\figptstra 11=1,2,3,4/\ep,103/
\figptstra 21=1,2,3,4/-\ep,103/
\figdrawline[12,13]
\figvectC 101(1,0,0)
\figptsrot 31=11,12,13,14/0,90,101/
\figptsrot 41=21,22,23,24/0,90,101/
\figdrawline[42,43]
\figvectC 102(0,1,0)
\figptsrot 51=11,12,13,14/0,90,102/
\figptsrot 61=21,22,23,24/0,90,102/
\figdrawline[53,54,51]
%
\figpt 16:(\ep,-\ep,\ep)
\figpt 17:(\ep,-1,\ep)
\drawpatch 1[72,14,17,16]
\figpt 45:(\ep,\ep,-\ep)
\figpt 46:(\ep,\ep,-1)
\drawpatch 1[74,44,46,45]
\figpt 15:(\ep,\ep,\ep)
\drawpatch 1[91,15,71,41,91]\drawpatch 1[91,15,81,52,91]\drawpatch 1[81,15,71,11,81]
\figset (dash=4)
\drawpatch \GrisClair[31,41,71,11,21,74,44,34,73,24,14,72,31]
\drawpatch 0.6[52,62,82,12,22,83,61,51,84,21,11,81,52]
\drawpatch 0.5[53,63,93,32,42,92,62,52,91,41,31,94,53]
\figdrawend
%
\figvisu{\figBoxA}{}{
\figptbary 0:[11,71; 1,1]\figwrites 0:{$\sY$}(0.04)
}
\centerline{\box\figBoxA}

%% file: fig4-1.tex
\figinit{1cm}
\def\Lmax{4}
\figpt 0:(0,0)
\figpt 1:(0,\Lmax)
\figpt 2:(\Lmax,0)
\figpt 4:(-1,-1)
\figpt 5:(-1,\Lmax)
\figpt 6:(\Lmax,-1)

\figdrawbegin{}
\figdrawline [1,0,2]
\figdrawline [5,4,6]
\figset (dash=8)
\figdrawline [4,0]
\figset (fill=yes)
\figdrawcirc 0 (\RadBullet)
\figdrawcirc 4 (\RadBullet)
\figset (fill=no)
\figdrawend

\figvisu{\figBoxA}{}{%
\figptbary 7:[0,6; 1,1]\figwritee 7:{$\Gamma^1$}(0)
\figptbary 7:[0,5; 1,1]\figwriten 7:{$\Gamma^2$}(0)
\figwrites 4:$\figcoord{0}$ (0.25)
\figwritene 0:$\figcoord{0}$ (0.25)
}
\centerline{\box\figBoxA}

%% file: fig4-2.tex
\figinit{1cm}
\def\Lmax{4}
\figpt 0:(0,0)
\figpt 1:(0,\Lmax)
\figpt 2:(\Lmax,0)
\figpt 4:(-1,-1)
\figpt 5:(-1,\Lmax)
\figpt 6:(\Lmax,-1)
\def\Lmed{2.5}
\figpt 13:(\Lmed,-1)
\figpt 14:(\Lmed,0)
\figpt 23:(-1,\Lmed)
\figpt 24:(0,\Lmed)

\figdrawbegin{}
\figset (fillmode=yes, color=\GrisClair)
\figdrawline[4,13,14,0,4]
\figdrawline[4,23,24,0,4]
\figset (fillmode=no, color=default)
\figdrawline [1,0,2]
\figdrawline [5,4,6]
\figset (width=1.5)
\figdrawline [13,14]
\figdrawline [23,24]
\figset (width=0.5,dash=8)
\figdrawline [4,0]
\figset (fill=yes)
\figdrawcirc 0 (\RadBullet)
\figdrawcirc 4 (\RadBullet)
\figdrawcirc 13 (\RadBullet)
\figdrawcirc 23 (\RadBullet)
\figset (fill=no)
\figdrawend

\figvisu{\figBoxA}{}{%
\figptbary 7:[0,13; 1,1]\figwritee 7:{$\Gamma_R^1$}(0)
\figptbary 7:[2,6; 1,1]\figwritew 7:{$\cS^1_R$}(0.2)
\figwriten 14:{$\Sigma_R^1$}(0.2)
\figptbary 7:[0,23; 1,1]\figwritec [7]{$\Gamma_R^2$}
\figptbary 7:[1,5; 1,1]\figwrites 7:{$\cS^2_R$}(0.2)
\figwritee 24:{$\Sigma_R^2$}(0.2)
\figwrites 4:$\figcoord{0}$ (0.25)
\figwritene 0:$\figcoord{0}$ (0.25)
\figwrites 13:{$(R,-1)$} (0.25)
\figwritew 23:{$(-1,R)$} (0.15)
}
\centerline{\box\figBoxA}

%% file: fig5.tex
\figinit{1.2cm}
\def\Lmax{3}
\figpt 0:(0,0)
\figpt 1:(0,\Lmax)
\figpt 2:(\Lmax,0)
\figpt 4:(-1,-1)
\figpt 5:(-1,\Lmax)
\figpt 6:(\Lmax,-1)

\figpt 11:(0,-1)
\figpt 12:(-1,0)

\figdrawbegin{}
\figset (fillmode=yes, color=\GrisClair)
\figdrawline[0,11,6,2,0]
\figdrawline[0,12,5,1,0]
\figset (fillmode=no,color=default)
\figdrawline [0,1]
\figdrawline [1,0,2,6,4,5,1]
\figset (dash=8)
\figdrawline [11,0]
\figdrawline [12,0]
\figset (fill=yes)
\figdrawcirc 0 (\RadBullet)
\figdrawcirc 4 (\RadBullet)
\figdrawcirc 5 (\RadBullet)
\figdrawcirc 6 (\RadBullet)
\figset (fill=no)
\figdrawend

\figvisu{\figBoxA}{}{%
\figptbary 7:[0,6; 1,1]\figwritee 7:{$\cT_R^1$}(0)
\figptbary 7:[0,5; 1,1]\figwriten 7:{$\cT_R^2$}(0)
\figptbary 7:[0,4; 1,1]\figwritec[7]{$\Gamma_0$}
\figwrites 4:$\figcoord{0}$ (0.25)
\figwritene 0:$\figcoord{0}$ (0.25)
\figwritew 5:{$(-1,R)$} (0.1)
\figwrites 6:{$(R,-1)$} (0.2)
}
\centerline{\box\figBoxA}

%% file: Fig-Lambda3.tex

\def\drawpatch#1[#2]{%
\figset (color=#1,fill=yes)\figdrawline[#2]
\figset (color=default,fill=no)\figdrawline[#2]}

\def\LambdaTroisInt#1#2#3#4{%
\figinit{6mm,realistic}
\def\ep{1}
\figpt 1:(3,0,0)
\figpt 2:(0,4,0)
\figpt 8:(0,0,5.4)

\figpt 3:(0,0,0)
\figpt 5:(0,#4,#4)
\figpt 6:(0,0,#4)
\figpt 7:(#4,0,#4)
\figpt 13:(\ep,\ep,\ep)
\figpt 15:(\ep,#4,#4)
\figpt 16:(\ep,\ep,#4)
\figpt 17:(#4,\ep,#4)
%
\figset projection(longitude=#1, latitude=#2, dist=#3)
\figdrawbegin{}
\figset (join=round)
\drawpatch 0.6[3,5,15,13,3]
\drawpatch 0.7[3,7,17,13,3]
\drawpatch 1[16,13,17]
\drawpatch 1[16,13,15]
\figset (dash=4)
\drawpatch \GrisClair[6,5,15,16,17,7,6]
\figdrawend
%
\figvisu{\figBoxA}{}{}
\centerline{\box\figBoxA}
}

\def\LambdaTroisExt#1#2#3#4{%
\figinit{6mm,realistic}
\def\ep{1}
\figpt 3:(0,0,0)
\figpt 5:(0,#4,#4)
\figpt 6:(0,0,#4)
\figpt 7:(#4,0,#4)
\figpt 8:(0,0,6.4)
\figpt 13:(\ep,\ep,\ep)
\figpt 15:(\ep,#4,#4)
\figpt 16:(\ep,\ep,#4)
\figpt 17:(#4,\ep,#4)
%
\figset projection(longitude=#1, latitude=#2, dist=#3)
\figdrawbegin{}
\figset (join=round)
\drawpatch 1[6,3,7]
\drawpatch 1[6,3,5]
\figset (dash=4)
\drawpatch \GrisClair[6,5,15,16,17,7,6]
\figdrawend
%
\figvisu{\figBoxA}{}{}
\centerline{\box\figBoxA}
}

%% file: FicheraLayer.bbl
\begin{thebibliography}{10}

\bibitem{ApelMehrmannWatkins02}
{\sc T.~Apel, V.~Mehrmann, and D.~Watkins}, {\em Structured eigenvalue methods
  for the computation of corner singularities in 3{D} anisotropic elastic
  structures}, Comput. Methods Appl. Mech. Engrg., 191 (2002), pp.~4459--4473.

\bibitem{ApelSW1996}
{\sc T.~Apel, A.-M. S\"andig, and J.~R. Whiteman}, {\em Graded mesh refinement
  and error estimates for finite element solutions of elliptic boundary value
  problems in non-smooth domains}, Math. Methods Appl. Sci., 19 (1996),
  pp.~63--85.

\bibitem{Bar52}
{\sc V.~Bargmann}, {\em On the number of bound states in a central field of
  force}, Proc. Nat. Acad. Sci. U. S. A., 38 (1952), pp.~961--966.

\bibitem{BEL14}
{\sc J.~Behrndt, P.~Exner, and V.~Lotoreichik}, {\em Schr\"odinger operators
  with {$\delta$}-interactions supported on conical surfaces}, J. Phys. A, 47
  (2014), pp.~355202, 16.

\bibitem{Besson1985}
{\sc G.~Besson}, {\em Comportement asymptotique des valeurs propres du
  laplacien dans un domaine avec un trou}, Bull. Soc. Math. France, 113 (1985),
  pp.~211--230.

\bibitem{BS87}
{\sc M.~S. Birman and M.~Z. Solomjak}, {\em Spectral theory of selfadjoint
  operators in {H}ilbert space}, Mathematics and its Applications (Soviet
  Series), D. Reidel Publishing Co., Dordrecht, 1987.
\newblock Translated from the 1980 Russian original by S. Khrushch\"ev and V.
  Peller.

\bibitem{BDP14}
{\sc V.~{Bonnaillie-No{\"e}l}, M.~{Dauge}, and N.~{Popoff}}, {\em {Ground state
  energy of the magnetic Laplacian on general three-dimensional corner
  domains}}, M\'emoires de la SMF, 145 (2016), pp.~viii + 138.

\bibitem{BPP17}
{\sc V.~Bruneau, K.~Pankrashkin, and N.~Popoff}, {\em Eigenvalue counting
  function for robin laplacians on conical domains}, The Journal of Geometric
  Analysis,  (2017), pp.~1--29.

\bibitem{BP16}
{\sc V.~Bruneau and N.~Popoff}, {\em On the negative spectrum of the {R}obin
  {L}aplacian in corner domains}, Anal. PDE, 9 (2016), pp.~1259--1283.

\bibitem{CEK04}
{\sc G.~Carron, P.~Exner, and D.~Krej\v{c}i\v{r}\'ik}, {\em Topologically
  nontrivial quantum layers}, J. Math. Phys., 45 (2004), pp.~774--784.

\bibitem{CDS81}
{\sc J.~M. Combes, P.~Duclos, and R.~Seiler}, {\em The Born-Oppenheimer
  Approximation}, Springer US, Boston, MA, 1981, pp.~185--213.

\bibitem{CoDa2003}
{\sc M.~Costabel and M.~Dauge}, {\em Computation of resonance frequencies for
  {M}axwell equations in non-smooth domains}, in Topics in computational wave
  propagation, vol.~31 of Lect. Notes Comput. Sci. Eng., Springer, Berlin,
  2003, pp.~125--161.

\bibitem{DaBook1988}
{\sc M.~Dauge}, {\em Elliptic boundary value problems on corner domains},
  vol.~1341 of Lecture Notes in Mathematics, Springer-Verlag, Berlin, 1988.
\newblock Smoothness and asymptotics of solutions.

\bibitem{DH93-1}
{\sc M.~Dauge and B.~Helffer}, {\em Eigenvalues variation. {I}. {N}eumann
  problem for {S}turm-{L}iouville operators}, J. Differential Equations, 104
  (1993), pp.~243--262.

\bibitem{DaLaRa11}
{\sc M.~Dauge, Y.~Lafranche, and N.~Raymond}, {\em Quantum waveguides with
  corners}, in Congr\`es {N}ational de {M}ath\'ematiques {A}ppliqu\'ees et
  {I}ndustrielles, vol.~35 of ESAIM Proc., EDP Sci., Les Ulis, 2011,
  pp.~14--45.

\bibitem{DOR15}
{\sc M.~Dauge, T.~Ourmi\`eres-Bonafos, and N.~Raymond}, {\em Spectral
  asymptotics of the {D}irichlet {L}aplacian in a conical layer}, Commun. Pure
  Appl. Anal., 14 (2015), pp.~1239--1258.

\bibitem{DE95}
{\sc P.~Duclos and P.~Exner}, {\em Curvature-induced bound states in quantum
  waveguides in two and three dimensions}, Rev. Math. Phys., 7 (1995),
  pp.~73--102.

\bibitem{DEK01}
{\sc P.~Duclos, P.~Exner, and D.~Krej\v{c}i{\v{r}}\'{\i}k}, {\em Bound states
  in curved quantum layers}, Comm. Math. Phys., 223 (2001), pp.~13--28.

\bibitem{ESS89}
{\sc P.~Exner, P.~{\v{S}}eba, and P.~{\v{S}}t'ovi{\v{c}}ek}, {\em On existence
  of a bound state in an l-shaped waveguide}, Czechoslovak Journal of Physics
  B, 39 (1989), pp.~1181--1191.

\bibitem{ET10}
{\sc P.~Exner and M.~Tater}, {\em Spectrum of {D}irichlet {L}aplacian in a
  conical layer}, J. Phys. A, 43 (2010), pp.~474023, 11.

\bibitem{GR09}
{\sc C.~Geuzaine and J.-F. Remacle}, {\em Gmsh: a three-dimensional finite
  element mesh generator with built-in pre- and post-processing facilities},
  International Journal for Numerical Methods in Engineering, 79(11) (2009),
  pp.~1309--1331.

\bibitem{GJ92}
{\sc J.~Goldstone and R.~L. Jaffe}, {\em Bound states in twisting tubes}, Phys.
  Rev. B, 45 (1992), pp.~14100--14107.

\bibitem{Kat95}
{\sc T.~Kato}, {\em Perturbation theory for linear operators}, Classics in
  Mathematics, Springer-Verlag, Berlin, 1995.
\newblock Reprint of the 1980 edition.

\bibitem{KMSW92}
{\sc M.~Klein, A.~Martinez, R.~Seiler, and X.~P. Wang}, {\em On the
  {B}orn-{O}ppenheimer expansion for polyatomic molecules}, Comm. Math. Phys.,
  143 (1992), pp.~607--639.

\bibitem{LOB16}
{\sc V.~Lotoreichik and T.~Ourmi\`eres-Bonafos}, {\em On the bound states of
  {S}chr\"odinger operators with {$\delta$}-interactions on conical surfaces},
  Comm. Partial Differential Equations, 41 (2016), pp.~999--1028.

\bibitem{Mar89}
{\sc A.~Martinez}, {\em D\'eveloppements asymptotiques et effet tunnel dans
  l'approximation de {B}orn-{O}ppenheimer}, Ann. Inst. H. Poincar\'e Phys.
  Th\'eor., 50 (1989), pp.~239--257.

\bibitem{MazNazPla2000}
{\sc V.~Maz'ya, S.~Nazarov, and B.~Plamenevskij}, {\em Asymptotic theory of
  elliptic boundary value problems in singularly perturbed domains. {V}ol.
  {I}}, vol.~111 of Operator Theory: Advances and Applications, Birkh\"auser
  Verlag, Basel, 2000.
\newblock Translated from the German by Georg Heinig and Christian Posthoff.

\bibitem{NazSha14}
{\sc S.~A. Nazarov and A.~V. Shanin}, {\em Trapped modes in angular joints of
  2{D} waveguides}, Appl. Anal., 93 (2014), pp.~572--582.

\bibitem{OBP16}
{\sc T.~Ourmi{\`e}res-Bonafos and K.~Pankrashkin}, {\em Discrete spectrum of
  interactions concentrated near conical surfaces}, Applicable Analysis, doi
  \href{https://doi.org/10.1080/00036811.2017.1325472}{10.1080/00036811.2017.1325472}
  (2017), pp.~1--22.

\bibitem{P16}
{\sc K.~Pankrashkin}, {\em On the discrete spectrum of {R}obin {L}aplacians in
  conical domains}, Math. Model. Nat. Phenom., 11 (2016), pp.~100--110.

\bibitem{Pa17}
{\sc K.~Pankrashkin}, {\em Eigenvalue inequalities and absence of threshold
  resonances for waveguide junctions}, J. Math. Anal. Appl., 449 (2017),
  pp.~907--925.

\bibitem{RS78}
{\sc M.~Reed and B.~Simon}, {\em Methods of modern mathematical physics. {IV}.
  {A}nalysis of operators}, Academic Press, New York, 1978.

\bibitem{Xlifepp}
{\sc Xlife++}, {\em eXtended Library of Finite Elements in C++}, available at
  \href{http://uma.ensta-paristech.fr/soft/XLiFE++/}{http://uma.ensta-paristech.fr/soft/XLiFE++/}.

\end{thebibliography}
